\documentclass[final,onefignum,onetabnum]{siamart171218}
\usepackage{braket,amsfonts}
\usepackage{amsmath}
\usepackage{amssymb}
\usepackage{graphicx}
\usepackage{placeins}
\usepackage{longtable}
\usepackage{alltt}
\usepackage{algorithm}
	\usepackage{pdflscape}
\usepackage[utf8]{inputenc}
\usepackage{babel}
\usepackage{rotating}
\usepackage[font=small,labelfont=bf]{caption}
\usepackage{array}
\usepackage{cancel}
\usepackage{longtable}

\usepackage[caption=false]{subfig}
\usepackage{pgfplots}

\allowdisplaybreaks

\newcommand{\tnorm}[1]{{\left\vert\kern-0.25ex\left\vert\kern-0.25ex\left\vert #1 
		\right\vert\kern-0.25ex\right\vert\kern-0.25ex\right\vert}}

\newsiamthm{claim}{Claim}
\newsiamremark{remark}{Remark}
\newsiamremark{hypothesis}{Hypothesis}
\crefname{hypothesis}{Hypothesis}{Hypotheses}

\usepackage{algorithmic}

\usepackage{graphicx,epstopdf}

\crefname{ALC@unique}{Line}{Lines}

\usepackage{amsopn}

\usepackage{xspace}
\usepackage{bold-extra}
\usepackage[most]{tcolorbox}

\colorlet{texcscolor}{blue!50!black}
\colorlet{texemcolor}{red!70!black}
\colorlet{texpreamble}{red!70!black}
\colorlet{codebackground}{black!25!white!25}

\allowdisplaybreaks

\newcommand{\R}{{\mathbb R}}

\def\epsilon{\varepsilon}

\def\s{\mathbf{s}}
\def\t{\mathbf{t}}
\definecolor{red}{rgb}{0,0,0}
\renewcommand{\v}[1]{\underline{#1}}
\usepackage{showlabels}
\newcounter{ind}

\def\undertilde#1{{\baselineskip=0pt\vtop
		{\hbox{$#1$}\hbox{$\scriptscriptstyle\sim$}}}{}}

\def\epsilon{\varepsilon}

\def\undertilde#1{{\baselineskip=0pt\vtop
		{\hbox{$#1$}\hbox{$\scriptscriptstyle\sim$}}}{}}

\def\nabr1{\undertilde{\nabla}_{r_1}\,}
\def\nabr2{\undertilde{\nabla}_{r_2}\,}

\renewcommand{\omega}{\kappa}
\def\W{\mathrm{W}}
\def\N{\mathbb{N}}

\def\f{\mathbf{f}}
\def\u{\mathbf{u}}
\def\v{\mathbf{v}}
\def\w{\mathbf{w}}
\def\z{\mathbf{z}}
\def\P{\mathbb{P}}
\def\y{\mathbf{y}}
\def\n{\mathbf{n}}
\def\r{\mathbf{r}}
\def\V{\boldsymbol{V}}
\def\W{\boldsymbol{W}}

\def\x{x}

\lstdefinestyle{siamlatex}{%
	style=tcblatex,
	texcsstyle=*\color{texcscolor},
	texcsstyle=[2]\color{texemcolor},
	keywordstyle=[2]\color{texemcolor},
	moretexcs={cref,cref,maketitle,mathcal,text,headers,email,url},
}

\tcbset{%
	colframe=black!75!white!75,
	coltitle=white,
	colback=codebackground, 
	colbacklower=white, 
	fonttitle=\bfseries,
	arc=0pt,outer arc=0pt,
	top=1pt,bottom=1pt,left=1mm,right=1mm,middle=1mm,boxsep=1mm,
	leftrule=0.3mm,rightrule=0.3mm,toprule=0.3mm,bottomrule=0.3mm,
	listing options={style=siamlatex}
}

\newtcblisting[use counter=example]{example}[2][]{%
	title={Example~\thetcbcounter: #2},#1}

\newtcbinputlisting[use counter=example]{\examplefile}[3][]{%
	title={Example~\thetcbcounter: #2},listing file={#3},#1}

\DeclareTotalTCBox{\code}{ v O{} }
{ 
	fontupper=\ttfamily\color{black},
	nobeforeafter,
	tcbox raise base,
	colback=codebackground,colframe=white,
	top=0pt,bottom=0pt,left=0mm,right=0mm,
	leftrule=0pt,rightrule=0pt,toprule=0mm,bottomrule=0mm,
	boxsep=0.5mm,
	#2}{#1}

\makeatletter
\newcommand*{\bdiv}{%
	\nonscript\mskip-\medmuskip\mkern5mu%
	\mathbin{\operator@font div}\penalty900\mkern5mu%
	\nonscript\mskip-\medmuskip
}
\makeatother

\patchcmd\newpage{\vfil}{}{}{}
\begin{tcbverbatimwrite}{tmp_\jobname_header.tex}
	\title{CONFORMING/NON-CONFORMING MIXED FINITE ELEMENT METHODS FOR OPTIMAL CONTROL OF VELOCITY-VORTICITY-PRESSURE FORMULATION FOR THE OSEEN PROBLEM WITH VARIABLE VISCOSITY}
\end{tcbverbatimwrite}
	\title{CONFORMING/NON-CONFORMING MIXED FINITE ELEMENT METHODS FOR OPTIMAL CONTROL OF VELOCITY-VORTICITY-PRESSURE FORMULATION FOR THE OSEEN PROBLEM WITH VARIABLE VISCOSITY}

\author{Harpal Singh
	\thanks{Department of Mathematics, Indian Institute of Technology Roorkee, Roorkee 247667, India. E-mail: {\tt harpal\_s@ma.iitr.ac.in}. HS is supported by Council of Scientific and Industrial Research (CSIR), India (Award no. 09/0143(16048)/2022-EMR-I).}
	\and Arbaz Khan 
	\thanks{Department of Mathematics, Indian Institute of Technology Roorkee, Roorkee 247667, India. Corresponding author. E-mail: {\tt arbaz@ma.iitr.ac.in}.}
}
\ifpdf
\hypersetup{pdftitle={New \LaTeX\ Style} }
\fi

\begin{document}
	\maketitle
	\date{\today}
	\begin{abstract}
	This work examines the distributed optimal control of generalized Oseen equations with non-constant viscosity. We propose and analyze a new conforming augmented mixed finite element method and a Discontinuous Galerkin (DG) method for the velocity-vorticity-pressure formulation. The continuous formulation, which incorporates least-squares terms from both the constitutive equation and the incompressibility condition, is well-posed under certain assumptions on the viscosity parameter. The CG method is divergence-conforming and suits any Stokes inf-sup stable velocity-pressure finite element pair, while a generic discrete space approximates vorticity. The DG scheme employs a stabilization technique, and a piecewise constant discretization estimates the control variable. We establish optimal a priori and residual-based a posteriori error estimates for the proposed schemes. Finally, we provide numerical experiments to showcase the method's performance and effectiveness. 
	\end{abstract}
\section{Introduction} \label{Introduction.} \setcounter{equation}{0}
When fluid is injected through a narrow opening into a container already filled with fluid, it creates intense shear zones within the flow. These zones can cause turbulent bursts in areas with high vorticity (the curl of the velocity). However, if the injected fluid has high viscosity, the turbulence can be reduced. One way to prevent turbulence is to dynamically control fluid injection at another boundary location. By adjusting specific flow parameters, we can lower the vorticity levels within the domain, thereby reducing the chances of turbulence forming. The incompressible flow equations in vorticity formulation are important for describing rotational flows naturally \cite{SCG}. Controlling viscous flows to achieve desired physical characteristics of the fluid is crucial for many scientific and engineering applications. This evolution has become a key focus in computational fluid dynamics and various scientific fields. Flow manipulation has many applications, such as controlling turbulence in wall flows \cite{TURB}, calculating boundary temperature in thermally convected flows, preventing flood flows through dam water gates \cite{DAMC}, environmental sciences, controlling transmission of impulses in a nerve axon \cite{SMAHAJAN2}, and reservoir simulations \cite{RESS}, etc.
	\subsection{Model problem}
	The generalized Oseen equations are obtained by simplifying the steady-state Navier–Stokes equations or by applying backward Euler time discretization for unsteady scenarios. These equations have many applications in engineering fields such as aircraft, automotive, marine engineering, and environmental fluid dynamics etc. Let $\Omega \subset \mathbb{R}^d$ ($d = 2, 3$) denote an open and bounded Lipschitz polygon, with boundary $\Gamma = \partial \Omega$. For given control cost parameter $\gamma > 0$, desired velocity field $\mathbf{y}_d \in [L^{2}(\Omega)]^d$, and desired vorticity field $\boldsymbol{\omega}_d \in [L^{2}(\Omega)]^{\frac{d(d-1)}{2}}$, we define the minimization functional as:
	\begin{align}\label{2.1.1}
		\mathcal{J}(\mathbf{y,\boldsymbol{\omega},u}) := \frac{1}{2} \int_{\Omega} |\mathbf{y}-\mathbf{y}_d|^{2} \ dx +  \frac{1}{2} \int_{\Omega} |\boldsymbol{\omega}-\boldsymbol{\omega}_d|^{2} \ dx 
		+\frac{\gamma }{2} \int_{\Omega} |\u|^{2} \ dx.
	\end{align}
Consider the distributed optimal control problem governed by generalized Oseen equations, formulated in terms of the velocity field $\y:\Omega \rightarrow \R^{d}$ and pressure $p:\Omega \rightarrow \R$, as follows:
	\begin{align}\label{2.1.2}
		\begin{cases}
			- 2\nabla \cdot (\nu(x) \boldsymbol{\varepsilon}(\y)) + (\boldsymbol{\beta} \cdot \nabla)\mathbf{y} + \sigma \mathbf{y}  + \nabla p &= \ \  \mathbf{f} + \mathbf{u} \hspace{2.9mm} \text{in} \  \Omega, \\
			\hspace{4cm}\nabla \cdot \mathbf{y} &= \ \ 0 \ \ \ \  \ \ \ \ \text{in} \  \Omega, \\
			\hspace{4.4cm} \mathbf{y} &= \ \ \mathbf{0} \ \ \ \  \ \ \ \ \text{on} \ \Gamma,\\
			\hspace{4cm} \int_{\Omega} p \ dx &= \ \ 0,
		\end{cases} 
	\end{align}
	with the control constraints 
	\begin{align}\label{2.1.3}
		\mathbf{a}(\x) \le \mathbf{u}(\x) \le \mathbf{b}(\x) \ \ \ \ \text{for} \ \text{a.e.} \  \x \in \Omega,
	\end{align}
	where 
	\begin{itemize}
		\item $\nu(\x) \in W^{1,\infty}(\Omega)$ is the variable viscosity of the fluid and $\sigma \in L^{\infty}(\Omega)$ is a scalar function such that, for positive constants $\nu_0, \sigma_{\min}, \nu_1, \sigma_{\max} \in (0,\infty)$, the following relations hold:
		\begin{align*}
			 \nu_0 < \nu(x) < \nu_1 \quad \quad \text{and} \quad \quad \sigma_{\min} < \sigma (x) < \sigma_{\max} \quad \quad \forall x \in \Omega.
		\end{align*}
		\item $\boldsymbol{\varepsilon}(\y) = \frac{1}{2} (\nabla \y + (\nabla \y)^{T})$ is the symmetric strain rate tensor.
\vspace{1mm}
		\item $\boldsymbol{\beta} \in [W^{1,\infty}(\Omega)]^{d}$ is the convective velocity field and $\mathbf{\f} \in [L^{2}(\Omega)]^d$ is given external body force. \vspace{1mm}
		\item $\mathbf{a}, \mathbf{b} \in \R^d$ with $\mathbf{a} < \mathbf{b}$ (componentwise).
	\end{itemize}
	\subsection{Literature review} 
	The system of equations \eqref{2.1.2} is crucial in many situations where the viscosity of the fluid changes due to variations in flow rate caused by temperature, concentration, or the presence of different substances in the fluid. There are numerous methods in the literature to address incompressible flow problems with both constant and variable viscosity. These methods use the velocity-vorticity-pressure formulation and include techniques like mixed finite element, stabilized, least-squares, discontinuous Galerkin, hybrid discontinuous Galerkin, and spectral methods. These approaches have been applied to problems such as Brinkman equations \cite{VAGVB, AVMD}, Stokes flows \cite{MAMARA, VAMDRR, DFS}, Oseen equations \cite{VABAM, VAPANI}, Navier-Stokes equations \cite{AMCP, AVCRRT, BVB, CJNS}, and elasticity problems \cite{AKHAN}.
	In a recent study, Anaya et al. \cite{VANAYA} introduced a new augmented mixed finite element technique for the Oseen equations with a more general friction term of the form $\nabla \cdot (\nu \boldsymbol{\varepsilon}(\y))$, where $\boldsymbol{\varepsilon}(\y)$ is the strain rate tensor and $\nu$ is variable viscosity. The velocity-vorticity-pressure formulation used is non-symmetric, and the augmentation terms arise from least-squares contributions associated with the constitutive relation and the incompressibility constraint. These terms help in deriving the Babu\^{s}ka-Brezzi property of ellipticity on the kernel, with regularity assumptions on the viscosity gradient.
	
	 \noindent In comparison to traditional conforming finite element methods, DG methods have several attractive and well-documented features. These include high-order accuracy, hp-adaptivity, ease of implementation on complex geometries, and superior robustness with rough coefficients. DG methods have been employed for tackling the Oseen problem, as evidenced in references \cite{HDGOSEEN, LDGOSEEN,  EDGOSEEN, AGO}. Anaya et al. \cite{VABAM} achieved optimal convergence rates by employing DG discretizations with a three-field formulation to solve Oseen equations with constant viscosity. An advantage of this scheme is the robustness with respect to rough coefficients and the relaxation of inter-element continuity.
	The finite element approximation of optimal control problems has been thoroughly explored in the literature. Relevant work to this paper includes studies such as \cite{ CNS4, CNS1, CNS6, CNS2, HSAK, YHK},
	 and the references cited therein. In addition to this, \cite{AFB, AAER1, MBO} specifically tackle optimal control problems governed by Oseen equations. These works employ a conforming scheme for the velocity-pressure formulation, specifically addressing scenarios with constant viscosity. In a notable study, M. Berggren \cite{VORT} utilized an optimal-control approach to minimize the vorticity field in a least-squares sense, aiming to laminarize an unsteady internal flow. Recently, Singh and Khan \cite{HSAK} developed a divergence-conforming DG finite element method for the optimal control of the Oseen equations with variable viscosity. Based on these studies, we propose a new augmented mixed finite element method and a DG method for the distributed optimal control of Oseen equations, expressed in terms of velocity, vorticity, and pressure.
	\subsection{Fundamental contributions}
	As per our knowledge, no existing literature addresses optimal control problems using a velocity-vorticity-pressure formulation. This research work is a new contribution into this unexplored domain, offering opportunities for novel insights and advancements. Here, we highlight the key contributions of our work:
\begin{itemize}
\item \textbf{Existence of an optimal control:} Our primary contribution revolves around establishing the well-posedness of both continuous and discrete optimal control problems governed by Oseen equations with variable viscosity by utilizing a velocity-vorticity-pressure formulation and suitable assumptions outlined in Lemmas~\ref{Lemma-2.2.2.3.} and \ref{Lemma 2.3.1.11}. We propose a novel DG scheme which incorporates appropriate numerical fluxes for the \textbf{curl}-\textbf{curl} and grad-div operators, taking into account variable viscosity. Some key features of the proposed schemes include the liberty to choose different Stokes inf-sup stable finite element families, direct and accurate access to vorticity (without applying postprocessing), and flexibility in handling Dirichlet boundary conditions for velocity.\vspace{1mm}

\item \textbf{A-priori error estimates:} We derive optimal a priori estimates for the control, state, and co-state variables for both conforming and non-conforming schemes, ensuring accuracy across different discretization strategies. This ensures that the estimates are applicable to a wide range of computational scenarios. \vspace{1mm}

\item \textbf{A-posteriori error estimates:} Adaptive mesh refinement strategies, guided by a posteriori error indicators, are crucial in solving flow problems numerically. These strategies ensure the convergence of finite element solutions, especially in complex geometries that might otherwise produce erroneous results. 
Another significant contribution is the development of a reliable and efficient a posteriori error estimator suitable for both conforming and non-conforming schemes. This estimator can be calculated locally with low computational cost, even on complex geometries. Additionally, we introduce a novel method for assessing error bounds, which allows for more accurate predictions of convergence rates. \vspace{1mm}

\item{\textbf{Unified analysis:}} 
The analysis presented covers optimal control problems governed by the Stokes equations (with a uniformly bounded variable viscosity) 
 and the Brinkman equations
   under the assumptions discussed in \cite[Lemma~2.2]{VAGVB}.
	\end{itemize}
	\subsection{Structure of the paper}
	The subsequent sections of the paper are structured as follows: In Section~\ref{Preliminaries and continuous formulation.}, we lay the foundational groundwork by introducing the necessary function spaces required for our analysis. This section delves into the mathematical framework, discussing the existence and uniqueness of the optimal control for the continuous formulation and the optimality conditions. Section~\ref{Discrete formulation and priori error estimates.} is dedicated to the development of a mixed conforming scheme and a DG scheme for discretizing the continuous optimality system. We provide a comprehensive analysis of the well-posednes, accompanied by detailed derivations of both a priori and a posteriori error estimates. In Section~\ref{Numerical Experiments.}, we present a series of numerical tests designed to validate our theoretical findings and showcase the efficacy and convergence of methodology across various scenarios.
	\section{Function spaces  and continuous formulation} \label{Preliminaries and continuous formulation.} \setcounter{equation}{0}
	\subsection{Preliminaries} Let $\Omega \subset \mathbb{R}^d$ ($d = 2, 3$) be an open and bounded polygonal domain with Lipschitz boundary $\Gamma = \partial \Omega$. The notation $W^{s,p}(\Omega)$ represents standard Sobolev spaces intended for scalar-valued functions, equipped with norms $\|\cdot\|_{W^{s,p}(\Omega)}$, where $s \geq 0$ and $1 \leq p \leq \infty$. When $s = 0$, we write $W^{0,p}(\Omega):= L^p(\Omega)$. For $p=2$, the notation is simplified to $H^{s}(\Omega)$ with the norm $\|\cdot\|_{s}$. Bold letters are used to represent the vector-valued counterparts of these spaces. By $(\cdot,\cdot)$, we denote standard $L^2$ inner-product. We introduce the following spaces:
	\begin{align*}
		\quad &L_{0}^{2}(\Omega) := \biggl\{v \in L^2(\Omega) \ : \int_{\Omega} v(x) \ dx = 0 \biggl\},  &\boldsymbol{H}_{0}^{1}(\Omega) := \big \{\mathbf{v} \in \boldsymbol{H}^{1}(\Omega) \ : \  \mathbf{v}|_{\Gamma} = 0 \big \}.
	\end{align*} 
	The notation $m \precsim n$ means that for a positive constant $C$, $m \leq Cn$. 
	For simplicity, we write $\boldsymbol{V} = \boldsymbol{H}_{0}^{1}(\Omega),\ \boldsymbol{W} = [L^{2}(\Omega)]^{\frac{d(d-1)}{2}}$, and $Q = L_{0}^{2}(\Omega)$.
	\subsubsection{\textbf{Useful identities}}
	\begin{itemize}
		\item  For any vector fields $\mathbf{v} = (v_1, v_2, v_3)^{T}$ and $\mathbf{w} = (w_1, w_2)^{T}$, we define the $\textbf{curl}$ operator as:
		\begin{align*}
			\textbf{curl} (\mathbf{v}) = \begin{pmatrix}
				\frac{\partial v_3}{\partial y} - \frac{\partial v_2}{\partial z} &
				\frac{\partial v_1}{\partial z} - \frac{\partial v_3}{\partial x} &
				\frac{\partial v_2}{\partial x} - \frac{\partial v_1}{\partial y}
			\end{pmatrix}^{T}, \quad \hspace{1cm}
			\textbf{curl} (\mathbf{w}) = \frac{\partial w_2}{\partial x} - \frac{\partial w_1}{\partial y}.
		\end{align*}
		\item The integration by parts formula \cite[Theorem~2.11]{VP} produces:
		\begin{align*}
		\textbf{2-D} \boldsymbol{\rightarrow}	\int_{\Omega} \textbf{curl} (\omega) \cdot  \v \ dx  &= \int_{\Omega}  \omega \ \textbf{curl} (\v) \ dx  - \int_{\Gamma} \omega \ \v \cdot \t \ ds, \\
		\textbf{3-D} \boldsymbol{\rightarrow}	\int_{\Omega} \textbf{curl} (\boldsymbol{\omega}) \cdot  \v \ dx  &= \int_{\Omega}  \boldsymbol{\omega} \cdot \textbf{curl} (\v) \ dx  + \int_{\Gamma} (\boldsymbol{\omega} \times \n) \cdot \v \ ds.
		\end{align*}
		\item We will also use the following relations:
		\begin{align}
			\nonumber - \int_{\Omega} \nabla \cdot (\boldsymbol{\beta} (\y \cdot \v)) \ dx	 &= \int_{\Omega} [(\boldsymbol{\beta} \cdot \nabla)\mathbf{y}] \cdot \v \ dx + \int_{\Omega} [(\boldsymbol{\beta} \cdot \nabla)\mathbf{v}] \cdot \y \ dx, \quad \text{\big(\cite[Lemma~2.2]{VP}\big)}\\
			\nonumber	 \textbf{curl} (\nu \y) &= \nabla \nu \times \y + \nu \ \textbf{curl} (\y),\\
			\label{2.2.1.1}	- 2\nabla \cdot (\nu(\x) \boldsymbol{\varepsilon}(\y)) &= -\nu \Delta \y - 2 \boldsymbol{\varepsilon}(\y) \nabla \nu = \nu \ \textbf{curl} (\textbf{curl} (\y)) - \nu \nabla (\nabla \cdot \y) - 2 \boldsymbol{\varepsilon}(\y) \nabla \nu.
		\end{align} 
	\end{itemize}
	By a use of \eqref{2.2.1.1}, the state system in a \textbf{velocity-vorticity-pressure formulation} is expressed as:
	\begin{align}\label{2.2.1.2}
		\begin{cases}
			- 2 \boldsymbol{\varepsilon}(\y) \nabla \nu + \nu \ \textbf{curl} (\boldsymbol{\omega}) + (\boldsymbol{\beta} \cdot \nabla)\mathbf{y} + \sigma \mathbf{y}  + \nabla p &= \mathbf{f} + \mathbf{u} \ \ \text{in} \  \Omega,\\
			\hspace{3.4cm} \boldsymbol{\omega} - \textbf{curl} (\y) &= \mathbf{0} \ \ \ \quad \ \ \text{in} \  \Omega, \\
			\hspace{4cm}\nabla \cdot \mathbf{y} &= 0 \ \ \ \  \ \ \ \ \text{in} \  \Omega, \\
			\hspace{4.4cm} \mathbf{y} &= \mathbf{0} \ \ \ \  \ \ \ \ \text{on} \ \Gamma,\hspace{1.52cm}
		\end{cases} 
	\end{align}
	where we've employed the definition of vorticity. The equations within the system \eqref{2.2.1.2} represent momentum conservation, constitutive relation, mass balance, and the boundary condition.
	\subsection{Continuous formulation and well-posedness} \label{Continuous formulation.}
	In this subsection, we propose a mixed variational formulation of the optimal control problem and discuss its well-posedness. The augmented variational formulation of the state system \eqref{2.2.1.2} is to find $(\mathbf{y}, \boldsymbol{\omega}, p) \in \boldsymbol{V} \times \boldsymbol{W} \times Q$ such that
	\begin{align}\label{2.2.2.1}
		\begin{cases}
			\boldsymbol{\mathcal{A}}((\mathbf{y},\boldsymbol{\omega}), (\v, \boldsymbol{\theta})) + \boldsymbol{\mathcal{B}}((\v, \boldsymbol{\theta}),p) &= (\mathbf{f} + \mathbf{u}, \mathbf{v}) \ \quad \hspace{3cm} \forall \ (\mathbf{v}, \boldsymbol{\theta}) \in \boldsymbol{V} \times \boldsymbol{W}, \\
			\hspace{2.87cm}\boldsymbol{\mathcal{B}}((\mathbf{y},\boldsymbol{\omega}), \phi) &= 0  \hspace{4.77cm}   \forall \ \phi \in Q,
		\end{cases} 
	\end{align}
	where the bilinear forms  $\boldsymbol{\mathcal{A}} : [\boldsymbol{V} \times \boldsymbol{W}]^{2} \rightarrow \R$ and $ \boldsymbol{\mathcal{B}} : [\boldsymbol{V} \times \boldsymbol{W}] \times Q \rightarrow \R$ are defined as:
	\begin{align*}
	\boldsymbol{\mathcal{A}}((\mathbf{y},\boldsymbol{\omega}), (\v, \boldsymbol{\theta})) &:= - \int_{\Omega} 2 \boldsymbol{\varepsilon}(\y) \nabla \nu \cdot \v \ dx + \int_{\Omega} \nu \boldsymbol{\omega} \cdot \textbf{curl}(\v) \ dx + \int_{\Omega} \boldsymbol{\omega} \cdot \nabla \nu \times \v \ dx \\
		\nonumber & \ \ \ \ + \int_{\Omega} \nu \boldsymbol{\omega} \cdot \boldsymbol{\theta} \ dx - \int_{\Omega}  \nu \boldsymbol{\theta} \cdot \textbf{curl}(\y) \ dx + \int_{\Omega} (\sigma \y + (\boldsymbol{\beta} \cdot \nabla)\mathbf{y}) \cdot \v \ dx \\
		 \nonumber & \ \ \ \ + \rho_1 \int_{\Omega} (\textbf{curl}(\y) - \boldsymbol{\omega}) \cdot \textbf{curl}(\v) \ dx + \rho_2 \int_{\Omega} (\nabla \cdot \y) \cdot (\nabla \cdot \v) \ dx,\\
		\boldsymbol{\mathcal{B}}((\mathbf{y},\boldsymbol{\omega}), \phi)  &:= -\int_{\Omega} \phi \ \nabla \cdot \y \ dx.
	\end{align*}
	The terms with positive parameters $\rho_1$ and $\rho_2$ simplify the analysis by incorporating residuals from the constitutive relation and the incompressibility condition, and satisfy the following relations:
	\begin{align*}	
		\rho_1 \int_{\Omega} (\textbf{curl}(\y)-\boldsymbol{\omega}) \cdot \textbf{curl}(\v) \ dx = 0, \quad \rho_2 \int_{\Omega} (\nabla \cdot \y) \ (\nabla \cdot \v) \ dx = 0 \quad \quad  \forall \ \v \in \boldsymbol{V},
	\end{align*}
	We define norms on the spaces $\boldsymbol{V}$ and $\boldsymbol{V} \times \boldsymbol{W}$ as follows:
	\begin{align}\label{2.2.2.5}
		|\!|\!|\mathbf{v}|\!|\!|_{1}^{2} &:= \|\mathbf{v}\|_{0}^2+ \|\textbf{curl}(\v)\|_{0}^{2} + \|\nabla \cdot \v\|_{0}^{2}, &&\|(\v, \boldsymbol{\theta})\|^{2} := |\!|\!|\mathbf{v}|\!|\!|_{1}^{2} + \|\boldsymbol{\theta}\|_{0}^{2}.
	\end{align}
	\begin{lemma} \label{Lemma-2.2.2.1.}
		For all $(\y, \boldsymbol{\omega}), (\v, \boldsymbol{\theta}) \in \boldsymbol{V} \times \boldsymbol{W}$, along with the regularity assumptions $\nu(\x) \in W^{1,\infty}(\Omega)$, $\boldsymbol{\beta} \in [W^{1,\infty}(\Omega)]^{d}$, $\sigma \in L^{\infty}(\Omega)$, and a positive constant $\tilde{C}$, the subsequent estimates are valid:
		\begin{align*}
			\biggl|\int_{\Omega} \sigma \y \cdot \v \ dx \bigg| &\le \sigma_{\max} \|\mathbf{y}\|_{0} \|\mathbf{v}\|_{0},
			&&\biggl|\int_{\Omega} \nu \boldsymbol{\omega} \cdot \boldsymbol{\theta} \ dx \bigg| \ \ \ \ \ \ \ \le \nu_{1} \|\boldsymbol{\omega}\|_{0} \|\boldsymbol{\theta}\|_{0},\\
			\biggl|\int_{\Omega} (\boldsymbol{\beta} \cdot \nabla)\mathbf{y} \cdot \v \ dx \bigg| &\le \widetilde{C} |\!|\!|\boldsymbol{\beta}|\!|\!|_{1} |\!|\!|\y|\!|\!|_{1} |\!|\!|\mathbf{v}|\!|\!|_{1},
			&&\biggl|\int_{\Omega} \nu \boldsymbol{\theta} \cdot \textbf{curl}(\v) \ dx \bigg| \le \nu_{1}  \|\boldsymbol{\theta}\|_{0} |\!|\!|\mathbf{v}|\!|\!|_{1}, \\
			\biggl|\int_{\Omega} \boldsymbol{\varepsilon}(\y) \nabla \nu \cdot \v \ dx \bigg| &\le \|\nabla \nu\|_{\infty} \|\boldsymbol{\varepsilon}(\y)\|_{0} \|\v\|_{0},
			&&\biggl|\int_{\Omega} \boldsymbol{\theta} \cdot \nabla \nu \times \v \ dx \bigg| \ \  \le 2\|\nabla \nu\|_{\infty} \|\boldsymbol{\theta}\|_{0} \|\v\|_{0}.
		\end{align*}
		\begin{proof}
			These estimates are derived straightforwardly by applying the Cauchy-Schwarz inequality and the definition of norms.
		\end{proof}
	\end{lemma}
	\begin{lemma} \label{Lemma-2.2.2.2.}
		There exist positive constants $C_1$ and $C_2$
		  such that
		\begin{align*}
			|\boldsymbol{\mathcal{A}}((\mathbf{y},\boldsymbol{\omega}), (\v, \boldsymbol{\theta}))| &\le C_1 \|(\y, \boldsymbol{\omega})\| \|(\v, \boldsymbol{\theta})\| && \forall \ (\y, \boldsymbol{\omega}), (\v, \boldsymbol{\theta}) \in \boldsymbol{V} \times \boldsymbol{W}, \\
			|\boldsymbol{\mathcal{B}}((\v, \boldsymbol{\theta}),\phi)| &\le  C_2 \|(\v, \boldsymbol{\theta})\| \|\phi\|_{0} && \forall \ ((\v, \boldsymbol{\theta}), \phi) \in [\boldsymbol{V} \times \boldsymbol{W}] \times Q.
		\end{align*}
	\end{lemma}
	\begin{proof}
		Both estimates are directly obtained by applying Lemma \ref{Lemma-2.2.2.1.}.
	\end{proof}
	\noindent We now present a result demonstrating the ellipticity and the inf-sup condition of the bilinear forms $\boldsymbol{\mathcal{A}}(\cdot, \cdot)$ and $\boldsymbol{\mathcal{B}}(\cdot, \cdot)$, respectively, as illustrated in \cite[Lemma~2, 3]{VANAYA}.
	\begin{lemma}\label{Lemma-2.2.2.3.}
		(i) Suppose that 
		\begin{align}\label{2.2.2.6}
			\sigma_{\min} \ge \frac{9 \|\nabla \nu\|_{\infty}^{2}}{\nu_{0}} \quad \text{and} \quad \widetilde{C} \|\nabla \cdot \boldsymbol{\beta}\|_{0} < \min \Biggl\{\sigma_{\min} - \frac{9 \|\nabla \nu\|_{\infty}^{2}}{\nu_{0}}, \frac{\nu_{0}}{12}\Biggr\}.
		\end{align}
		Then, if we choose $\rho_1= \frac{2}{3}\nu_{0}$ and $\rho_2 > \frac{\nu_{0}}{3}$, there exists a constant $C_{3} > 0$ such that
		\begin{align*}
			\boldsymbol{\mathcal{A}}((\v, \boldsymbol{\theta}), (\v, \boldsymbol{\theta})) &\ge C_{3} \|(\v, \boldsymbol{\theta})\|^{2} \quad \quad \forall \ (\v, \boldsymbol{\theta}) \in \boldsymbol{V} \times \boldsymbol{W}.
		\end{align*}
		(ii) There exists a constant $\alpha >0$, independent of $\nu$ such that 
		\begin{align*}
			\alpha \|\phi\|_{0} &\le  \underset{0 \neq (\v, \boldsymbol{\theta}) \in \boldsymbol{V} \times \boldsymbol{W}}{\sup} \frac{\boldsymbol{|\mathcal{B}}((\v, \boldsymbol{\theta}),\phi)|}{\|(\v, \boldsymbol{\theta})\|} \quad \quad \forall \ \phi \in Q.
		\end{align*}
	\end{lemma}
	\begin{theorem}\label{Theorem-2.2.2.4.}
		Suppose the hypothesis of Lemma \ref{Lemma-2.2.2.3.} holds true. Then, there exists a unique solution $(\y, \boldsymbol{\omega}, p) \in \boldsymbol{H}_{0}^{1}(\Omega) \times [L^{2}(\Omega)]^{\frac{d(d-1)}{2}} \times L_{0}^{2}(\Omega)$ to the state system (\ref{2.2.2.1}). Furthermore, we have the estimate
		\begin{align}\label{2.2.2.7}
			\|(\y, \boldsymbol{\omega})\| + \|p\|_{0} \precsim \|\f\|_{0} + \|\u\|_{0}.
		\end{align}
	\end{theorem}
	\begin{proof}
		The well-posedness of the state problem is established through the application of Lemmas \ref{Lemma-2.2.2.2.} and \ref{Lemma-2.2.2.3.}, along with a direct implication of the Babu\^{s}ka-Brezzi theory \cite[Theorem~II.1.1]{BDMM}.
	\end{proof}
	\begin{remark}
		If the convective velocity field $\boldsymbol{\beta}$ satisfies the divergence-free condition,
		  then the state-system \eqref{2.2.2.1} becomes well-posed for the following:
		\begin{align*}
			\rho_1=\frac{2}{3} \nu_{0}, \quad \rho_2 > \frac{\nu_{0}}{3} \quad \text{provided that} \quad \sigma_{\min} \nu_{0} > 9 \|\nabla \nu\|_{\infty}^{2}.
		\end{align*}
	\end{remark}
	\begin{definition}
		The \textbf{admissible set} of controls is defined as 
		\begin{align*}
			\boldsymbol{\mathcal{A}_d} := \{\u = (u_1,\ldots, u_d) \in \boldsymbol{L}^{2}(\Omega) \ : \ a_i(\x) \le u_i(\x) \le b_i(\x)  \ \text{for} \ a.e. \  \x \in \Omega, \ i = 1,\ldots,d \}.
		\end{align*}
	\end{definition}
	\noindent Utilizing de Rham's Theorem \cite[Section~4.1.3 and Theorem~B73]{AJ}, an equivalent representation of (\ref{2.2.2.1}) is to find $(\mathbf{y}, \boldsymbol{\omega}) \in \boldsymbol{V}_{0} \times \boldsymbol{W}$, where $\boldsymbol{V}_{0} = \{\mathbf{v} \in \boldsymbol{V}: \nabla \cdot \mathbf{v} = 0\}$, such that
	\begin{align}\label{2.2.2.8}
		\boldsymbol{\mathcal{A}}((\mathbf{y},\boldsymbol{\omega}), (\v, \boldsymbol{\theta})) &= (\mathbf{f} + \mathbf{u}, \mathbf{v}) \hspace{1cm}  \forall  \ (\mathbf{v}, \boldsymbol{\theta}) \in \boldsymbol{V}_{0} \times \boldsymbol{W}.
	\end{align}
	This problem is well-posed by the Lax-Milgram theorem. We define the control-to-state map $\boldsymbol{\mathcal{S}}: \boldsymbol{L}^{2}(\Omega) \rightarrow \boldsymbol{L}^{2}(\Omega) \times [L^{2}(\Omega)]^{d(d-1)/2}$, which associates the velocity-vorticity pair $(\mathbf{y}, \boldsymbol{\omega})$ with a given control $\mathbf{u} \in \boldsymbol{\mathcal{A}_d}$. The set $\boldsymbol{\mathcal{A}_d}$, being a nonempty, bounded, convex, and closed subset of the reflexive Banach space $\boldsymbol{L}^{2}(\Omega)$, is weakly sequentially compact. We introduce the reduced functional $\mathcal{F}:\boldsymbol{L}^{2}(\Omega) \rightarrow \R$ as:
	\begin{align*}
		\mathcal{F}(\mathbf{u}) := \frac{1}{2} \|\boldsymbol{\mathcal{S}}_{\y}(\mathbf{u})-\mathbf{y}_d\|_{0}^{2} + \frac{1}{2} \|\boldsymbol{\mathcal{S}}_{\boldsymbol{\omega}}(\mathbf{u})-\boldsymbol{\omega}_d\|_{0}^{2} +\frac{\gamma }{2}  \|\mathbf{u}\|_{0}^{2},
	\end{align*}
	where $\boldsymbol{\mathcal{S}}_{\y}(\u)$ and $\boldsymbol{\mathcal{S}}_{\boldsymbol{\omega}}(\u)$ denote the velocity and voerticity fields corresponding to control $\u$. The weak lower semicontinuity and strict convexity of $\mathcal{F}$ imply that the following problem
	\begin{align*}
		\underset{\u \in \boldsymbol{\mathcal{A}_d}}{\min} \mathcal{F}(\u) \quad \quad \text{subject to} \ \ (\ref{2.2.2.8})
	\end{align*}
	has a unique optimal solution $\bar{\mathbf{u}}$ \cite[Theorem~2.14]{FTO} and corresponding optimal velocity $\bar{\mathbf{y}}$ and vorticity $\bar{\boldsymbol{\omega}}$. The existence of pressure state $\bar{p}$ such that $(\bar{\mathbf{y}}, \bar{\boldsymbol{\omega}}, \bar{p})$ solves (\ref{2.2.2.1}) is due to de Rham's Theorem. Hence, the optimal control problem (\ref{2.1.1}-\ref{2.1.3}) is well-posed.
	The optimal solution $\bar{\mathbf{u}}$ satisfies the variational inequality (\textbf{first-order necessary optimality condition}) \cite[Lemma~2.21, Theorem~2.25]{FTO}:
	\begin{align}
	\label{2.2.2.9}	\mathcal{F}'(\bar{\mathbf{u}})(\mathbf{u}-\bar{\mathbf{u}}) \ge \mathbf{0} \quad \forall \ \mathbf{u} \in \boldsymbol{\mathcal{A}_d} \qquad i.e. \qquad  (\bar{\mathbf{w}} + \gamma \bar{\mathbf{u}},\mathbf{u}-\bar{\mathbf{u}}) \ge \mathbf{0} \quad \forall \ \mathbf{u} \ \in \boldsymbol{\mathcal{A}_d},
	\end{align}
	where $(\bar{\mathbf{w}}, \bar{\boldsymbol{\vartheta}}, \bar{q}) \in \boldsymbol{V} \times \boldsymbol{W} \times Q$ is a unique solution to the co-state problem
	\begin{align}\label{2.2.2.10}
		\begin{cases}
			\boldsymbol{\mathcal{C}}((\mathbf{w},\boldsymbol{\vartheta}), (\z, \boldsymbol{\tau})) - \boldsymbol{\mathcal{B}}((\z, \boldsymbol{\tau}),q) &= (\mathbf{y} - \y_{d}, \mathbf{z}) + (\boldsymbol{\omega} - \boldsymbol{\omega}_{d}, \boldsymbol{\tau}) \ \quad \hspace{1.36cm} \forall \ (\z, \boldsymbol{\tau}) \in \boldsymbol{V} \times \boldsymbol{W}, \\
			\hspace{2.87cm}\boldsymbol{\mathcal{B}}((\mathbf{w},\boldsymbol{\vartheta}), \psi) &= 0  \hspace{5.516cm}   \forall \ \psi \in Q,
		\end{cases} 
	\end{align}
	with the bilinear form $\boldsymbol{\mathcal{C}} : [\boldsymbol{V} \times \boldsymbol{W}]^{2} \rightarrow \R$ defined as
	\begin{align}
	\label{2.12}	 \boldsymbol{\mathcal{C}}((\mathbf{w},\boldsymbol{\vartheta}), (\z, \boldsymbol{\tau})) &:= - \int_{\Omega} 2 \boldsymbol{\varepsilon}(\w) \nabla \nu \cdot \z \ dx + \int_{\Omega} \nu \boldsymbol{\vartheta} \cdot \textbf{curl}(\z) \ dx + \int_{\Omega} \boldsymbol{\vartheta} \cdot \nabla \nu \times \z \ dx  \\
		\nonumber & \ \ \ \ + \int_{\Omega} \nu \boldsymbol{\vartheta} \cdot \boldsymbol{\tau} \ dx - \int_{\Omega}  \nu \boldsymbol{\tau} \cdot \textbf{curl}(\w) \ dx + \int_{\Omega} (\sigma \w - (\boldsymbol{\beta} \cdot \nabla) \y - (\nabla \cdot \boldsymbol{\beta}) \w) \cdot \z \ dx \\
		\nonumber  & \ \ \ \ + \rho_1 \int_{\Omega} (\textbf{curl}(\w) - \boldsymbol{\vartheta})  \cdot \textbf{curl}(\z) \ dx + \rho_2 \int_{\Omega} (\nabla \cdot \w) \cdot (\nabla \cdot \z) \ dx.
	\end{align}
	Combining the state and co-state equations with the variational inequality, the optimality system is stated as follows: $(\mathbf{y},\boldsymbol{\omega}, p,\mathbf{u}) \in \boldsymbol{V} \times \boldsymbol{W} \times Q \times \boldsymbol{\mathcal{A}_d}$ is an optimal solution of the optimal control problem (\ref{2.1.1}, \ref{2.2.1.2}, \ref{2.1.3}) if and only if $(\y, \boldsymbol{\omega}, p, \w, \boldsymbol{\vartheta}, q, \mathbf{u}) \in  \boldsymbol{V} \times \boldsymbol{W} \times Q \times  \boldsymbol{V} \times \boldsymbol{W} \times Q \times \boldsymbol{\mathcal{A}_d}$ satisfies the following:
	\begin{subequations}
		\begin{align}
			\label{2.2.2.11} \boldsymbol{\mathcal{A}}((\mathbf{y},\boldsymbol{\omega}), (\v, \boldsymbol{\theta})) + \boldsymbol{\mathcal{B}}((\v, \boldsymbol{\theta}),p) &= (\mathbf{f} + \mathbf{u}, \mathbf{v})  \quad \quad \quad \quad  &&  \forall \ (\mathbf{v}, \boldsymbol{\theta}) \in \boldsymbol{V} \times \boldsymbol{W}, \\
			\label{2.2.2.12} \hspace{1.57cm} \boldsymbol{\mathcal{B}}((\mathbf{y},\boldsymbol{\omega}), \phi) &= 0 \quad   && \forall \ \phi \in Q, \\
			\label{2.2.2.13} \boldsymbol{\mathcal{C}}((\mathbf{w},\boldsymbol{\vartheta}), (\z, \boldsymbol{\tau})) - \boldsymbol{\mathcal{B}}((\z, \boldsymbol{\tau}),q) &=  (\mathbf{y} - \y_{d}, \mathbf{z}) + (\boldsymbol{\omega} - \boldsymbol{\omega}_{d}, \boldsymbol{\tau}) \ \quad \hspace{1.5cm} && \forall \ (\z, \boldsymbol{\tau}) \in \boldsymbol{V} \times \boldsymbol{W},\\
			\label{2.2.2.14} \hspace{1.56cm}\boldsymbol{\mathcal{B}}((\mathbf{w},\boldsymbol{\vartheta}), \psi) &= 0 \  \quad \quad \quad \quad    && \forall \ \psi \in Q,\\
			\label{2.2.2.15} (\mathbf{w} + \gamma \u,\tilde{\u}-\u) &\ge \mathbf{0} \quad \quad \quad \quad \quad \quad    && \forall \ \tilde{u} \in \boldsymbol{\mathcal{A}_d}.
		\end{align} 
	\end{subequations}
By applying the projection formula \cite[Theorem~2.28]{FTO} to the optimal control variable, the variational inequality \eqref{2.2.2.15} can be reformulated as:
\begin{align}
\label{2.2.2.16}	\u = \boldsymbol{\Pi}_{[\mathbf{a}, \mathbf{b}]}\left(- \gamma^{-1} \w \right) \ \text{a.e. in} \ \Omega, \qquad \text{where} \qquad
\boldsymbol{\Pi}_{[\mathbf{a}, \mathbf{b}]}(\mathbf{v})(\x) := \min \{\mathbf{b}(\x), \max \{\mathbf{a}(\x), \mathbf{v}(\x)\}\}.
\end{align}
	\subsection{Second-order conditions}
	In conducting a numerical analysis of the problem and evaluating optimization algorithms, we establish the second-order conditions for the optimal control problem with velocity-vorticity-pressure formulation by following \cite{HSAK}.
\begin{definition}
	A control $\bar{\u} \in \boldsymbol{\mathcal{A}_d}$ is termed \textbf{locally optimal} if there exists a positive constant $\epsilon$, such that for any $\u$ in $\boldsymbol{\mathcal{A}_d}$ with $\|\bar{\u}-\u\|_{0}^{2} \le \epsilon$, following inequality holds:
	\begin{align*}
		\mathcal{J}(\bar{\y}, \bar{\boldsymbol{\omega}}, \bar{\u}) \le \mathcal{J}(\y, \boldsymbol{\omega}, \u).
	\end{align*}
	Here, $(\y, \boldsymbol{\omega})$ and $(\bar{\y}, \bar{\boldsymbol{\omega}})$ represent the velocity-vorticity states associated to the controls $\u$ and $\bar{\u}$, respectively.
\end{definition}
	\begin{definition}
		A pair $(\bar{\y}, \bar{\boldsymbol{\omega}}, \bar{\u}) \in \boldsymbol{V} \times \boldsymbol{W} \times \boldsymbol{\mathcal{A}_d}$ is considered a \textbf{globally optimal solution} if 
		\begin{align*}
			\mathcal{J}(\bar{\y}, \bar{\boldsymbol{\omega}}, \bar{\u}) = \underset{(\y, \boldsymbol{\omega}, \u) \in \boldsymbol{V} \times \boldsymbol{W} \times \boldsymbol{\mathcal{A}_d}}{\min} \mathcal{J}(\y, \boldsymbol{\omega}, \u).
		\end{align*}
	\end{definition}
	\noindent  For $\boldsymbol{X} := \boldsymbol{V} \times Q$, and $\y = (\y^{v}, \y^{p}), \ \z = (\z^v, \z^p) \in \boldsymbol{X}$, define $\boldsymbol{\mathcal{R}} : [\boldsymbol{X} \times \boldsymbol{W}]^{2} \rightarrow \R$ as: 
	\begin{align*}
		\nonumber	\boldsymbol{\mathcal{R}}((\y,\boldsymbol{\omega}), (\z,\boldsymbol{\theta})) &= - 2 (\boldsymbol{\varepsilon}(\y^{v}) \nabla \nu, \z^v) +  (\nu \boldsymbol{\omega}, \textbf{curl}(\z^v)) + (\boldsymbol{\omega}, \nabla \nu \times \z^v) + (\sigma \y^{v} + (\boldsymbol{\beta} \cdot \nabla)\mathbf{y}^{v}, \z^v) + (\nu \boldsymbol{\omega}, \boldsymbol{\theta})  \\
		\nonumber & \ \ \ - (\nu \boldsymbol{\theta}, \textbf{curl}(\y^{v})) + \rho_1(\textbf{curl}(\y^{v})-\boldsymbol{\omega}, \textbf{curl}(\z^v)) + \rho_2 (\nabla \cdot \y^{v}, \nabla \cdot \z^v) - (\y^{p}, \nabla \cdot \z^{v})\\
		& \ \ \ \ + (\z^{p},\nabla \cdot \y^{v}).
	\end{align*}
	The Lagrange function $\boldsymbol{\mathcal{L}}: \boldsymbol{X} \times \boldsymbol{W} \times \boldsymbol{L}^{2}(\Omega) \times \boldsymbol{X} \times \boldsymbol{W} \rightarrow \mathbb{R}$, is defined as: 
	\begin{align}\label{2.2.3.1}
		\boldsymbol{\mathcal{L}}(\y, \boldsymbol{\omega}, \u,\z, \boldsymbol{\theta}) := \ & \mathcal{J}(\y, \boldsymbol{\omega}, \u) - \boldsymbol{\mathcal{R}}((\y,\boldsymbol{\omega}), (\z,\boldsymbol{\theta})) + (\f+\u,\z^{v}).
	\end{align}
\begin{lemma}\label{lem: 2.2.3.1}
	The Lagrangian $\boldsymbol{\mathcal{L}}$ exhibits Fréchet differentiability of order two w.r.t. the vector $\v = (\tilde{\boldsymbol{y}}, \tilde{\boldsymbol{\kappa}}, \tilde{\boldsymbol{u}})$. The second-order derivative evaluated at $\v = (\tilde{\boldsymbol{y}}, \tilde{\boldsymbol{\kappa}}, \tilde{\boldsymbol{u}})$, along with the associated adjoint state $\tilde{\boldsymbol{z}}$, satisfies the following conditions:
	\begin{align}
		\label{2.2.3.2} 	\mathcal{L}_{\v \v} (\tilde{\v},\tilde{\z})[(\r_1,\s_1,\t_1),(\r_2,\s_2,\t_2)] &= \mathcal{L}_{\y \y} (\tilde{\v},\tilde{\z})[\r_1,\r_2] + \mathcal{L}_{\boldsymbol{\kappa} \boldsymbol{\kappa} } (\tilde{\v},\tilde{\z})[\s_1,\s_2] + \mathcal{L}_{\u \u} (\tilde{\v},\tilde{\z})[\t_1,\t_2],\\
		\label{2.2.3.3} 	|\mathcal{L}_{\y \y} (\tilde{\v},\tilde{\z})[\r_1,\r_2]| &\le C_{\mathcal{L}} |\r_1| |\r_2|
	\end{align}
	for all $(\r_i,\s_i,\t_i) \in \boldsymbol{V} \times \boldsymbol{W} \times \boldsymbol{L}^{2}(\Omega)$; $i=1,2$ , and $C_{\mathcal{L}}$ is a positive constant independent of $\tilde{\v},\r_1$ and $\r_2.$
\end{lemma}
\begin{proof}
	The first order derivatives of $\boldsymbol{\mathcal{L}}$  w.r.t. $\y$, $\boldsymbol{\kappa}$ and $\u$ are 
	\begin{align*}
		\mathcal{L}_{\y} (\tilde{\v},\tilde{\z}) \r &= (\r,\tilde{\y}-\y_{d}) - \boldsymbol{\mathcal{R}}_{\y}((\tilde{\y},\tilde{\boldsymbol{\kappa}}), (\z,\s)), \qquad &&\mathcal{L}_{\boldsymbol{\kappa}} (\tilde{\v},\tilde{\z}) \s = (\s,\tilde{\boldsymbol{\kappa}}-\boldsymbol{\kappa}_{d}) - \boldsymbol{\mathcal{R}}_{\boldsymbol{\kappa}}((\tilde{\y},\tilde{\boldsymbol{\kappa}}), (\z,\s)) \\
		\mathcal{L}_{\u} (\tilde{\v},\tilde{\z}) \t &= \gamma (\t,\tilde{\u}) + (\u,\tilde{\z}).
	\end{align*}
	The mappings $\tilde{\y} \mapsto \boldsymbol{\mathcal{L}}_{\y}(\tilde{\v}, \tilde{\w})$, $\tilde{\boldsymbol{\kappa}} \mapsto \boldsymbol{\mathcal{L}}_{\boldsymbol{\kappa}}(\tilde{\v}, \tilde{\w})$ and $\tilde{\u} \mapsto \boldsymbol{\mathcal{L}}_{\u}(\tilde{\w}, \tilde{\z})$ exhibit an affine linear structure with bounded linear components, ensuring continuity. Consequently, both mappings are Fréchet-differentiable. This observation establishes that $\boldsymbol{\mathcal{L}}$ is twice Fréchet-differentiable. The second-order derivative of $\boldsymbol{\mathcal{L}}$ with respect to $\v$ is then expressed as:
	\begin{align*}
		\boldsymbol{\mathcal{L}}_{\v \v}(\tilde{\v}, \tilde{\z})[(\r_1,\s_1,\t_1),(\r_2,\s_2,\t_2)]  &= \boldsymbol{\mathcal{L}}_{\u \u}(\tilde{\v}, \tilde{\z})[\t_1,\t_2]+ \mathcal{L}_{\boldsymbol{\kappa} \boldsymbol{\kappa} } (\tilde{\v},\tilde{\z})[\s_1,\s_2]  +\boldsymbol{\mathcal{L}}_{\y \y}(\tilde{\v}, \tilde{\z})[\r_1,\r_2] \\
		&= \gamma (\t_1,\t_2) + (\r_1,\r_2).
	\end{align*}
	The second estimate is derived by applying the Cauchy-Schwarz inequality.
\end{proof}
	\section{Mixed formulations and error analysis} \label{Discrete formulation and priori error estimates.} \setcounter{equation}{0}
	In this section, we propose a conforming and a non-conforming scheme for the continuous optimality system by selecting appropriate finite element spaces for the velocity, vorticity, and pressure variables, ensuring necessary stability conditions. Furthermore, we will investigate the well-posedness of the discrete problem and derive a priori and a posteriori error estimates for both schemes.\\ \\
	Firstly, we introduce the notations associated with the discretization of the domain $\Omega$. Let $\mathcal{T}_h$ represent a shape-regular partition of the polygonal or polyhedral domain $\bar{\Omega}$ into closed triangles or tetrahedrons $K$, in the sense of \cite{AJ}, such that $\bigcup\limits_{K \in \mathcal{T}_h}K = \bar{\Omega}.$ Let $h= \max \{h_K: K \in \mathcal{T}_h \}$ denote the global mesh-size, where $h_K$ represents the diameter of an element $K$. The sets $\mathcal{E}^{i}(\mathcal{T}_h)$, $\mathcal{E}^{b}(\mathcal{T}_h)$, and $\mathcal{E}(\mathcal{T}_h)$ consist of interior edges, boundary edges, and all edges of $\mathcal{T}_h$, respectively. $h_E$ signifies the length of an edge $E$, and $\boldsymbol{n}_E$ indicates its unit outward normal vector. We denote the broken Sobolev space norm on an element $K$ by $\|\cdot\|_{s,K}$. For an integer $k \ge 0$, let $\P_{k}(K)$ denotes the space of polynomials of degree atmost $k$ on an element $K$. \\ \\
\subsection{Conforming formulation}\label{Conforming Scheme}
Consider a finite element family using the MINI-element \cite[Sections 8.6 and 8.7]{BMO} for velocity-pressure, and continuous or discontinuous piecewise polynomials for vorticity as:
\begin{align}
	\label{2.3.2.1}	\boldsymbol{V}_h &:= \boldsymbol{U}_h \oplus \mathbb{B}(b_{K} \nabla Q_h) \cap \boldsymbol{H}_{0}^{1}(\Omega), \\
	\label{2.3.2.2}	Q_h &:= \big \{\phi_h \in C(\bar{\Omega}) : \phi_h|_{K} \in \P_{k}(K) \ \forall \ K \in \mathcal{T}_h \big \} \cap L_{0}^{2}(\Omega),
\end{align}
where 
\begin{align*}
\boldsymbol{U}_h &:= \big \{\v_h \in [C(\bar{\Omega})]^{d} : \v_h|_{K} \in [\P_{k}(K)]^{d} \ \forall \ K \in \mathcal{T}_h \big \}, \\
\mathbb{B}(b_{K} \nabla \phi_h) &:= \{\v_{hb} \in \boldsymbol{H}^{1}(\Omega) : \v_{hb}|_{K} = b_{K} \nabla (\phi_h)|_{K} \ \text{for some} \ \phi_h \in Q_h \},
\end{align*}
and $b_K$ is the standard (cubic or quartic) bubble function $\lambda_1, \ldots,\lambda_{d+1} \in [\P_{k}(K)]^{d+1}$. The primary reason for selecting such a pair of discrete spaces $(\boldsymbol{V}_h, Q_h)$ is to ensure that the following discrete inf-sup condition is satisfied for a positive constant $\alpha_0$ (invariant w.r.t. $h$), as discussed in \cite{DFO}:
\begin{align*}
	\alpha_{0} \|\phi_h\|_{0} &\le  \underset{0 \neq (\v_h, \boldsymbol{\theta}_h) \in \boldsymbol{V}_h \times \boldsymbol{W}_h}{\sup} \frac{|\boldsymbol{\mathcal{B}}((\v_h, \boldsymbol{\theta}_h),\phi_h)|}{\|(\v_h, \boldsymbol{\theta}_h)\|} \quad \quad \forall \ \phi_h \in Q_h.
\end{align*}
	Define $\boldsymbol{W}_h^{i}$ as a continuous or piecewise discontinuous polynomial subspace of $\boldsymbol{W}$, in the following manner:
	\begin{subequations}\label{2.3.2.3}
		\begin{align}
			\label{2.3.2.3.a}	\boldsymbol{W}_h^{1} &:= \Big \{\boldsymbol{\theta}_h \in [C(\bar{\Omega})]^{\frac{d(d-1)}{2}} \ : \boldsymbol{\theta}_h|_{K} \in [\P_{k}(K)]^{\frac{d(d-1)}{2}} \ \forall \ K \in \mathcal{T}_h \Big \}, \\
			\label{2.3.2.3.b}	\boldsymbol{W}_h^{2} &:= \Big \{\boldsymbol{\theta}_h \in [L^{2}(\Omega)]^{\frac{d(d-1)}{2}} : \boldsymbol{\theta}_h|_{K} \in [\P_{k}(K)]^{\frac{d(d-1)}{2}} \ \forall \ K \in \mathcal{T}_h \Big \}.
		\end{align}
	\end{subequations}
	For the control variable, we define $\boldsymbol{\mathcal{A}_{dh}}$ as a discrete subspace of $\boldsymbol{\mathcal{A}_d}$, that is nonempty, closed, and convex. Utilizing a piecewise constant discretization, the discrete control space is defined as:
	\begin{align}
	\label{2.3.2.5}	\boldsymbol{\mathcal{A}_{dh}} :=\{\u_h \in \boldsymbol{L}^2(\Omega) \ : \u_h|_{K} \in [\P_{0}(K)]^{d} \ \forall \ K \in \mathcal{T}_h \}.
	\end{align}
	The discrete optimality system for the subspaces introduced in (\ref{2.3.2.1}-\ref{2.3.2.5}) is to find $(\y_h, \boldsymbol{\omega}_h, p_h, \w_h, \boldsymbol{\vartheta}_h, q_h, \mathbf{u}_h) \in  \boldsymbol{V}_h \times \boldsymbol{W}_h \times Q_h \times  \boldsymbol{V}_h \times \boldsymbol{W}_h \times Q_h \times \boldsymbol{\mathcal{A}_{dh}}$ such that
	\begin{subequations}
		\begin{align}
			\label{2.3.2.6} \boldsymbol{\mathcal{A}}((\mathbf{y}_h,\boldsymbol{\omega}_h), (\v_h, \boldsymbol{\theta}_h)) + \boldsymbol{\mathcal{B}}((\v_h, \boldsymbol{\theta}_h),p_h) &= (\mathbf{f} + \mathbf{u}_h, \mathbf{v}_h) &&  \forall  (\mathbf{v}_h, \boldsymbol{\theta}_h) \in \boldsymbol{V}_h \times \boldsymbol{W}_h,\\
			\label{2.3.2.7} \hspace{1.57cm} \boldsymbol{\mathcal{B}}((\mathbf{y}_h,\boldsymbol{\omega}_h), \phi_h) &= 0 && \hspace{-0.1cm} \forall  \phi_h \in Q_h,\\
			\label{2.3.2.8} \boldsymbol{\mathcal{C}}((\mathbf{w}_h,\boldsymbol{\vartheta}_h), (\z_h, \boldsymbol{\tau}_h)) - \boldsymbol{\mathcal{B}}((\z_h, \boldsymbol{\tau}_h),q_h) &=  (\mathbf{y}_h - \y_{d}, \mathbf{z}_h) + (\boldsymbol{\omega}_h - \boldsymbol{\omega}_{d}, \boldsymbol{\tau}_h) && \hspace{-0.1cm} \forall (\z_h, \boldsymbol{\tau}_h) \in \boldsymbol{V}_h \times \boldsymbol{W}_h,\\
			\label{2.3.2.9} \hspace{1.56cm}\boldsymbol{\mathcal{B}}((\mathbf{w}_h,\boldsymbol{\vartheta}_h), \psi_h) &= 0 && \hspace{-0.1cm} \forall  \psi_h \in Q_h,\\
			\label{2.3.2.10} (\mathbf{w}_h + \gamma \u_h,\tilde{\u}_h -\u_h) &\ge \mathbf{0} && \hspace{-0.1cm} \forall \tilde{\u}_h \in \boldsymbol{\mathcal{A}_{dh}}.
		\end{align} 
	\end{subequations}
	\noindent The application of Babu\^{s}ka-Brezzi theory to saddle point problems, along with the continuity-coercivity properties of bilinear form $\boldsymbol{\mathcal{A}}$ and the discrete inf-sup stability of $\boldsymbol{\mathcal{B}}$, guarantees the unique solvability of the discrete optimal control problem.
	\begin{remark}
		As detailed in \cite[Section~3.1.1]{VANAYA}, the proposed scheme allows the use of generalised Taylor-Hood-$\P_k$ finite elements for approximating velocity and pressure, and continuous or discontinuous piecewise polynomial spaces for vorticity.
	\end{remark}
	\subsubsection{A priori error estimates}\label{A priori error estimates}
	The main objective of this subsection is to derive the a priori error estimates for the control, state, and co-state variables. Let $k \geq 1$ be an integer.
	Throughout this subsection, we enforce the following regularity assumptions:
	\begin{align*}
		\y, \w \in \boldsymbol{H}^{s+1} (\Omega), \ \boldsymbol{\omega}, \boldsymbol{\vartheta} \in [H^{s} (\Omega)]^{\frac{d(d-1)}{2}}, \ p, q \in H^{s} (\Omega), \ \text{and} \ \u \in \boldsymbol{H}^{1} (\Omega)\  \text{for some } \ s \in (1/2,k].
	\end{align*}
	For a specified control $\u$, let $(\y_h(\u), \boldsymbol{\omega}_h(\u), p_h(\u)) \in  \boldsymbol{V}_h \times \boldsymbol{W}_h \times Q_h$ be a solution to the problem
	\begin{subequations}\label{3.10}
		\begin{align}
			\label{2.4.1.1.a}  \boldsymbol{\mathcal{A}}((\mathbf{y}_h(\u),\boldsymbol{\omega}_h(\u)), (\v_h, \boldsymbol{\theta}_h)) + \boldsymbol{\mathcal{B}}((\v_h, \boldsymbol{\theta}_h),p_h(\u)) &= (\mathbf{f} + \mathbf{u}, \mathbf{v}_h) &&\forall \ (\v_h, \boldsymbol{\theta}_h) \in \boldsymbol{V}_h \times \boldsymbol{W}_h,\\
			\label{2.4.1.1.b} \hspace{1.57cm} \boldsymbol{\mathcal{B}}((\mathbf{y}_h(\u),\boldsymbol{\omega}_h(\u)), \phi_h) &= 0 &&\forall \ \phi_h \in Q_h.
		\end{align} 
	\end{subequations}
	Similarly, let $(\w_h(\y), \boldsymbol{\vartheta}_h(\y), r_h(\y)) \in  \boldsymbol{V}_h \times \boldsymbol{W}_h \times Q_h$ be a solution to the following problem:
	\begin{subequations}\label{3.11}
		\begin{align}
			\label{2.4.1.2.a}	\boldsymbol{\mathcal{C}}((\mathbf{w}_h(\y),\boldsymbol{\vartheta}_h(\y)), (\z_h, \boldsymbol{\tau}_h)) - \boldsymbol{\mathcal{B}}((\z_h, \boldsymbol{\tau}_h),q_h(\y)) &=  (\mathbf{y} - \y_{d}, \mathbf{z}_h) + (\boldsymbol{\omega} - \boldsymbol{\omega}_{d}, \boldsymbol{\tau}_h),\\
			\label{2.4.1.2.b} \hspace{1.56cm}\boldsymbol{\mathcal{B}}((\mathbf{w}_h(\y),\boldsymbol{\vartheta}_h(\y)), \psi_h) &= 0,
		\end{align}
	\end{subequations}
	for all $(\z_h, \boldsymbol{\tau}_h, \psi_h) \in \boldsymbol{V}_h \times \boldsymbol{W}_h \times Q_h$.
	\begin{lemma}\label{Lemma-2.4.1.1}
		Let $(\y_h, \boldsymbol{\omega}_h, p_h)$ and $(\w_h, \boldsymbol{\vartheta}_h, q_h)$  be the solutions to the state and co-state discrete systems (\ref{2.3.2.6}-\ref{2.3.2.7}) and (\ref{2.3.2.8}-\ref{2.3.2.9}), respectively. Let $(\y_h(\u), \boldsymbol{\omega}_h(\u), p_h(\u))$ and $(\w_h(\y), \boldsymbol{\vartheta}_h(\y), r_h(\y))$ be the auxiliary variables. Then, we have the following estimates:
		\begin{align}
			\label{2.4.1.3}	\|(\y_h(\u)-\y_h,\boldsymbol{\omega}_h(\u)-\boldsymbol{\omega}_h)\| + \|p_h(\u)-p_h\|_{0} &\precsim \|\u-\u_h\|_{0}, \\
			\label{2.4.1.4}	\|(\w_h(\y)-\w_h,\boldsymbol{\vartheta}_h(\y)-\boldsymbol{\vartheta}_h)\| + \|q_h(\y)-q_h\|_{0} &\precsim \|\y-\y_h\|_{0} + \|\boldsymbol{\omega} - \boldsymbol{\omega}_{h}\|_{0}.
		\end{align}
	\end{lemma}
	\begin{proof}
	Subtracting the system (\ref{2.3.2.6}-\ref{2.3.2.7}) from (\ref{2.4.1.1.a}-\ref{2.4.1.1.b}), we have
		\begin{align*}
		 \boldsymbol{\mathcal{A}}((\mathbf{y}_h(\u)-\y_h,\boldsymbol{\omega}_h(\u)-\boldsymbol{\omega}_h), (\v_h, \boldsymbol{\theta}_h)) + \boldsymbol{\mathcal{B}}((\v_h, \boldsymbol{\theta}_h),p_h(\u)-p_h) &= (\mathbf{u}-\u_h, \mathbf{v}_h),\\
			\label{2.4.1.3.1b} \hspace{1.57cm} \boldsymbol{\mathcal{B}}((\mathbf{y}_h(\u)-\y_h,\boldsymbol{\omega}_h(\u)-\boldsymbol{\omega}_h), \phi_h) &= 0,
		\end{align*}
for all $(\v_h, \boldsymbol{\theta}_h,\phi_h) \in \boldsymbol{V}_h \times \boldsymbol{W}_h \times Q_h$. After substituting $\v_h =  \y_h(\u)-\y_h,\boldsymbol{\theta}_h= \boldsymbol{\omega}_h(\u)-\boldsymbol{\omega}_h$ and  $\phi_h =  p_h(\u)-p_h$, in the above set of equations, we obtain
	\begin{align*}
 \boldsymbol{\mathcal{A}}((\mathbf{y}_h(\u)-\y_h,\boldsymbol{\omega}_h(\u)-\boldsymbol{\omega}_h), (\mathbf{y}_h(\u)-\y_h,\boldsymbol{\omega}_h(\u)-\boldsymbol{\omega}_h)) = (\mathbf{u}-\u_h, \y_h(\u)-\y_h).
	\end{align*} 
By a use of Lemma~\ref{Lemma-2.2.2.3.} and an application of Cauchy-Schwarz inequality, we have
\begin{align}
\label{2.4.1.3.3} 	\|(\y_h(\u)-\y_h,\boldsymbol{\omega}_h(\u)-\boldsymbol{\omega}_h)\| &\precsim  \|\u-\u_h\|_{0}.
\end{align}
Similarly, by subtracting the system (\ref{2.3.2.8}-\ref{2.3.2.9}) from (\ref{2.4.1.2.a}-\ref{2.4.1.2.b}), we get
	\begin{align*} \boldsymbol{\mathcal{C}}((\mathbf{w}_h(\y)-\w_h,\boldsymbol{\vartheta}_h(\y)-\boldsymbol{\vartheta}_h), (\z_h, \boldsymbol{\tau}_h)) - \boldsymbol{\mathcal{B}}((\z_h, \boldsymbol{\tau}_h),q_h(\y)-q_h) &= (\mathbf{y}-\y_h, \mathbf{z}_h)+ (\boldsymbol{\omega} - \boldsymbol{\omega}_{h}, \boldsymbol{\tau}_h),\\
	\hspace{1.57cm} \boldsymbol{\mathcal{B}}((\mathbf{w}_h(\y)-\w_h,\boldsymbol{\vartheta}_h(\y)-\boldsymbol{\vartheta}_h), \psi_h) &= 0,
	\end{align*} 
	for all $(\z_h,\boldsymbol{\tau}_h,\psi_h) \in \boldsymbol{V}_h \times \boldsymbol{W}_h \times Q_h$. Substituting $\z_h = \w_h(\y)- \w_h$, $\boldsymbol{\tau}_h=\boldsymbol{\vartheta}_h(\y)-\boldsymbol{\vartheta}_h$, and $\psi_h =  q_h(\y)-q_h$, in the above system of equations, we arrive at:
\begin{align*}
	\boldsymbol{\mathcal{C}}((\mathbf{w}_h(\y)-\w_h,\boldsymbol{\vartheta}_h(\y)-\boldsymbol{\vartheta}_h), (\w_h(\y)- \w_h, \boldsymbol{\vartheta}_h(\y)-\boldsymbol{\vartheta}_h)) = (\mathbf{y}-\y_h, \w_h(\y)- \w_h)+ (\boldsymbol{\omega} - \boldsymbol{\omega}_{h}, \boldsymbol{\vartheta}_h(\y)-\boldsymbol{\vartheta}_h).
\end{align*} 
Using Lemma~\ref{Lemma-2.2.2.3.} and Cauchy-Schwarz inequality, we obtain
\begin{align}
	\label{2.4.1.3.4} 	\|(\w_h(\y)-\w_h,\boldsymbol{\vartheta}_h(\y)-\boldsymbol{\vartheta}_h)\| &\precsim \|\y-\y_h\|_{0} + \|\boldsymbol{\omega} - \boldsymbol{\omega}_{h}\|_{0}.
\end{align}
Since $\boldsymbol{\mathcal{B}}((\cdot,\cdot),\cdot)$ satisfies the inf-sup condition, so for a constant $C>0$, we have
\begin{align}
	\nonumber C \|p_h(\u)-p_h\|_{0}  &\le \underset{0 \neq (\v_h, \boldsymbol{\theta}_h) \in \boldsymbol{V}_h \times \boldsymbol{W}_h}{\sup} \frac{|\boldsymbol{\mathcal{B}}((\v_h, \boldsymbol{\theta}_h),p_h(\u)-p_h)|}{\|(\v_h, \boldsymbol{\theta}_h)\|} \\
	\nonumber &= \underset{0 \neq (\v_h, \boldsymbol{\theta}_h) \in \boldsymbol{V}_h \times \boldsymbol{W}_h}{\sup}  \frac{|(\mathbf{u}-\u_h, \mathbf{v}_h) - \boldsymbol{\mathcal{A}}((\mathbf{y}_h(\u)-\y_h,\boldsymbol{\omega}_h(\u)-\boldsymbol{\omega}_h), (\v_h, \boldsymbol{\theta}_h))|}{\|(\v_h, \boldsymbol{\theta}_h)\|}\\
	\label{2.4.1.3.5} &\precsim \|\u-\u_h\|_{0}.
\end{align}
By following similar steps for the co-state problem, we get
\begin{align}
	\label{2.4.1.3.6}	 \|q_h(\y)-q_h\|_{0} &\precsim \|\y-\y_h\|_{0} + \|\boldsymbol{\omega} - \boldsymbol{\omega}_{h}\|_{0}.
\end{align}
Combining (\ref{2.4.1.3.3}) with (\ref{2.4.1.3.5}) and (\ref{2.4.1.3.4}) with (\ref{2.4.1.3.6}), we get the desired estimates.
	\end{proof}
	\begin{lemma}\label{Lemma-2.4.1.2}
		Let $(\y, \boldsymbol{\omega}, p), (\w, \boldsymbol{\vartheta}, q) \in  \boldsymbol{V} \times \boldsymbol{W} \times Q$ be the solutions to the state and co-state systems (\ref{2.2.2.11}-\ref{2.2.2.12}) and (\ref{2.2.2.13}-\ref{2.2.2.14}), respectively. Then, for positive constants $C_1$ and $C_2$, we have the estimates:
		\begin{align}
			\label{2.4.1.5.}	\|(\y- \y_h(\u),\boldsymbol{\omega} - \boldsymbol{\omega}_h(\u))\| + \|p-p_h(\u)\|_{0} &\le C_1 h^{s}\big(\|\y\|_{s+1} + \|\boldsymbol{\omega}\|_{s+1} + \|p\|_{s}\big), \\
			\label{2.4.1.6}	\|(\w - \w_h(\y),\boldsymbol{\vartheta} - \boldsymbol{\vartheta}_h(\y))\| + \|q-q_h(\y)\|_{0} &\le C_2 h^{s}\big(\|\w\|_{s+1} + \|\boldsymbol{\vartheta}\|_{s+1} + \|q\|_{s}\big).
		\end{align}
	\end{lemma}
	\begin{proof}
		The first estimate is derived directly from the result presented in \cite[Theorem~3.3]{VANAYA}. The second estimate for the co-state follows by a similar approach.
	\end{proof}	
	\begin{theorem}\label{Lemma-2.4.1.3}
		Let $(\y, \boldsymbol{\omega}, p, \w, \boldsymbol{\vartheta}, q, \mathbf{u})$ be a solution to the system (\ref{2.2.2.11}-\ref{2.2.2.15}), with corresponding discrete approximation $(\y_h, \boldsymbol{\omega}_h, p_h, \w_h, \boldsymbol{\vartheta}_h, q_h, \mathbf{u}_h)$. Then, for a constant $C>0$ independent of $h$, we have:
		\begin{align}
			\label{2.4.1.7}	\|\u-\u_h\|_{0} \le Ch \|\u\|_{1}.
		\end{align}
	\end{theorem}
	\begin{proof}
		For $\u_h \in \boldsymbol{\mathcal{A}_{d}}$, let $(\y(\u_h), \boldsymbol{\omega} (\u_h), p(\u_h)) \in \boldsymbol{V} \times \boldsymbol{W} \times Q$ denotes the solution to the problem
		\begin{subequations}
			\begin{align}
				\label{2.4.1.8}	\boldsymbol{\mathcal{A}}((\mathbf{y}(\u_h),\boldsymbol{\omega}(\u_h)), (\v, \boldsymbol{\theta})) + \boldsymbol{\mathcal{B}}((\v, \boldsymbol{\theta}),p(\u_h)) &= (\mathbf{f} + \mathbf{u}_h, \mathbf{v}) &&\forall \ (\v, \boldsymbol{\theta}) \in \boldsymbol{V} \times \boldsymbol{W}, \\
				\label{2.4.1.9}	\hspace{2.87cm}\boldsymbol{\mathcal{B}}((\mathbf{y}(\u_h),\boldsymbol{\omega}(\u_h)), \phi) &= 0 &&\forall \ \phi \in Q.
			\end{align}
		\end{subequations} 
		Let $(\w(\u_h), \boldsymbol{\vartheta} (\u_h), q(\u_h)) \in \boldsymbol{V} \times \boldsymbol{W} \times Q$ be the solution to the following problem:
		\begin{subequations}
			\begin{align}
				\label{2.4.1.10}	\boldsymbol{\mathcal{C}}((\mathbf{w}(\u_h),\boldsymbol{\vartheta}(\u_h)), (\z, \boldsymbol{\tau})) - \boldsymbol{\mathcal{B}}((\z, \boldsymbol{\tau}),q(\u_h)) &= (\mathbf{y}(\u_h) - \y_{d}, \mathbf{z}) + (\boldsymbol{\omega}(\u_h) - \boldsymbol{\omega}_{d}, \boldsymbol{\tau}), \\
				\label{2.4.1.11} \hspace{2.87cm}\boldsymbol{\mathcal{B}}((\mathbf{w}(\u_h),\boldsymbol{\vartheta}(\u_h)), \psi) &= 0,
			\end{align}
		\end{subequations}
		for all $(\mathbf{z}, \boldsymbol{\tau}, \psi) \in \boldsymbol{V} \times \boldsymbol{W} \times Q$. The reduced functional $\mathcal{F}$ 
		exhibits the following properties:
		\begin{subequations}\label{2.4.1.111}
		\begin{align}
		\label{2.4.1.111a}	\mathcal{F}'(\u) (\boldsymbol{\lambda}) &= - \ \gamma (\u,\boldsymbol{\lambda}) + (\boldsymbol{\lambda},\w) \hspace{1.8cm} &&\forall \ \boldsymbol{\lambda} \in \boldsymbol{\mathcal{A}_{d}}, \\ 
		\label{2.4.1.111b}	\mathcal{F}'(\u_{h}) (\boldsymbol{\lambda}_{h}) &= - \ \gamma (\u_{h},\boldsymbol{\lambda}_{h}) + (\boldsymbol{\lambda}_{h},\w(\u_h)) \hspace{0.5cm} &&\forall \ \boldsymbol{\lambda}_{h} \in \boldsymbol{\mathcal{A}_{dh}}.
		\end{align}
	\end{subequations}
		By using second-order conditions as illustrated in \cite{HSAK}, we have the following identities:
		\begin{align*}
			-(\gamma \u,\u - \u_{h}) + (\u - \u_{h},\w) &= \mathbf{0} = -(\gamma \u,\u - \Pi_{h} \u) + (\u - \Pi_{h} \u,\w), \\
			-(\gamma \u_{h},\u_{h} - \Pi_{h} \u) + (\u_{h} - \Pi_{h} \u,\w_{h}) &= \mathbf{0},
		\end{align*}
		where $\Pi_h$ represents the standard $L^2$-projection operator. Additionally, we also have:
		\begin{align}
		\label{2.4.1.12} 	\gamma \|\u-\u_{h}\|_{0}^{2} &\le \mathcal{F}'(\u)(\u-\u_h) - \mathcal{F}'(\u_h)(\u-\u_h) =  {\gamma(\u-\u_h,\Pi_h \u- \u)} \\
		\nonumber	& \ \ + {(\Pi_h \u - \u + \u - \u_{h},\w_h-\w(\u_h))} + {(\Pi_h \u - \u,\w(\u_h)-\w)}.
		\end{align}
		Using the Cauchy-Schwarz and Young's inequalities for all the terms, we obtain the following:
		\begin{align}
			\label{2.4.1.13}	\gamma \|\u-\u_{h}\|_{0}^{2} &\le \frac{\gamma}{\mu} \|\u-\Pi_h \u\|_{0}^{2} + 2\gamma \mu \|\u-\u_h\|_{0}^{2} + \frac{\mu}{\gamma}  \|\w-\w(\u_h)\|_{0}^{2} \\
		\nonumber	& \ \ \ + \bigg( \frac{\mu}{2\gamma}+\frac{1}{4\gamma \mu}\bigg)\|\w(\u_h)-\w_h\|_{0}^{2},
		\end{align}
		where $\mu$ is a positive constant. Subtracting (\ref{2.4.1.10}-\ref{2.4.1.11}) from (\ref{2.2.2.13}-\ref{2.2.2.14}), we get
		\begin{align*}
			\boldsymbol{\mathcal{C}}((\w-\mathbf{w}(\u_h),\boldsymbol{\vartheta} - \boldsymbol{\vartheta}(\u_h)), (\z, \boldsymbol{\tau})) - \boldsymbol{\mathcal{B}}((\z, \boldsymbol{\tau}),q-q(\u_h)) &= (\mathbf{y} - \y(\u_h), \mathbf{z}) + (\boldsymbol{\omega} - \boldsymbol{\omega}(\u_h), \boldsymbol{\tau}), \\
			\hspace{2.87cm}\boldsymbol{\mathcal{B}}((\w-\mathbf{w}(\u_h),\boldsymbol{\vartheta} - \boldsymbol{\vartheta}(\u_h)), \psi) &= 0,
		\end{align*}
		for all $(\z, \boldsymbol{\tau}, \psi) \in \boldsymbol{V} \times \boldsymbol{W} \times Q$.
		Substituting $\z = \w-\w(\u_h), \boldsymbol{\tau} = \boldsymbol{\vartheta} - \boldsymbol{\vartheta}(\u_h)$ and $\psi = q-q(\u_h)$, we have
		\begin{align*}
			\boldsymbol{\mathcal{C}}((\w-\mathbf{w}(\u_h),\boldsymbol{\vartheta} - \boldsymbol{\vartheta}(\u_h)), (\w-\mathbf{w}(\u_h), \boldsymbol{\vartheta} - \boldsymbol{\vartheta}(\u_h))) = (\y - \mathbf{y}(\u_h), \w-\w(\u_h)) + (\boldsymbol{\omega} - \boldsymbol{\omega}(\u_h), \boldsymbol{\vartheta} - \boldsymbol{\vartheta}(\u_h)).
		\end{align*}
		Using the constitutive relation and the norm definition \eqref{2.2.2.5}, we obtain
\begin{align}
	\label{2.4.1.14}  |\!|\!|\w-\w(\u_h)|\!|\!|_{1} \le (C_{a}^{c})^{-1} |\!|\!|\y-\y(\u_h)|\!|\!|_{1} \le \eta \|\u-\u_h\|_{0},
\end{align}
		where $\eta = (C_{a}^{c})^{-1}$, with $C_{a}^{c}$ being coercivity constant. Using the estimate (\ref{2.4.1.14}) in (\ref{2.4.1.13}), we have
		\begin{align*}
			\gamma \|\u-\u_{h}\|_{0}^{2} &\le \frac{\gamma}{\mu} \|\u-\Pi_h \u\|_{0}^{2} + \mu \bigg(2\gamma + \frac{\eta^{2}}{\gamma}\bigg) \|\u-\u_h\|_{0}^{2} + \bigg( \frac{\mu}{2\gamma}+\frac{1}{4\gamma \mu}\bigg)\|\w(\u_h)-\w_h\|_{0}^{2}.
		\end{align*}
		By selecting $\mu = \frac{\gamma}{2} \left(2\gamma + \eta^{2} \gamma^{-1}\right)^{-1}$, we get
		\begin{align*}
			\|\u-\u_{h}\|_{0}^{2} &\le \bigg(8 \gamma + \frac{4\eta^{2}}{\gamma^{2}}\bigg) \|\u- \Pi_h \u\|_{0}^{2} + \bigg( \frac{\left(2\gamma + \eta^{2} \gamma^{-1}\right)^{-1}}{ 2 \gamma}+\frac{2\gamma + \eta^{2} \gamma^{-1}}{\gamma^{3}}\bigg)\|\w(\u_h)-\w_h\|_{0}^{2}.
		\end{align*}
		Use of estimates for the second term and the $L^2$-projection gives
		\begin{align*}
			\|\u-\u_{h}\|_{0} &\le C_1 \bigg(\sum_{K \in \mathcal{T}_h} h_{K}^{2} \|\u\|_{1,K}^{2}\bigg)^{1/2} \le Ch \|\u\|_{1}.
		\end{align*}
	\end{proof}
	\begin{theorem}\label{lem: Lemma-2.4.1.4}
		Let $(\y, \boldsymbol{\omega}, p, \w, \boldsymbol{\vartheta}, q, \mathbf{u})$ be a solution to the system (\ref{2.2.2.11}-\ref{2.2.2.15}), with the discrete approximation $(\y_h, \boldsymbol{\omega}_h, p_h, \w_h, \boldsymbol{\vartheta}_h, q_h, \mathbf{u}_h)$. Then, for positive constants $C_s$ and $C_a$, we have the estimates:
		\begin{align}
			\label{2.4.1.15}	\|(\y - \y_h,\boldsymbol{\omega} - \boldsymbol{\omega}_h)\| + \|p-p_h\|_{0} &\le C_s h^{s}\big(\|\y\|_{s+1} + \|\boldsymbol{\omega}\|_{s+1} + \|p\|_{s} + \|\u\|_{s} \big), \\
			\label{2.4.1.16}	\|(\w - \w_h, \boldsymbol{\vartheta} - \boldsymbol{\vartheta}_h)\| + \|q-q_h\|_{0} &\le C_a h^{s}\big(\|\y\|_{s+1} + \|p\|_{s} + \|\u\|_{s} + \|\w\|_{s+1} + \|\boldsymbol{\vartheta}\|_{s+1} + \|q\|_{s}\big).
		\end{align}
	\end{theorem}
	\begin{proof}
			To prove the first estimate, we use Triangle inequality to get
			\begin{align*}
				\|(\y - \y_h,\boldsymbol{\omega} - \boldsymbol{\omega}_h)\| + \|p-p_h\|_{0} \le \ &\ \|(\y- \y_h(\u), \boldsymbol{\omega} -\boldsymbol{\omega}_h(\u))\| + \|p-p_h(\u)\|_{0} \\
				&+ \|(\y_h(\u) - \y_h,\boldsymbol{\omega}_h(\u) - \boldsymbol{\omega}_h)\| + \|p_h(\u)-p_h\|_{0}.
			\end{align*}
		Now, by a use of the estimates derived in \eqref{2.4.1.3}, \eqref{2.4.1.5.} and Theorem~\ref{Lemma-2.4.1.3}, we obtain the first estimate. Similarly, a use of Triangle inequality for the second estimate gives 
		\begin{align*}
				\|(\w - \w_h, \boldsymbol{\vartheta} - \boldsymbol{\vartheta}_h)\|  + \|q-q_h\|_{0} \le \ &\ 	\|(\w - \w_h(\y),\boldsymbol{\vartheta} -\boldsymbol{\vartheta}_h(\y))\| + \|q-q_h(\y)\|_{0} \\
			& + 	\|(\w_h(\y) - \w_h,\boldsymbol{\vartheta}_h(\y) - \boldsymbol{\vartheta}_h)\| + \|q_h(\y)-q_h\|_{0}.
		\end{align*}
	By using the estimates derived in \eqref{2.4.1.4}, \eqref{2.4.1.6} and Theorem~\ref{Lemma-2.4.1.3}, we get the desired estimate.
	\end{proof}
	\subsubsection{A posteriori error estimates}\label{A posteriori error estimates} 
	A posteriori error estimators are computable quantities that rely solely on the approximate solution and known data. They provide insight into the local accuracy of the approximate solution, making them a crucial component of adaptive finite element methods. The iterative techniques aim to enhance the approximation's quality while maintaining an efficient allocation of computational resources. In this subsection, we develop a residual-based a posteriori error estimator and illustrate its reliability and efficiency in the context of the optimal control problem. The analysis is limited to the two-dimensional scenario, using continuous finite element approximations for vorticity. However, extending the analysis to three dimensions and incorporating discontinuous vorticity is straightforward.\\ \\
	Let $(\y, \omega, p, \w, \vartheta, q, \mathbf{u}) \in \boldsymbol{V} \times \boldsymbol{W} \times Q \times \boldsymbol{V} \times \boldsymbol{W} \times Q \times \boldsymbol{\mathcal{A}_{d}}$ and $(\y_h, \omega_h, p_h, \w_h, \vartheta_h, q_h, \mathbf{u}_h) \in \boldsymbol{V}_h \times \boldsymbol{W}_{h}^{1} \times Q_h \times \boldsymbol{V}_h \times \boldsymbol{W}_{h}^{1} \times Q_h \times \boldsymbol{\mathcal{A}_{dh}}$ be the unique solutions to the continuous and discrete problems (\ref{2.2.2.11}-\ref{2.2.2.15}) and (\ref{2.3.2.6}-\ref{2.3.2.10}), respectively.
	For an element $K \in \mathcal{T}_h$, we introduce local error indicators denoted as $\eta_{c,K}^{\y}$, $\eta_{c,K}^{\w}$, and $\eta_{c,K}^{\u}$, where:
	\begin{align*}
		\begin{cases}
			\big(\eta_{c,K}
			^{\y}\big)^{2} &:= h_{K}^{2}\|\f + \u_h + 2 \boldsymbol{\varepsilon}(\y_h) \nabla \nu - \nu \  \textbf{curl}(\omega_h) - (\boldsymbol{\beta} \cdot \nabla) \y_h - \sigma \y_h - \nabla p_h\|_{0,K}^{2} \\
			& \ \ \ \ + \|\omega_h - \textbf{curl}(\y_h)\|_{0,K}^{2}+ \|\nabla \cdot \y_h\|_{0,K}^{2}, \\ 
			\big(\eta_{c,K}^{\w}\big)^{2} &:= h_{K}^{2}\|\y_h-\y_{d} + 2 \boldsymbol{\varepsilon}(\w_h) \nabla \nu - \nu \  \textbf{curl}(\vartheta_h) + (\boldsymbol{\beta} \cdot \nabla) \w_h + (\nabla \cdot \boldsymbol{\beta}) \w_h - \sigma \w_h + \nabla q_h\|_{0,K}^{2}  \\
			& \ \ \ \ + \|\vartheta_h - \textbf{curl}(\w_h) - \omega_h + \omega_d\|_{0,K}^{2} + \|\nabla \cdot \w_h\|_{0,K}^{2}, \\
			(\eta_{c,K}^{\u})^{2}  &:= h_{K}^{2} \|\w_h + \gamma \u_h\|_{0,K}^{2}.
		\end{cases} 
	\end{align*}
	We define the \textbf{global error estimators} $\eta_c^{\y}, \eta_c^{\w}$, and  $\eta_c^{\u}$ as:
	\begin{align*}
		(\eta_c^{\y})^2 &:= \sum_{K \in \mathcal{T}_h}(\eta_{c,K}^{\y})^{2} , &&  (\eta_c^{\w})^2 := \sum_{K \in \mathcal{T}_h}(\eta_{c,K}^{\w})^{2}, \quad \quad \quad \quad (\eta_c^{\u})^{2} := \sum_{K \in \mathcal{T}_h}(\eta_{c,K}^{\u})^{2}.
	\end{align*}
	\textbf{Reliability:} Firstly, we establish a reliability estimate for the a posteriori error estimator of the optimal control problem. The continuous dependence estimate expressed in \eqref{2.2.2.7} is essentially a counterpart to the global inf-sup condition for the continuous formulation outlined in \eqref{2.2.2.1}. Thus, employing this estimate for the error $(\y-\y_h, \omega-\omega_h, p-p_h)$ yields: 
	\begin{align}\label{2.5.1.2}
		\|(\y - \y_h,\omega - \omega_h)\| + \|p-p_h\|_{0} \precsim \underset{(\v, \theta, \phi) \in \boldsymbol{H}_{0}^{1}(\Omega) \times L^{2}(\Omega) \times L_{0}^{2}(\Omega)}{\sup} \frac{\mathcal{R}(\v, \theta, \phi)}{\|(\v, \theta, \phi)\|},
	\end{align}
	where for all $(\v, \theta, \phi) \in \boldsymbol{H}_{0}^{1}(\Omega) \times L^{2}(\Omega) \times L_{0}^{2}(\Omega)$, the residual functional $\mathcal{R}$ is defined by
	\begin{align*}
			\mathcal{R}(\v, \theta, \phi) = \boldsymbol{\mathcal{A}}((\mathbf{y}-\y_h,\omega-\omega_h), (\v, \theta)) + \boldsymbol{\mathcal{B}}((\v, \theta),p-p_h) + \boldsymbol{\mathcal{B}}((\mathbf{y}-\y_h,\omega-\omega_h), \phi).
	\end{align*}
	\begin{lemma}\label{Lemma-2.5.1.2.}
		Let $(\y, \omega, p, \w, \vartheta, q)$ and $(\y_h, \omega_h, p_h, \w_h, \vartheta_h, q_h)$ be the solutions to the continuous and discrete problems (\ref{2.2.2.11}-\ref{2.2.2.14}) and (\ref{2.3.2.6}-\ref{2.3.2.9}), respectively. Then, the following estimates hold true:
		\begin{align}
			\label{2.5.1.4}	\|(\y - \y_h,\omega - \omega_h)\| + \|p-p_h\|_{0} &\precsim \eta_c^{\y}, \\
			\label{2.5.1.5}	\|(\w - \w_h, \vartheta - \vartheta_h)\| + \|q-q_h\|_{0} &\precsim \eta_c^{\w}.
		\end{align}
	\end{lemma}
	\begin{proof}
		The first estimate is derived directly from the proof presented in \cite[Theorem~4.1]{VANAYA}. A similar approach applies to the subsequent estimate for the co-state.
	\end{proof}
	Now we state and prove the main reliability result of this section for the optimal control problem.
	\begin{theorem}\label{Theorem-2.5.1.3.}
		Let $(\y, \omega, p, \w, \vartheta, q, \u)$ and $(\y_h, \omega_h, p_h, \w_h, \vartheta_h, q_h, \u_h)$ be the solutions to the continuous and discrete optimality systems (\ref{2.2.2.11}-\ref{2.2.2.15}) and (\ref{2.3.2.6}-\ref{2.3.2.10}), respectively. Then, the following \textbf{reliability} estimate holds:
		\begin{align}
			\label{2.5.1.6}	\|\u-\u_h\|_{0}+\|(\y - \y_h,\omega - \omega_h)\| + \|p-p_h\|_{0} + \|(\w - \w_h, \vartheta - \vartheta_h)\|  + \|q-q_h\|_{0} \precsim \eta_c^{\y}+ \eta_c^{\w}+\eta_c^{\u}.
		\end{align}
	\end{theorem}
	\begin{proof}
		For given $\u_h \in \boldsymbol{L}^{2}(\Omega)$, let $(\y(\u_h), \omega(\u_h), p(\u_h)) \in \boldsymbol{V} \times \boldsymbol{W} \times Q$ be a solution of the system (\ref{2.4.1.8}-\ref{2.4.1.9}) and $(\w(\u_h), \vartheta(\u_h), q(\u_h))\in \boldsymbol{V} \times \boldsymbol{W} \times Q$ be a solution to the system (\ref{2.4.1.10}-\ref{2.4.1.11}). Subtracting these equations separately from (\ref{2.2.2.11}-\ref{2.2.2.12}) and (\ref{2.2.2.13}-\ref{2.2.2.14}), respectively, we obtain the following system:
		\begin{subequations}
			\begin{align}
				\label{2.5.1.7}	\boldsymbol{\mathcal{A}}((\y-\mathbf{y}(\u_h),\omega - \omega(\u_h)), (\v, \theta)) + \boldsymbol{\mathcal{B}}((\v, \theta),p-p(\u_h)) &= (\u - \mathbf{u}_h, \mathbf{v}), \\
				\label{2.5.1.8}	\hspace{2.87cm}\boldsymbol{\mathcal{B}}((\y-\mathbf{y}(\u_h),\omega - \omega(\u_h)), \phi) &= 0,\\
				\label{2.5.1.9}	\boldsymbol{\mathcal{C}}((\w-\mathbf{w}(\u_h), \vartheta - \vartheta(\u_h)), (\z, \tau)) - \boldsymbol{\mathcal{B}}((\z, \tau), q - q(\u_h)) &= (\y - \mathbf{y}(\u_h), \mathbf{z}) + (\omega - \omega(\u_h), \tau), \\
				\label{2.5.1.10} \hspace{2.87cm}\boldsymbol{\mathcal{B}}((\w - \mathbf{w}(\u_h), \vartheta - \vartheta(\u_h)), \psi) &= 0,
			\end{align} 
		\end{subequations}
		for all $(\mathbf{v}, \theta, \phi), (\mathbf{z}, \tau, \psi) \in \boldsymbol{V} \times \boldsymbol{W} \times Q$. Upon Substituting $\v = \w-\w(\u_h),\ \theta = \vartheta - \vartheta(\u_h), \ \phi = q - q(\u_h),\ \z = \y-\y(\u_h), \tau = \omega - \omega(\u_h)$, and $\psi = p - p(\u_h)$ into these set of equations, we obtain:
		\begin{align}
		\label{2.5.1.11} 	 \boldsymbol{\mathcal{A}}((\y-\mathbf{y}(\u_h),\omega - \omega(\u_h), (\w-\w(\u_h), \vartheta - \vartheta(\u_h))) &=  (\u-\u_h,\w-\w(\u_h)) \\
		\nonumber&= \|\y-\y(\u_h)\|_{0}^{2} + \|\omega - \omega(\u_h)\|_{0}^{2} \ge 0.
		\end{align}
		To establish a connection between the control and the co-state, consider the following:
		\begin{align*}
			(\mathcal{F}'(\u),\v) &= (\gamma \u+ \w, \v),  \qquad \qquad (\mathcal{F}'(\u_h),\v) = (\gamma \u_h + \w(\u_h), \v), \qquad \forall \ \v \in \boldsymbol{V}.
		\end{align*}
		After subtracting and using the substitution $\v = \u - \u_h$, we obtain
		\begin{align}
			\label{2.5.1.14}	 (\mathcal{F}'(\u)-\mathcal{F}'(\u_h),\u - \u_h) &= \gamma(\u - \u_h, \u - \u_h) + (\w-\w(\u_h),\u - \u_h).
		\end{align}
		By using (\ref{2.5.1.11}) and the variational inequality (\ref{2.3.2.10}) into (\ref{2.5.1.14}), we derive the following:
		\begin{align*}
			\nonumber   \gamma \|\u - \u_h\|_{0}^{2}  &\le (\mathcal{F}'(\u) - \mathcal{F}'(\u_h),\u - \u_h) \le -(\gamma \u_h + \w_h,\u-\u_h) \\
			\nonumber 	&\le (\w_h - \w(\u_h),\u-\u_h) - (\w_h + \gamma \u_h, \u-\v_h) - (\w_h + \gamma \u_h, \v_h - \u_h) \\
			&\le (\w_h - \w(\u_h),\u-\u_h) - (\w_h + \gamma \u_h, \u-\v_h).
		\end{align*}
		Considering $\v_h = \Pi_{h} \u \in \boldsymbol{U}_{ad,h}$, we apply Young’s inequality to derive:
		\begin{align}
			\label{2.5.1.15}  \|\u - \u_h\|_{0}  \precsim \ & \eta_c^{\u} +  \|(\w_h - \w(\u_h),\vartheta_h - \vartheta(\u_h))\|.
		\end{align}
		This result establishes connection between the control and the co-state velocity-vorticity. Now, let $(\tilde{\w}, \tilde{\vartheta}, \tilde{q})$ solves (\ref{2.2.2.13}-\ref{2.2.2.14}) with $\y=\y_h$. Then, $(\w(\u_h)-\tilde{\w}, \vartheta(\u_h)- \tilde{\vartheta}, q(\u_h)-\tilde{q})$ solves 
		\begin{subequations}\label{2.5.1.15.a}
			\begin{align}
				\label{2.5.1.15.1}	\boldsymbol{\mathcal{C}}((\w(\u_h) - \tilde{\w}, \vartheta(\u_h) - \tilde{\vartheta}), (\z, \tau)) - \boldsymbol{\mathcal{B}}((\z, \tau),  q(\u_h)-\tilde{q}) &= (\mathbf{y}(\u_h) - \y_h, \mathbf{z}) + (\omega(\u_h) - \omega_h, \tau), \\
				\label{2.5.1.15.2}	\hspace{2.87cm}\boldsymbol{\mathcal{B}}((\mathbf{w}(\u_h) - \tilde{\w}, \vartheta(\u_h) - \tilde{\vartheta}), \psi) &= 0.
			\end{align}
		\end{subequations}
		By an application of Theorem \ref{Theorem-2.2.2.4.} and Lemma~\ref{Lemma-2.5.1.2.},
		  we obtain:
		\begin{align}
			\label{2.5.1.17}	\|(\w(\u_h)-\tilde{\w},\vartheta(\u_h)-\tilde{\vartheta})\| + \|q(\u_h)-\tilde{q}\|_{0} &\precsim \|\mathbf{y}(\u_h) - \y_h\|_{0} + \|\omega(\u_h) - \omega_h\|_{0},\\
			\label{2.5.1.17.1}	\|(\tilde{\w}-\w_h,\tilde{\vartheta} - \vartheta_h)\| + \|\tilde{q}-q_h\|_{0} &\precsim \eta_{c}^{\w}.
		\end{align}
		Using the Triangle Inequality and (\ref{2.5.1.17}-\ref{2.5.1.17.1}), we have
		\begin{align}
		  \label{2.5.1.18} 	 \|(\w_h - \w(\u_h),\vartheta_h - \vartheta(\u_h))\| + \|q(\u_h)-q_h)\|_{0} \precsim  \|(\y_h - \y(\u_h),\omega_h - \omega(\u_h))\| + \eta_c^{\w}.
		\end{align}
		For the state equation, Lemma \ref{Lemma-2.5.1.2.} leads to:
		\begin{align}
			\label{2.5.1.19}	\|(\y_h - \y(\u_h),\omega_h - \omega(\u_h))\| + \|p(\u_h)-p_h\|_{0} \precsim \eta_c^{\y}.
		\end{align}
		Using the estimate \eqref{2.2.2.7} and substituting (\ref{2.5.1.18}-\ref{2.5.1.19}) into (\ref{2.5.1.15}), we achieve the desired estimate.
	\end{proof}
\textbf{Efficiency:} Now, we demonstrate the effectiveness of the a posteriori error estimator by conventional element and edge bubble function technique. This bound shows the relationship between the total error and corresponding approximation, showing the effectiveness of the computational approach.\\
	For an element $K \in \mathcal{T}_h$ and an edge $E \in \mathcal{E}(\mathcal{T}_h),$ let $\chi_{K}$ and $\chi_{E}$ be the interior and edge bubble functions, respectively, as defined in \cite{AMO}. Let $\chi_{K} \in \P_{3}(K)$ with support($\chi_{K}$) $\subset K,\ \chi_{K} = 0$ on $\partial K$, and $0 \le \chi_{K} \le 1$ in $K$. Similarly, let $\chi_{E} \in \P_{2}(K)$ with support($\chi_{E}$) $\subset \Omega_{E}:= \{K' \in \mathcal{T}_h: E \in \mathcal{E}(K')\},\ \chi_{E} = 0$ on $\partial K \setminus E$, and $0 \le \chi_{K} \le 1$ in $\Omega_{E}$. We define an extension operator $\boldsymbol{E}: C^{0}(E) \rightarrow C^{0}(T)$ that satisfies $\boldsymbol{E}(q) \in \P_{k}(K)$ and $\boldsymbol{E}(q)|_E = q $ for all $q \in \P_{k}(E)$ and for all $k \in \N \cup \{0\}$.\\
	The element and edge bubble functions $\chi_{K}$ and $\chi_{E}$, and the extension operator $\boldsymbol{E}$ satisfies the following properties proven in \cite{AMO, VERFP}.
	\begin{lemma}\label{Lemma-4.5}
		(i) For $K \in \mathcal{T}_h$ and $v \in \P_{k}(K)$, there exists a positive constant $C_1$ such that
		\begin{align*}
			C_{1}^{-1} \|v\|_{0,K}^{2} \le \int_{K} \chi_{K} v^{2} \ dx \le C_1 \|v\|_{0,K}^{2}, \quad C_{1}^{-1} \|v\|_{0,K}^{2} \le \|\chi v\|_{0,K}^{2} + h_{K}^{2} |\chi v|_{1,K}^{2} \le C_1 \|v\|_{0,K}^{2}.
		\end{align*}
		(ii) For $E \in \mathcal{E}(\mathcal{T}_h)$ and $v \in \P_{k}(E)$, there exists a positive constant $C_2$ such that
		\begin{align*}
			C_{2}^{-1} \|v\|_{0,E}^{2} \le \int_{E} \chi_{E} v^{2} \ ds \le C_2 \|v\|_{0,E}^{2}.
		\end{align*}
		(iii) For $K \in \mathcal{T}_h$, $e \in \mathcal{E}(\mathcal{T}_h)$ and $v \in \P_{k}(E)$, there exists a positive constant $C_3$ such that
		\begin{align*}
			\|\chi_{E} \boldsymbol{E}(v)\|_{0,K}^{2} \le C_3 h_{E} \|v\|_{0,E}^{2}.
		\end{align*}
	\end{lemma}
	\begin{lemma}\label{Lemma-2.5.1.4}
		Let $(\y, \omega, p, \w, \vartheta, q)$ and $(\y_h, \omega_h, p_h, \w_h, \vartheta_h, q_h)$ be the solutions to the continuous and discrete problems (\ref{2.2.2.11}-\ref{2.2.2.14}) and (\ref{2.3.2.6}-\ref{2.3.2.9}), respectively. Then, the following estimates hold true:
		\begin{align}
			\label{2.5.1.20}	\eta_c^{\y} &\precsim	\|(\y-\y_h,\omega - \omega_h)\| + \|p-p_h\|_{0}, \\
			\label{2.5.1.21}	\eta_c^{\w} &\precsim \|(\w-\w_h,\vartheta - \vartheta_h)\| + \|q-q_h\|_{0}.
		\end{align}
	\end{lemma}
	\begin{proof}
		The first estimate is directly deduced from the demonstration outlined in \cite[Theorem~4.2]{VANAYA}. A similar approach can be applied to establish the subsequent approximation for the co-state.
	\end{proof}
	\begin{theorem}\label{Theorem-2.5.1.5}
		Let $(\y, \omega, p, \w, \vartheta, q, \u)$ and $(\y_h, \omega_h, p_h, \w_h, \vartheta_h, q_h, \u_h)$ be the solutions to the systems (\ref{2.2.2.11}-\ref{2.2.2.15}) and (\ref{2.3.2.6}-\ref{2.3.2.10}), respectively. Then, we have the following \textbf{efficiency} estimate:
		\begin{align}
			\label{2.5.1.22} \eta_c^{\y}+ \eta_c^{\w}+\eta_c^{\u} \precsim 	\|\u-\u_h\|_{0}+\|(\y-\y_h, \omega - \omega_h)\| + \|p-p_h\|_{0} + \|(\w-\w_h,\vartheta - \vartheta_h)\|  + \|q-q_h\|_{0}.
		\end{align}
	\end{theorem}
	\begin{proof}
		Firstly, suppose that $(\y(\u_h), \omega(\u_h), p(\u_h)) \in \boldsymbol{V} \times \boldsymbol{W} \times Q$ is a solution to the system (\ref{2.4.1.8}-\ref{2.4.1.9}).
		By utilizing Lemma~\ref{Lemma-2.5.1.4}, we obtain
		\begin{align}
			\label{2.5.1.23} \eta_c^{\y} \precsim	 \|(\y(\u_h)-\y_h,\omega(\u_h) - \omega_h)\| + \|p(\u_h)-p_h\|_{0}.
		\end{align}
		Similarly, assume that $(\w(\u_h), \vartheta(\u_h), q(\u_h))\in \boldsymbol{V} \times \boldsymbol{W} \times Q$ is a solution to the system (\ref{2.4.1.10}-\ref{2.4.1.11}).
		Then, by Lemma~\ref{Lemma-2.5.1.4}, we have
		\begin{align}
			\label{2.5.1.24}  \eta_c^{\w} & \precsim \|(\w(\u_h)-\w_h,\vartheta(\u_h) - \vartheta_h)\| + \|q(\u_h)-q_h\|_{0}.
		\end{align}
		Now, an application of Triangle inequality and Theorem \ref{Theorem-2.2.2.4.} gives
		\begin{align}
			\nonumber  \eta_c^{\y} &\precsim \|(\y-\y_h,\omega - \omega_h)\| + \|p-p_h\|_{0} + \|(\y-\y(\u_h),\omega - \omega(\u_h))\| + \|p-p(\u_h)\|_{0} \\
			\label{2.5.1.25}	&\precsim \|(\y-\y_h,\omega - \omega_h)\| + \|p-p_h\|_{0} + \|\u-\u_h\|_{0}.
		\end{align}
		Similarly, from \eqref{2.2.2.7} and (\ref{2.5.1.24}), we obtain
		\begin{align}
			\nonumber  \eta_c^{\w} &\precsim \|(\w-\w_h,\vartheta - \vartheta_h)\| + \|q-q_h\|_{0} + \|(\w-\w(\u_h),\vartheta - \vartheta(\u_h))\| + \|q-q(\u_h)\|_{0} \\
			\label{2.5.1.26}	&\precsim \|(\w-\w_h,\vartheta - \vartheta_h)\| + \|q-q_h\|_{0} + \|\u-\u_h\|_{0}.
		\end{align}
		Combining equations (\ref{2.5.1.25}) and (\ref{2.5.1.26}), we attain the desired efficiency bound.
	\end{proof}
	\subsection{Discontinuous Galerkin formulation}\label{Discontinuous Galerkin method}
	In this subsection, we present a discontinuous Galerkin scheme for the optimal control problem and derive a priori and a posteriori error estimates. For the discretization $\mathcal{T}_h$ of $\Omega$,
	let $K^{+}$ and $K^{-}$ be two adjacent elements (sharing an edge $E \in \mathcal{E}(\mathcal{T}_h)$) with the outward unit normal vectors $\mathbf{n}^{+}$ and $\mathbf{n}^{-}$, respectively. 
	For a vector field $\v$ and a scalar $\phi$ with traces $\v^{\pm}$ and $\phi^{\pm}$ on $K^{\pm}$, respectively, the tangential jump $([\![\cdot]\!]_{T})$, normal jump $([\![\cdot]\!]_{N})$ and average $(\{\!\!\{\cdot\}\!\!\})$ across an edge $E$ are defined as:
	\begin{align*}
		[\![\v]\!]_{T} &:= \v^{+} \times \mathbf{n}^{+} + \v^{-} \times \mathbf{n}^{-}, \ \ \ \ \
		&&[\![\v]\!]_{N} := \v^{+} \cdot \mathbf{n}^{+} + \v^{-} \cdot \mathbf{n}^{-}, \ \ \ \ \
		[\![\phi]\!] := \phi^{+} \mathbf{n}^{+} + \phi^{-} \mathbf{n}^{-},\\
		\{\!\!\{\v\}\!\!\} &:= \frac{\v^{+}+\v^{-}}{2}, \ \ \ \ \ &&\{\!\!\{\phi\}\!\!\} := \frac{\phi^{+}+\phi^{-}}{2}.
	\end{align*}
	On all the boundary edges, $[\![\v]\!]_{T} = \v \times \mathbf{n}, \ [\![\v]\!]_{N} = \v \cdot \mathbf{n}, \ [\![q]\!] := q \mathbf{n}, \ \{\!\!\{\v\}\!\!\} = \v,$ and $\{\!\!\{q\}\!\!\} = q$.
	The inflow and outflow parts of the boundary \(\Gamma\) are \(\Gamma_{\text{in}} = \{\mathbf{x} \in \Gamma : \boldsymbol{\beta} \cdot \n < 0\}\) and \(\Gamma_{\text{out}} = \{\mathbf{x} \in \Gamma : \boldsymbol{\beta} \cdot \n \ge 0\}\), respectively, and the inflow and outflow parts of $\partial K$ are \(\partial K_{\text{in}} = \{\mathbf{x} \in \partial K : \boldsymbol{\beta} \cdot \n_{K} < 0\}\) and \(\partial K_{\text{out}} = \{\mathbf{x} \in \partial K : \boldsymbol{\beta} \cdot \n_{K} \ge 0\}\).
	 For $k \ge 0$, the discontinuous finite-dimensional spaces for velocity, vorticity and pressure variables are defined as:
	\begin{align}
	\label{2.3.3.1}	\boldsymbol{V}_h &:= \{\v_h \in \boldsymbol{L}^{2}(\Omega)  && \hspace{-3.1cm}: \v_h \in [\P_{k+1}(K)]^{d} \ \ \ \ \forall K \in \mathcal{T}_h\},\\
		\label{2.3.3.2}	\boldsymbol{W}_h &:= \{\boldsymbol{\theta}_h \in [L^{2}(\Omega)]^{\frac{d(d-1)}{2}} && \hspace{-3.1cm}: \boldsymbol{\theta}_h \in [\P_{k}(K)]^{\frac{d(d-1)}{2}} \hspace{0.15cm} \forall K \in \mathcal{T}_h\},\\
	\label{2.3.3.3}	Q_h &:= \{\phi_h \in L_{0}^{2}(\Omega) && \hspace{-3.1cm}: \phi_h \in \P_{k}(K) \hspace{1.22cm} \forall K \in \mathcal{T}_h\}.
\end{align}
	For the discrete subspaces (\ref{2.3.3.1}-\ref{2.3.3.3}) and piecewise constant discretization of control, the DG discrete formulation (see \cite[Section~4.2]{VABAM})  for the optimal control problem is to find $(\y_h, \boldsymbol{\omega}_h, p_h, \w_h, \boldsymbol{\vartheta}_h, q_h, \mathbf{u}_h) \in  \boldsymbol{V}_h \times \boldsymbol{W}_h \times Q_h \times  \boldsymbol{V}_h \times \boldsymbol{W}_h \times Q_h \times \boldsymbol{\mathcal{A}_{dh}}$, such that:
	\begin{subequations}
		\begin{align}
			\label{2.3.3.4} \boldsymbol{\mathcal{A}}_{DG}((\mathbf{y}_h,\boldsymbol{\omega}_h), (\v_h, \boldsymbol{\theta}_h)) + \boldsymbol{\mathcal{O}}(\boldsymbol{\beta}; \y_h, \v_h) + \boldsymbol{\mathcal{B}}_{DG}((\v_h, \boldsymbol{\theta}_h),p_h) &= (\mathbf{f} + \mathbf{u}_h, \mathbf{v}_h), \\
			\label{2.3.3.5} \hspace{1.57cm} \boldsymbol{\mathcal{B}}_{DG}((\mathbf{y}_h,\boldsymbol{\omega}_h), \phi_h) &= 0, \\
			\label{2.3.3.6} \boldsymbol{\mathcal{A}}_{DG}((\mathbf{w}_h,\boldsymbol{\vartheta}_h), (\z_h, \boldsymbol{\tau}_h)) + \boldsymbol{\mathcal{O}}(\boldsymbol{\beta}; \z_h, \w_h) - \boldsymbol{\mathcal{B}}_{DG}((\z_h, \boldsymbol{\tau}_h),q_h) &=  (\mathbf{y}_h - \y_{d}, \mathbf{z}_h) + (\boldsymbol{\omega}_h - \boldsymbol{\omega}_{d}, \boldsymbol{\tau}_h),\\
			\label{2.3.3.7} \hspace{1.56cm}\boldsymbol{\mathcal{B}}_{DG}((\mathbf{w}_h,\boldsymbol{\vartheta}_h), \psi_h) &= 0,\\
			\label{2.3.3.8} (\mathbf{w}_h + \gamma \u_h,\tilde{\u}_h -\u_h) &\ge \mathbf{0},
		\end{align} 
	\end{subequations}
for all $(\mathbf{v}_h, \boldsymbol{\theta}_h), (\mathbf{z}_h, \boldsymbol{\tau}_h) \in \boldsymbol{V}_h \times \boldsymbol{W}_h, \ \phi_h, \psi_h \in Q_h$, and $\tilde{\u}_h \in \boldsymbol{\mathcal{A}_{dh}}$. Here, the bilinear forms $\boldsymbol{\mathcal{A}}_{DG}:[\boldsymbol{V}_h \times \boldsymbol{W}_h]^{2} \rightarrow \R, \ \boldsymbol{\mathcal{B}}_{DG}:[\boldsymbol{V}_h \times \boldsymbol{W}_h] \times Q_h \rightarrow \R,$ and $\boldsymbol{\mathcal{O}}: \boldsymbol{V}_h \times \boldsymbol{V}_h \rightarrow \R$ are defined as:
\begin{align*}
	\boldsymbol{\mathcal{A}}_{DG}((\mathbf{y}_h,\boldsymbol{\omega}_h), (\v_h, \boldsymbol{\theta}_h)) &:=  -2 \sum_{K \in \mathcal{T}_h} \int_{K} \boldsymbol{\varepsilon}(\y_h) \nabla \nu \cdot \v_h \ dx + \textcolor{red}{\sum_{K \in \mathcal{T}_h} \int_{K} \nu \ \boldsymbol{\kappa}_h \cdot \textbf{curl} (\v_h) \ dx} + \sum_{K \in \mathcal{T}_h} \int_{K} \nu \boldsymbol{\omega}_h \cdot \boldsymbol{\theta}_h \ dx\\
	& \ \ \ \  - \textcolor{red}{\sum_{K \in \mathcal{T}_h} \int_{K} \nu \ \boldsymbol{\theta}_h \cdot \textbf{curl} (\y_h) \ dx + \sum_{E \in \mathcal{E}(\mathcal{T}_h)} \int_{E} \big( \{\!\!\{\boldsymbol{\kappa}_h\}\!\!\} \cdot [\![\nu \v_h]\!]_{T} - \{\!\!\{\boldsymbol{\theta}_h\}\!\!\} \cdot [\![\nu \y_h]\!]_{T}\big) \ ds}\\
	& \ \ \ \ + \sum_{E \in \mathcal{E}(\mathcal{T}_h)}\int_{E} \big(C_{11} [\![\nu \y_h]\!]_{T} \cdot [\![\v_h]\!]_{T} + A_{11} [\![\y_h]\!]_{N} [\![\v_h]\!]_{N}  \big) ds + \textcolor{red}{\sum_{K \in \mathcal{T}_h} \int_{K} \boldsymbol{\kappa}_h \cdot (\nabla \nu \times \v_h) dx}\\
	& \ \ \ \ + \textcolor{red}{\sum_{K \in \mathcal{T}_h} \int_{K}\big(\rho_1 (\textbf{curl}(\y_h)-\boldsymbol{\omega}_h) \cdot \textbf{curl}(\v_h)
	+ \rho_2 (\nabla \cdot \y_h) \cdot (\nabla \cdot \v_h)\big) dx}, \\
	\boldsymbol{\mathcal{O}}(\boldsymbol{\beta}; \y_h, \v_h)&:= \sum_{K \in \mathcal{T}_h} \int_{K} ((\sigma - \nabla\cdot \boldsymbol{\beta})\y_h\cdot \mathbf{v}_h -  \y_h \boldsymbol{\beta}^{T}: \nabla \mathbf{v}_h) \ dx + \sum_{K \in \mathcal{T}_h} \int_{\partial K_{\text{out}} \cap \Gamma_{\text{out}}} (\boldsymbol{\beta} \cdot \boldsymbol{n}_{K}) \y_h \cdot \mathbf{v}_h \ ds \\
	& \ \ \ \ +  \sum_{K \in \mathcal{T}_h} \int_{\partial K_{\text{out}} \setminus \Gamma} (\boldsymbol{\beta} \cdot \boldsymbol{n}_{K}) \y_h \cdot (\mathbf{v}_h\textcolor{red}{-\mathbf{v}_h^{e}}) \ ds,\\
	\boldsymbol{\mathcal{B}}_{DG}((\v_h, \boldsymbol{\theta}_h),p_h) &:= - \sum_{K \in \mathcal{T}_h} \int_{K} p_h (\nabla \cdot \v_h) \ dx + \sum_{E \in \mathcal{E}(\mathcal{T}_h)} \int_{E} \big(\{\!\!\{p_h\}\!\!\} \cdot [\![\v_h]\!]_{N} + D_{11} [\![p_h]\!] \cdot [\![\phi_h]\!] \big) \ ds.
\end{align*}
\textcolor{red}{The term with parameter $D_{11}$ in $\boldsymbol{\mathcal{B}}_{DG}$ appears only for \eqref{2.3.3.5} and \eqref{2.3.3.7}, not for \eqref{2.3.3.4} and \eqref{2.3.3.6}.} Here $A_{11}, C_{11}$ and $D_{11}$ are positive and bounded stabilisation parameters defined as in \cite[Section~2.4]{LDGSTOKES}:
		\[
	A_{11}(x) = \begin{cases}
		a_{11} \max \{h_{K^{+}}^{-1}, h_{K^{-}}^{-1}\}& \text{if} \ x \in \partial K^{+} \cup \partial K^{-}\\
		a_{11} h_{K}^{-1} & \text{if} \ x \in K \cap \Gamma
	\end{cases},
	\
	 C_{11}(x) = \begin{cases}
	 	c_{11} \max \{h_{K^{+}}^{-1}, h_{K^{-}}^{-1}\}& \text{if} \ x \in \partial K^{+} \cup \partial K^{-}\\
	 	c_{11} h_{K}^{-1} & \text{if} \ x \in \partial K \cap \Gamma
	 \end{cases},
	\]
	\[
	D_{11}(x) =  \begin{cases}
		d_{11} \max \{h_{K^{+}}, h_{K^{-}}\}& \text{if} \ x \in \partial K^{+} \cup \partial K^{-} \\
		d_{11} h_K & \text{if} \ x \in \partial K \cap \Gamma
	\end{cases},
	\]
	where $a_{11}, c_{11}$ and $d_{11}$ are positive constants independent of the global mesh-size. For the analysis, we define norms on the discrete spaces $\boldsymbol{V}_h, \boldsymbol{W}_h$ and $Q_h$ as:
\textcolor{red}{
	\begin{align}\label{2.3.3.9}
	|\!|\!|\mathbf{v}_h|\!|\!|_{1,h}^{2} &:=  |\!|\!|\mathbf{v}_h|\!|\!|_{1}^{2} + |\v_h|_{j}^{2},
	& \|(\v_h, \boldsymbol{\theta}_h)\|_{h}^{2} := |\!|\!|\mathbf{v}_h|\!|\!|_{1,h}^{2} + \|\boldsymbol{\theta}_h\|_{0}^{2}, \qquad \qquad
	 &\|\phi_h\|_{h}^{2} := \|\phi_h\|_{0}^{2} + |\phi_h|_{j}^{2},
\end{align}
where the jump norms $|\v_h|_{j}$ and $|\phi_h|_{j}$ are defined as:
	\begin{align*}
	 |\v_h|_{j}^{2} &:= \sum_{E \in \mathcal{E}(\mathcal{T}_h)} \int_{E} \Big(C_{11} [\![\v_h]\!]_{T}^{2} + A_{11} [\![\v_h]\!]_{N}^{2} \Big) ds, \quad 
	 &|\phi_h|_{j}^{2} := \sum_{E \in \mathcal{E}(\mathcal{T}_h)} \int_{E} D_{11} [\![\phi_h]\!]^{2} \ ds.
\end{align*}}
\textcolor{red}{
\begin{lemma}\label{Lemma 2.3.1.11}
	(i) There exist positive constants $C_1$ and $C_2$ such that
	\begin{align*}
		|\boldsymbol{\mathcal{A}}_{DG}((\mathbf{y}_h,\boldsymbol{\omega}_h), (\v_h, \boldsymbol{\theta}_h))+\boldsymbol{\mathcal{O}}(\boldsymbol{\beta}; \y_h, \v_h)| &\le C_1 \|(\y_h, \boldsymbol{\omega}_h)\|_{h} \|(\v_h, \boldsymbol{\theta}_h)\|_{h} && \forall \ (\y_h, \boldsymbol{\omega}_h), (\v_h, \boldsymbol{\theta}_h) \in \boldsymbol{V}_h \times \boldsymbol{W}_h, \\
		|\boldsymbol{\mathcal{B}}_{DG}((\v_h, \boldsymbol{\theta}_h),\phi_h)| &\le  C_2 \|(\v_h, \boldsymbol{\theta}_h)\|_h \|\phi_h\|_{h} && \forall \ (\v_h, \boldsymbol{\theta}_h) \in \boldsymbol{V}_h \times \boldsymbol{W}_h, \phi_h \in  Q_h.
	\end{align*}
	(ii) Suppose that 
	\begin{align}\label{2.3.3.9.1}
		\|\sigma - \nabla \cdot \boldsymbol{\beta}\|_{\infty} \ge \sigma_{\boldsymbol{\beta}} > 0,\qquad \qquad \textbf{and} \qquad \qquad \sigma_{\boldsymbol{\beta}} > \frac{9\|\nabla \nu\|_{\infty}^{2}}{\nu_{0}}.
	\end{align}
Then, if we select $\kappa_1 = \frac{2}{3}\nu_{0}$ and $\kappa_2 > \frac{\nu_{0}}{6}$, there exists a positive constant $C_3$ such that
	\begin{align*}
		\boldsymbol{\mathcal{A}}_{DG}((\v_h, \boldsymbol{\theta}_h), (\v_h, \boldsymbol{\theta}_h)) +\boldsymbol{\mathcal{O}}(\boldsymbol{\beta}; \v_h, \v_h) &\ge C_{3} \|(\v_h, \boldsymbol{\theta}_h)\|_h^{2} \quad \quad \forall \ (\v_h, \boldsymbol{\theta}_h) \in \boldsymbol{V}_h \times \boldsymbol{W}_h.
	\end{align*}
\end{lemma}
\begin{proof}
	(i) Both the estimates are derived by using Lemma~\ref{Lemma-2.2.2.1.}, the Cauchy-Schwarz inequality and the definition of DG norms \eqref{2.3.3.9}.\\
	(ii) Let $(\v_h, \boldsymbol{\theta}_h) \in \boldsymbol{V}_h \times \boldsymbol{W}_h$. By a use of Lemma~\ref{Lemma-2.2.2.1.} and the Young's inequality, we have
	\begin{align}
	\nonumber	\left|-2 \sum_{K \in \mathcal{T}_h} \int_{K} \boldsymbol{\varepsilon}(\v_h) \nabla \nu \cdot \v_h \ dx \right| &\le 2\|\nabla \nu\|_{\infty} \bigg(\frac{\nu_{0}}{12\|\nabla \nu\|_{\infty}}\|\nabla \v_h\|_{0}^{2} + \frac{3\|\nabla \nu\|_{\infty}}{\nu_{0}}\|\v_h\|_{0}^{2} \bigg) \\
	\label{2.4.2.1.1}	&= \frac{\nu_{0}}{6}\big(\|\textbf{curl}(\v_h)\|_{0}^{2} + \|\nabla \cdot \v_h\|_{0}^{2}\big) + \frac{6\|\nabla \nu\|_{\infty}^{2}}{\nu_{0}}\|\v_h\|_{0}^{2}.
	\end{align}
Using the fact that $\|(\nabla \nu \times \v_h)\|_{0} \le 2\|\nabla \nu\|_{\infty}\|\v_h\|_{0}$, we get
\begin{align}
\nonumber	\left|\sum_{K \in \mathcal{T}_h} \int_{K} \boldsymbol{\theta}_h \cdot (\nabla \nu \times \v_h) \ dx\right| &\le 2\|\nabla \nu\|_{\infty} \bigg(\frac{\nu_{0}}{6\|\nabla \nu\|_{\infty}}\|\boldsymbol{\theta}_h\|_{0}^{2} + \frac{3\|\nabla \nu\|_{\infty}}{2\nu_{0}}\|\v_h\|_{0}^{2} \bigg)\\
	\label{2.4.2.1.2}&= \frac{\nu_{0}}{3}\|\boldsymbol{\theta}_h\|_{0}^{2} + \frac{3\|\nabla \nu\|_{\infty}^{2}}{\nu_{0}}\|\v_h\|_{0}^{2},
\end{align}
and
\begin{align}
\nonumber	\left|\rho_1 \sum_{K \in \mathcal{T}_h} \int_{K} \boldsymbol{\theta}_h \cdot (\textbf{curl}(\v_h)) \ dx\right| &\le \rho_1 \bigg(\frac{\nu_{0}}{3\rho_1}\|\boldsymbol{\theta}_h\|_{0}^{2} + \frac{3\rho_1}{4\nu_{0}}\|\textbf{curl}(\v_h)\|_{0}^{2} \bigg)\\
\label{2.4.2.1.3}	&= \frac{\nu_{0}}{3}\|\boldsymbol{\theta}_h\|_{0}^{2} + \frac{3\rho_1^{2}}{4\nu_{0}}\|\textbf{curl}(\v_h)\|_{0}^{2},
\end{align}
Now, by using (\ref{2.4.2.1.1}-\ref{2.4.2.1.3}) and \cite[(17)]{LDGO}, we have
\begin{align*}
	&\boldsymbol{\mathcal{A}}_{DG}((\v_h,\boldsymbol{\theta}_h), (\v_h, \boldsymbol{\theta}_h)) +\boldsymbol{\mathcal{O}}(\boldsymbol{\beta}; \v_h, \v_h) \\
	&\ge  \sigma_{\boldsymbol{\beta}} |\!|\!|\v_h|\!|\!|_{1,h}^{2} + \sum_{K \in \mathcal{T}_h} \int_{K} \nu |\boldsymbol{\theta}_h|^{2} \ dx + \rho_2 \|\nabla \cdot \v_h\|_{0}^{2}+ \rho_{1} \|\textbf{curl}(\v_h)\|_{0}^{2} -2 \sum_{K \in \mathcal{T}_h} \int_{K} \boldsymbol{\varepsilon}(\v_h) \nabla \nu \cdot \v_h \ dx \\
	& \ \ \ - \sum_{K \in \mathcal{T}_h} \int_{K} \boldsymbol{\theta}_h \cdot (\rho_1 \textbf{curl}(\v_h) - \nabla \nu \times \v_h)\ dx  + \sum_{E \in \mathcal{E}(\mathcal{T}_h)} \int_{E}\big(C_{11} [\![\nu \v_h]\!]_{T} \cdot [\![\v_h]\!]_{T} + A_{11} [\![\v_h]\!]_{N} [\![\v_h]\!]_{N}  \big) ds \\
	&\ge \sigma_{\boldsymbol{\beta}} |\!|\!|\v_h|\!|\!|_{1,h}^{2} + \nu_{0} \|\boldsymbol{\theta}_h\|_{0}^{2} + \rho_2 \|\nabla \cdot \v_h\|_{0}^{2}+ \rho_{1} \|\textbf{curl}(\v_h)\|_{0}^{2} - \frac{\nu_{0}}{6}\big(\|\textbf{curl}(\v_h)\|_{0}^{2} + \|\nabla \cdot \v_h\|_{0}^{2}\big)  \\
	& \ \ \ - \frac{6\|\nabla \nu\|_{\infty}^{2}}{\nu_{0}}\|\v_h\|_{0}^{2} -\frac{\nu_{0}}{3}\|\boldsymbol{\theta}_h\|_{0}^{2} - \frac{3\rho_1^{2}}{4\nu_{0}}\|\textbf{curl}(\v_h)\|_{0}^{2} - \frac{\nu_{0}}{3}\|\boldsymbol{\theta}_h\|_{0}^{2} - \frac{3\|\nabla \nu\|_{\infty}^{2}}{\nu_{0}}\|\v_h\|_{0}^{2}+ \nu_{0} |\v_h|_{j}^{2} \\
	&= 
	 \frac{\nu_{0}}{3} \|\boldsymbol{\theta}_h\|_{0}^{2} + \bigg(\sigma_{\boldsymbol{\beta}}- \frac{9\|\nabla \nu\|_{\infty}^{2}}{\nu_{0}}\bigg)\|\v_h\|_{0}^{2} + \Big(\sigma_{\boldsymbol{\beta}} + \frac{\nu_{0}}{6}\Big) \|\textbf{curl}(\v_h)\|_{0}^{2} + \Big(\rho_2 +\sigma_{\boldsymbol{\beta}} - \frac{\nu_{0}}{6}\Big) \|\nabla \cdot \v_h\|_{0}^{2} + \nu_{0} |\v_h|_{j}^{2}.
\end{align*}
Under the assumptions outlined in \eqref{2.3.3.9.1}, we get the coercivity property where the constant
\begin{align*}
	C_{3} = \min \left\{\frac{\nu_{0}}{3}, \sigma_{\boldsymbol{\beta}}- \frac{9\|\nabla \nu\|_{\infty}^{2}}{\nu_{0}}, \sigma_{\boldsymbol{\beta}}+ \frac{\nu_{0}}{6}, \rho_2 +\sigma_{\boldsymbol{\beta}} - \frac{\nu_{0}}{6}, \nu_{0} \right\},
\end{align*}
is positive by \eqref{2.3.3.9.1} and the assumptions on augmentation constants $\rho_1$ and $\rho_2$.
\end{proof}}
\noindent The well-posedness of the discrete system follows by an application of Babu\^{s}ka-Brezzi theory, along with the continuity-coercivity properties of bilinear form $\boldsymbol{\mathcal{A}}_{DG}$ and the discrete inf-sup stability of $\boldsymbol{\mathcal{B}}_{DG}$ under the assumptions highlighted in Lemma \ref{Lemma 2.3.1.11}.
\subsubsection{A priori error estimates}
Let $k \geq 0$ be an integer. We impose the following regularity assumptions throughout this subsection:
\begin{align*}
	\y, \w \in \boldsymbol{H}^{s+1} (\Omega), \ \boldsymbol{\omega}, \boldsymbol{\vartheta} \in [H^{s} (\Omega)]^{\frac{d(d-1)}{2}}, \ p, q \in H^{s} (\Omega), \ \text{and} \ \u \in \boldsymbol{H}^{1} (\Omega)\  \text{for some } \ s \ge 1.
\end{align*}
For a control $\u \in \boldsymbol{L}^{2}(\Omega)$, let $(\y_h(\u), \boldsymbol{\omega}_h(\u), p_h(\u)) \in  \boldsymbol{V}_h \times \boldsymbol{W}_h \times Q_h$ be a solution to the following auxiliary problem:
\begin{subequations}
	\begin{align} 
	\label{2.4.2.1.a}	\boldsymbol{\mathcal{A}}_{DG}((\mathbf{y}_h(\u),\boldsymbol{\omega}_h(\u)), (\v_h, \boldsymbol{\theta}_h)) + \boldsymbol{\mathcal{O}}(\boldsymbol{\beta}; \y_h(\u), \v_h) + \boldsymbol{\mathcal{B}}_{DG}((\v_h, \boldsymbol{\theta}_h),p_h(\u)) &= (\mathbf{f} + \mathbf{u}, \mathbf{v}_h), \\
		\label{2.4.2.1.b} \hspace{1.57cm} \boldsymbol{\mathcal{B}}_{DG}((\mathbf{y}_h(\u),\boldsymbol{\omega}_h(\u)), \phi_h) &= 0
	\end{align} 
\end{subequations}
for all $(\v_h, \boldsymbol{\theta}_h, \phi_h) \in \boldsymbol{V}_h \times \boldsymbol{W}_h \times Q_h$. Similarly, let $(\w_h(\y), \boldsymbol{\vartheta}_h(\y), q_h(\y)) \in  \boldsymbol{V}_h \times \boldsymbol{W}_h \times Q_h$ be a solution to the following problem:
\begin{subequations}
	\begin{align}
	\label{2.4.2.2.a}
	\boldsymbol{\mathcal{A}}_{DG}((\mathbf{w}_h(\y),\boldsymbol{\vartheta}_h(\y)), (\z_h, \boldsymbol{\tau}_h)) &+ \boldsymbol{\mathcal{O}}(\boldsymbol{\beta}; \z_h, \w_h(\y)) \\
	\nonumber  &- \boldsymbol{\mathcal{B}}_{DG}((\z_h, \boldsymbol{\tau}_h),q_h(\y))=  (\mathbf{y} - \y_{d}, \mathbf{z}_h) + (\boldsymbol{\omega} - \boldsymbol{\omega}_{d}, \boldsymbol{\tau}_h),\\
	\label{2.4.2.2.b} &\hspace{-0.25cm} \boldsymbol{\mathcal{B}}_{DG}((\mathbf{w}_h(\y),\boldsymbol{\vartheta}_h(\y)), \psi_h) = 0,
	\end{align}
\end{subequations}
for all $(\z_h, \boldsymbol{\tau}_h, \psi_h) \in \boldsymbol{V}_h \times \boldsymbol{W}_h \times Q_h$.
\begin{lemma}\label{Lemma-2.4.2.1.0}
	Let $(\y_h, \boldsymbol{\omega}_h, p_h), (\w_h, \boldsymbol{\vartheta}_h, q_h)$  be the solutions to the state and co-state discrete systems (\ref{2.3.2.6}-\ref{2.3.2.7}), and (\ref{2.3.2.8}-\ref{2.3.2.9}), respectively, and $(\y_h(\u), \boldsymbol{\omega}_h(\u), p_h(\u)), (\w_h(\y), \boldsymbol{\vartheta}_h(\y), r_h(\y))$ be the auxiliary variables. Then, we have the following estimates:
	\begin{align}
		\label{2.4.2.3}	\|(\y_h(\u)-\y_h,\boldsymbol{\omega}_h(\u) - \boldsymbol{\omega}_h)\|_{h} + \|p_h(\u)-p_h\|_{h} &\precsim \|\u-\u_h\|_{0}, \\
		\label{2.4.2.4}	\|(\w_h(\y)-\w_h,\boldsymbol{\vartheta}_h(\y) - \boldsymbol{\vartheta}_h)\|_{h} + \|q_h(\y)-q_h\|_{h} &\precsim \|\y-\y_h\|_{0} + \|\boldsymbol{\omega} - \boldsymbol{\omega}_{h}\|_{0}.
	\end{align}
\end{lemma}
\begin{proof}
			Subtracting the system (\ref{2.3.3.4}-\ref{2.3.3.5}) from (\ref{2.4.2.1.a}-\ref{2.4.2.1.b}), we have
			\begin{align*}
				\boldsymbol{\mathcal{A}}_{DG}((\mathbf{y}_h(\u)-\y_h,\boldsymbol{\omega}_h(\u)-\boldsymbol{\omega}_h), (\v_h, \boldsymbol{\theta}_h)) &+ \boldsymbol{\mathcal{O}}(\boldsymbol{\beta}; \y_h(\u)-\y_h, \v_h) \\
				&+ \boldsymbol{\mathcal{B}}_{DG}((\v_h, \boldsymbol{\theta}_h),p_h(\u)-p_h) = (\mathbf{u}-\u_h, \mathbf{v}_h),\\
				 \hspace{1.57cm} &\boldsymbol{\mathcal{B}}((\mathbf{y}_h(\u)-\y_h,\boldsymbol{\omega}_h(\u)-\boldsymbol{\omega}_h), \phi_h) = 0,
			\end{align*}
			for all $(\v_h, \boldsymbol{\theta}_h,\phi_h) \in \boldsymbol{V}_h \times \boldsymbol{W}_h \times Q_h$. We substitute $\v_h =  \y_h(\u)-\y_h,\boldsymbol{\theta}_h= \boldsymbol{\omega}_h(\u)-\boldsymbol{\omega}_h$ and  $\phi_h =  p_h(\u)-p_h$, in the above set of equations to get
			\begin{align*}
				\boldsymbol{\mathcal{A}}_{DG}((\mathbf{y}_h(\u)-\y_h,\boldsymbol{\omega}_h(\u)-\boldsymbol{\omega}_h),& (\mathbf{y}_h(\u)-\y_h,\boldsymbol{\omega}_h(\u)-\boldsymbol{\omega}_h)) \\
				&+ \boldsymbol{\mathcal{O}}(\boldsymbol{\beta}; \y_h(\u)-\y_h, \y_h(\u)-\y_h) = (\mathbf{u}-\u_h,\mathbf{y}_h(\u)-\y_h).
			\end{align*} 
			By using the coercivity property
			and Cauchy-Schwarz inequality, we have
			\begin{align}
				\label{2.4.2.4.1} 	\|(\y_h(\u)-\y_h,\boldsymbol{\omega}_h(\u)-\boldsymbol{\omega}_h)\|_{h} &\precsim  \|\u-\u_h\|_{0}.
			\end{align}
			Similarly, by subtracting the system (\ref{2.3.3.6}-\ref{2.3.3.7}) from (\ref{2.4.2.2.a}-\ref{2.4.2.2.b}), we get
			\begin{align*}
				\boldsymbol{\mathcal{A}}_{DG}((\mathbf{w}_h(\y)-\w_h,\boldsymbol{\vartheta}_h(\y)-\boldsymbol{\vartheta}_h), (\z_h, \boldsymbol{\tau}_h)) &+ \boldsymbol{\mathcal{O}}(\boldsymbol{\beta}; \z_h, \w_h(\y)-\w_h) \\
				\nonumber  &- \boldsymbol{\mathcal{B}}_{DG}((\z_h, \boldsymbol{\tau}_h),q_h(\y)-q_h) =  (\mathbf{y} - \y_{h}, \mathbf{z}_h) + (\boldsymbol{\omega} - \boldsymbol{\omega}_{h}, \boldsymbol{\tau}_h),\\
			 & \hspace{-1.25cm}\boldsymbol{\mathcal{B}}_{DG}((\mathbf{w}_h(\y)-\w_h,\boldsymbol{\vartheta}_h(\y)-\boldsymbol{\vartheta}_h), \psi_h) = 0,
			\end{align*} 
			for all $(\z_h,\boldsymbol{\tau}_h,\psi_h) \in \boldsymbol{V}_h \times \boldsymbol{W}_h \times Q_h$. Substituting $\z_h = \w_h(\y) - \w_h$, $\boldsymbol{\tau}_h=\boldsymbol{\vartheta}_h(\y)-\boldsymbol{\vartheta}_h$, and $\psi_h =  q_h(\y) - q_h$, in the above pair of equations, we have
			\begin{align*}
				\boldsymbol{\mathcal{A}}_{DG}(&(\mathbf{w}_h(\y)-\w_h,\boldsymbol{\vartheta}_h(\y)-\boldsymbol{\vartheta}_h), (\w_h(\y)- \w_h, \boldsymbol{\vartheta}_h(\y)-\boldsymbol{\vartheta}_h)) \\
				&+ \boldsymbol{\mathcal{O}}(\boldsymbol{\beta}; \mathbf{w}_h(\y)-\w_h, \mathbf{w}_h(\y)-\w_h) = (\mathbf{y}-\y_h, \w_h(\y)- \w_h)+ (\boldsymbol{\omega} - \boldsymbol{\omega}_{h}, \boldsymbol{\vartheta}_h(\y)-\boldsymbol{\vartheta}_h).
			\end{align*} 
			Using a similar argument as \eqref{2.4.2.4.1}, we obtain
			\begin{align}
				\label{2.4.2.4.2} 	\|(\w_h(\y)-\w_h,\boldsymbol{\vartheta}_h(\y)-\boldsymbol{\vartheta}_h)\|_{h} &\precsim \|\y-\y_h\|_{0} + \|\boldsymbol{\omega} - \boldsymbol{\omega}_{h}\|_{0}.
			\end{align}
			Now, by using the inf-sup stability of $\boldsymbol{\mathcal{B}}_{DG}((\cdot,\cdot),\cdot)$, for a constant $C>0$, we have
			\begin{align}
			\label{2.4.2.4.3} C \|p_h(\u)-p_h\|_{h}  \le \underset{0 \neq (\v_h, \boldsymbol{\theta}_h) \in \boldsymbol{V}_h \times \boldsymbol{W}_h}{\sup} \frac{|\boldsymbol{\mathcal{B}}_{DG}((\v_h, \boldsymbol{\theta}_h),p_h(\u)-p_h)|}{\|(\v_h, \boldsymbol{\theta}_h)\|_{h}}
			 \precsim \|\u-\u_h\|_{0}.
			\end{align}
			Similarly, for the co-state problem, we get
			\begin{align}
				\label{2.4.2.4.4}	 \|q_h(\y)-q_h\|_{h} &\precsim \|\y-\y_h\|_{0} + \|\boldsymbol{\omega} - \boldsymbol{\omega}_{h}\|_{0}.
			\end{align}
			Combining (\ref{2.4.2.4.1}) with (\ref{2.4.2.4.3}) and (\ref{2.4.2.4.2}) with (\ref{2.4.2.4.4}), we get the desired estimates.
\end{proof}
\textcolor{red}{
\begin{lemma}\label{Lemma-2.4.2.1}
	Let $(\y, \boldsymbol{\omega}, p), (\w, \boldsymbol{\vartheta}, q) \in  \boldsymbol{V} \times \boldsymbol{W} \times Q$ be solutions to the state and co-state systems (\ref{2.2.2.11}-\ref{2.2.2.12}), and (\ref{2.2.2.13}-\ref{2.2.2.14}), respectively, and $(\y_h(\u), \boldsymbol{\omega}_h(\u), p_h(\u)), (\w_h(\y), \boldsymbol{\vartheta}_h(\y), r_h(\y))$ be the auxiliary variables. Then, for positive constants $C_1, C_2$ (independent of $h$), the following estimates hold: 
	\begin{align*}
		\|(\y-\y_h(\u),\boldsymbol{\omega} - \boldsymbol{\omega}_h(\u))\|_{h} + \|p-p_h(\u)\|_{h} & \le C_1 h^{\min\{s,k+1\}} \big(\|\y\|_{s+1} + \|\boldsymbol{\omega}\|_{s+1} + \|p\|_{s} + \|\u\|_{1} \big), \\
	\|(\w-\w_h(\y),\boldsymbol{\vartheta} - \boldsymbol{\vartheta}_h(\y))\|_{h} + \|q-q_h(\y)\|_{h} & \le C_2 h^{\min\{s,k+1\}} \big( \|\y\|_{s+1} + \|\w\|_{s+1} + \|\boldsymbol{\vartheta}\|_{s+1} + \|q\|_{s}\big).
	\end{align*}
\end{lemma}
\begin{proof}
	To prove the first estimate, by a use of Triangle inequality we get
	\begin{align}
	\nonumber	\|(\y-\y_h(\u),\boldsymbol{\omega} - \boldsymbol{\omega}_h(\u))\|_{h} + \|p-p_h(\u)\|_{h} &\le \|(\y - \boldsymbol{\Pi}_{\boldsymbol{V}}\y,\boldsymbol{\kappa} - \boldsymbol{\Pi}_{\boldsymbol{W}}\boldsymbol{\kappa})\|_{h} + \|p - \Pi_{Q}p\|_{h} \\
	\nonumber	& \ \ \ + \|(\boldsymbol{\Pi}_{\boldsymbol{V}}\y - \y_h(\u),\boldsymbol{\Pi}_{\boldsymbol{W}}\boldsymbol{\kappa} - \boldsymbol{\kappa}_h(\u))\|_{h} + \|\Pi_{Q}p - p_h(\u)\|_{h}\\
	\label{1.00}	&\le \|(\boldsymbol{\xi}_{\y},\boldsymbol{\xi}_{\boldsymbol{\kappa}})\|_{h} + \|\xi_{p}\|_{h} + \|(\boldsymbol{\Psi}_{\y},\boldsymbol{\Psi}_{\boldsymbol{\kappa}})\|_{h} + \|\Psi_{p}\|_{h},
	\end{align}
where the numerical and approximation errors are defined by:
	\begin{subequations}\label{2.4.2.9.00}
		\begin{align}
		\label{2.4.2.9.00a}	&\boldsymbol{\xi}_{\y} = \y - \boldsymbol{\Pi}_{\boldsymbol{V}}\y,&&\boldsymbol{\xi}_{\boldsymbol{\kappa}} = \boldsymbol{\kappa} - \boldsymbol{\Pi}_{\boldsymbol{W}}\boldsymbol{\kappa}, &&& \xi_{p} = p - \Pi_{Q}p,\\
		\label{2.4.2.9.00b}	&\boldsymbol{\Psi}_{\y} = \boldsymbol{\Pi}_{\boldsymbol{V}}\y - \y_h(\u),&&\boldsymbol{\Psi}_{\boldsymbol{\kappa}} = \boldsymbol{\Pi}_{\boldsymbol{W}}\boldsymbol{\kappa} - \boldsymbol{\kappa}_h(\u), &&& \Psi_{p} = \Pi_{Q}p - p_h(\u),
		\end{align} 
	\end{subequations}
where $\boldsymbol{\Pi}_{\boldsymbol{V}}, \boldsymbol{\Pi}_{\boldsymbol{W}}$ and $\Pi_{Q}$ denote the $L^2$-projection onto the discrete spaces $\V_h, \W_h$ and $Q_h$, respectively. For the regularity assumptions discussed in the beginning of this subsection, by following standard $L^{2}$-projection estimates, we get the following estimates \cite[Lemma~7]{VABAM}:
\begin{subequations}\label{2.4.2.9.01}
	\begin{align}
		&|\!|\!|\boldsymbol{\xi}_{\y}|\!|\!|_{1}
		\le C_{1} h^{\min\{s,k+1\}}\|\y\|_{s+1}, && \|\boldsymbol{\xi}_{\boldsymbol{\kappa}}\|_{0} \le C_{2} h^{\min\{s,k\}+1}\|\y\|_{s+1},\\
		& \ |\boldsymbol{\xi}_{\y}|_{j} \le C_{3} h^{\min\{s,k+1\}}\|\y\|_{s+1}, && \ \ |\xi_{p}|_{j} \le C_{4} h^{\min\{s,k\}+1}\|p\|_{s},
	\end{align}
\end{subequations}
where $C_1, C_2, C_3$ and $C_4$ are positive constants independent of the mesh size. Combining these estimates, we obtain
\begin{align}\label{1.1}
	\|(\boldsymbol{\xi}_{\y},\boldsymbol{\xi}_{\boldsymbol{\kappa}})\|_{h} + |\xi_{p}|_{j} \le C_{1} h^{\min\{s,k+1\}}\big(\|\y\|_{s+1}+\|p\|_{s}\big).
\end{align} Additionally, we have the following estimates:
\begin{subequations}\label{2}
	\begin{align}
	\label{2a}	\nu_{1}\bigg|\sum_{K \in \mathcal{T}_h} \int_{K} \boldsymbol{\theta}_h \cdot \textbf{curl}(\boldsymbol{\xi}_{\y}) dx + \sum_{E \in \mathcal{E}(\mathcal{T}_h)}\int_{E} \{\!\!\{\boldsymbol{\theta}_h\}\!\!\} \cdot [\![\boldsymbol{\xi}_{\y}]\!]_{T} ds\bigg|&\le C_{5} h^{\min\{s,k+1\}}\|\y\|_{s+1} \|\boldsymbol{\theta}_h\|_{0},\\
	\label{2b} \nu_{1}\bigg|\sum_{K \in \mathcal{T}_h} \int_{K} \boldsymbol{\xi}_{\boldsymbol{\kappa}} \cdot \textbf{curl}(\v_h) dx + \sum_{E \in \mathcal{E}(\mathcal{T}_h)}\int_{E} \{\!\!\{\boldsymbol{\xi}_{\boldsymbol{\kappa}}\}\!\!\} \cdot [\![\v_h]\!]_{T} ds\bigg| &\le C_{6}h^{\min\{s,k\}+1}\|\y\|_{s+1} |\v_h|_{j},\\
	\label{2c} \bigg|-\sum_{K \in \mathcal{T}_h} \int_{K} \xi_{p} \nabla \cdot \v_h \ dx + \sum_{E \in \mathcal{E}(\mathcal{T}_h)}\int_{E} \{\!\!\{\xi_{p}\}\!\!\} \cdot [\![\v_h]\!]_{N} ds\bigg|&\le C_{7}h^{\min\{s,k\}+1}\|p\|_{s} |\v_h|_{j},\\
	\label{2d} \bigg|-\sum_{K \in \mathcal{T}_h} \int_{K} \phi_h \nabla \cdot \boldsymbol{\xi}_{\y} \ dx + \sum_{E \in \mathcal{E}(\mathcal{T}_h)}\int_{E} \{\!\!\{\phi_h\}\!\!\} \cdot [\![\boldsymbol{\xi}_{\y}]\!]_{N} ds\bigg| &\le C_{8}h^{\min\{s,k+1\}}\|\y\|_{s+1} |\phi_h|_{j},\\
	\label{2e} \bigg|\sum_{E \in \mathcal{E}(\mathcal{T}_h)} \big(C_{11} [\![\nu \boldsymbol{\xi}_{\y}]\!]_{T} \cdot [\![\v_h]\!]_{T} + A_{11} [\![\boldsymbol{\xi}_{\y}]\!]_{N} [\![\v_h]\!]_{N}  \big) ds\bigg| &\le C_{9}h^{\min\{s,k+1\}}\|\y\|_{s+1} |\v_h|_{j},\\
	\label{2f}\bigg|\sum_{E \in \mathcal{E}(\mathcal{T}_h)} \int_{E}  D_{11} [\![\xi_{p}]\!] \cdot [\![\phi_h]\!]\ ds\bigg|&\le C_{10}h^{\min\{s,k\}+1}\|p\|_{s} |\phi_h|_{j},
	\end{align}
\end{subequations}
where $C_5, C_6, C_7, C_8, C_9$ and $C_{10}$ are positive constants dependent on the stabilisation parameters and viscosity parameter $\nu$. The estimates (\ref{2a}-\ref{2b}, \ref{2e}-\ref{2f})  follow from the results provided in \cite[Lemma~8]{VABAM} and (\ref{2c}-\ref{2d}) follow exactly with same arguments from \cite[Section~3.3]{LDGSTOKES}.\\
By using Galerkin orthogonality and the orthogonality of $L^2$-projections, along with the bounds from (\ref{2.4.2.9.01},\ref{2}), we obtain
\begin{align}\label{1.02}
\nonumber	C \|(\boldsymbol{\Psi}_{\y},\boldsymbol{\Psi}_{\boldsymbol{\kappa}})\|_{h}^{2}&\le \boldsymbol{\mathcal{A}}_{DG}((\boldsymbol{\Psi}_{\y},\boldsymbol{\Psi}_{\boldsymbol{\kappa}}), (\boldsymbol{\Psi}_{\y}, \boldsymbol{\Psi}_{\boldsymbol{\kappa}})) + \boldsymbol{\mathcal{O}}(\boldsymbol{\beta}; \boldsymbol{\Psi}_{\y}, \boldsymbol{\Psi}_{\y}) \\
\nonumber	&\le \bigg|\sum_{K \in \mathcal{T}_h} \int_{K} 2\boldsymbol{\varepsilon}(\boldsymbol{\Psi}_{\y}) \nabla \nu \cdot \boldsymbol{\xi}_{\y} \ dx\bigg| + \bigg|\sum_{K \in \mathcal{T}_h} \int_{K} \boldsymbol{\Psi}_{\boldsymbol{\kappa}} \cdot (\nabla \nu \times \boldsymbol{\xi}_{\y}) dx\bigg| \\
\nonumber & \ \ \ + \nu_{1} \bigg|\sum_{K \in \mathcal{T}_h} \int_{K} \boldsymbol{\Psi}_{\boldsymbol{\kappa}} \cdot \textbf{curl} (\boldsymbol{\xi}_{\y}) \ dx + \sum_{E \in \mathcal{E}(\mathcal{T}_h)} \int_{E}  \{\!\!\{\boldsymbol{\Psi}_{\boldsymbol{\kappa}}\}\!\!\} \cdot [\![\boldsymbol{\xi}_{\y}]\!]_{T}\ ds\bigg|\\
\nonumber	& \ \ \ + \nu_{1} \bigg|\sum_{K \in \mathcal{T}_h} \int_{K} \boldsymbol{\xi}_{\boldsymbol{\kappa}} \cdot \textbf{curl} (\boldsymbol{\Psi}_{\y}) \ dx - \sum_{E \in \mathcal{E}(\mathcal{T}_h)} \int_{E} \{\!\!\{\boldsymbol{\xi}_{\boldsymbol{\kappa}}\}\!\!\} \cdot [\![\boldsymbol{\Psi}_{\y}]\!]_{T} \ ds\bigg|\\
\nonumber	& \ \ \ + \nu_{1} \bigg|\sum_{E \in \mathcal{E}(\mathcal{T}_h)} \big(C_{11} [\![\boldsymbol{\xi}_{\y}]\!]_{T} \cdot [\![\boldsymbol{\Psi}_{\y}]\!]_{T} + A_{11} [\![\boldsymbol{\xi}_{\y}]\!]_{N} [\![\boldsymbol{\Psi}_{\y}]\!]_{N}  \big) ds\bigg| + \bigg|\sum_{K \in \mathcal{T}_h} \int_{K} \boldsymbol{\Psi}_{\y} \boldsymbol{\beta}^{T}: \nabla \boldsymbol{\xi}_{\y} \ dx\bigg| \\
\nonumber	& \ \ \ + \bigg|\sum_{K \in \mathcal{T}_h} \int_{K}\big(\rho_1 (\textbf{curl}(\boldsymbol{\Psi}_{\y})-\boldsymbol{\Psi}_{\boldsymbol{\kappa}}) \cdot \textbf{curl}(\boldsymbol{\xi}_{\y})
+ \rho_2 (\nabla \cdot \boldsymbol{\Psi}_{\y}) \cdot (\nabla \cdot \boldsymbol{\xi}_{\y})\big) dx\bigg| \\
\nonumber	& \ \ \ + \bigg|\sum_{K \in \mathcal{T}_h} \int_{\partial K_{\text{out}} \cap \Gamma_{\text{out}}} (\boldsymbol{\beta} \cdot \boldsymbol{n}_{K}) \boldsymbol{\Psi}_{\y} \cdot \boldsymbol{\xi}_{\y} \ ds\bigg| + \bigg|\sum_{K \in \mathcal{T}_h} \int_{\partial K_{\text{out}} \setminus \Gamma} (\boldsymbol{\beta} \cdot \boldsymbol{n}_{K}) \boldsymbol{\Psi}_{\y} \cdot (\boldsymbol{\xi}_{\y}-\boldsymbol{\xi}_{\y}^{e}) \ ds\bigg|\\
\nonumber&\le \Big(C_{5}  \|\boldsymbol{\Psi}_{\boldsymbol{\kappa}}\|_{0} + (C_{6} + C_{9}) |\boldsymbol{\Psi}_{\y}|_{j} + C_{12}\big(\|\boldsymbol{\Psi}_{\y}\|_{0} + \|\textbf{curl}( \boldsymbol{\Psi}_{\y})\|_{0} + \|\nabla \cdot \boldsymbol{\Psi}_{\y}\|_{0}\big) \Big)h^{\min\{s,k+1\}}\|\y\|_{s+1}\\
\therefore \|(\boldsymbol{\Psi}_{\y},\boldsymbol{\Psi}_{\boldsymbol{\kappa}})\|_{h} &\le C_{13}h^{\min\{s,k+1\}}\|\y\|_{s+1}.
\end{align}
Similarly, using the estimates (\ref{2c}-\ref{2e}), we get
\begin{align}\label{4}
	|\Psi_{p}|_{j} \le C_{13} h^{\min\{s,k+1\}}\|p\|_{s}.
\end{align}
The following $L^{2}$-norm error of the pressure follows from \cite[Lemma~5]{VABAM}:
\begin{align}\label{1.01}
	\|\xi_{p}\|_{0} + \|\Psi_{p}\|_{0} \le C_{14} h^{\min\{s,k+1\}}\|p\|_{s}.
\end{align}
Using the estimates \eqref{1.1}, \eqref{1.02}, \eqref{4} and \eqref{1.01} in \eqref{1.00}, we get the desired estimate for the state problem.
Similarly, the second estimate for the co-state is derived using a comparable approach.
\end{proof}
}
\begin{theorem}\label{Lemma-2.4.2.2}
	Let $(\y, \boldsymbol{\omega}, p, \w, \boldsymbol{\vartheta}, q, \mathbf{u})$ be the solution to the continuous system (\ref{2.2.2.11}-\ref{2.2.2.15}), and \\ $(\y_h, \boldsymbol{\omega}_h, p_h, \w_h, \boldsymbol{\vartheta}_h, q_h, \mathbf{u}_h)$ be their discrete DG approximation. Then, we have the following estimate:
	\begin{align}
		\label{2.4.2.9}	\|\u-\u_h\|_{0} \le Ch \|\u\|_{1}.
	\end{align}
\end{theorem}
\begin{proof}
	To prove this estimate, we follow a technique similar to the proof of Lemma \ref{Lemma-2.4.1.3}. For some $\u_h \in \boldsymbol{\mathcal{A}_{d}}$, let $(\y(\u_h), \boldsymbol{\omega} (\u_h), p(\u_h)), (\w(\u_h), \boldsymbol{\vartheta} (\u_h), q(\u_h)) \in \boldsymbol{V} \times \boldsymbol{W} \times Q$ be the solution of the systems (\ref{2.4.1.8}-\ref{2.4.1.9}) and (\ref{2.4.1.10}-\ref{2.4.1.11}), respectively.
	The reduced functional $\mathcal{F}$ exhibits the properties mentioned in (\ref{2.4.1.111a}-\ref{2.4.1.111b}).
Now, by using the second order conditions, we obtain
%
\begin{align}
	\label{2.4.2.9.1}	\gamma \|\u-\u_{h}\|_{0}^{2} &\le {\gamma(\u-\u_h,\Pi_h \u- \u)} + {(\Pi_h \u - \u,\w_h-\w(\u_h))} + {(\Pi_h \u - \u,\w(\u_h)-\w)}  \\
	\nonumber& + {(\u_{h}-\u,\w(\u_h)-\w_h)}.
\end{align}
Using the Cauchy-Schwarz and Young's inequalities, for a constant $\zeta>0$, we have:
\begin{align}
	\label{2.4.2.9.2}	\gamma \|\u-\u_{h}\|_{0}^{2} &\le \frac{\gamma}{\zeta} \|\u-\Pi_h \u\|_{0}^{2} + 2\gamma \zeta \|\u-\u_h\|_{0}^{2} + \frac{\zeta}{\gamma}  \|\w-\w(\u_h)\|_{0}^{2}+ \big( \frac{\zeta}{2\gamma}+\frac{1}{4\gamma \zeta}\big)\|\w(\u_h)-\w_h\|_{0}^{2}.
\end{align}
Now, we subtract (\ref{2.4.1.10}-\ref{2.4.1.11}) from (\ref{2.2.2.11}-\ref{2.2.2.12}), and substitute 
 $\z = \w-\w(\u_h), \boldsymbol{\tau} = \boldsymbol{\vartheta} - \boldsymbol{\vartheta}(\u_h)$ and $\psi = q-q(\u_h)$, to obtain
\begin{align*}
	\boldsymbol{\mathcal{C}}((\w-\mathbf{w}(\u_h),\boldsymbol{\vartheta} - \boldsymbol{\vartheta}(\u_h)), (\w-\mathbf{w}(\u_h), \boldsymbol{\vartheta} - \boldsymbol{\vartheta}(\u_h))) = (\y - \mathbf{y}(\u_h), \w-\w(\u_h)) + (\boldsymbol{\omega} - \boldsymbol{\omega}(\u_h), \boldsymbol{\vartheta} - \boldsymbol{\vartheta}(\u_h)).
\end{align*}
Using the constitutive relation from \eqref{2.2.1.2} and coercivity of $\boldsymbol{\mathcal{C}}((\cdot,\cdot),(\cdot,\cdot))$, we have
\begin{align}
	\label{2.4.2.14}  |\!|\!|\w-\w(\u_h)|\!|\!|_{1} \le (C_{a}^{c})^{-1} |\!|\!|\y-\y(\u_h)|\!|\!|_{1} \le \eta \|\u-\u_h\|_{0},
\end{align}
where $C_{a}^{c}$ is the coercivity constant and $\eta = (C_{a}^{c})^{-1}$. Applying (\ref{2.4.2.14}) to (\ref{2.4.2.9.2}), and selecting the constant $\zeta = \frac{\gamma}{2} \left(2\gamma + \eta^{2} \gamma^{-1}\right)^{-1}$, we obtain
\begin{align*}
	\|\u-\u_{h}\|_{0}^{2} &\le \frac{2}{\zeta} \|\u- \Pi_h \u\|_{0}^{2} + \bigg( \frac{\zeta}{\gamma^{2}}+\frac{1}{2\gamma^{2} \zeta}\bigg) |\!|\!|\w(\u_h)-\w_h|\!|\!|_{1,h}^{2}.
\end{align*}
A use of the estimates for the second term and $L^2$-projection gives the estimate for the control variable.
\end{proof}
\begin{theorem}\label{lem: Lemma-2.4.2.3}
	Let $(\y, \boldsymbol{\omega}, p, \w, \boldsymbol{\vartheta}, q, \mathbf{u})$ be a solution to the system (\ref{2.2.2.11}-\ref{2.2.2.15}), with corresponding DG approximation
	$(\y_h, \boldsymbol{\omega}_h, p_h, \w_h, \boldsymbol{\vartheta}_h, q_h, \mathbf{u}_h)$.  Then, for positive constants $C_3$ and $C_4$ independent of the mesh-size, following estimates hold:
	\begin{align*}
	\|(\y-\y_h,\boldsymbol{\omega} - \boldsymbol{\omega}_h)\|_{h} + \|p-p_h\|_{h} & \le C_3 h^{\min\{s,k+1\}} \big(\|\y\|_{s+1} + \|\boldsymbol{\omega}\|_{s+1} + \|p\|_{s} + \|\u\|_{s} \big), \\
	\|(\w-\w_h,\boldsymbol{\vartheta} - \boldsymbol{\vartheta}_h)\|_{h} + \|q-q_h\|_{h} & \le C_4 h^{\min\{s,k+1\}} \big( \|\y\|_{s+1} + \|p\|_{s} + \|\u\|_{s} + \|\w\|_{s+1} + \|\boldsymbol{\vartheta}\|_{s+1} + \|q\|_{s}\big).
\end{align*}
\end{theorem}
\begin{proof}
	To prove the first estimate, we utilize the Triangle Inequality, which gives
\begin{align*}
	\|(\y-\y_h,\boldsymbol{\omega} - \boldsymbol{\omega}_h)\|_{h} + \|p-p_h\|_{h} \le \ &\ \|(\y-\y_h(\u),\boldsymbol{\omega} - \boldsymbol{\omega}_h(\u))\|_{h} + \|p-p_h(\u)\|_{h} \\
	&+ \|(\y_h(\u)-\y_h,\boldsymbol{\omega}_h(\u) - \boldsymbol{\omega}_h)\|_{h} + \|p_h(\u)-p_h\|_{h}.
\end{align*}
By applying the estimates from Lemmas~\ref{Lemma-2.4.2.1.0}, \ref{Lemma-2.4.2.1}, and Theorem~\ref{Lemma-2.4.2.2}, we derive the first estimate. Similarly, a use of Triangle inequality for the second estimate gives 
\begin{align*}
	\|(\w-\w_h,\boldsymbol{\vartheta} - \boldsymbol{\vartheta}_h)\|_{h} + \|q-q_h\|_{h} \le \ &\ 	\|(\w-\w_h(\y),\boldsymbol{\vartheta} - \boldsymbol{\vartheta}_h(\y))\|_{h} + \|q-q_h(\y)\|_{h} \\
	& + \|(\w_h(\y)-\w_h,\boldsymbol{\vartheta}_h(\y) - \boldsymbol{\vartheta}_h)\|_{h} + \|q_h(\y)-q_h\|_{h}.
\end{align*}
An application of Lemmas~\ref{Lemma-2.4.2.1.0}, \ref{Lemma-2.4.2.1}, and Theorem~\ref{Lemma-2.4.2.2} completes the proof.
\end{proof}
\begin{remark}
	Unlike the conforming scheme explained in Section~\ref{Conforming Scheme}, the discrete velocity obtained from the DG scheme does not necessarily satisfy the divergence-free criterion. As a result, this approach is not pressure robust, which means that pressure errors influence a priori velocity error estimates. This might result in less accurate velocity approximations. This can be addressed by employing divergence-conforming discrete spaces, as mentioned in \cite{LPHILIP}.
\end{remark}
\subsubsection{A posteriori error estimates}
Let $k \ge 1$ be an integer, and $\boldsymbol{V}_h, \boldsymbol{W}_h, Q_h$ and $\boldsymbol{\mathcal{A}_{dh}}$ be the discontinuous discrete spaces. Let $(\y, \omega, p, \w, \vartheta, q, \mathbf{u}) \in \boldsymbol{V} \times \boldsymbol{W} \times Q \times \boldsymbol{V} \times \boldsymbol{W} \times Q \times \boldsymbol{\mathcal{A}_{d}}$ and $(\y_h, \omega_h, p_h, \w_h, \vartheta_h, q_h, \mathbf{u}_h) \in \boldsymbol{V}_h \times \boldsymbol{W}_{h} \times Q_h \times \boldsymbol{V}_h \times \boldsymbol{W}_{h} \times Q_h \times \boldsymbol{\mathcal{A}_{dh}}$ be the unique solutions to the continuous and discrete problems (\ref{2.2.2.11}-\ref{2.2.2.15}) and (\ref{2.3.3.4}-\ref{2.3.3.8}), respectively.
Let  $\nu_h,\ \boldsymbol{\beta}_{h},\ \sigma_{h},\ \f_{h},\ \y_{d,h}$ and $\boldsymbol{\omega}_{d,h}$ represent the piecewise polynomial approximations of the viscosity coefficient $\nu$, convective velocity field $\boldsymbol{\beta}$, reaction term coefficient $\sigma$, source function $\f$ and the desired velocity $\y_d$ and vorticity $\boldsymbol{\omega}_{d}$, respectively. These approximations may exhibit discontinuities across elemental edges.
For an element $K\in\mathcal{T}_h$, we define local error indicators $\eta_{d,K}^{\y}$, $\eta_{d,K}^{\w}$, and $\eta_{d,K}^{\u}$ as:
\begin{align*}
	(\eta_{d,K}^{\y})^{2} := (\eta_{R_K}^{\y})^{2}   + (\eta_{E_K}^{\y})^{2}   +  (\eta_{J_K}^{\y})^{2} , \quad \quad (\eta_{d,K}^{\w})^{2} := (\eta_{R_K}^{\w})^{2}   + (\eta_{E_K}^{\w})^{2}   +  (\eta_{J_K}^{\w})^{2}, \quad \quad
	(\eta_{d,K}^{\u})^{2} := (\eta_{R_K}^{\u})^{2} 
\end{align*} 
where the interior residual terms are defined as
\begin{align*}
	\begin{cases}
		\big(\eta_{R_{K}}
		^{\y}\big)^{2} &:= h_{K}^{2}\|\f_{h} + \u_h + 2 \boldsymbol{\varepsilon}(\y_h) \nabla \nu_h - \nu_h \  \textbf{curl}(\boldsymbol{\omega_h}) - (\boldsymbol{\beta}_{h} \cdot \nabla) \y_h - \nabla p_h - \sigma_h \y_h\|_{0,K}^{2} \\
		& \ \ \ \ + \|\boldsymbol{\omega_h} - \textbf{curl}(\y_h)\|_{0,K}^{2}, \\
		\big(\eta_{R_{K}}^{\w}\big)^{2} &:= h_{K}^{2} \|\y_{h} - \y_{d,h}  + 2 \boldsymbol{\varepsilon}(\w_h) \nabla \nu_h - \nu_h \  \textbf{curl}(\boldsymbol{\vartheta_h}) + (\boldsymbol{\beta}_{h} \cdot \nabla) \w_h + \nabla q_h - (\sigma_h-\nabla \cdot \boldsymbol{\beta}_h)	\w_h\|_{0,K}^{2}\\
		& \ \ \ \ + \|\boldsymbol{\vartheta}_h - \textbf{curl}(\w_h) - \boldsymbol{\omega}_h + \boldsymbol{\omega}_{d,h}\|_{0,K}^{2}, \\
		(\eta_{R_K}^{\u})^{2}  &:= h_{K}^{2} \|\w_h + \lambda \u_h\|_{0,K}^{2},
	\end{cases} 
\end{align*}
	the edge residuals are defined as
\begin{align*}
		\big(\eta_{E_{K}}^{\y}\big)^{2} &:= \frac{1}{2} \sum \limits_{E \in \partial K \setminus \Gamma} h_{E} \|[\![ (p_h I - \boldsymbol{ \omega}_h)\times \n]\!]\|_{0,E}^{2}, \qquad \big(\eta_{E_{K}}^{\w}\big)^{2} := \frac{1}{2} \sum \limits_{E \in \partial K \setminus \Gamma} h_{E} \|[\![ (q_h I - \boldsymbol{ \vartheta}_h)\times \n]\!]\|_{0,E}^{2},
\end{align*}
and the trace residuals are defined as
\begin{align*}
	\big(\eta_{J_{K}}^{\y}\big)^{2} &:= \frac{1}{2} \sum \limits_{E \in \partial K \setminus \Gamma} \Big( C_{11} \|[\![\y_h ]\!]_{T}\|_{0,E}^{2} + A_{11} \|[\![\y_h]\!]_{N}\|_{0,E}^{2} + D_{11} \|[\![p_h ]\!]\|_{0,E}^{2} \Big) \vspace{1mm}\\
	& \ \ \  \ + \sum \limits_{E \in \partial K \cap \Gamma} \Big(C_{11} \| \y_h\|_{0,E}^{2} + A_{11} \|\y_h\|_{0,E}^{2} + D_{11} \|p_h\|_{0,E}^{2} \Big),\\
	\big(\eta_{J_{K}}^{\w}\big)^{2} &:= \frac{1}{2} \sum \limits_{E \in \partial K \setminus \Gamma} \Big( C_{11} \|[\![\w_h ]\!]_{T}\|_{0,E}^{2} + A_{11} \|[\![\w_h]\!]_{N}\|_{0,E}^{2} + D_{11} \|[\![q_h ]\!]\|_{0,E}^{2} \Big) \vspace{1mm}\\
	& \ \ \  \ + \sum \limits_{E \in \partial K \cap \Gamma} \Big(C_{11} \|\w_h\|_{0,E}^{2} + A_{11} \|\w_h\|_{0,E}^{2} + D_{11} \|q_h\|_{0,E}^{2} \Big),
\end{align*}
where $I$ is the $d \times d$ identity matrix. We define the \textbf{global error estimators} $\eta_d^{\y}, \eta_d^{\w}$, and  $\eta_d^{\u}$ as:
\begin{align*}
	(\eta_d^{\y})^2 &:= \sum_{K \in \mathcal{T}_h}(\eta_{d,K}^{\y})^{2} , &&  (\eta_d^{\w})^2 := \sum_{K \in \mathcal{T}_h}(\eta_{d,K}^{\w})^{2}, \quad \quad \quad \quad \eta_d^{\u} := \sum_{K \in \mathcal{T}_h}(\eta_{d,K}^{\u})^{2}.
\end{align*}\\
For the local data oscillation terms $\Theta_{K}^{\y}$ and $\Theta_{K}^{\w}$ defined as:
\begin{align*}
	(\Theta_{K}^{\y})^{2} &:= h_{K}^{2} \big(\|\f-\f_h\|_{0,K}^{2} + \| 2 \boldsymbol{\varepsilon}(\y_h)(\nabla \nu - \nabla \nu_h)\|_{0,K}^{2} + \|(\nu-\nu_h) \textbf{curl}(\boldsymbol{\kappa}_h)\|_{0,K}^{2} + \| ((\boldsymbol{\beta}-\boldsymbol{\beta}_h) \cdot \nabla) \y_h \|_{0,K}^{2} \\
	& \ \ \ \ + \| (\sigma-\sigma_h)\y_h\|_{0,K}^{2}\big),\\
	(\Theta_{K}^{\w})^{2} &:= h_{K}^{2} \big(\|\y_{d,h}-\y_d\|_{0,K}^{2}  + \| 2 \boldsymbol{\varepsilon}(\w_h)(\nabla \nu - \nabla \nu_h)\|_{0,K}^{2}+ \|(\nu-\nu_h) \textbf{curl}(\boldsymbol{\vartheta}_h)\|_{0,K}^{2} \\
	& \ \ \ \ + \|((\boldsymbol{\beta}-\boldsymbol{\beta}_h) \cdot \nabla) \w_h \|_{0,K}^{2} + \|((\sigma - \nabla \cdot \boldsymbol{\beta})-(\sigma_h - \nabla \cdot \boldsymbol{\beta}_h))\w_h\|_{0,K}^{2} + \|\boldsymbol{\kappa}_{d,h}-\boldsymbol{\kappa}_d\|_{0,K}^{2}\big),
\end{align*}
we define the \textbf{global data oscillation terms} $\Theta^{\y}$ and $\Theta^{\w}$ as:
\begin{align*}
	 (\Theta^{\y})^2 &:= \sum_{K \in \mathcal{T}_h}(\Theta_{K}^{\y})^{2}, && (\Theta^{\w})^2 := \sum_{K \in \mathcal{T}_h}(\Theta_{K}^{\w})^{2},
\end{align*}
	\begin{lemma}\label{Lemma-2.5.2.1.}
	Let $(\y, \omega, p, \w, \vartheta, q)$ and $(\y_h, \omega_h, p_h, \w_h, \vartheta_h, q_h)$ be solutions to the continuous and discrete problems (\ref{2.2.2.11}-\ref{2.2.2.14}) and (\ref{2.3.3.4}-\ref{2.3.3.7}), respectively. Then, the following estimates hold true:
	\begin{align}
		\label{2.5.2.2}	\|(\y-\y_h,\boldsymbol{\omega} - \boldsymbol{\omega}_h)\|_{h} + \|p-p_h\|_{h} \ &  \precsim \eta_d^{\y} + \Theta^{\y}, \\
		\label{2.5.2.3}	\|(\w-\w_h,\boldsymbol{\vartheta} - \boldsymbol{\vartheta}_h)\|_{h} + \|q-q_h\|_{h} \ & \precsim \eta_d^{\w}+ \Theta^{\w}.
	\end{align}
	\end{lemma}
	\begin{proof}
		To derive reliability estimates for the state and co-state problem, we employ the idea from \cite[Section~4]{AGO}. We decompose the state and co-state velocity approximation as:
		\begin{align*}
			\y_h &= \y_h^{c} + \y_h^{r}, &&\w_h = \w_h^{c} + \w_h^{r},  &&&\y_h^{c}, \w_h^{c} \in \boldsymbol{V}_{h}^{c} &&&& \text{and} &&&&& \y_h^{r}, \w_h^{r} \in \boldsymbol{V}_{h}^{\perp},
		\end{align*}
	where $\boldsymbol{V}_{h}^{c} := \boldsymbol{V}_{h} \cap \boldsymbol{H}_{0}^{1}(\Omega)$ and $\boldsymbol{V}_{h}^{\perp}$
	denotes the conforming space and orthogonal complement of $\boldsymbol{V}_{h}^{c}$, respectively. First, we establish an upper bound for the remaining term $\y_h^{r}$ in the state velocity. Then, we demonstrate an upper bound for the continuous error $\y - \y_h^{c}$, as well as the vorticity and pressure errors $\boldsymbol{\kappa} - \boldsymbol{\kappa}_h$ and $p - p_h$. By using the Triangle Inequality, Lemma \ref{2.5.1.2}, and a similar method to that in \cite[Lemma~5.8]{HSAK}, we derive the estimate \eqref{2.5.2.2}. Similar sort of technique leads to the other estimate for the co-state.
	\end{proof}
	\begin{theorem}\label{Theorem-2.5.2.2.}
	Let $(\y, \omega, p, \w, \vartheta, q, \u)$ and $(\y_h, \omega_h, p_h, \w_h, \vartheta_h, q_h, \u_h)$ be the solutions to the continuous and discrete systems (\ref{2.2.2.11}-\ref{2.2.2.15}) and (\ref{2.3.3.4}-\ref{2.3.3.8}), respectively. Then, we have the  \textbf{reliability} estimate:
	\begin{align}
\label{2.5.2.4}	\|\u-\u_h\|_{0} \ & + \|(\y-\y_h,\boldsymbol{\omega} - \boldsymbol{\omega}_h)\|_{h} + \|p-p_h\|_{h} \\
	\nonumber & + \|(\w-\w_h,\boldsymbol{\vartheta} - \boldsymbol{\vartheta}_h)\|_{h} + \|q-q_h\|_{h} \precsim \eta_d^{\y}+ \eta_d^{\w}+\eta_d^{\u}+ \Theta^{\y}+ \Theta^{\w}.
	\end{align}
\end{theorem}
\begin{proof}
	To prove this result, we proceed in the same manner as in the proof of Theorem~\ref{Theorem-2.5.1.3.} to obtain the following relation between the control and the co-state variables:
	\begin{align}
		\label{2.5.2.5}  \|\u - \u_h\|_{0}  \precsim \ & \eta_d^{\u} +  \|(\w_h - \w(\u_h),\boldsymbol{\vartheta}_h - \boldsymbol{\vartheta}(\u_h))\|_{h}.
	\end{align}
	Since $(\w(\u_h)-\tilde{\w}, \vartheta(\u_h)- \tilde{\boldsymbol{\vartheta}}, q(\u_h)-\tilde{q})$ solves the system \eqref{2.5.1.15.a}, where $(\tilde{\w}, \tilde{\boldsymbol{\vartheta}}, \tilde{q})$ is a solution to the system (\ref{2.2.2.13}-\ref{2.2.2.14}), by an application of Lemma~\ref{Lemma-2.5.2.1.}, we have
	\begin{align}
		\label{2.5.2.6}	\|(\tilde{\w} - \w_h,\tilde{\boldsymbol{\vartheta}} - \vartheta_h)\|_{h} + \|\tilde{q}-q_h\|_{h} \precsim \eta_d^{\w} + \Theta^{\w}.
	\end{align}
	Using the Triangle Inequality and the estimate (\ref{2.5.2.6}), we have
	\begin{align}
	\label{2.5.2.7}	 \|(\w_h - \w(\u_h),& \boldsymbol{\vartheta}_h - \boldsymbol{\vartheta}(\u_h))\|_{h} + \|q(\u_h)-q_h)\|_{h} \precsim \|(\w(\u_h)- \tilde{\w},\boldsymbol{\vartheta}(\u_h) - \tilde{\boldsymbol{\vartheta}})\|_{h}+ \|q(\u_h)-\tilde{q}\|_{h} \\
		\nonumber & + \|(\tilde{\w} - \w_h,\tilde{\boldsymbol{\vartheta}} - \boldsymbol{\vartheta}_h)\|_{h} + \|\tilde{q}-q_h\|_{h} \precsim \ \|\mathbf{y}(\u_h) - \mathbf{y}_h\|_{0}+\eta_d^{\w} + \Theta^{\w}.
	\end{align}
	For the state equation, Lemma \ref{Lemma-2.5.2.1.} provides the following result:
	\begin{align}
		\label{2.5.2.8}	\|(\y_h - \y(\u_h),\omega_h - \omega(\u_h))\|_{h} + \|p(\u_h)-p_h\|_{h} \precsim \eta_d^{\y} + \Theta^{\y}.
	\end{align}
	Substituting (\ref{2.5.2.7}-\ref{2.5.2.8}) into (\ref{2.5.2.5}) and using \eqref{2.5.1.2}, we achieve the desired estimate.
	\end{proof}
\begin{remark}
	By following a similar approach as in the proof of Theorem~\ref{Theorem-2.5.1.5} and \cite[Theorem~5.9]{HSAK}, we have the following \textbf{efficiency} estimate:
	\begin{align*}
					\eta_d^{\y}+ \eta_d^{\w}+\eta_d^{\u} \precsim \	\|\u-\u_h\|_{0} & + \|(\y-\y_h,\boldsymbol{\omega} - \boldsymbol{\omega}_h)\|_{h} + \|p-p_h\|_{h} \\
					&+ \|(\w-\w_h,\boldsymbol{\vartheta} - \boldsymbol{\vartheta}_h)\|_{h} + \|q-q_h\|_{h} + \Theta^{\y}+ \Theta^{\w}.
	\end{align*}
\end{remark}
	\section{Numerical experiments}\label{Numerical Experiments.}
	In this section, we conduct numerical experiments to validate the theoretical convergence rates and to showcase effectiveness of the proposed methods in incompressible flows on different domains. The uniqueness of pressure ensured by the zero-mean condition, is enforced using a real Lagrange multiplier. To solve the linear systems, we employ the multifrontal massively parallel sparse direct solver MUMPS in Fenics \cite{AMB}. 
	Additionally, we integrate an adaptive mesh refinement technique inspired from \cite[Section~6]{SMAHAJAN} and \cite[Section~5]{VERFP}.
	Throughout all experiments, the control cost parameter is set to be $\gamma = 1$. We define the global estimators $\boldsymbol{\eta_{\mathbf{CG}}}$ (conforming), $\boldsymbol{\eta_{\mathbf{DG}}}$ (non-conforming), and total errors $\boldsymbol{\mathcal{T \hspace{-0.35mm} E}_{CG}}$ (conforming) and $\boldsymbol{\mathcal{T \hspace{-0.35mm} E}_{DG}}$ (non-conforming) as follows:
			\begin{align*}
		\boldsymbol{\eta_{\mathbf{CG}}} &:= \big((\eta_c^{\y})^2 + (\eta_c^{\w})^2 +(\eta_c^{\u})^2\big)^{1/2}, \qquad \boldsymbol{\eta_{\mathbf{DG}}}:= \big((\eta_d^{\y})^2 + (\eta_d^{\w})^2 +(\eta_d^{\u})^2\big)^{1/2},\\
		\boldsymbol{\mathcal{T \hspace{-0.35mm} E}_{\mathbf{CG}}} &:= \big(\|\u-\u_h\|_{0}^{2}+\|(\y-\y_h,\omega - \omega_h)\|^{2} + \|p-p_h\|_{0}^{2} + \|(\w-\w_h,\vartheta - \vartheta_h)\|^{2}  + \|q-q_h\|_{0}^{2}\big)^{1/2},\\
		\boldsymbol{\mathcal{T \hspace{-0.35mm} E}_{\mathbf{DG}}} & := \big(\|\u-\u_h\|_{0}^{2}+\|(\y-\y_h,\omega - \omega_h)\|_h^{2} + \|p-p_h\|_h^{2} + \|(\w-\w_h,\vartheta - \vartheta_h)\|_h^{2} + \|q-q_h\|_h^{2}\big)^{1/2}.
	\end{align*}
	\subsection{Convergence test using manufactured solutions}\label{Example 5.1.}
	The first experiment seeks to estimate precise solutions analytically within a two-dimensional domain $\Omega = (0,1)^{2}$. We construct the forcing term $\mathbf{f}$, target velocity and vorticity fields $\mathbf{y}_d$ and $\kappa_d$ such that the exact solutions to the optimal control problem are the subsequent smooth functions:
	\begin{align*}
		\y(x_1,x_2) &= \textbf{curl}\big((\sin(\pi x_1) \sin(\pi x_2))^{2}\big),  &&\kappa(x_1,x_2) = \textbf{curl}(\y), \ &&& p(x_1,x_2) = \cos(2\pi x_1)\cos(2\pi x_2),\\
		\w(x_1,x_2) &= \textbf{curl} \big((\sin(2 \pi x_1) \sin(2 \pi x_2))^{2}\big), \ && \vartheta(x_1,x_2) = \textbf{curl}(\w), \ &&& q(x_1,x_2) = \sin(2\pi x_1)\sin(2\pi x_2).
	\end{align*} 
	We set $\boldsymbol{\beta} = \mathbf{y}$, $\sigma = 100$, and $\nu(x_1,x_2) = 0.001 + 0.999 x_1 x_2$. The control bounds are chosen as $\mathbf{a}=(-0.5,-0.5)^{T}$ and $\mathbf{b}=(0.5,0.5)^{T}$. Figure~\ref{FIGURE 1} indicates the effectiveness of both conforming ($k=1$) and non-conforming ($k=0$) numerical schemes in effectively approximating the state, co-state, and control variables. We observe that global indicators and total errors decay at an optimal rate as shown in Figure~\ref{FIGURE 1}. Figure~\ref{FIGURE 2} visualises the behaviour of all numerical solutions. The smooth and continuous form of the solutions demonstrates the numerical methods stability and accuracy. Overall, these findings confirm the validity of both approaches for capturing the dynamics of the studied system, providing helpful insights into its behaviour and parameters.
\begin{figure}
	\centering
	\subfloat[CG scheme $(k=1)$]{\includegraphics[scale=0.18]{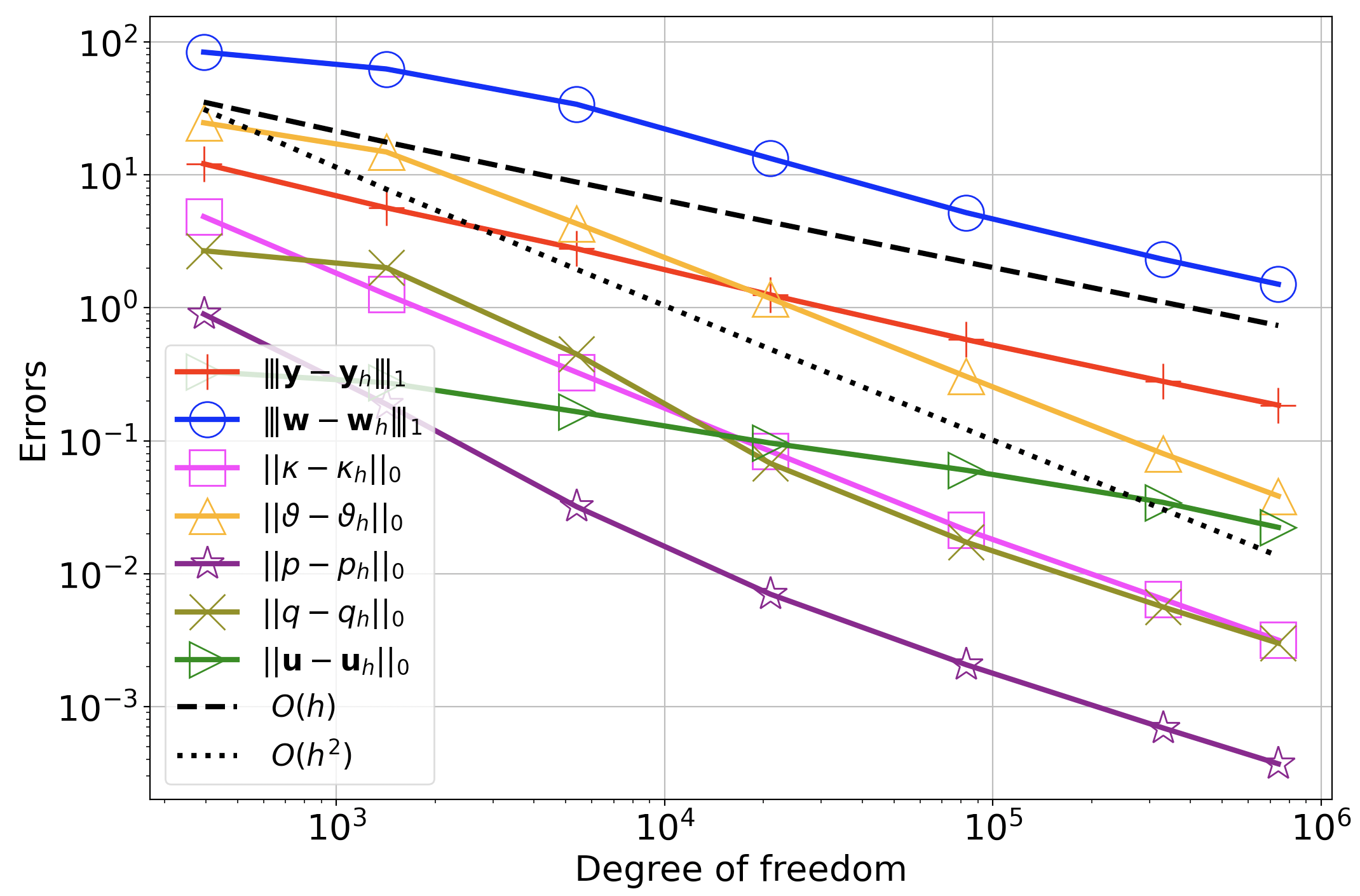}}
	\subfloat[DG scheme $(k=0)$]{\includegraphics[scale=0.18]{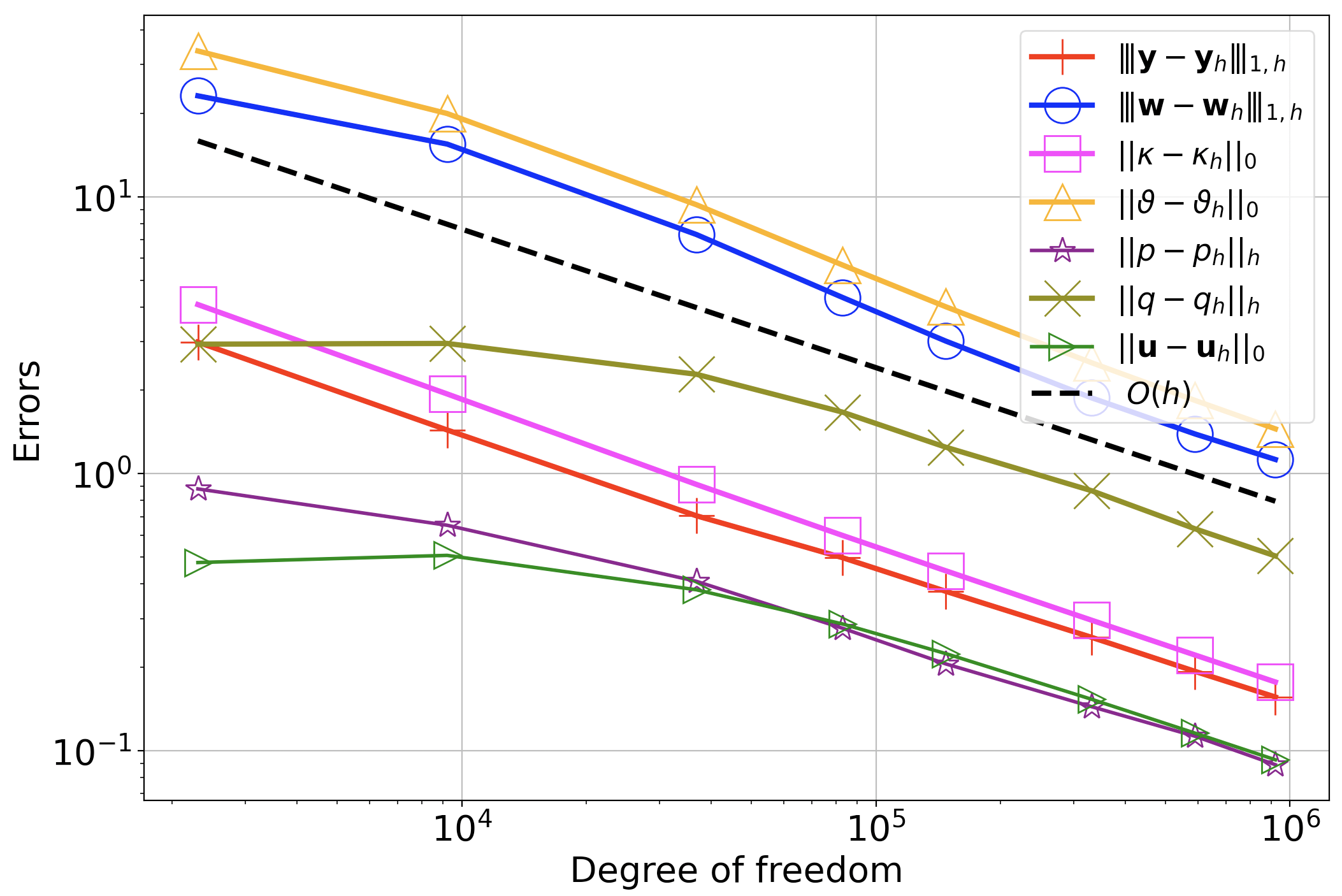}}\vspace{-2.25mm}
		\subfloat[Indicator-Total error (CG)]{\includegraphics[scale=0.162]{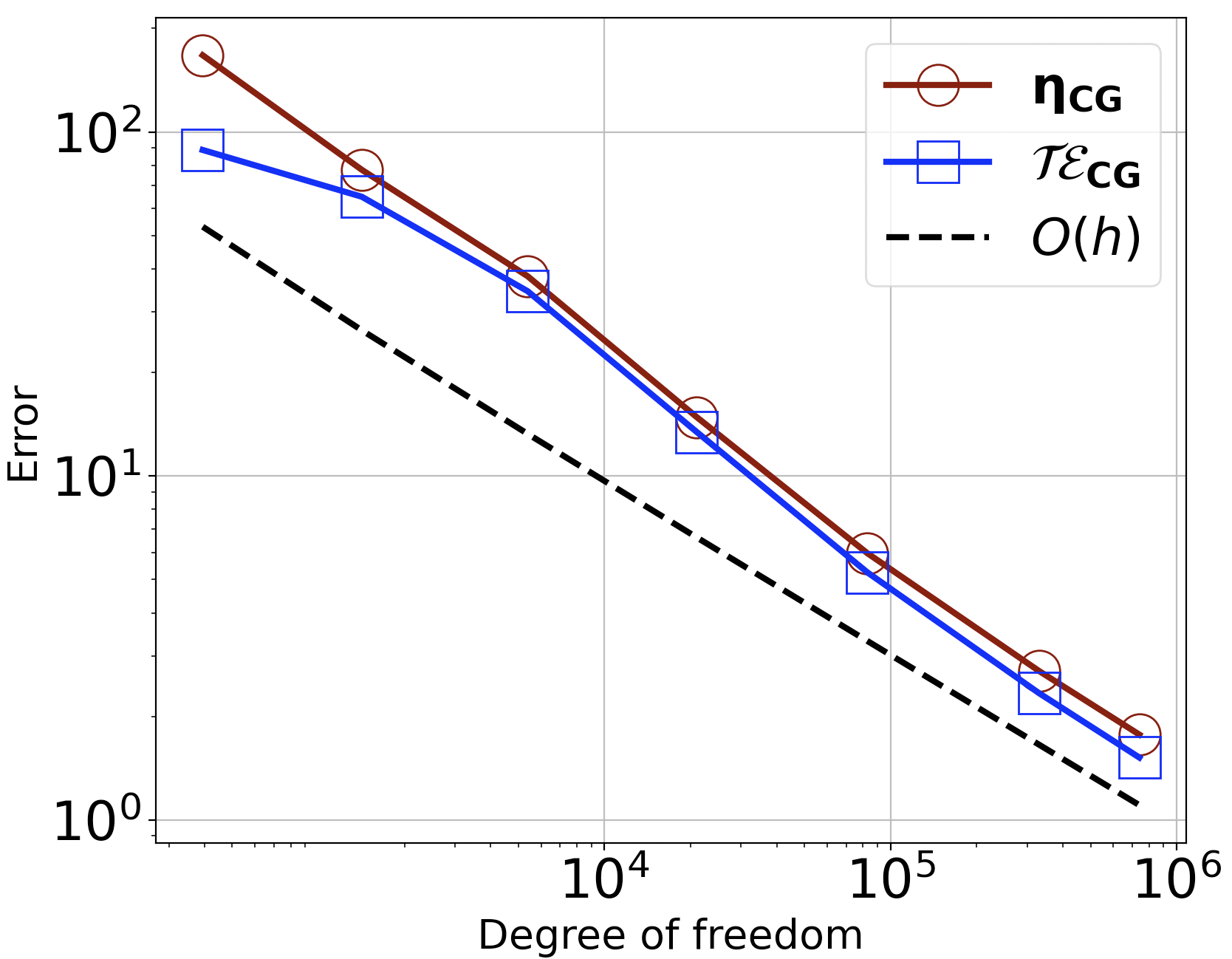}}
	\subfloat[Indicator-Total error (DG)]{\includegraphics[scale=0.162]{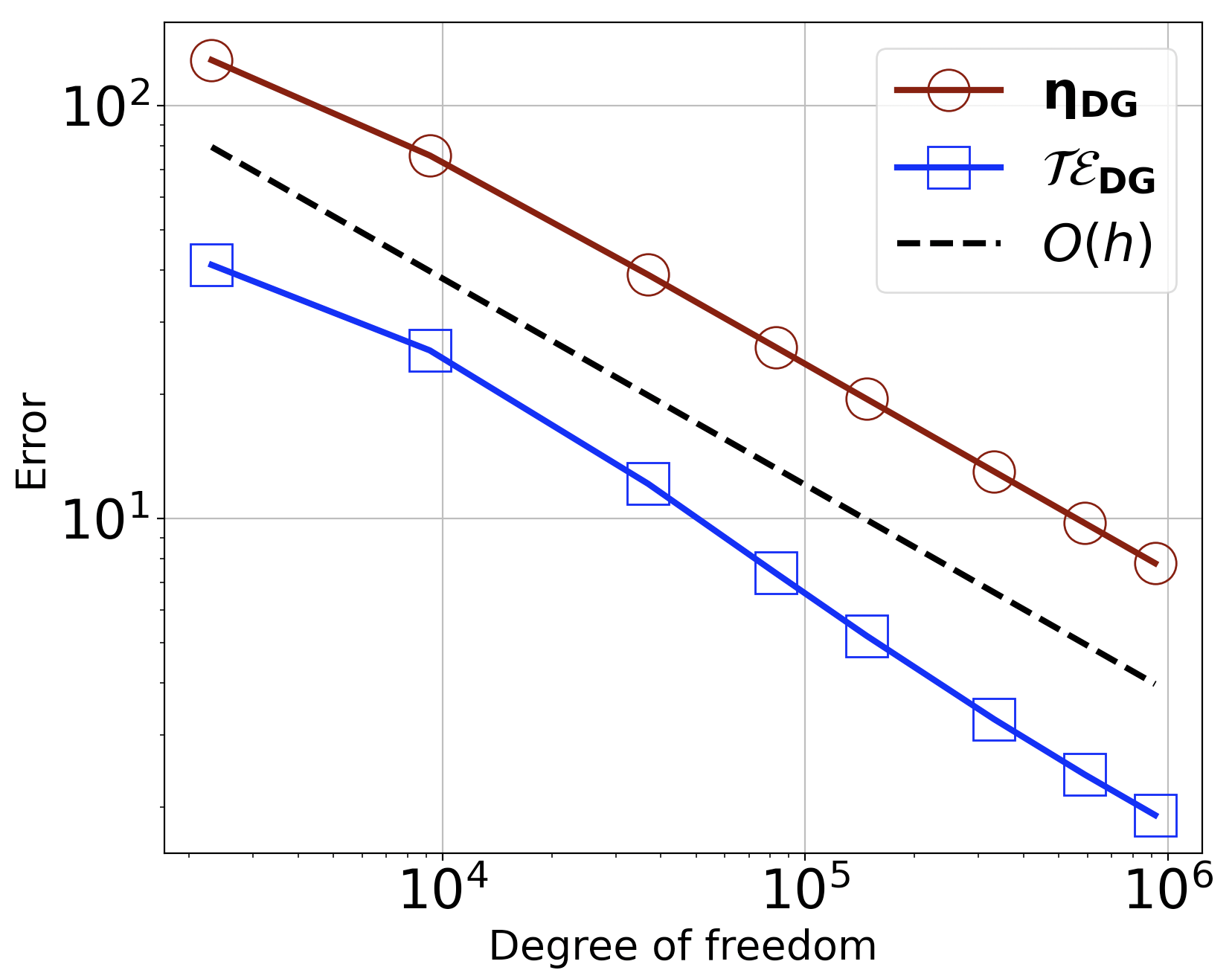}}
	\subfloat[Efficiency (CG-DG)]{\includegraphics[scale=0.162]{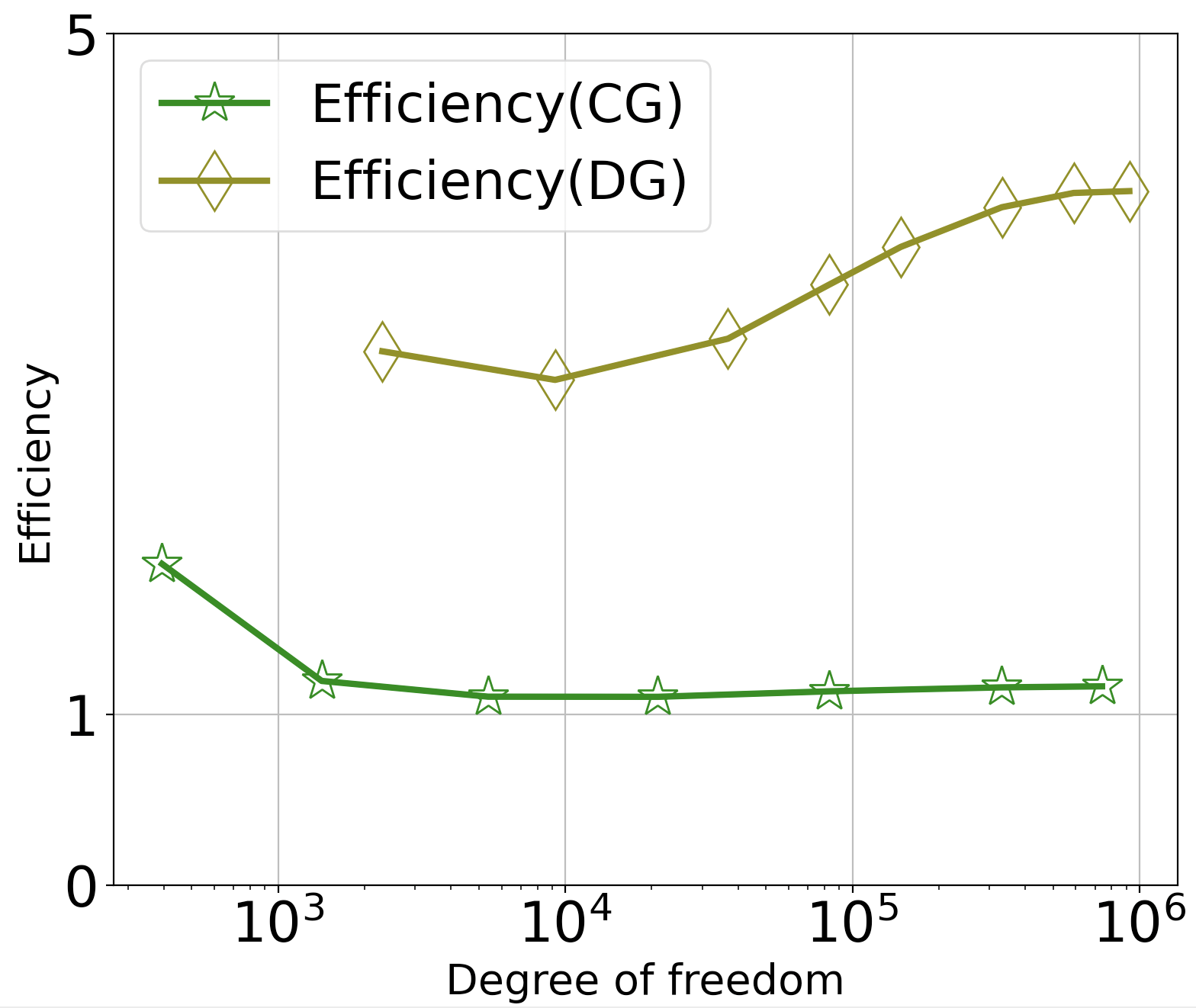}}\vspace{-2mm}
	\caption{Convergence plots for the (a) CG scheme (b) DG scheme (c) Indicator-Total error (CG) (d) Indicator-Total error (DG) and (e) Efficiency under uniform refinement for Example~\ref{Example 5.1.}.}
	\label{FIGURE 1}
\end{figure}
\begin{figure}
	\centering
	\subfloat[$\y_{h1}$]{\includegraphics[scale=0.145]{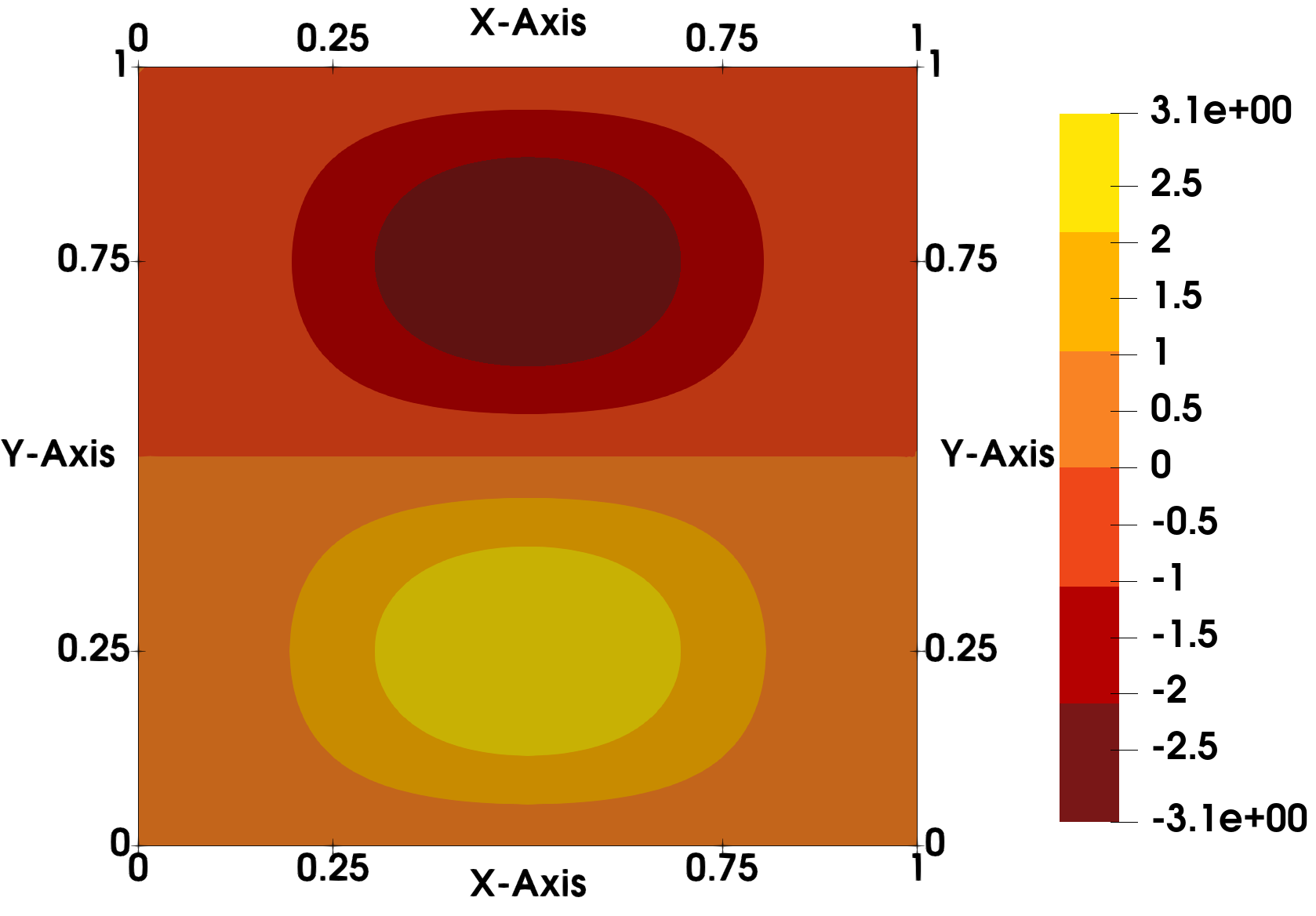}}
	\subfloat[$\y_{h2}$]{\includegraphics[scale=0.145]{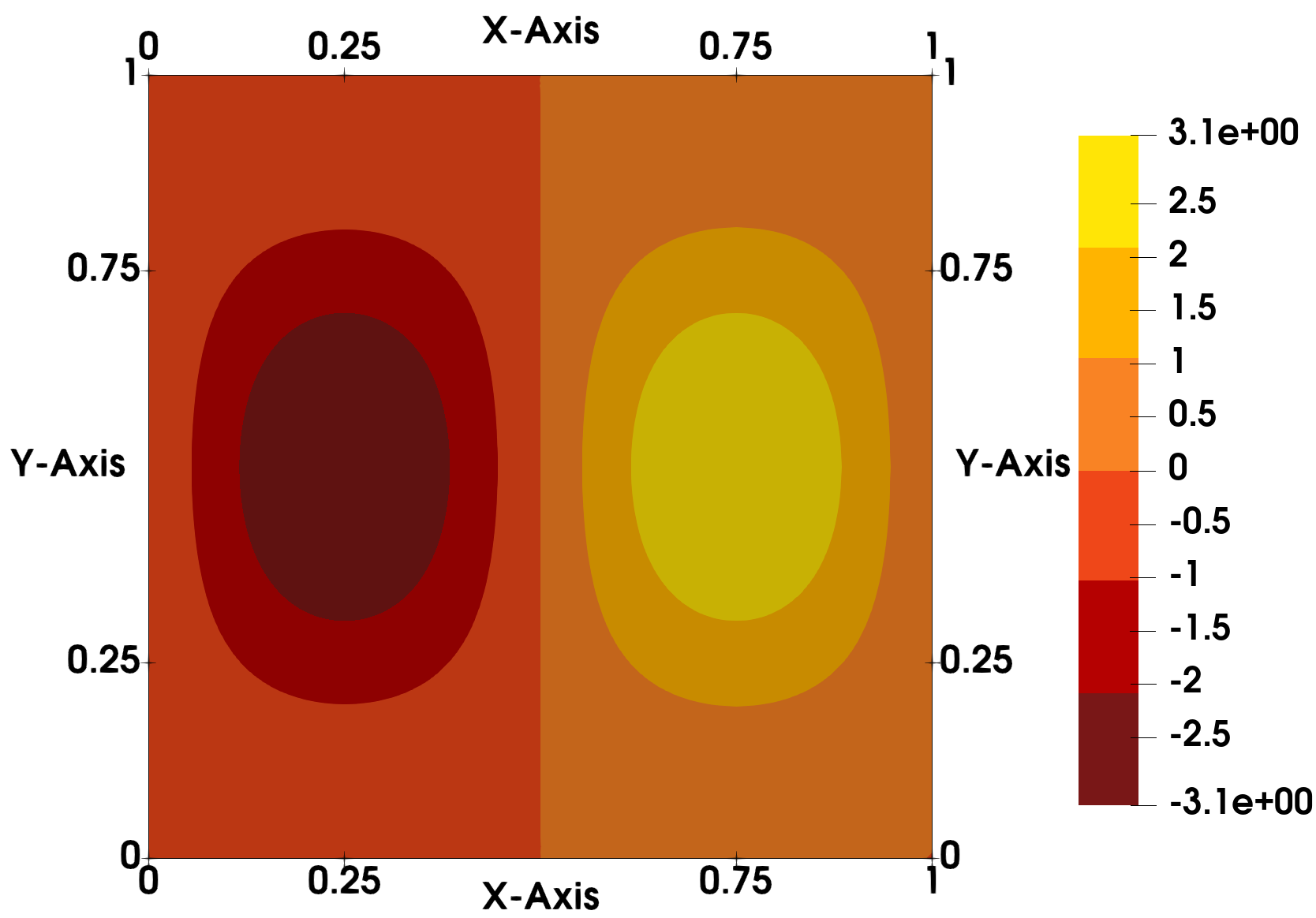}} 
	\subfloat[$\omega_h$]{\includegraphics[scale=0.145]{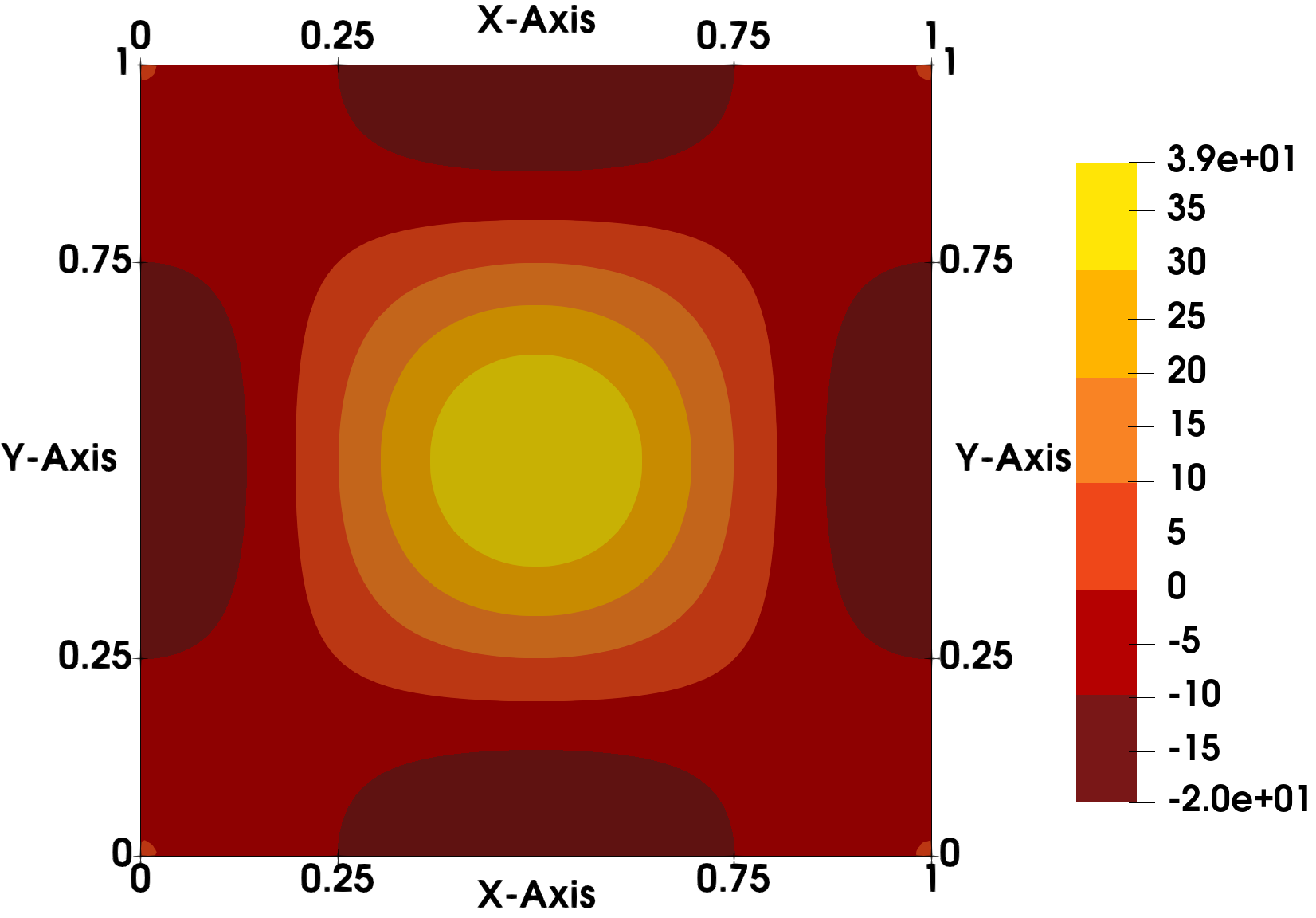}}
	\subfloat[$p_h$]{\includegraphics[scale=0.105]{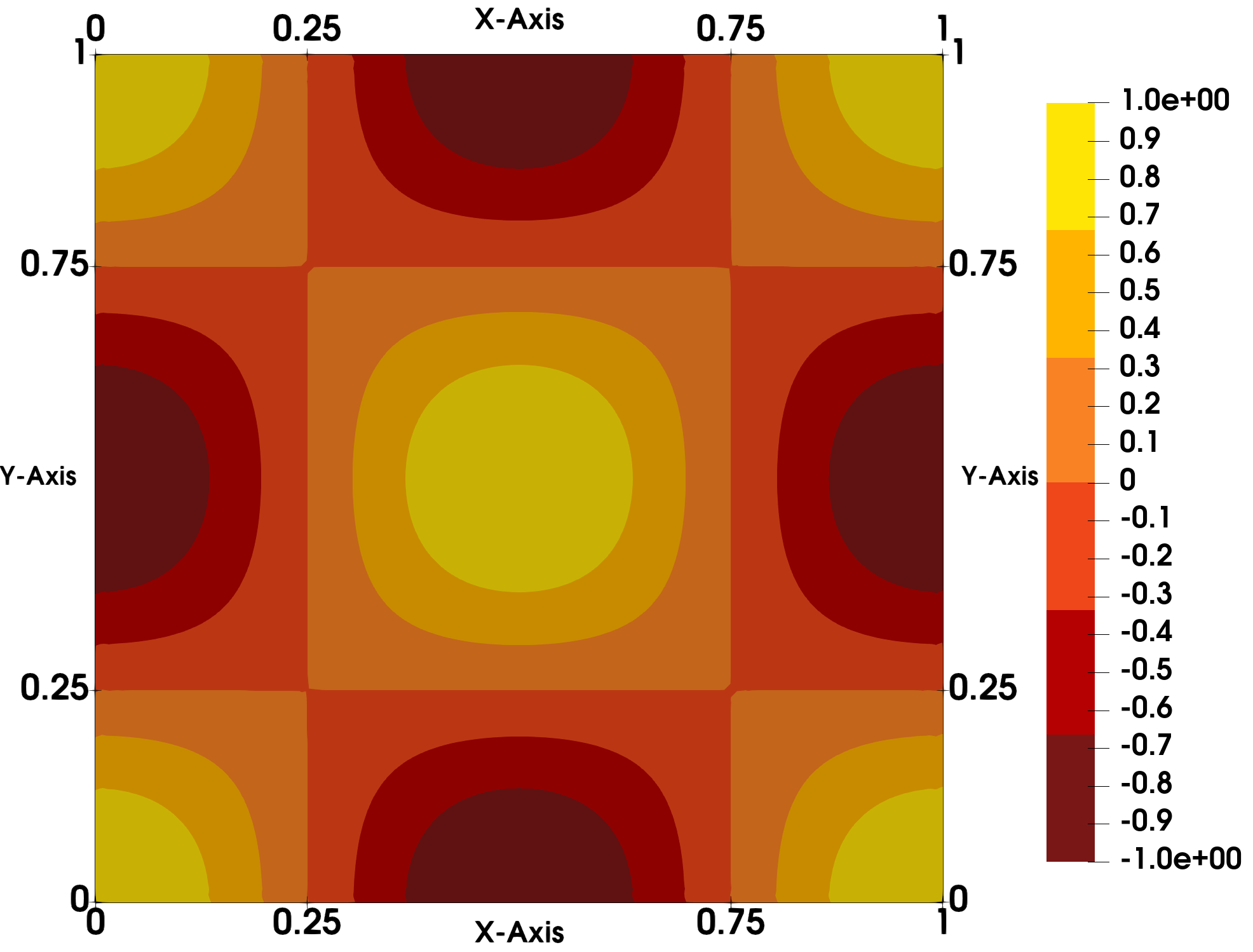}}\vspace{-2mm}\\
	\subfloat[$\w_{h1}$]{\includegraphics[scale=0.145]{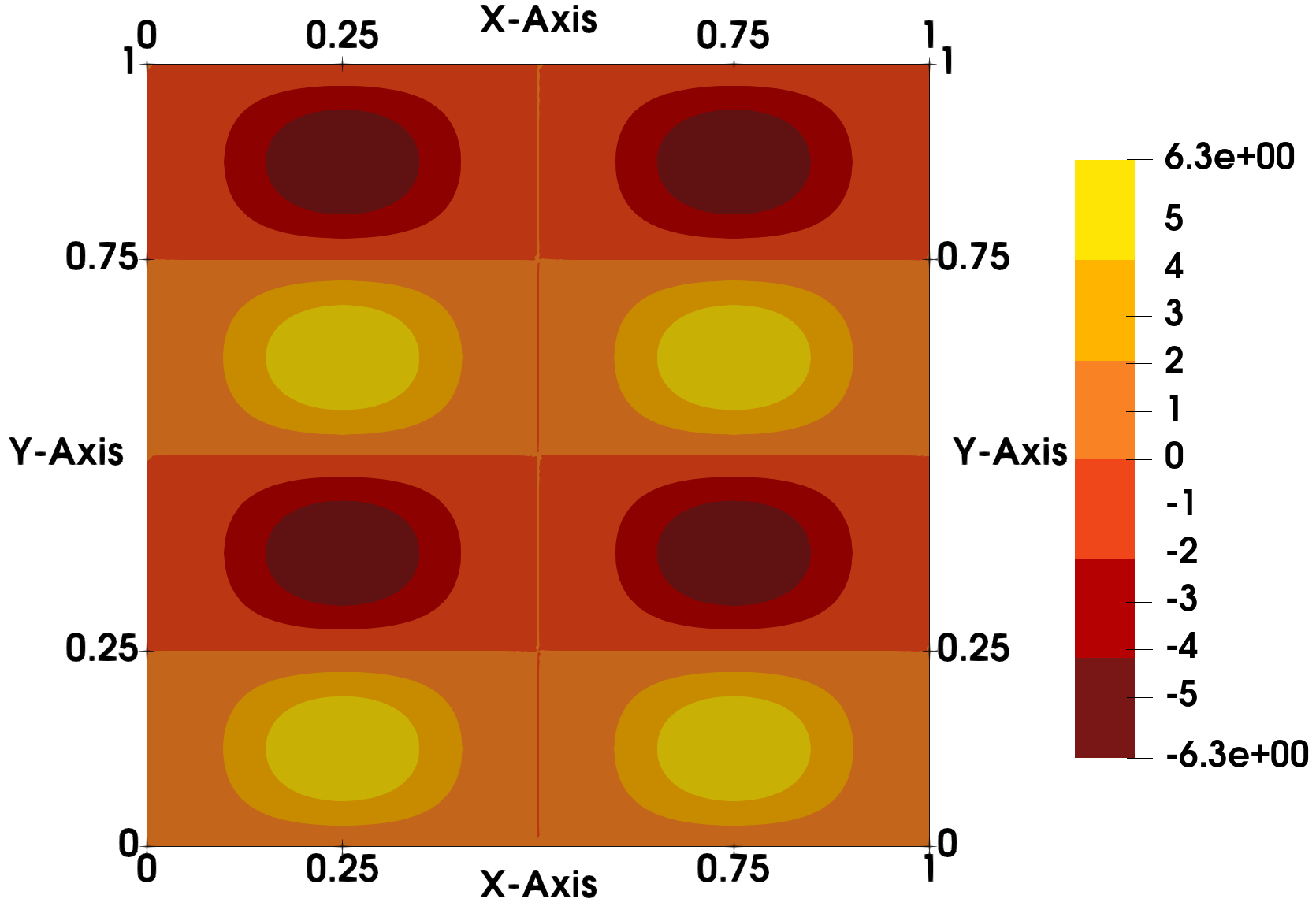}}
	\subfloat[$\w_{h2}$]{\includegraphics[scale=0.145]{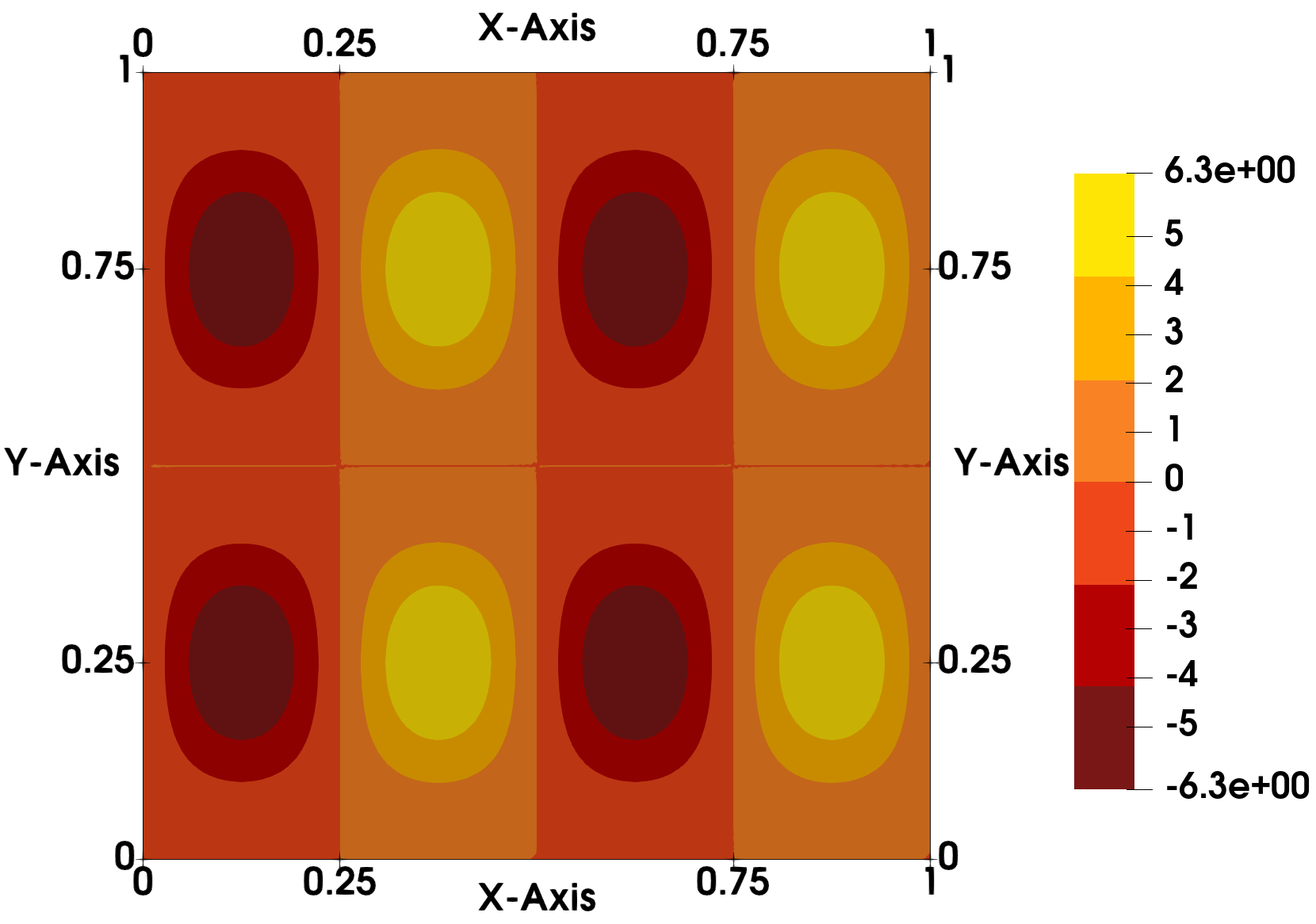}}
	\subfloat[$\vartheta_h$]{\includegraphics[scale=0.145]{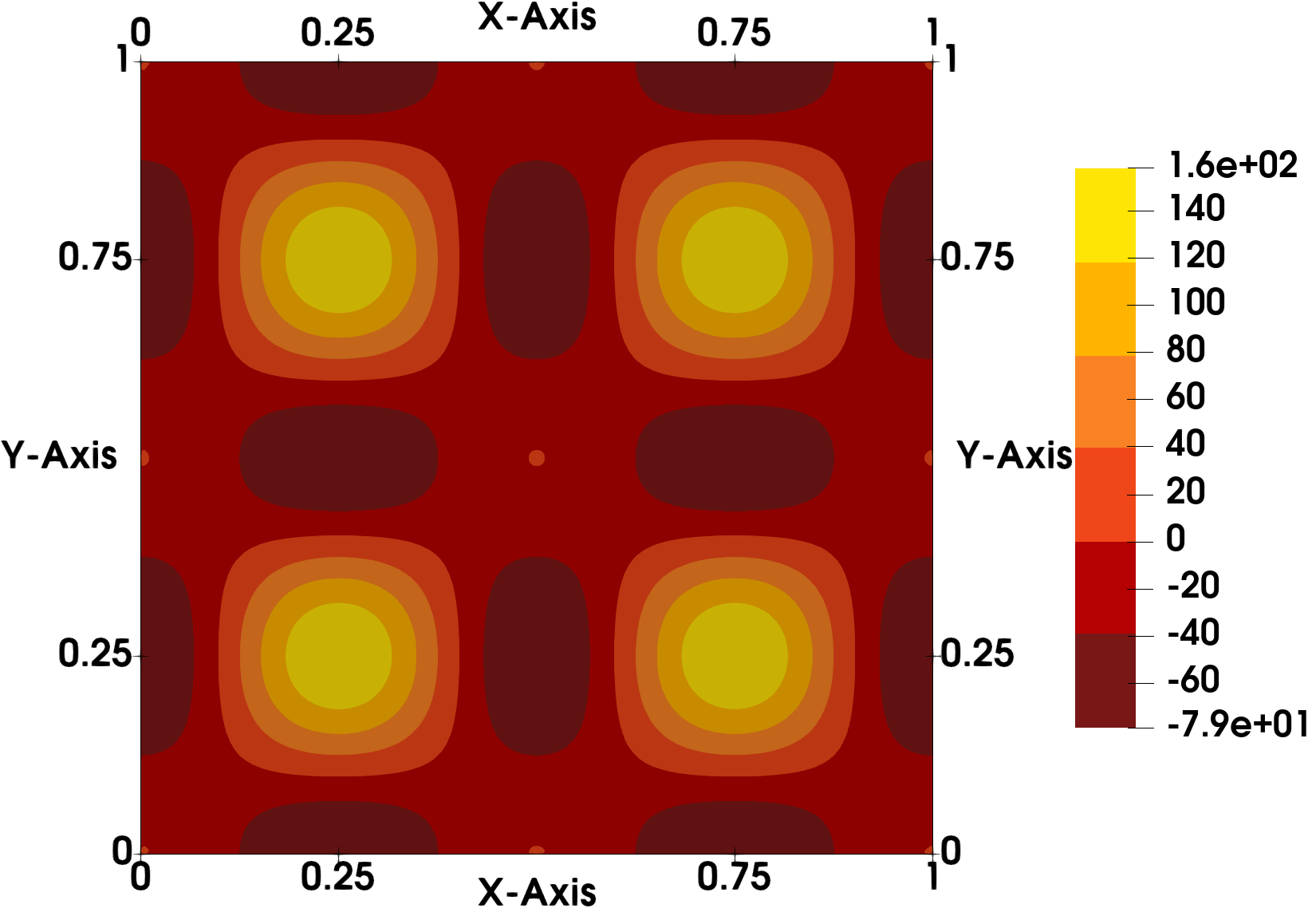}}
	\subfloat[$q_h$]{\includegraphics[scale=0.105]{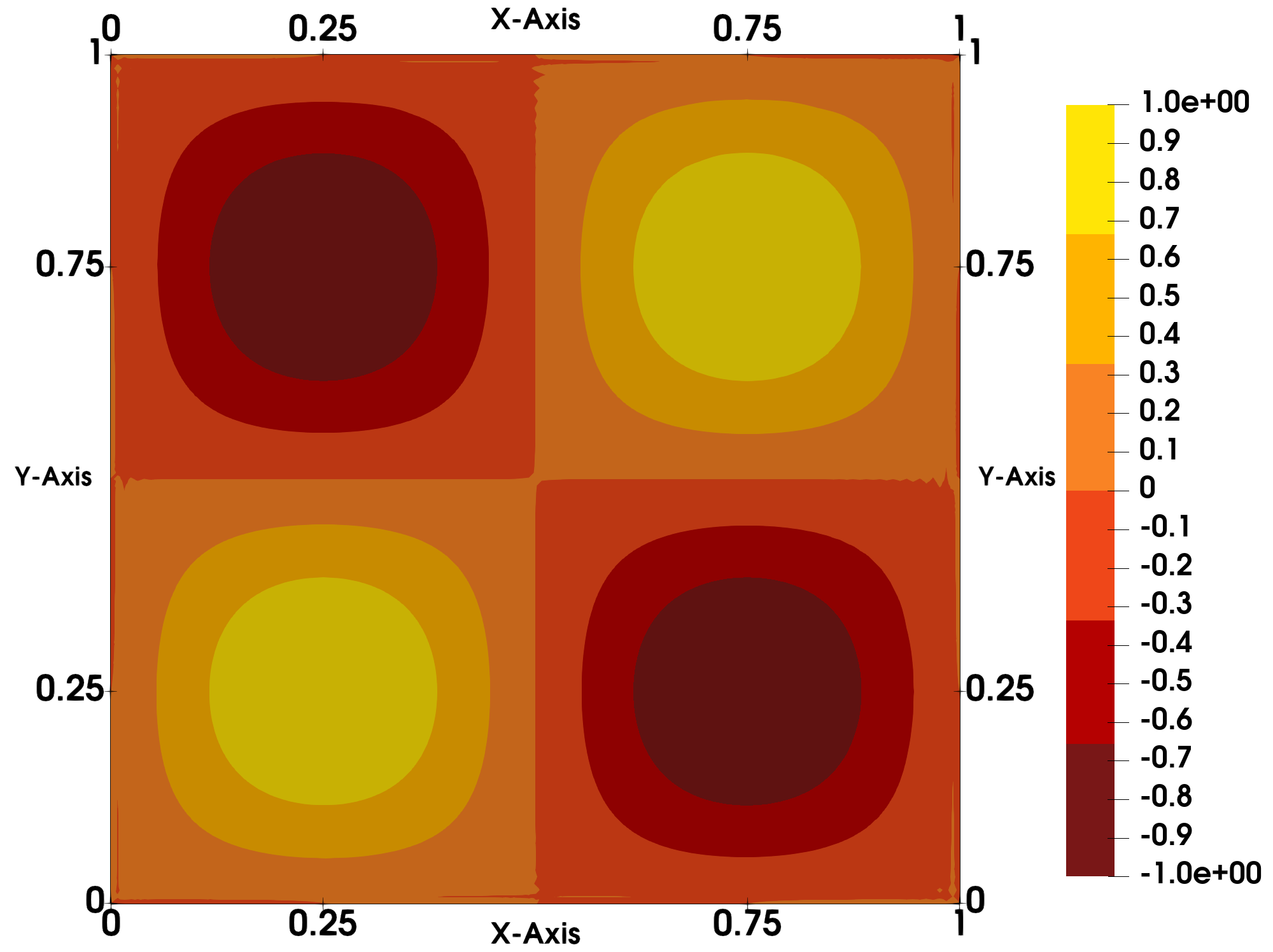}} \vspace{-2mm}\\
	\subfloat[$\u_{h1}$]{\includegraphics[scale=0.156]{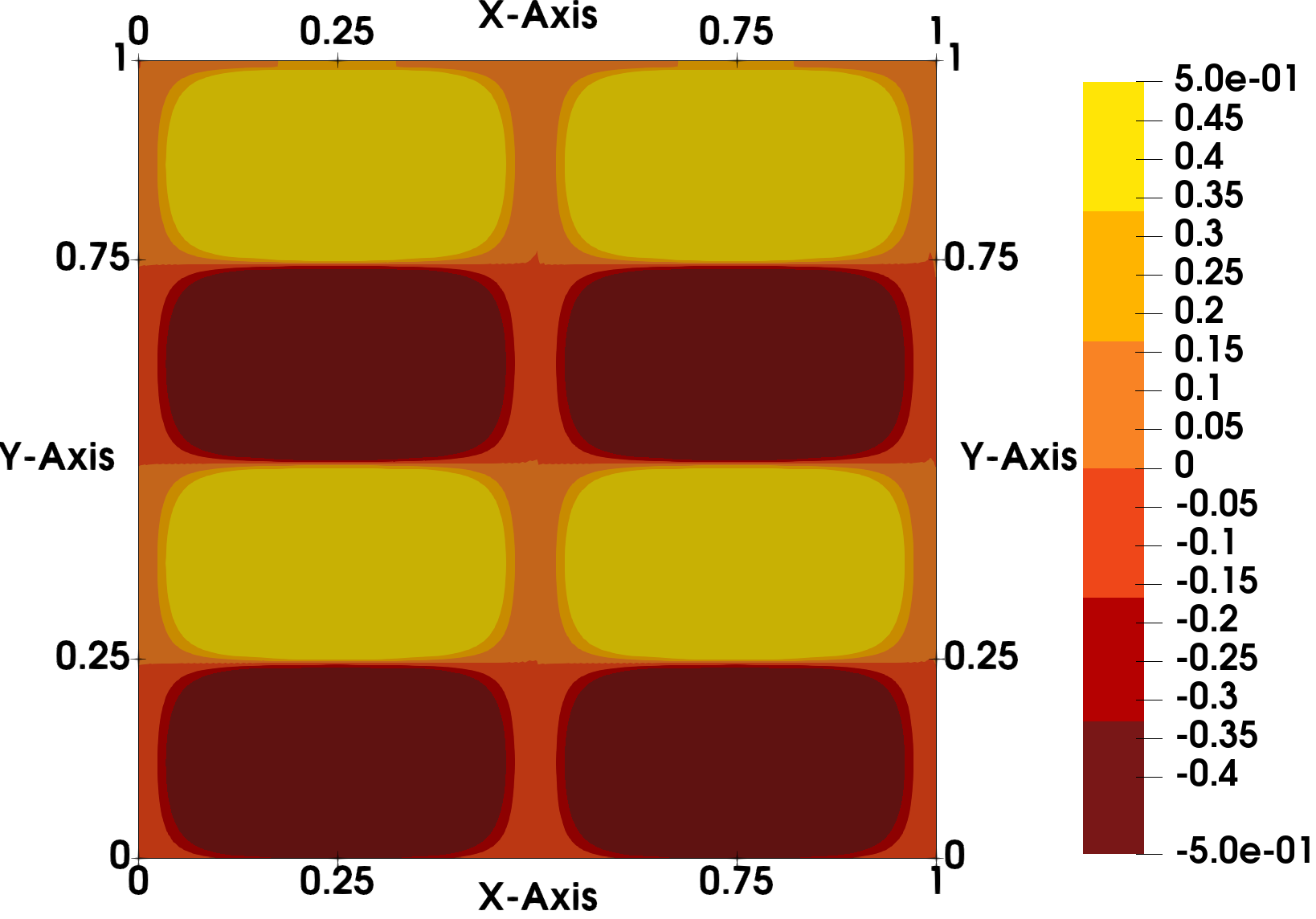}}
	\subfloat[$\u_{h2}$]{\includegraphics[scale=0.156]{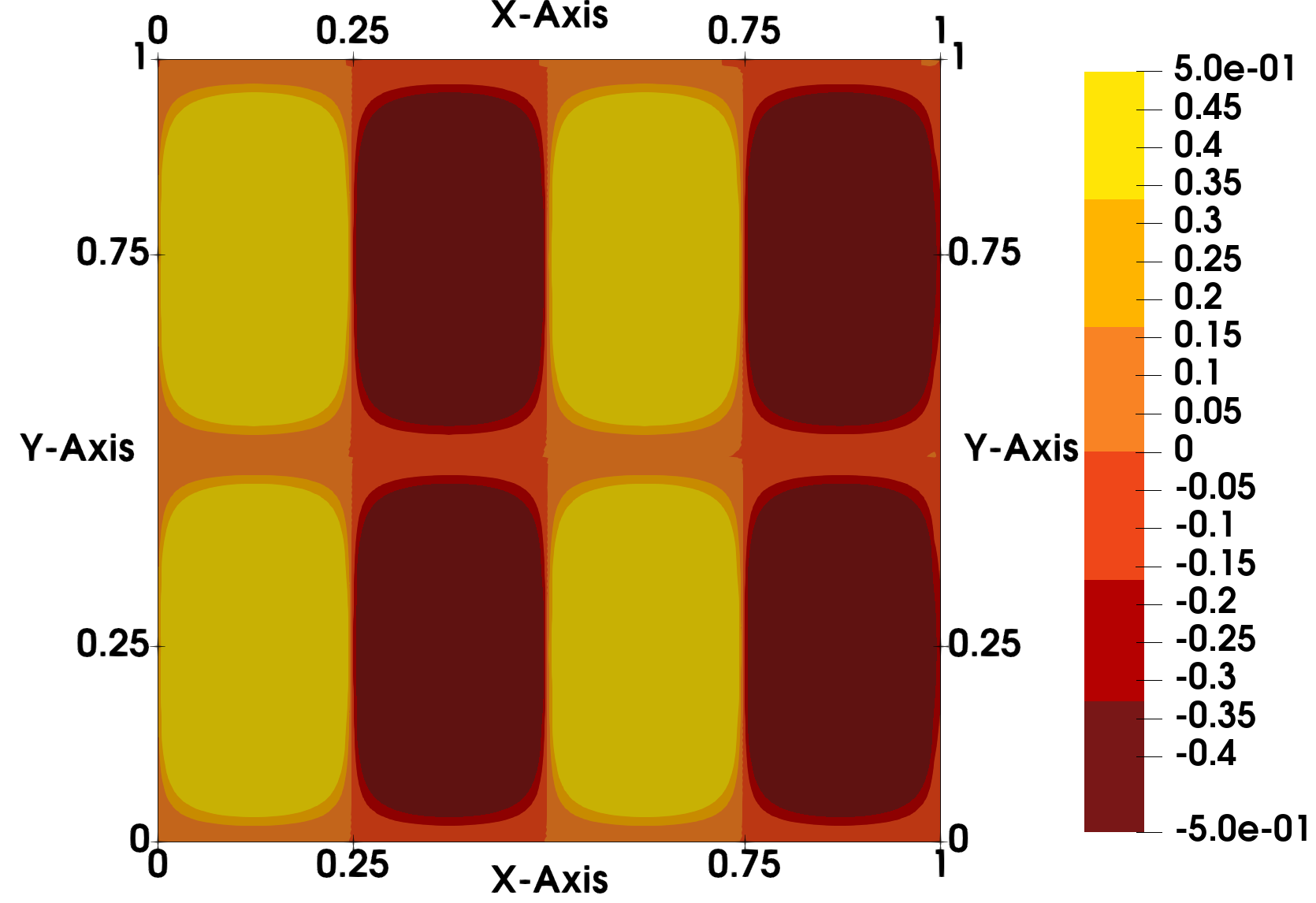}} \vspace{-2.25mm}
	\caption{Plots of numerical solutions of state velocity $(\y_{h1},\y_{h2})$, vorticity $(\kappa_h)$, pressure $(p_h)$, co-state velocity $(\w_{h1},\w_{h2})$, vorticity $(\vartheta_h)$, pressure $(q_h)$, and control $(\u_{h1},\u_{h2})$, respectively, for Example- \ref{Example 5.1.}.} \label{FIGURE 2}
\end{figure}
\begin{remark}
	The choice of augmentation constants $\rho_1$ and $\rho_2$ is important for achieving optimal convergence rates. An incorrect selection of these constants can negatively impact convergence, as demonstrated in Table-\ref{tab:example}. This table illustrates the effect of various $\rho_1$ and $\rho_2$ values on convergence when employing Mini-elements for velocity-pressure and continuous vorticity approximation with $(k=1)$ as compared to the optimal convergence rates in Figure \ref{FIGURE 1} (a).
\end{remark}
	\begin{sidewaystable}
		\caption{Effect of augmentation constants $\rho_1$ and $\rho_2$ on convergence rates. This study employs a conforming scheme on a $2^N \times 2^N$ mesh, utilizing lowest order Mini-elements for the velocity-pressure and a continuous piecewise linear polynomial space to approximate vorticity for Example~\ref{Example 5.1.}.}
		\label{tab:example}
		\begin{tabular}{|p{0.12cm}|p{1.21cm}|p{1.52cm}|p{0.55cm}|p{1.51cm}|p{0.55cm}|p{1.45cm}|p{0.55cm}|p{1.65cm}|p{0.6cm}|p{1.52cm}|p{0.6cm}|p{1.44cm}|p{0.8cm}|p{1.42cm}|p{0.55cm}|}
			\hline
			$N$ & DOF & $\tiny{|\!|\!|\y-\y_h|\!|\!|_{1}}$ & $r(\y)$ & $\tiny{\|\kappa-\kappa_h\|_{0}}$ & $r(\kappa)$ & $\tiny{\|p-p_h\|_{0}}$& $r(p)$ & $\tiny{|\!|\!|\w-\w_h|\!|\!|_{1}}$  & $r(\w)$ & $\tiny{\|\vartheta-\vartheta_h\|_{0}}$ & $r(\vartheta)$ & $\tiny{\|q-q_h\|_{0}}$& $r(q)$& $\|\u-\u_h\|_{0}$& $r(\u)$  \\
			\hline
			\multicolumn{16}{|l|}{\hspace{9.5cm}$\rho_1 = \rho_2 = 0 $} \\
			\hline
			$2$ & $394$ & $13.3122$ & - & $4.8394$ & - & $0.67089$ & - & $92.1523$ & - & $23.3484$ & - & $1.5521$ & - & $0.3358$ & - \\
			\hline
			$3$ & $1418$ & $7.0534$ & $\mathbf{0.92}$ & $1.2386$ & $1.97$ & $0.15799$ & $2.09$ & $86.6640$ & $\mathbf{0.09}$ & $15.0373$& $0.63$ & $2.3580$ & $-0.6$ & $0.3275$ & $0.04$ \\
			\hline
			$4$ & $5386$& $5.3275$ & $\mathbf{0.4}$ & $0.3174$ & $1.96$ & $0.04467$ & $1.82$ & $79.2136$ & $\mathbf{0.13}$ & $4.2023$ & $1.84$ & $0.83786$ & $1.49$ & $0.2049$ & $0.68$ \\
			\hline
			$5$ & $21002$ & $4.8844$ & $\mathbf{0.13}$ & $0.0853$ & $1.89$ & $0.01223$& $1.87$ & $75.1424$ & $\mathbf{0.08}$ & $1.11058$ & $1.92$ & $0.2517$ & $1.73$ & $0.1128$ & $0.86$ \\
			\hline
			$6$ & $82954$ & $4.7762$ & $\mathbf{0.03}$ & $0.02626$ & $1.7$ & $0.00334$ & $1.87$ & $73.3655$ & $\mathbf{0.03}$ & $0.3033$ & $1.87$ & $0.07246$ & $1.8$ & $0.0688$ & $0.71$ \\
			\hline
			$7$ & $329738$  & $4.7513$ & $\mathbf{0.01}$ & $0.0098$ & $1.42$ & $0.000924$ & $1.85$ & $72.8745$ & $\mathbf{0.01}$ & $0.09211$ & $1.72$& $0.0204$ & $1.82$ & $0.0398$ & $0.79$ \\
			\hline
			\multicolumn{16}{|l|}{\hspace{9.5cm} $\rho_1 = 2\nu_0/3, \quad \rho_2 = 0 $} \\ \hline
			$2$ & $394$ & $13.3031$ & - & $4.8385$ & - & $0.67068$ & - & $92.0294$ & - & $23.38158$ & - & $1.55918$ & - & $0.3358$ & - \\ \hline
			$3$ & $1418$ & $6.9983$ & $0.93$ & $1.23864$ & $1.97$ & $0.15737$ & $\mathbf{2.09}$ & $85.7812$ & $0.1$ & $15.0396$ & $0.64$ & $2.3396$ & $\mathbf{-0.59}$ & $0.3262$ & $0.04$ \\ \hline
			$4$ & $5386$ & $ 5.07996$ & $0.46$ & $0.31804$ & $1.96$ & $0.04301$ & $\mathbf{1.87}$ & $74.4968$ & $0.2$ & $4.2197$ & $1.83$ & $0.79374$ & $\mathbf{1.56}$ & $0.19939$ & $0.71$ \\ \hline
			$5$ & $21002$ & $4.0398$ & $0.33$& $ 0.08522$ & $1.9$ & $0.01038$ & $\mathbf{2.05}$ & $59.79037$ & $0.32$ & $1.13144$ & $1.9$ & $0.19893$ & $\mathbf{2.0}$ & $0.10641$ & $0.91$ \\ \hline
			$6$ & $82954$& $2.67845$ & $0.59$ & $0.02403$ & $1.83$ & $0.00200$ & $\mathbf{2.38}$ & $37.05938$ & $0.69$ & $0.30828$ & $1.88$ & $0.03479$& $\mathbf{2.52}$ & $0.06197$ & $0.78$ \\ \hline
			$7$ & $329738$ & $1.26965$ & $1.08$ & $0.00673$ & $1.84$ & $0.000296$ & $\mathbf{2.75}$ & $15.88539$ & $1.22$ & $0.08092$& $1.93$ & $0.00407$ & $\mathbf{3.1}$ & $0.03477$ & $0.83$ \\ \hline
			\multicolumn{16}{|l|}{\hspace{9.5cm} $\rho_1 =0, \quad \rho_2 = 0.5 \nu_{0} $} \\ \hline
			$2$ & $394$ & $13.29109$ & - & $4.8401$ & - & $0.67329$ & - & $92.0309$ & - & $23.36260$ & - & $1.5577$ & - & $0.33584$ & - \\ \hline
			$3$ & $1418$ & $7.00704$ & $0.92$ & $1.23988$ & $1.96$ & $0.157859$ & $\mathbf{2.09}$ & $86.06839$ & $0.1$ & $15.0436$ & $0.64$ & $2.34678$ & $\mathbf{-0.59}$ & $0.32680$ & $0.04$ \\ \hline
			$4$ & $5386$ & $5.190279$ & $0.43$ & $0.31848$ & $1.96$& $0.043507$ & $\mathbf{1.86}$ & $77.0860$ & $0.16$ & $4.21531$ & $1.84$ & $0.82901$ & $\mathbf{1.5}$& $0.20295$ & $0.69$ \\ \hline
			$5$ & $21002$ & $4.4229$ & $0.23$ & $0.085125$ & $1.9$ & $0.01102$ & $\mathbf{1.98}$ & $68.62444$ & $0.17$ & $1.11938$ & $1.91$ & $0.24099$ & $\mathbf{1.78}$ & $0.11013$ & $0.88$\\ \hline
			$6$ & $82954$& $3.44077$ & $0.36$ & $0.02331$ & $1.87$ & $0.00244$ & $\mathbf{2.18}$ & $55.11659$ & $0.32$ & $0.29621$ & $1.92$ & $0.06034$ & $\mathbf{2.0}$ & $0.06484$ & $0.76$ \\ \hline
			$7$ & $329738$ & $1.97081$ & $0.8$ & $0.00649$ & $1.84$ & $0.00043$ & $\mathbf{2.49}$ & $32.5027$ & $0.76$ & $0.07648$ & $1.95$ & $0.01092$ & $\mathbf{2.47}$ & $0.03566$ & $0.86$ \\ \hline
			\multicolumn{16}{|l|}{\hspace{9.5cm} $\rho_1 =2\nu_0/3, \quad \rho_2 = 0.1 \nu_{0} $} \\ \hline
			$2$ & $394$ & $13.29893$ & - & $4.8387$ & - & $0.67115$ & - & $92.00518$ & - & $23.3844$ & - & $1.56012$ & - & $0.33585$ & - \\ \hline
			$3$ & $1418$ &  $6.9889$ & $0.93$ & $1.23889$ & $1.97$ & $0.157346$ & $\mathbf{2.09}$ & $85.66354$ & $0.1$ & $15.04083$ & $0.64$ & $2.33741$ & $\mathbf{-0.58}$ & $0.32606$& $0.04$ \\ \hline
			$4$ & $5386$&  $5.05299$ & $0.47$ & $0.31825$ & $1.96$ & $0.04278$ & $\mathbf{1.88}$ & $74.10211$ & $0.21$ & $4.22227$ & $1.83$ & $0.792145$ & $\mathbf{1.56}$ & $0.19903$ & $0.71$ \\ \hline
			$5$ & $21002$ &  $3.95842$ & $0.35$ & $0.08518$ & $1.9$ & $0.01016$ & $\mathbf{2.07}$ & $58.79296$ & $0.33$ & $1.13287$ & $1.9$ & $0.19759$ & $\mathbf{2.0}$ & $0.10605$ & $0.91$ \\ \hline
			$6$ & $82954$&  $2.50213$ & $0.66$ & $0.02359$ & $1.85$ & $0.00187$ & $\mathbf{2.44}$  & $35.35574$ & $0.73$ & $0.30688$ & $1.88$ & $0.03418$ & $\mathbf{2.53}$ & $0.06166$ & $0.78$ \\ \hline
			$7$ & $329738$ & $1.08279$ & $1.21$ & $0.00648$ & $1.86$ & $0.000262$ & $\mathbf{2.83}$ & $14.25886$ & $1.31$ & $0.07979$ & $1.94$ & $ 0.00392$ & $\mathbf{3.12}$ & $0.03468$ & $0.83$ \\ \hline
		\end{tabular}
	\end{sidewaystable}
	\subsection{Adaptive mesh refinement for a boundary layers problem}\label{Example 5.2.}
	In this example, we explore a two-dimensional triangular region $\Omega = \{(x_1,x_2):x_1>0, \ x_2>0, \ x_1 + x_2 <1\}$, with coefficients defined as $\nu(x_1,x_2) = 1+0.001 x_1 x_2$,\ $ \boldsymbol{\beta} = (1,1)^{T},\ \sigma = 100$, and control constraints set as $\mathbf{a}=(0,0)^{T}$ and $\mathbf{b}=(0.1,0.1)^{T}$. The selection of source function $\mathbf{f}$, desired velocity $\y_d$, and vorticity $\omega_d$ is such that the manufactured solutions are:
	\begin{align*}
		\y(x_1,x_2) &= \textbf{curl} \ \bigg(x_1 x_2^{2}(1-x_1-x_2)^{2} \bigg(1-x_1- \frac{\exp(-50x_1)-\exp(-50)}{1-\exp(-50)}\bigg)\bigg), &&\kappa(x_1,x_2) = \textbf{curl}(\y), \\
		\w(x_1,x_2) &=  \textbf{curl} \ \bigg(x_1^{2} x_2(1-x_1-x_2)^{2} \bigg(1-x_2- \frac{\exp(-50x_2)-\exp(-50)}{1-\exp(-50)}\bigg)\bigg), && \vartheta(x_1,x_2) = \textbf{curl}(\w),\\
		p(x_1,x_2) &= \frac{\cos(2\pi x_2)}{1024}, &&q(\xi_1,x_2) = \frac{\cos(2\pi x_1)}{1024}.
	\end{align*}
	The presence of boundary layers in the solution impedes the convergence rates at which the state and co-state variables converge under uniform mesh refinement. We use an adaptive mesh refinement technique that targets regions connected with boundary layers specifically in order to overcome this. The adaptively refined mesh plots displayed in Figure~\ref{FIGURE 3} provide an illustration of this methodology. We find that ideal convergence rates emerge once mesh refinement reaches a significant level, as shown in Figure~\ref{FIGURE 2.1}. This notable increase in convergence rates shows that high-error locations are successfully targeted and resolved by the adaptive mesh refinement. Furthermore, as Figure~\ref{FIGURE 2.1} illustrates, the error indicator and total error both show an optimal decline, and their ratio (efficiency) becomes almost constant. Figure~\ref{FIGURE 5} provides more information on the numerical solutions of all state and co-state variables. The improved accuracy and resolution attained by adaptive refinement are demonstrated in these plots, highlighting the technique's significance in effectively resolving boundary layer issues.
	 \begin{figure}
	 	\centering
	 	\subfloat[1800 DOF]{\includegraphics[scale=0.322]{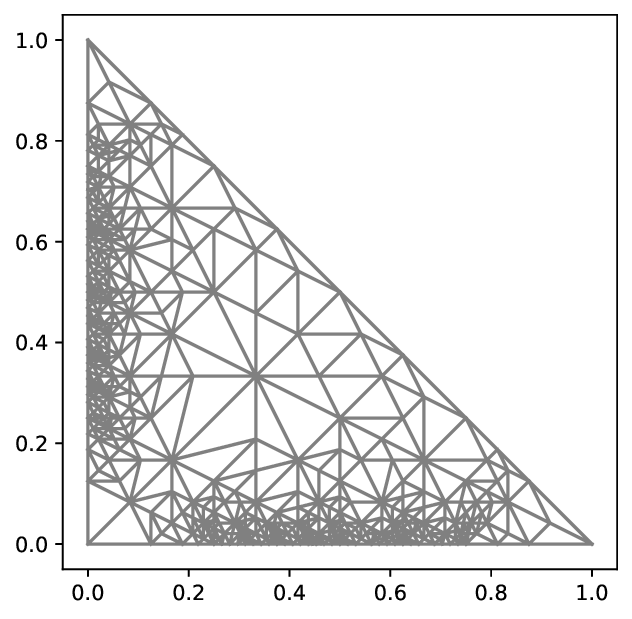}}
	 	\subfloat[8506 DOF]{\includegraphics[scale=0.322]{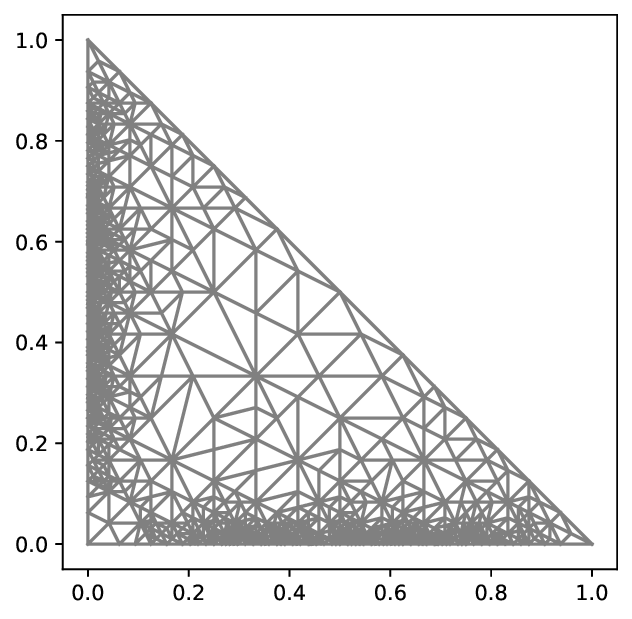}} 
	 	\subfloat[35756 DOF]{\includegraphics[scale=0.322]{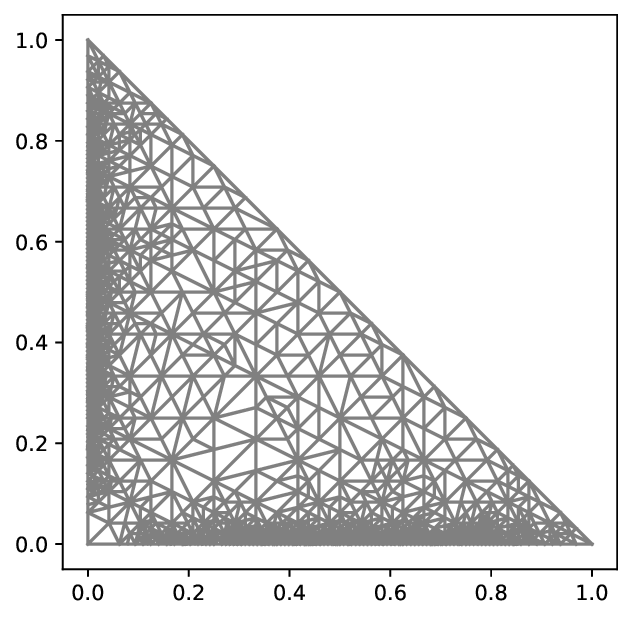}} \vspace{-3mm}
	 	\caption{Adaptively refined meshes showing refinement in the region of boundary layers for Example~\ref{Example 5.2.}.} \label{FIGURE 3}
	 \end{figure}
	\begin{figure} 
	\centering
	\subfloat[CG scheme $(k=1)$]{\includegraphics[scale=0.18]{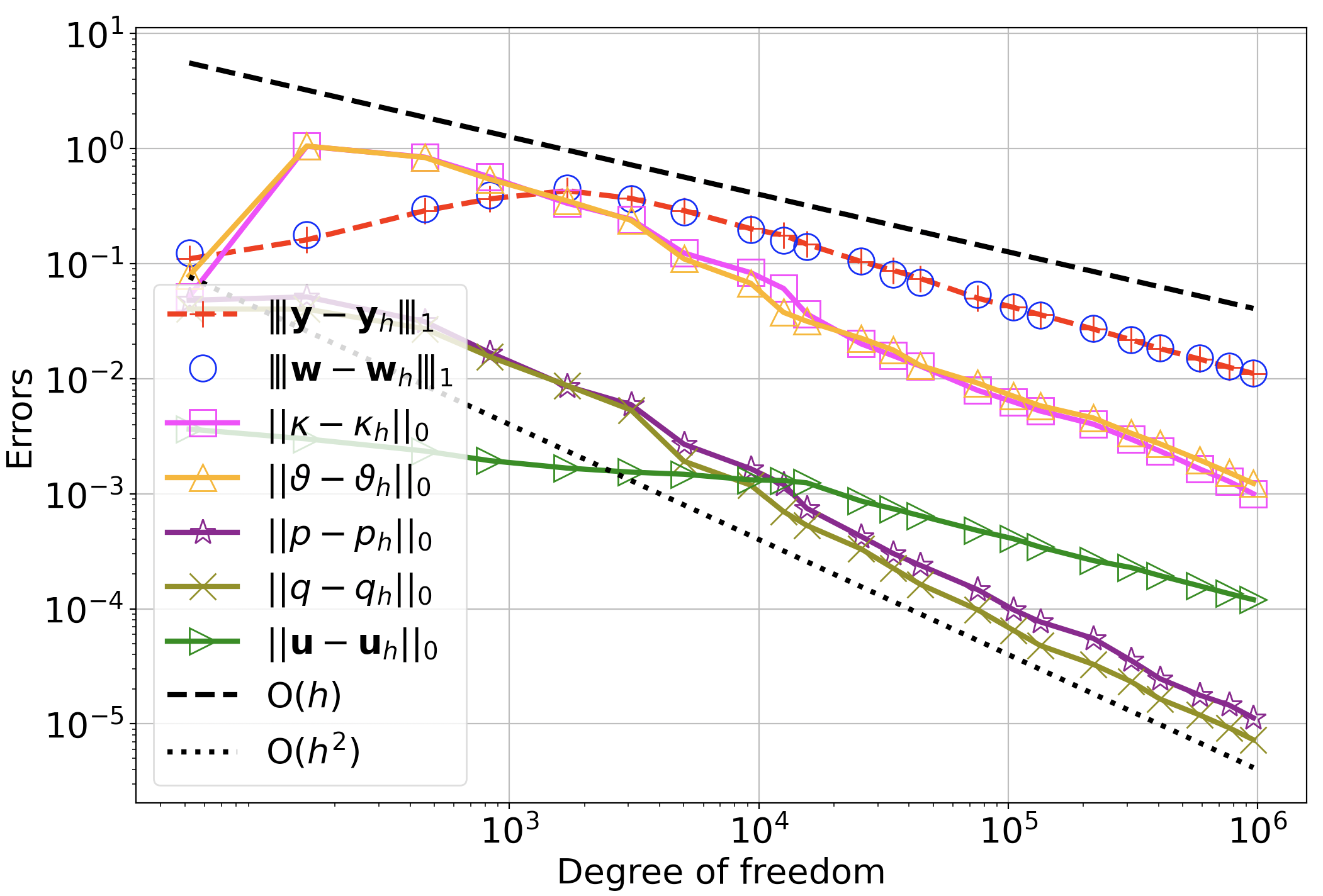}}
	\subfloat[DG scheme $(k=0)$]{\includegraphics[scale=0.18]{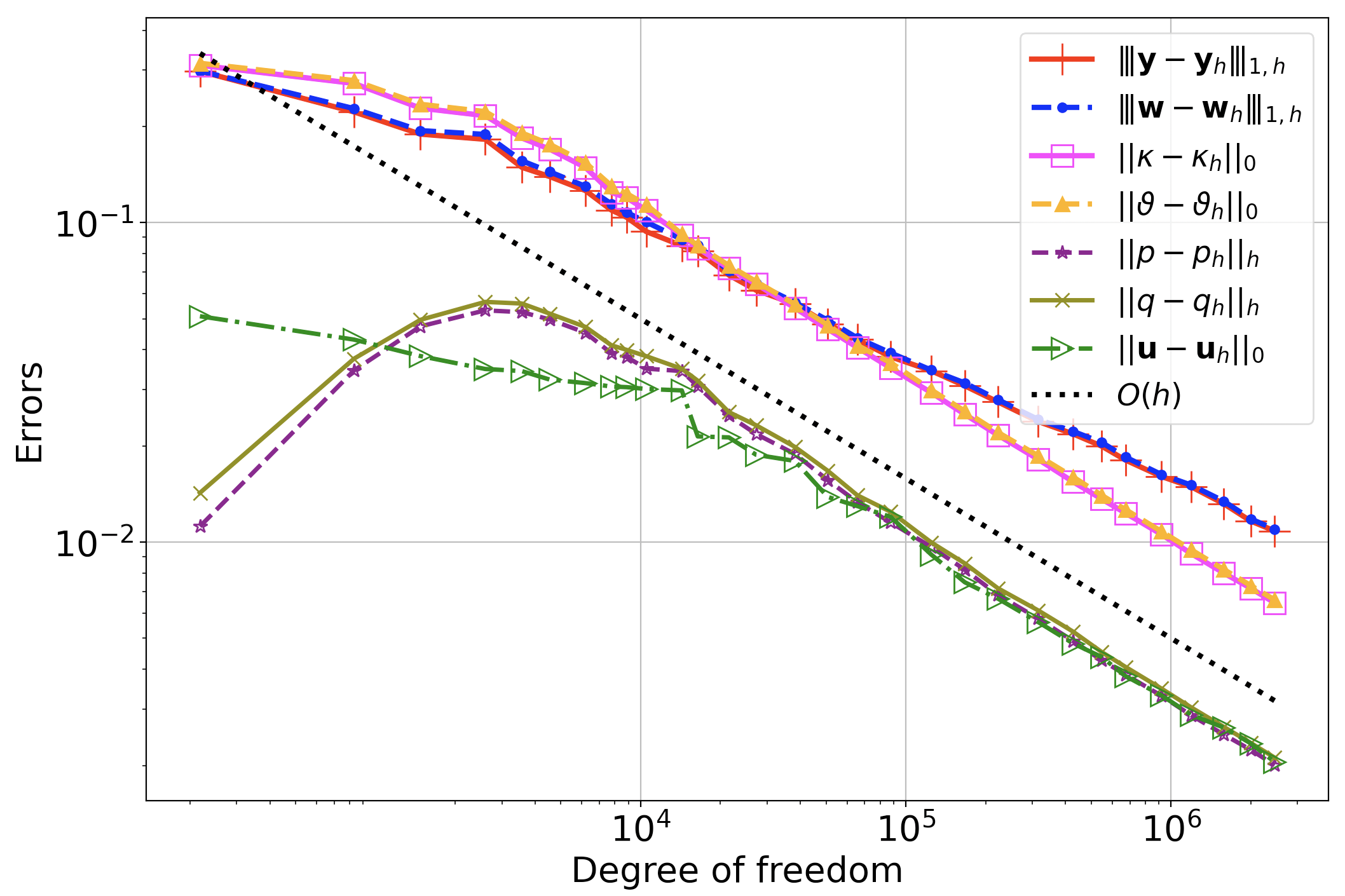}}
	\vspace{-2.5mm}
	\subfloat[Indicator-Total error (CG)]{\includegraphics[scale=0.158]{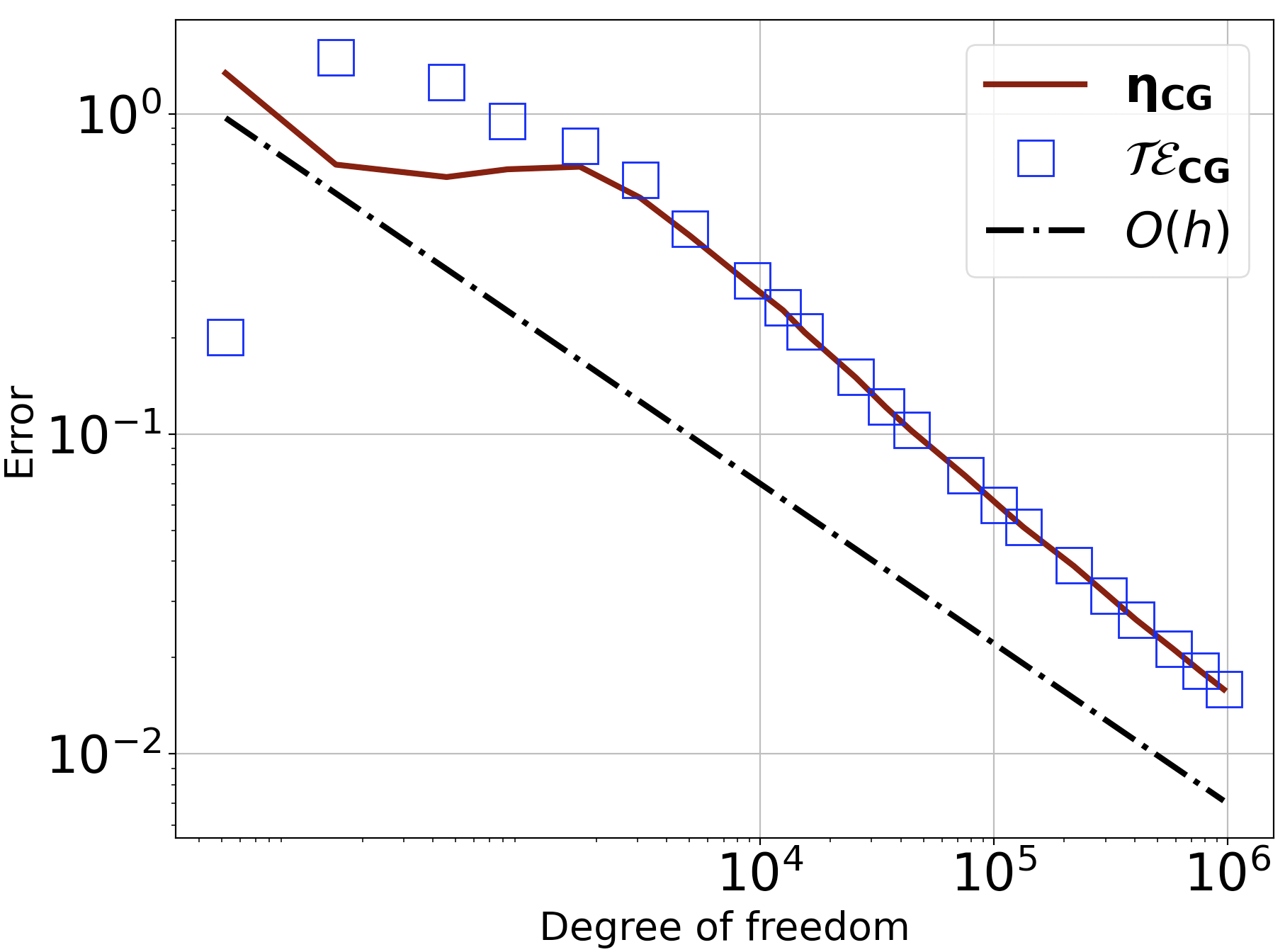}}
	\subfloat[Indicator-Total error (DG)]{\includegraphics[scale=0.158]{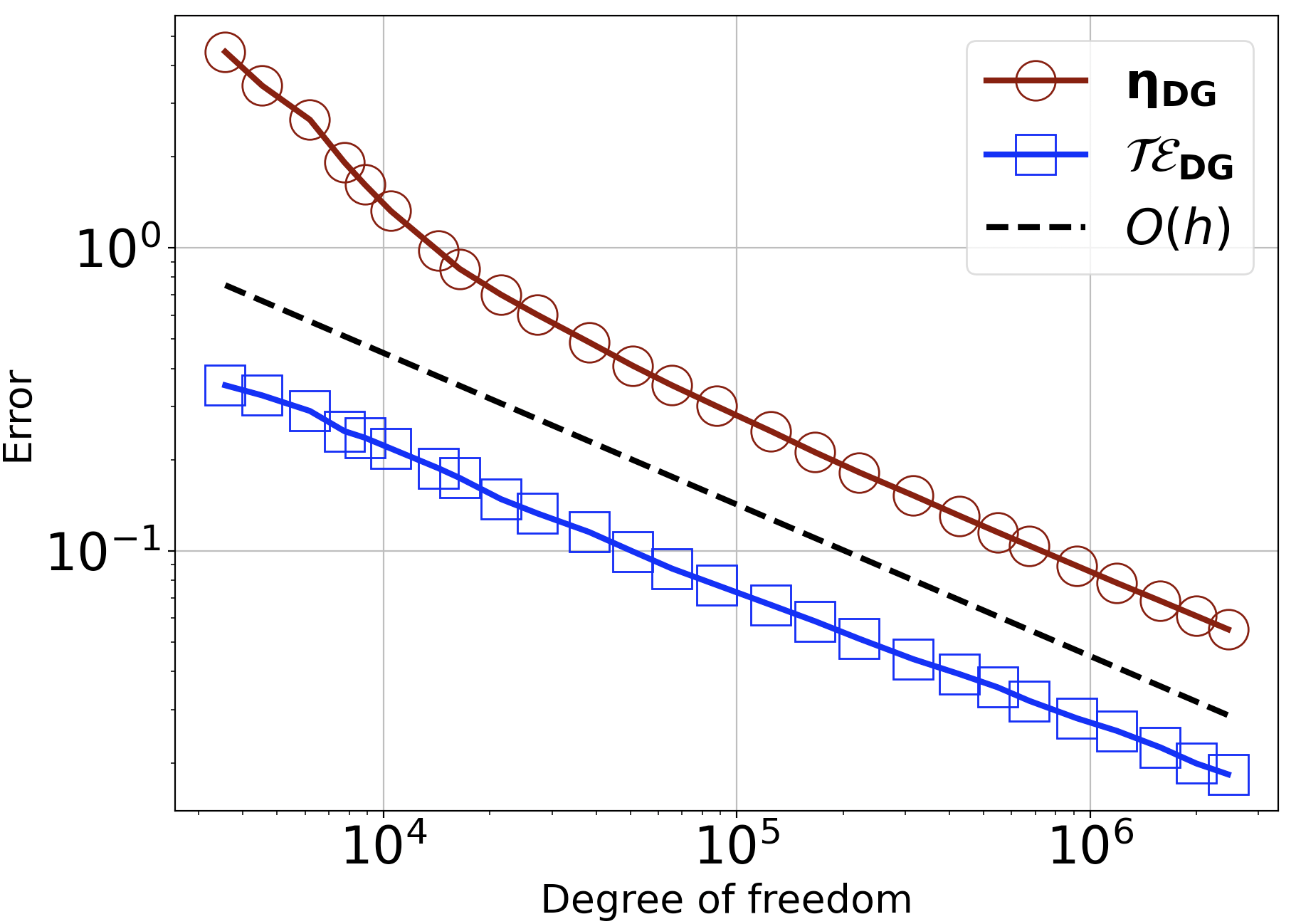}}
	\subfloat[Efficiency (CG-DG)]{\includegraphics[scale=0.158]{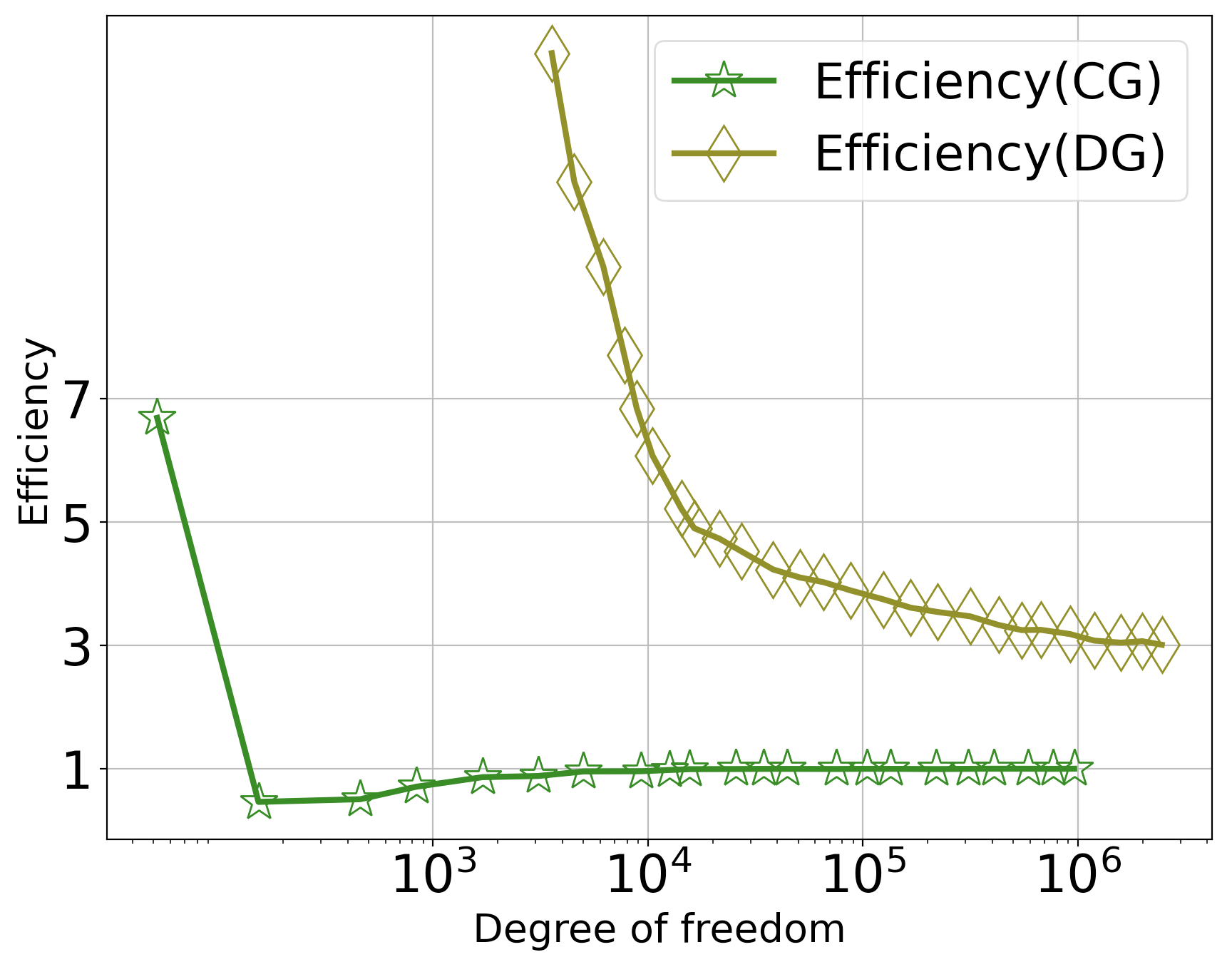}} \vspace{-2.5mm}
	\caption{Convergence plots for the (a) CG scheme (b) DG scheme (c) Indicator-Total error (CG) (d) Indicator-Total error (DG) and (e) Efficiency under uniform refinement for Example~\ref{Example 5.2.}.}
	\label{FIGURE 2.1}
\end{figure}
\begin{figure}
	\centering
	\subfloat[$\y_{h1}$]{\includegraphics[scale=0.112]{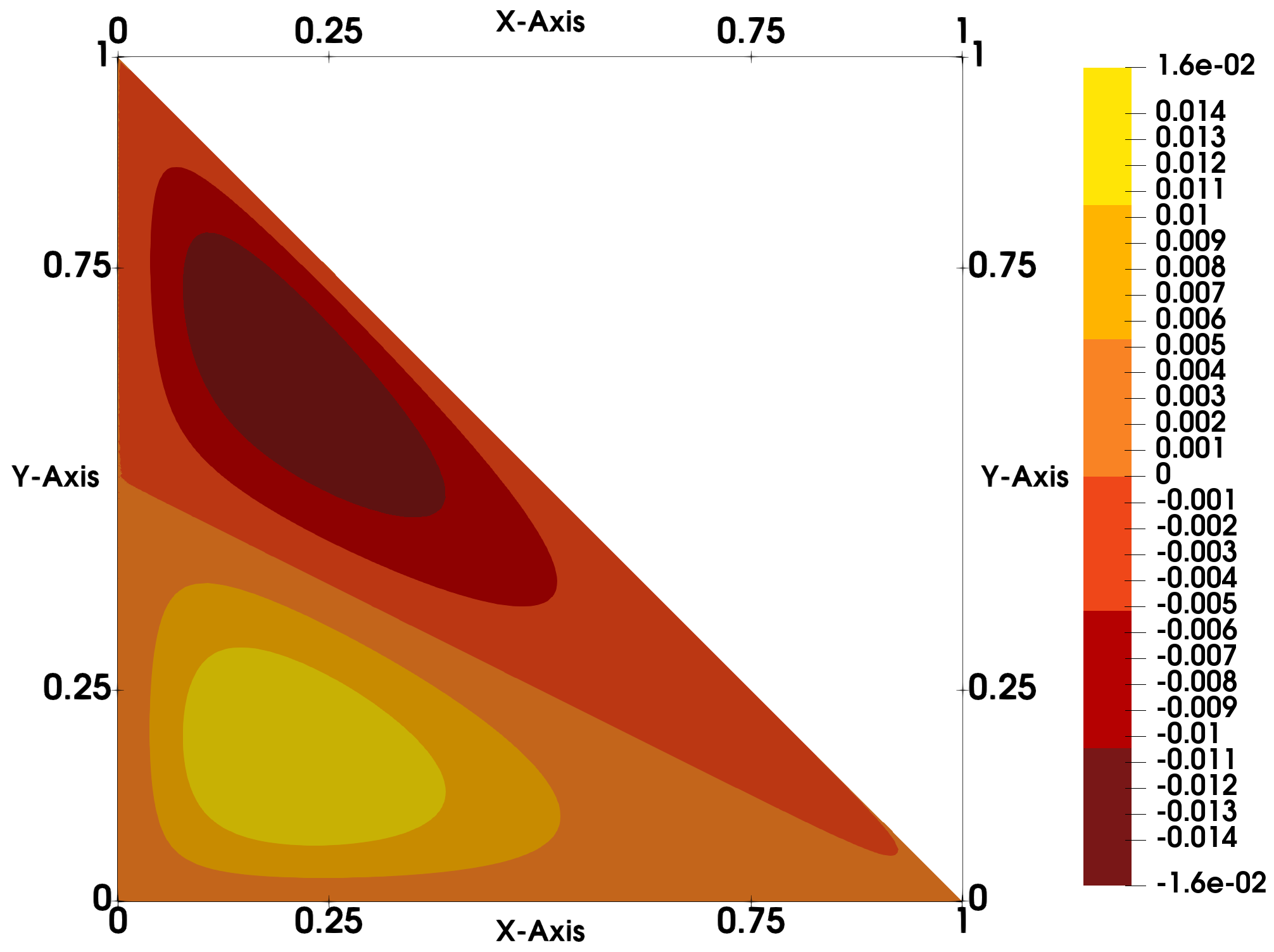}}
	\subfloat[$\y_{h2}$]{\includegraphics[scale=0.112]{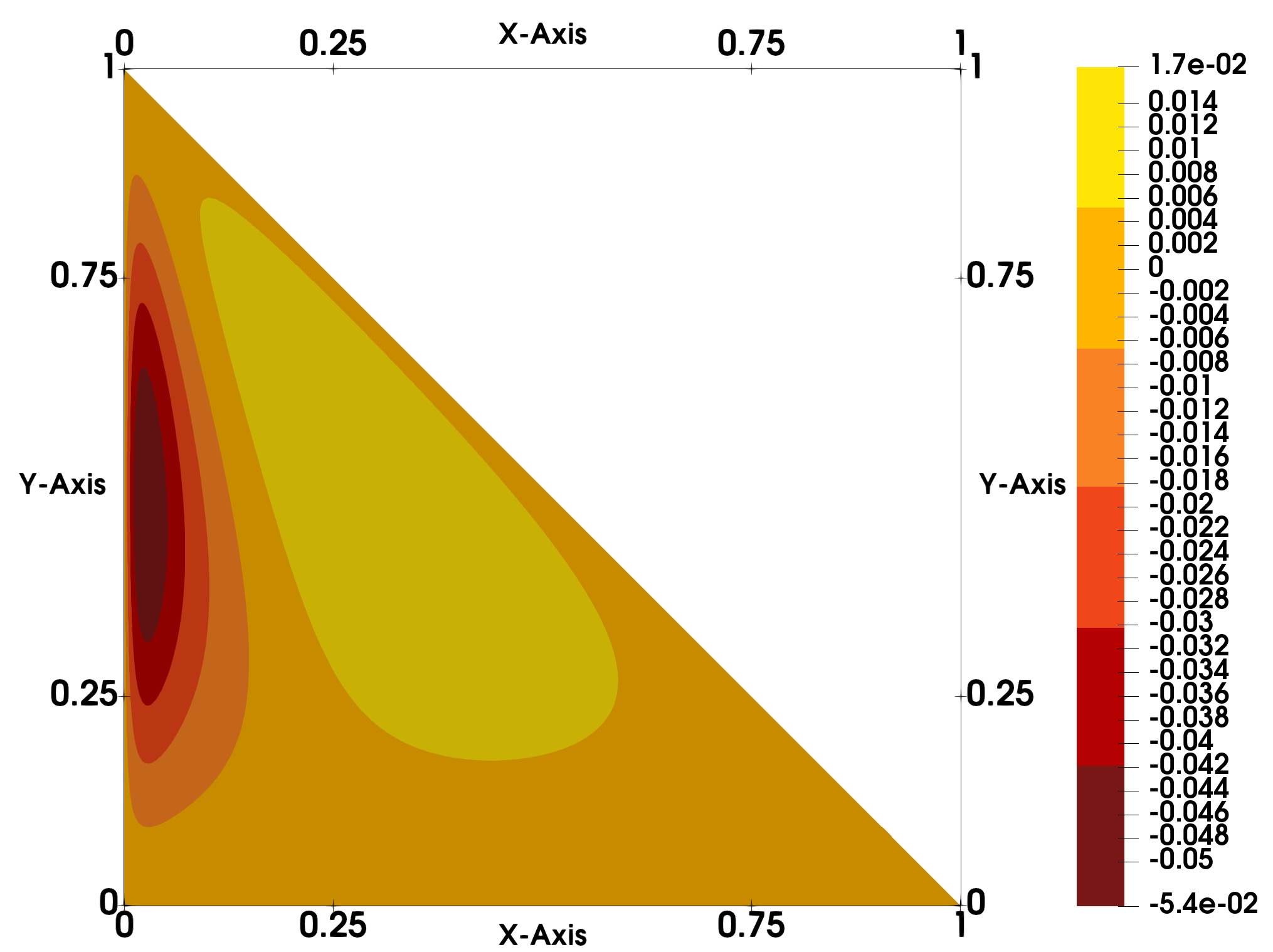}} 
	\subfloat[$\omega_h$]{\includegraphics[scale=0.112]{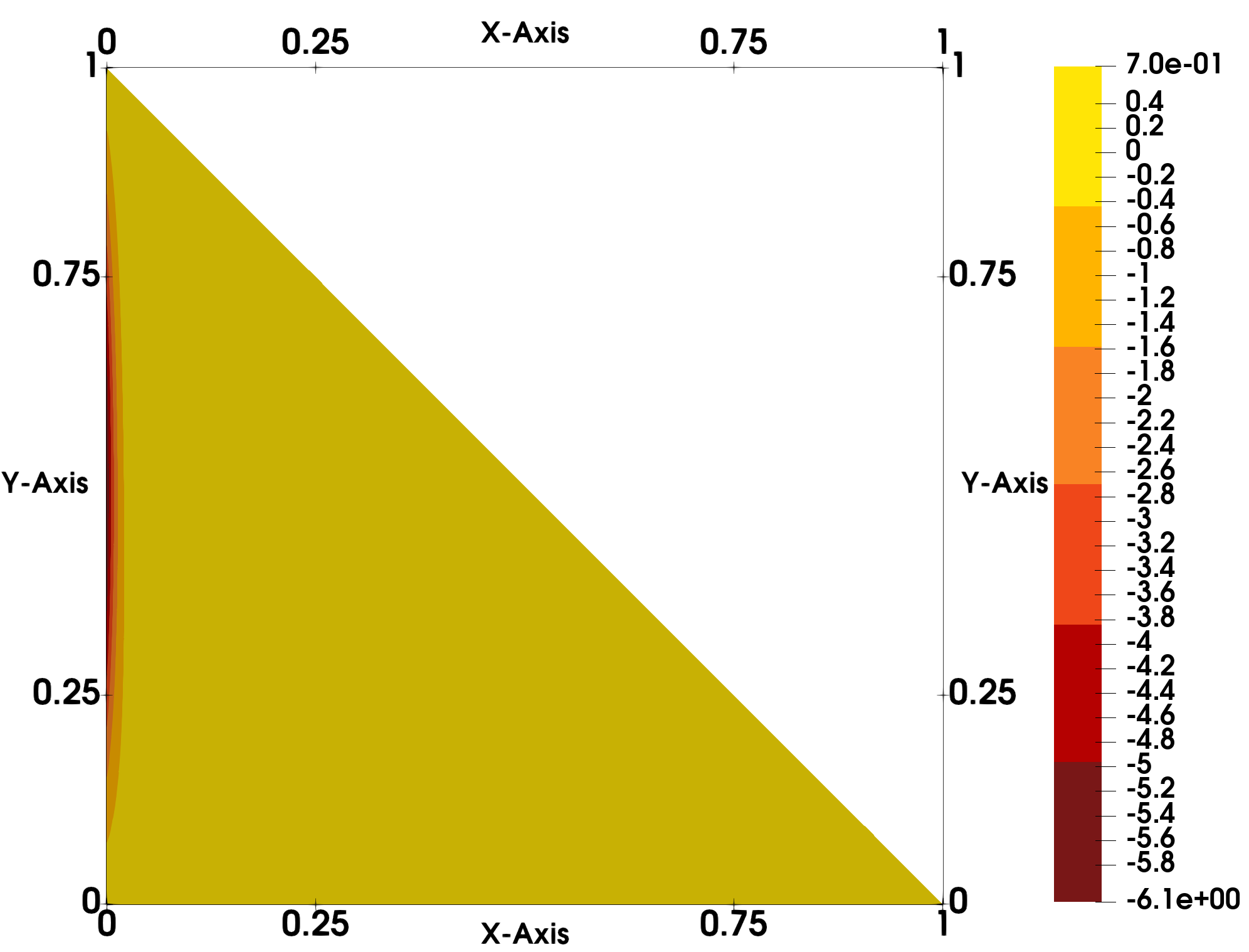}}
	\subfloat[$p_h$]{\includegraphics[scale=0.112]{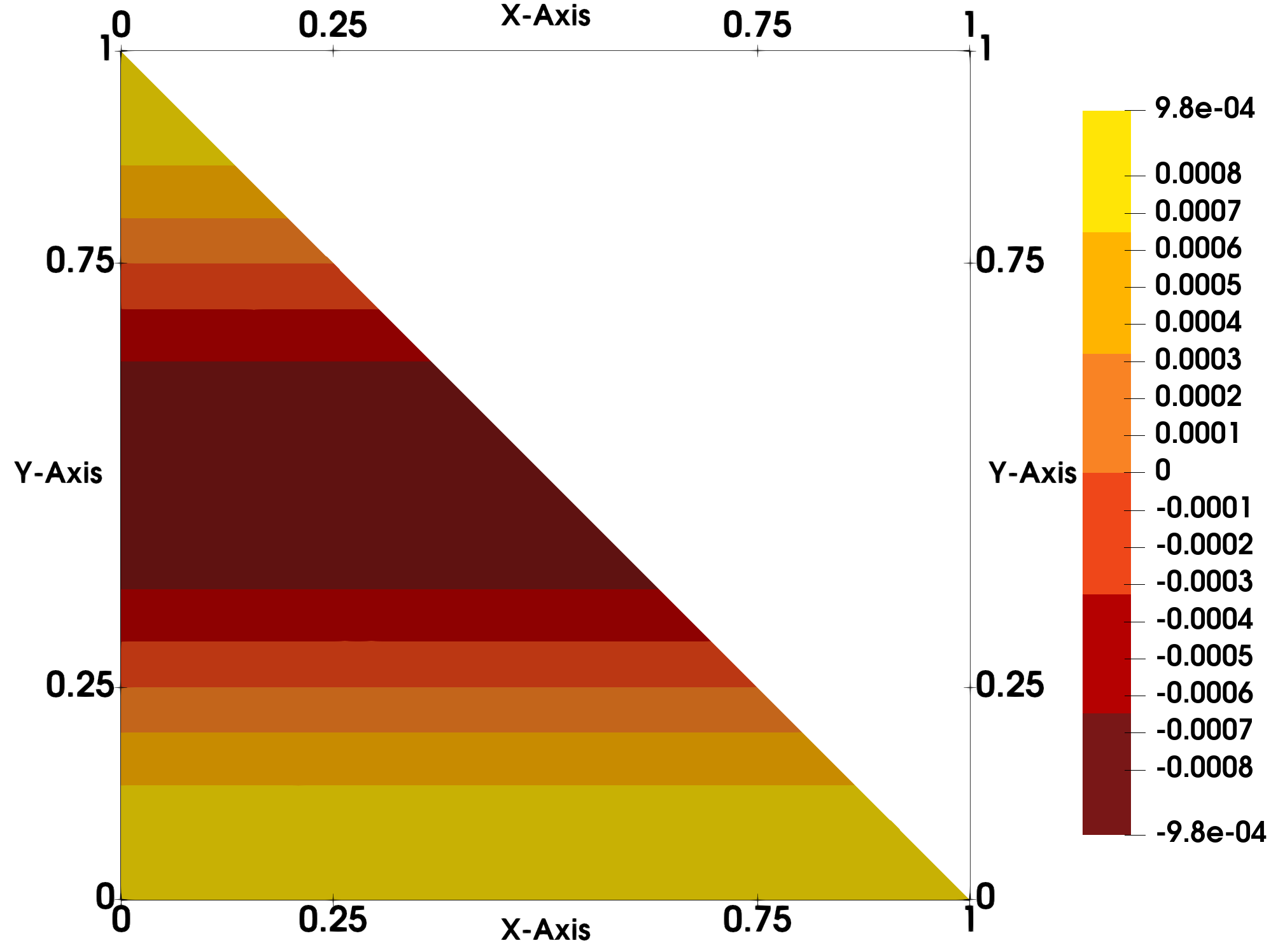}}\\
	\subfloat[$\w_{h1}$]{\includegraphics[scale=0.112]{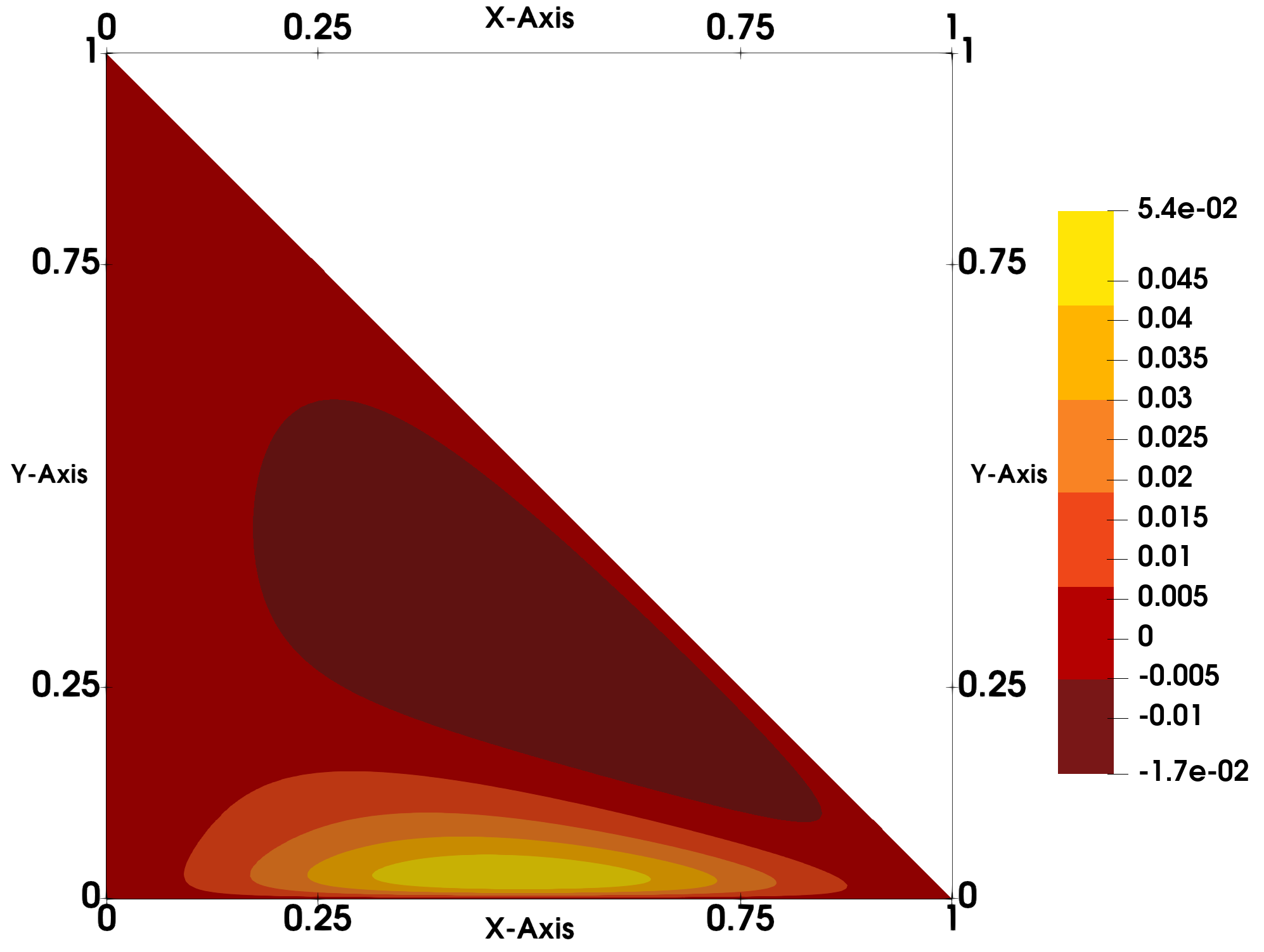}}
	\subfloat[$\w_{h2}$]{\includegraphics[scale=0.112]{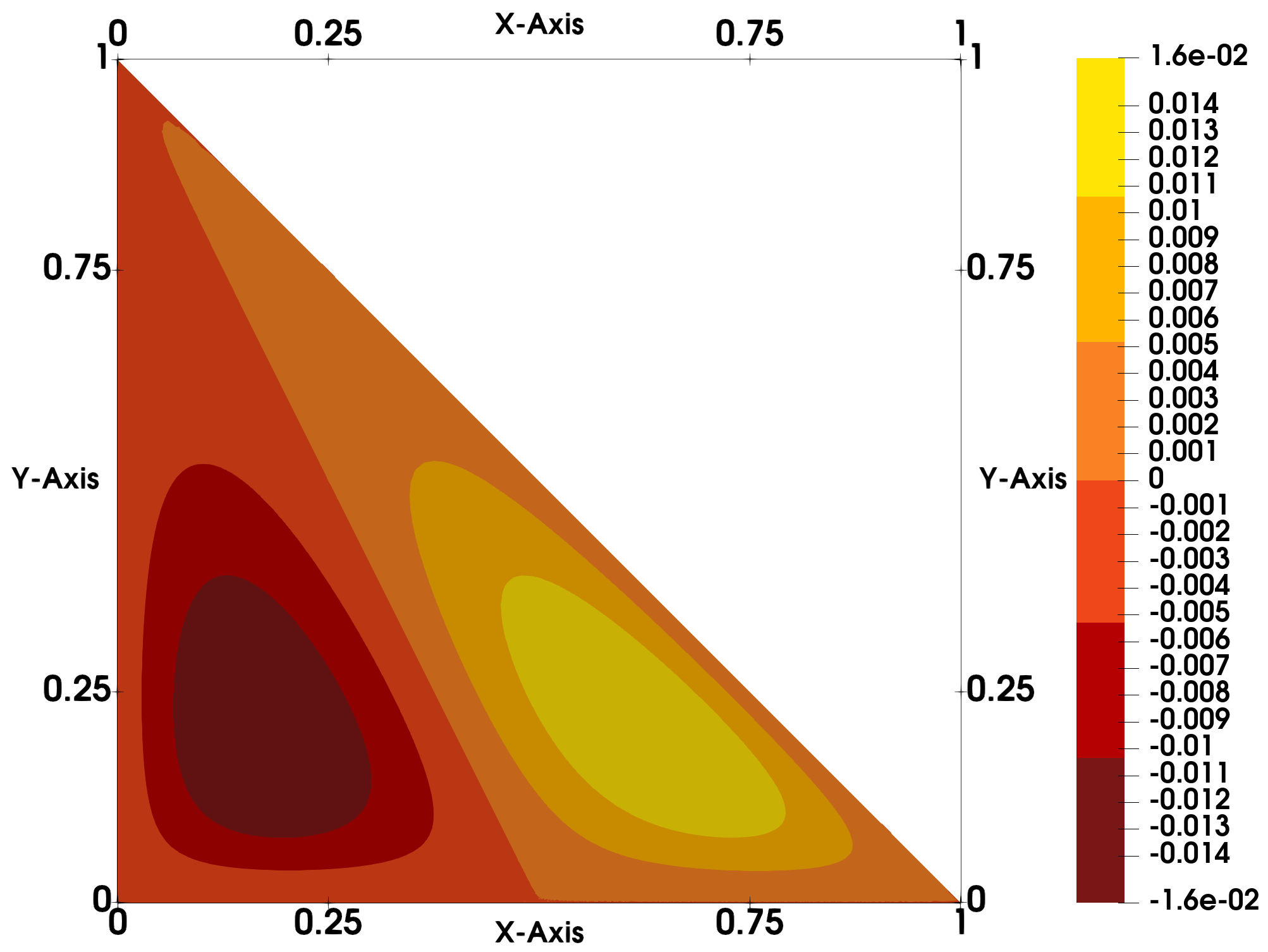}} 
	\subfloat[$\vartheta_h$]{\includegraphics[scale=0.112]{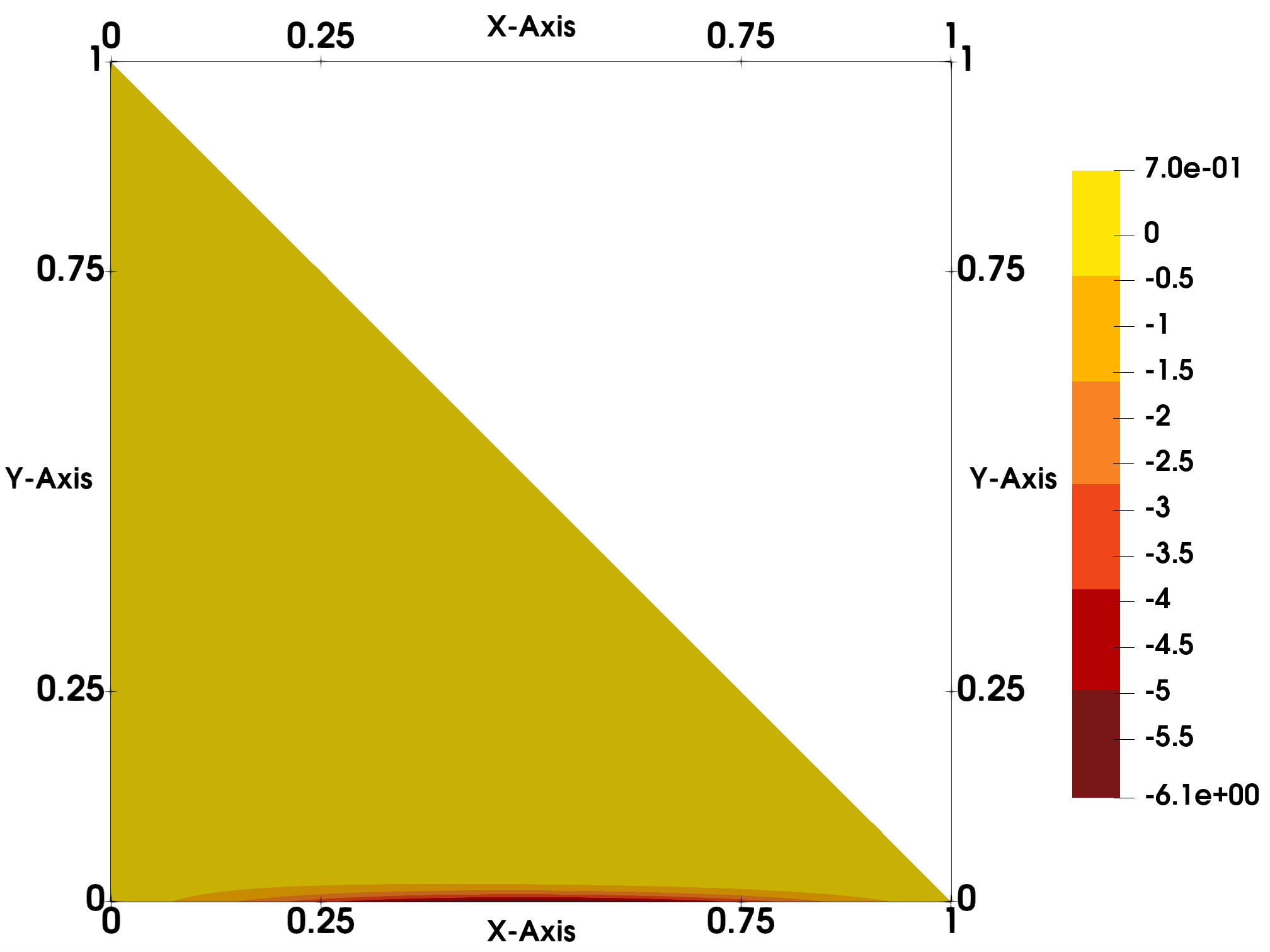}}
	\subfloat[$q_h$]{\includegraphics[scale=0.112]{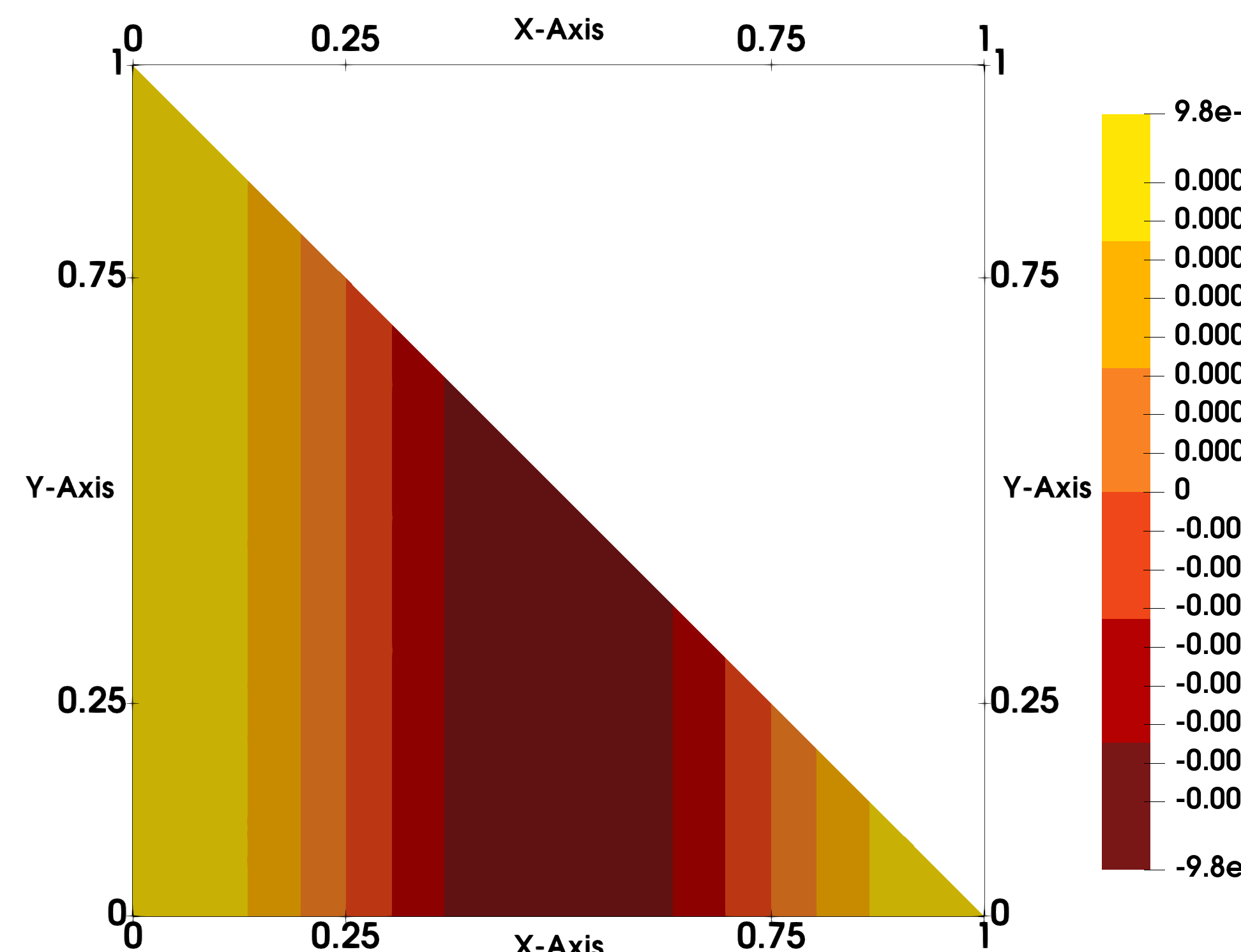}} \\
	\subfloat[$\u_{h1}$]{\includegraphics[scale=0.112]{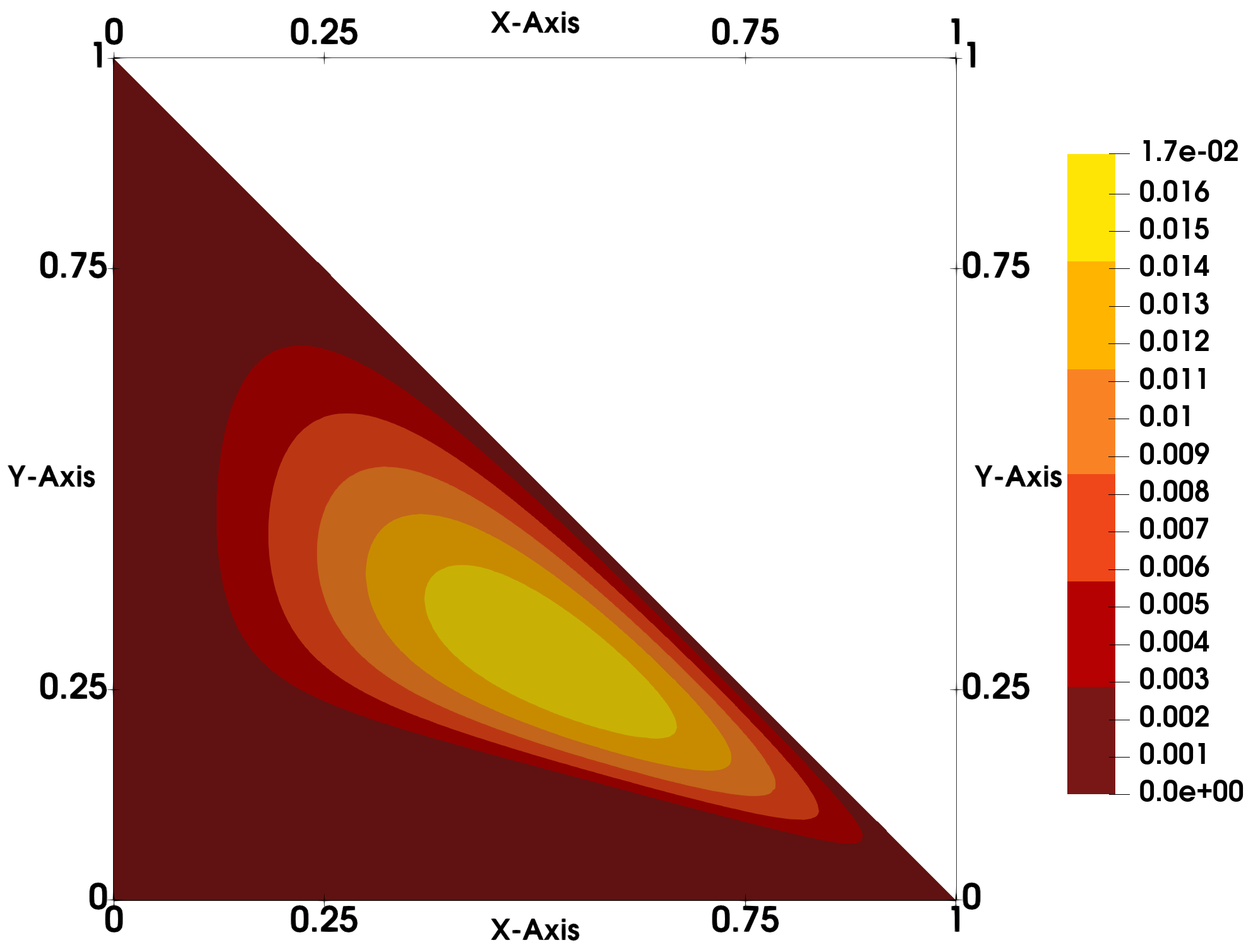}}
	\subfloat[$\u_{h2}$]{\includegraphics[scale=0.112]{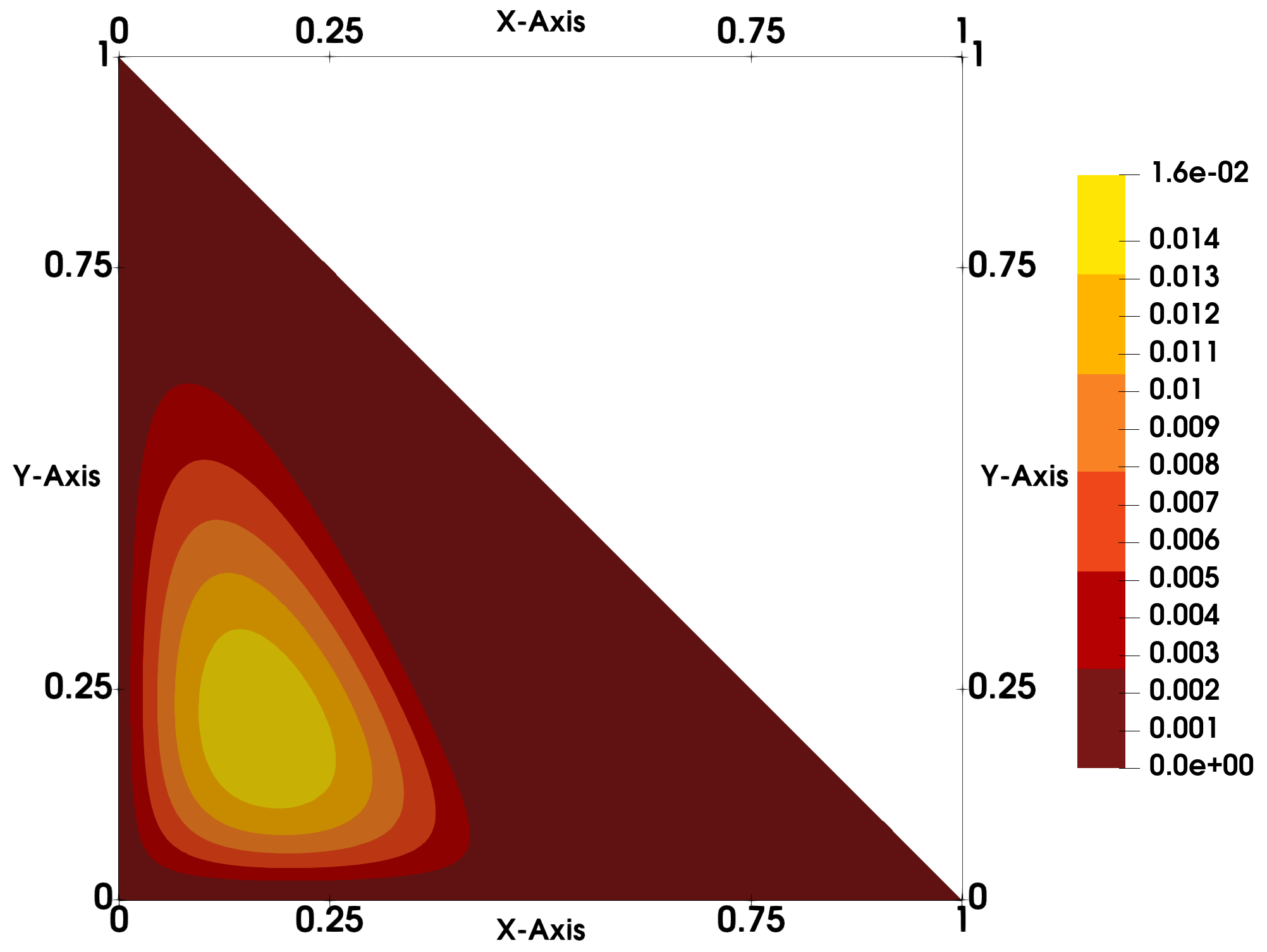}}
	\caption{Plots of numerical solutions of state velocity $(\y_{h1},\y_{h2})$, vorticity $(\kappa_h)$, pressure $(p_h)$, co-state velocity $(\w_{h1},\w_{h2})$, vorticity $(\vartheta_h)$, pressure $(q_h)$, and control $(\u_{h1},\u_{h2})$, respectively, for Example- \ref{Example 5.2.}.} \label{FIGURE 5}
\end{figure}
	\subsection{Non-convex L-shape and T-shape domains}\label{Example 5.3.}
	Consider the non-convex L-shaped and T-shaped domains $\Omega = (-1,1)^{2} \setminus (0,1)^{2}$, and $\Omega = ((-1.5,1.5) \times (0,1)) \cup ((-0.5,0.5) \times (-2,0])$, respectively. We chose coefficients $\nu = 1 + x_1^{2},\ \boldsymbol{\beta} = (1,1)^{T},\ \sigma = 0$, and control bounds $ \mathbf{a} =(0,0)^{T}$ and $\mathbf{b} =(1,1)^{T}$. We select the source function $\mathbf{f}$, desired velocity $\y_d$, and vorticity $\omega_d$ as:
	\begin{align*}
	\f(x_1,x_2) = (1,1)^{T}, 
	\qquad \y_{d} = (x_2, -x_1)^{T},
	\qquad \omega_{d} = \textbf{curl}(\y_d) = -2.
	\end{align*}
	The exact solutions for these problems remain unknown. However, we anticipate significant challenges in convergence when using uniform mesh refinement, primarily due to the presence of reentrant corners, which typically lead to singularities in the solution. In contrast, our adaptive refinement strategy proves to be much more effective in dealing with these issues. As demonstrated in Figure~\ref{FIGURE 6}, the adaptive method focuses the refinement in the regions surrounding the re-entrant corners. This targeted approach helps to accurately capture the singularities and complex behaviors in these areas, which uniform refinement often fails to do efficiently. As we continue refining the mesh adaptively, we observe that the global error estimators decrease optimally. This behavior depicted in Figure~\ref{FIGURE 7}, validates the efficacy of the adaptive scheme. Additionally, Figures~\ref{FIGURE 8} and \ref{FIGURE 9} provide detailed visualizations of the numerical solutions. These plots illustrate the improved accuracy and resolution achieved through our adaptive refinement strategy. The finer mesh around the reentrant corners allows for a more precise approximation of the solution, which is critical for capturing the true nature of the problem.
	 \begin{figure}
	\centering
	\subfloat[Initial L-shaped mesh]{\includegraphics[scale=0.34]{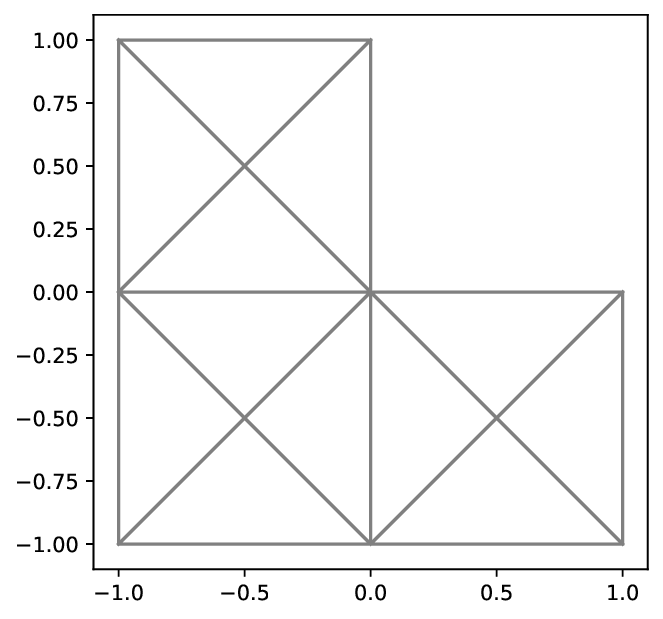}}
	\subfloat[6486 DOF]{\includegraphics[scale=0.34]{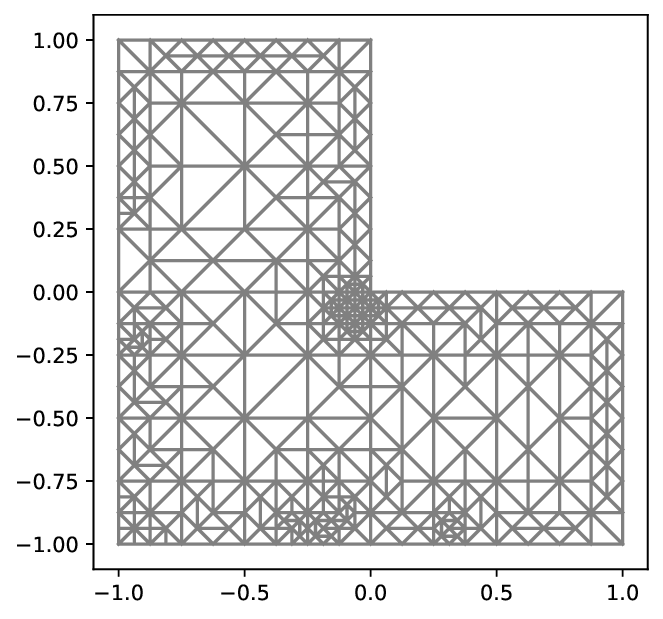}} 
	\subfloat[21770 DOF]{\includegraphics[scale=0.34]{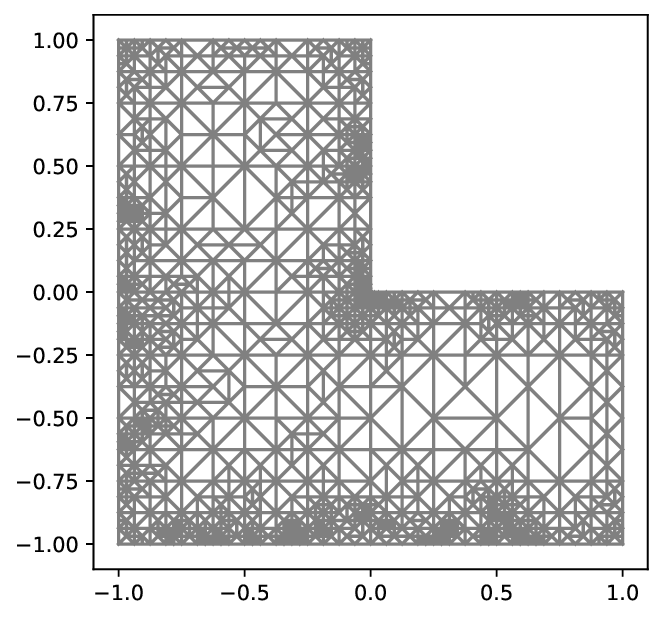}}\\
	\subfloat[Initial T-shaped mesh]{\includegraphics[scale=0.34]{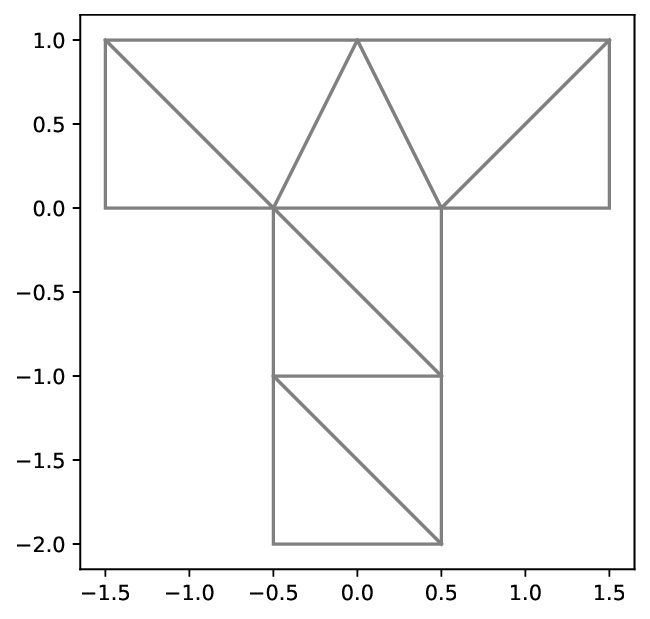}}
	\subfloat[8146 DOF]{\includegraphics[scale=0.34]{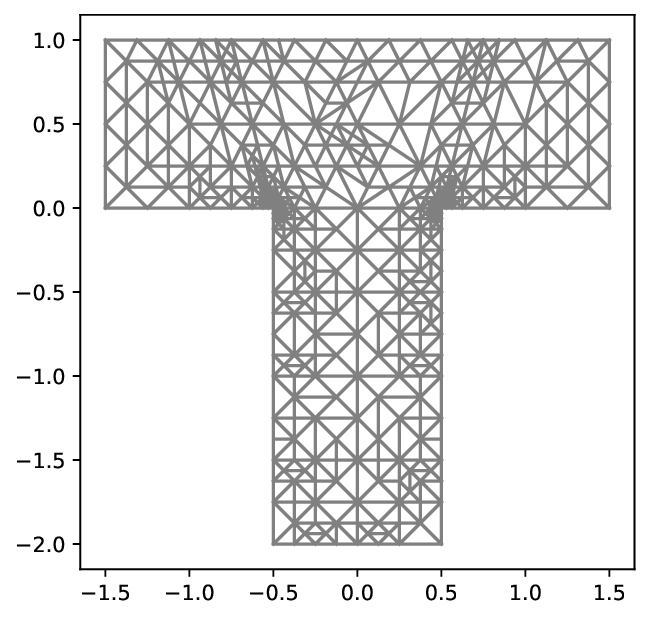}} 
	\subfloat[29786 DOF]{\includegraphics[scale=0.34]{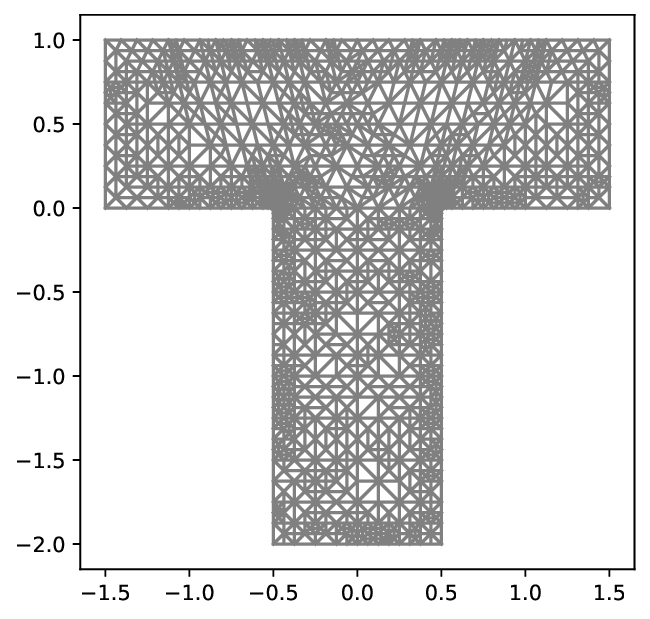}} \vspace{-2.5mm}
	\caption{Adaptively refined meshes showing refinement near re-entrant corners for Example~\ref{Example 5.3.}.} \label{FIGURE 6}
\end{figure}
\begin{figure}
	\centering
	\subfloat[L-shape]{\includegraphics[scale=0.15]{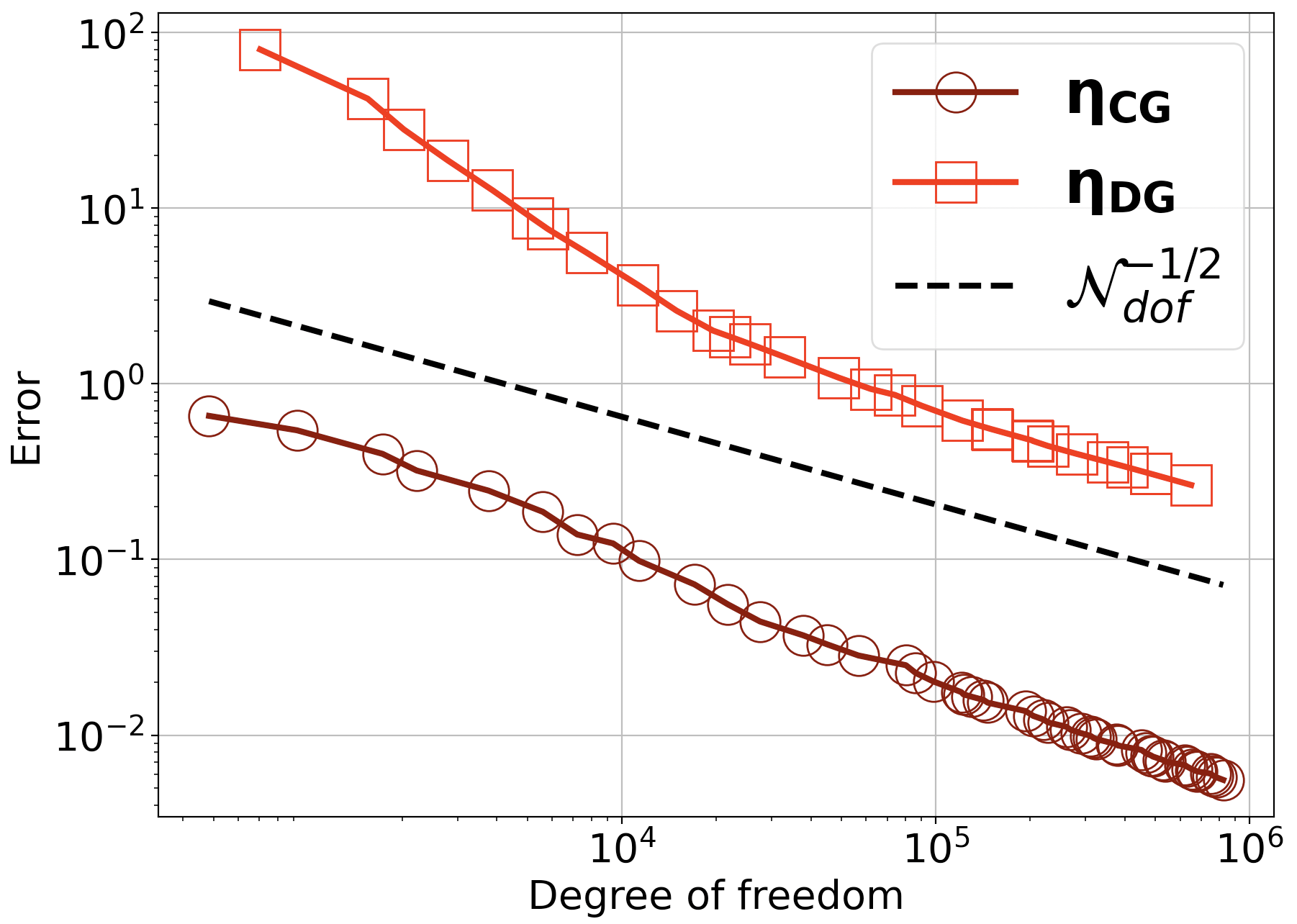}} 
	\subfloat[T-shape]{\includegraphics[scale=0.15]{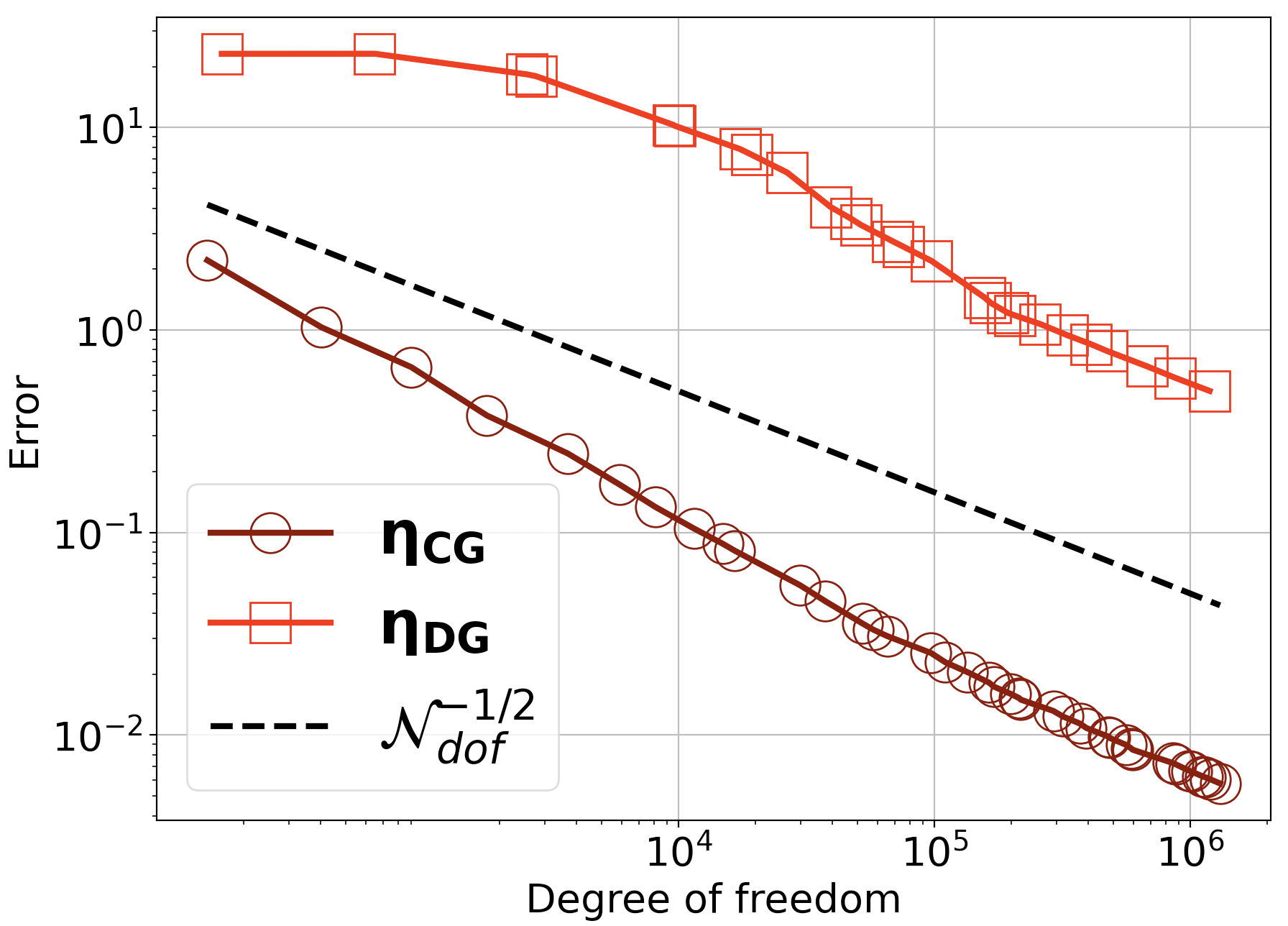}}
	\subfloat[Rectangular pipe]{\includegraphics[scale=0.15]{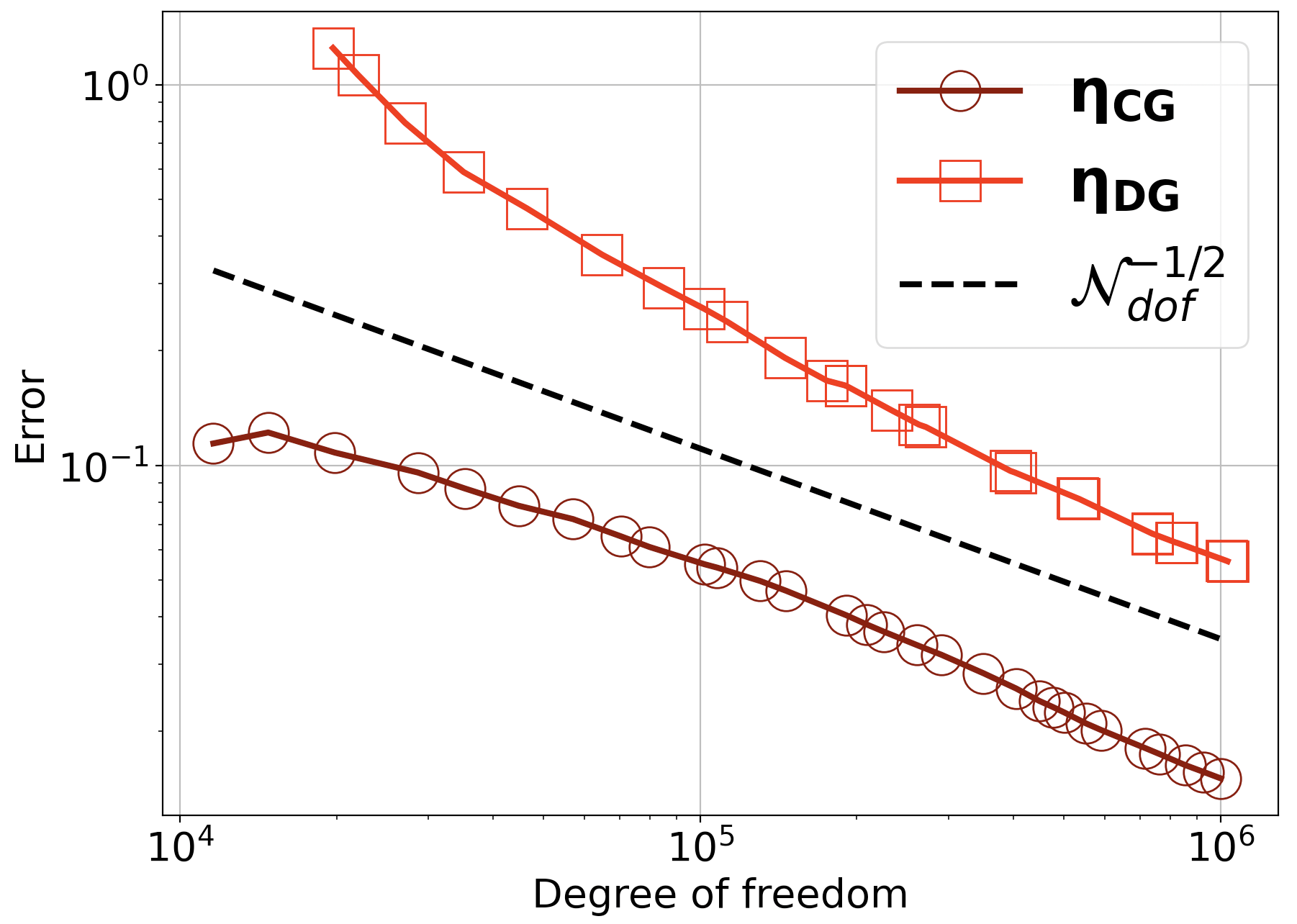}}
	\caption{Convergence plots for the global indicators (a) L-shape (b) T-shape (c) Rectangular pipe.} \label{FIGURE 7}
\end{figure}
\begin{figure}
	\centering
	\subfloat[$\y_{h1}$]{\includegraphics[scale=0.112]{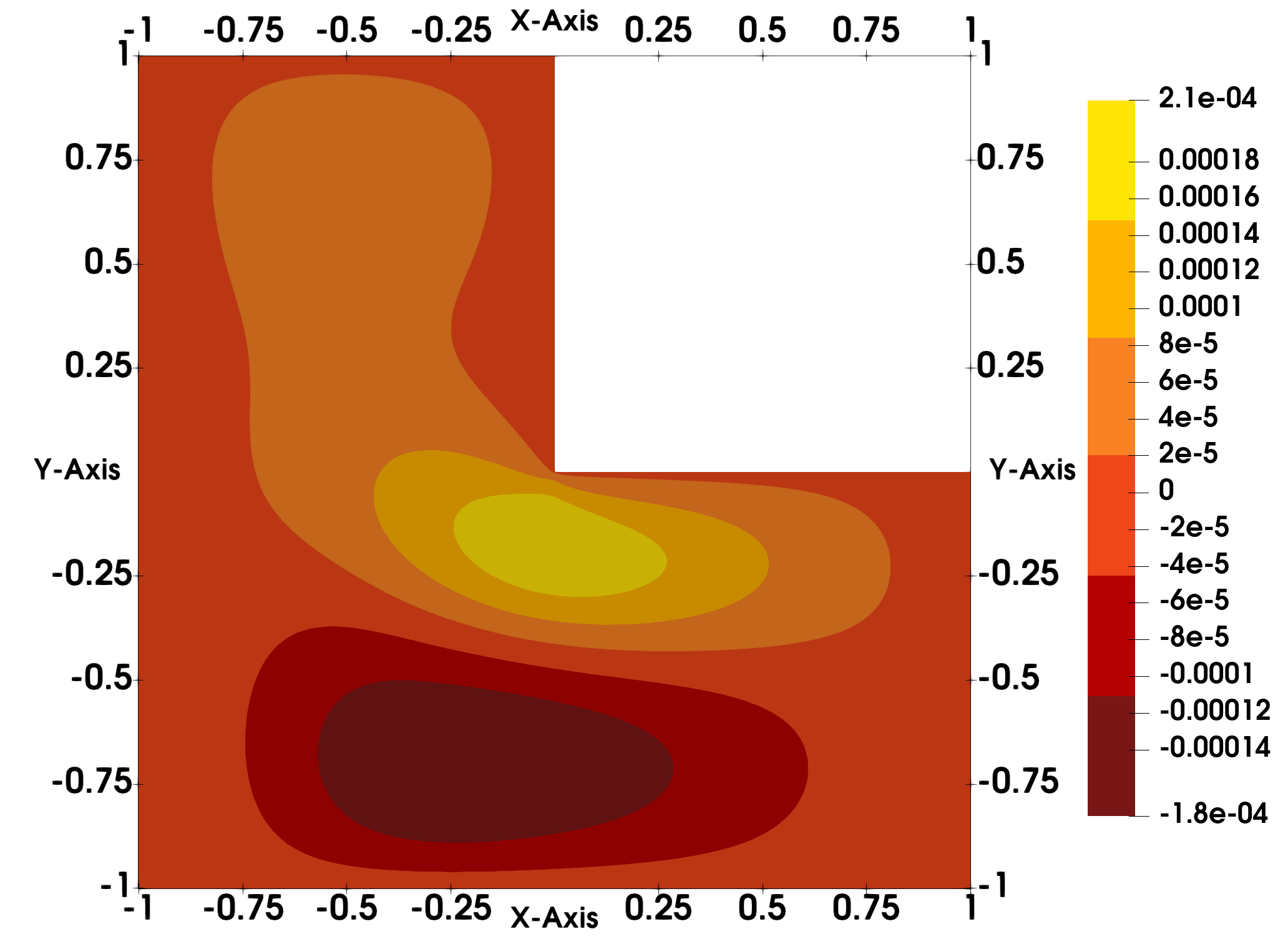}}
	\subfloat[$\y_{h2}$]{\includegraphics[scale=0.114]{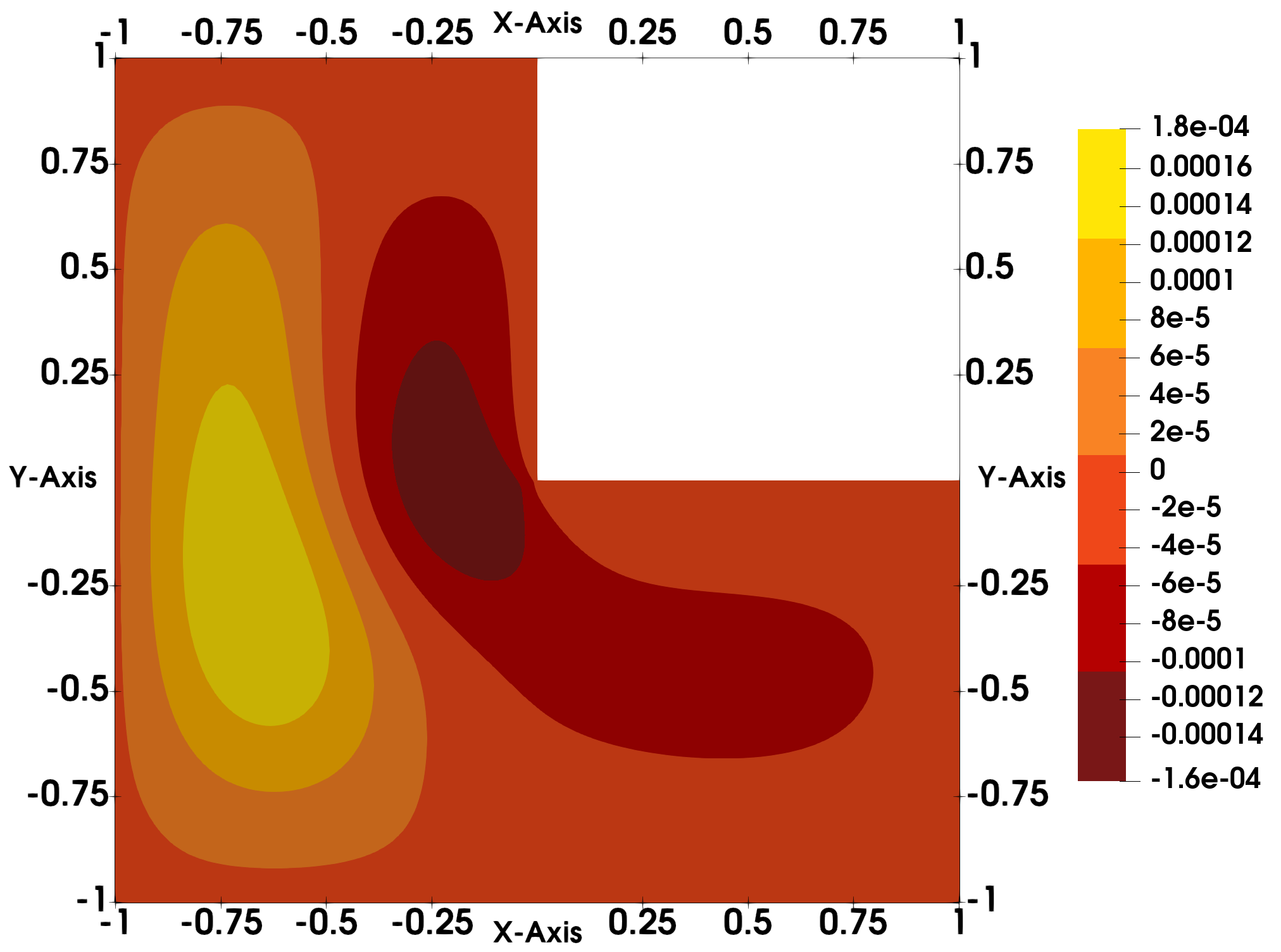}} 
	\subfloat[$\omega_h$]{\includegraphics[scale=0.114]{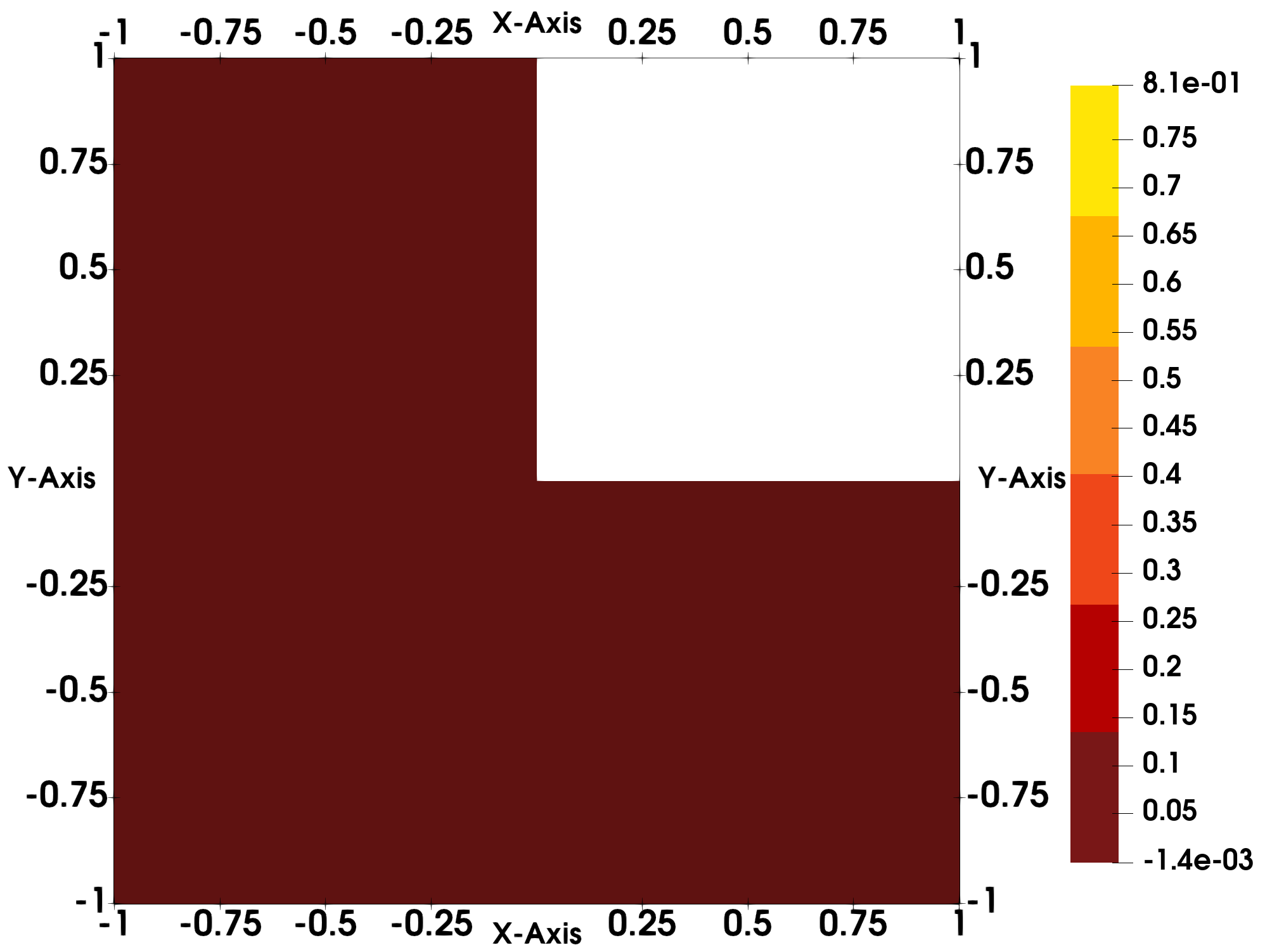}}
	\subfloat[$p_h$]{\includegraphics[scale=0.119]{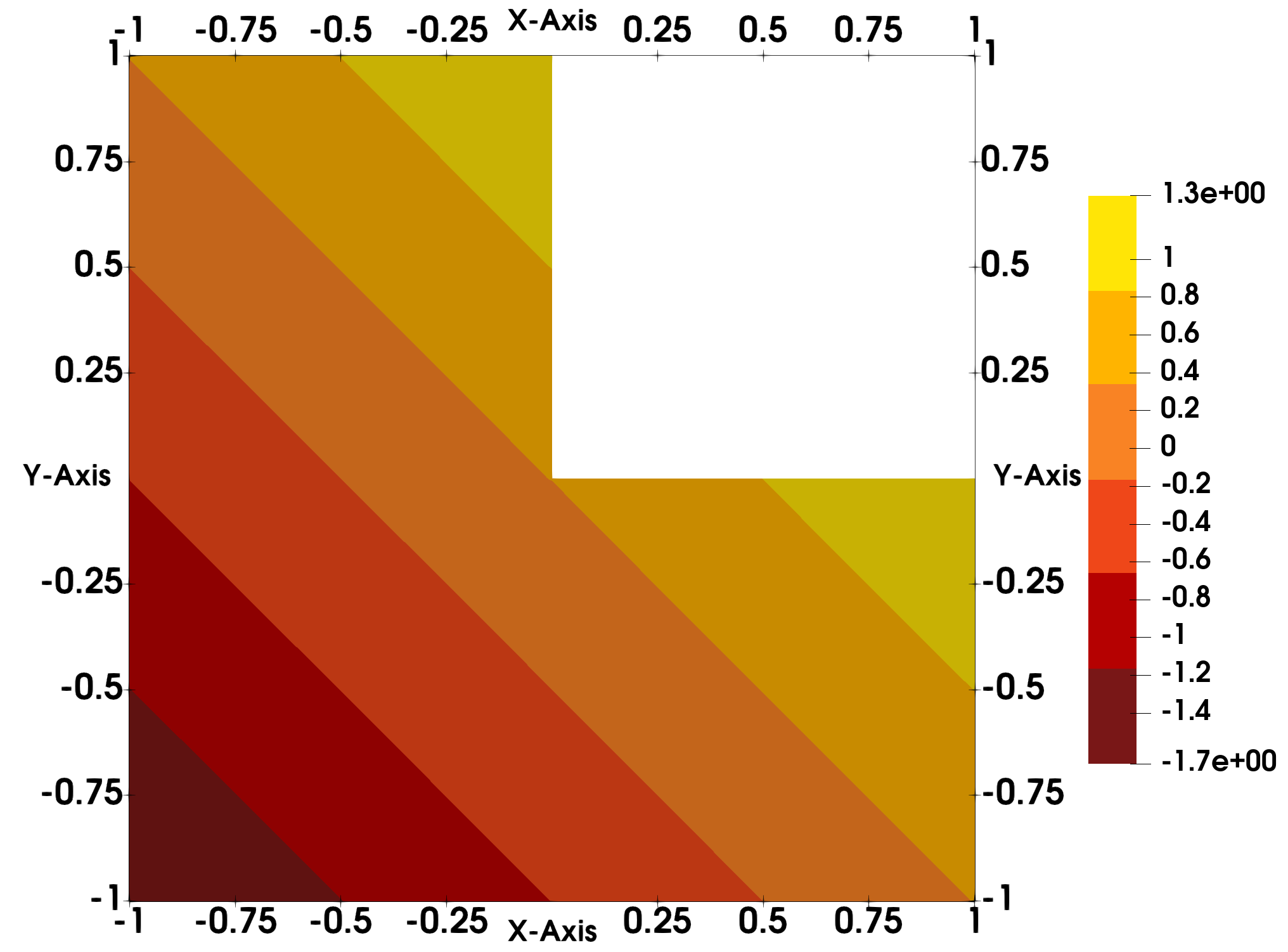}}\\
	\subfloat[$\w_{h1}$]{\includegraphics[scale=0.115]{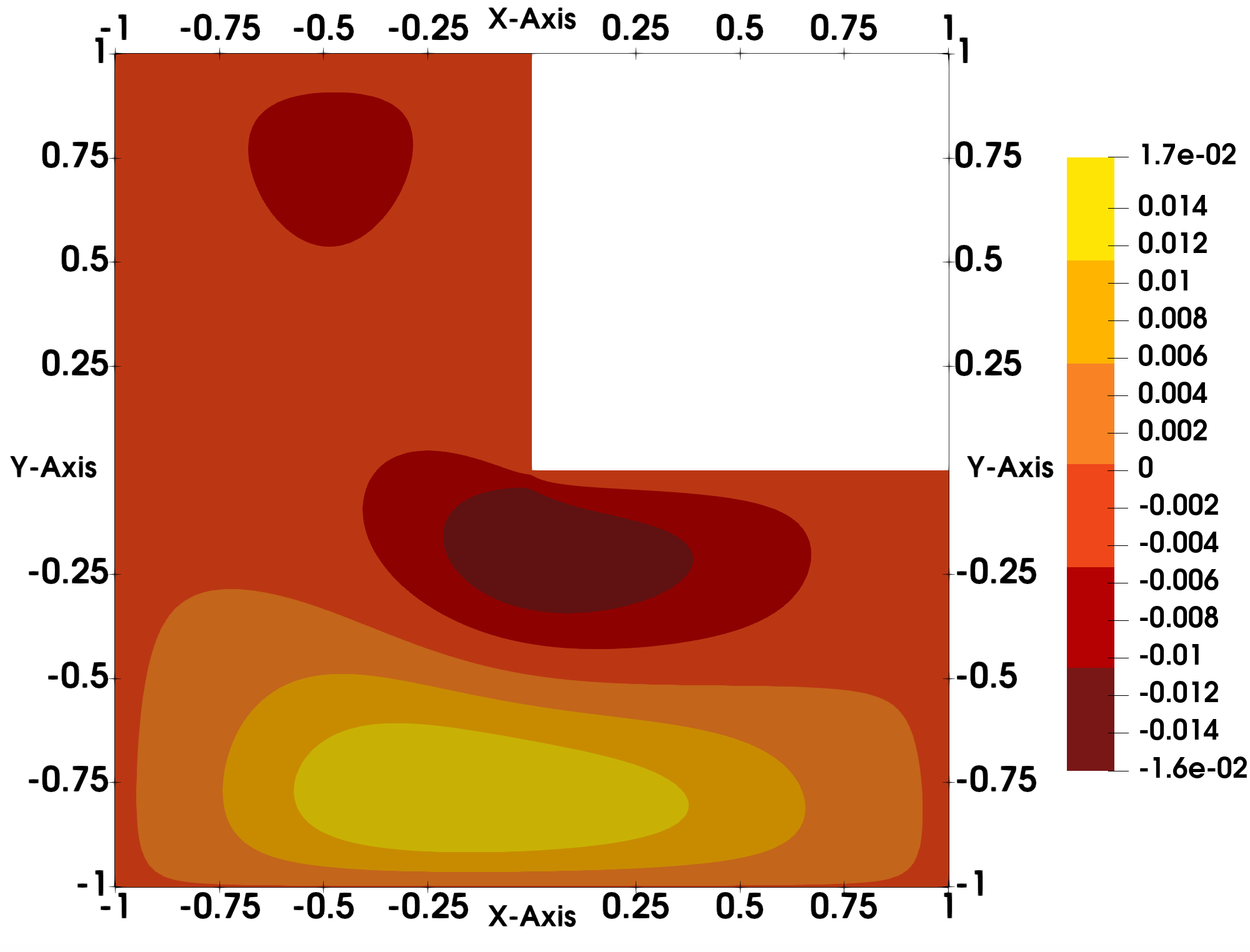}}
	\subfloat[$\w_{h2}$]{\includegraphics[scale=0.115]{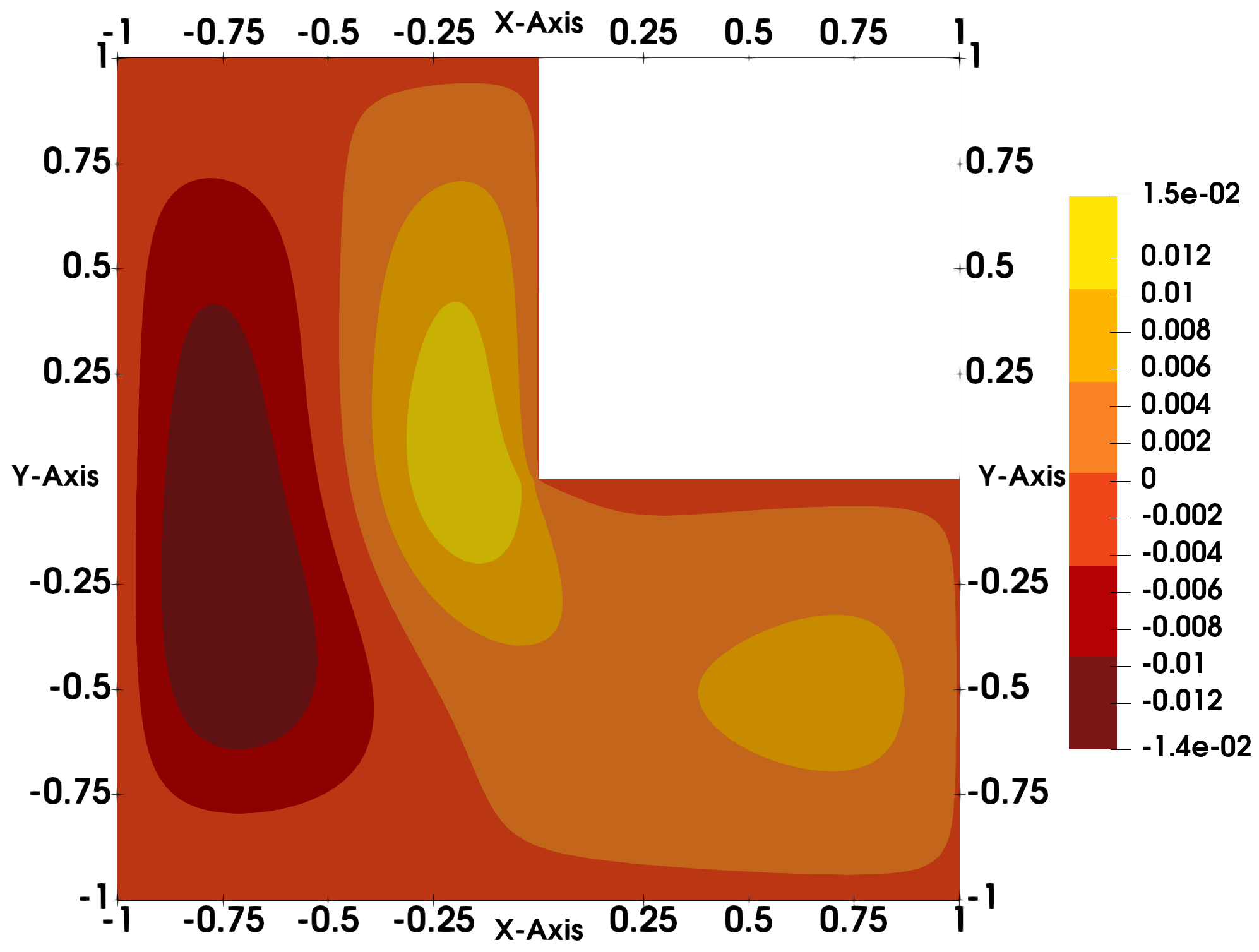}} 
	\subfloat[$\vartheta_h$]{\includegraphics[scale=0.12]{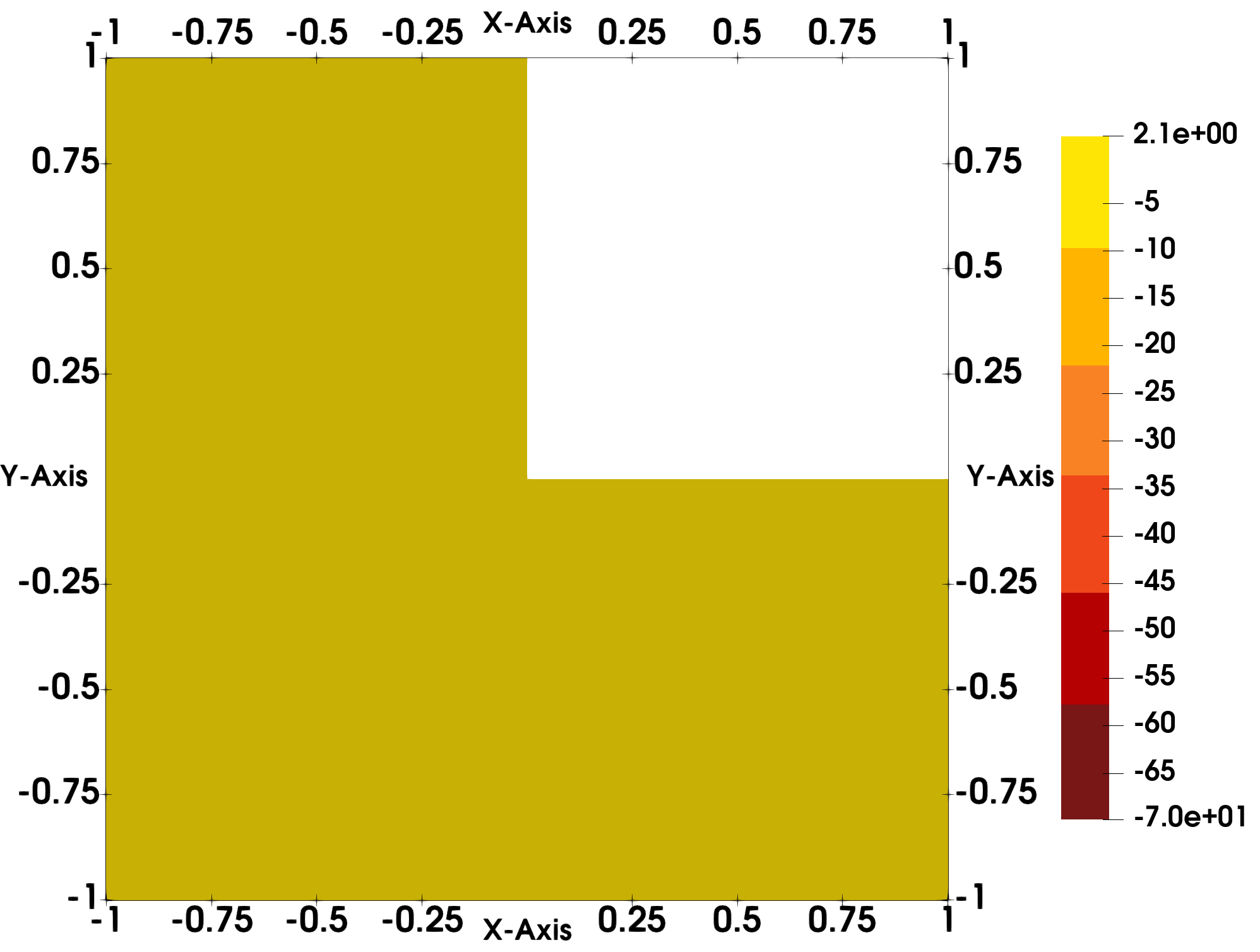}}
	\subfloat[$q_h$]{\includegraphics[scale=0.12]{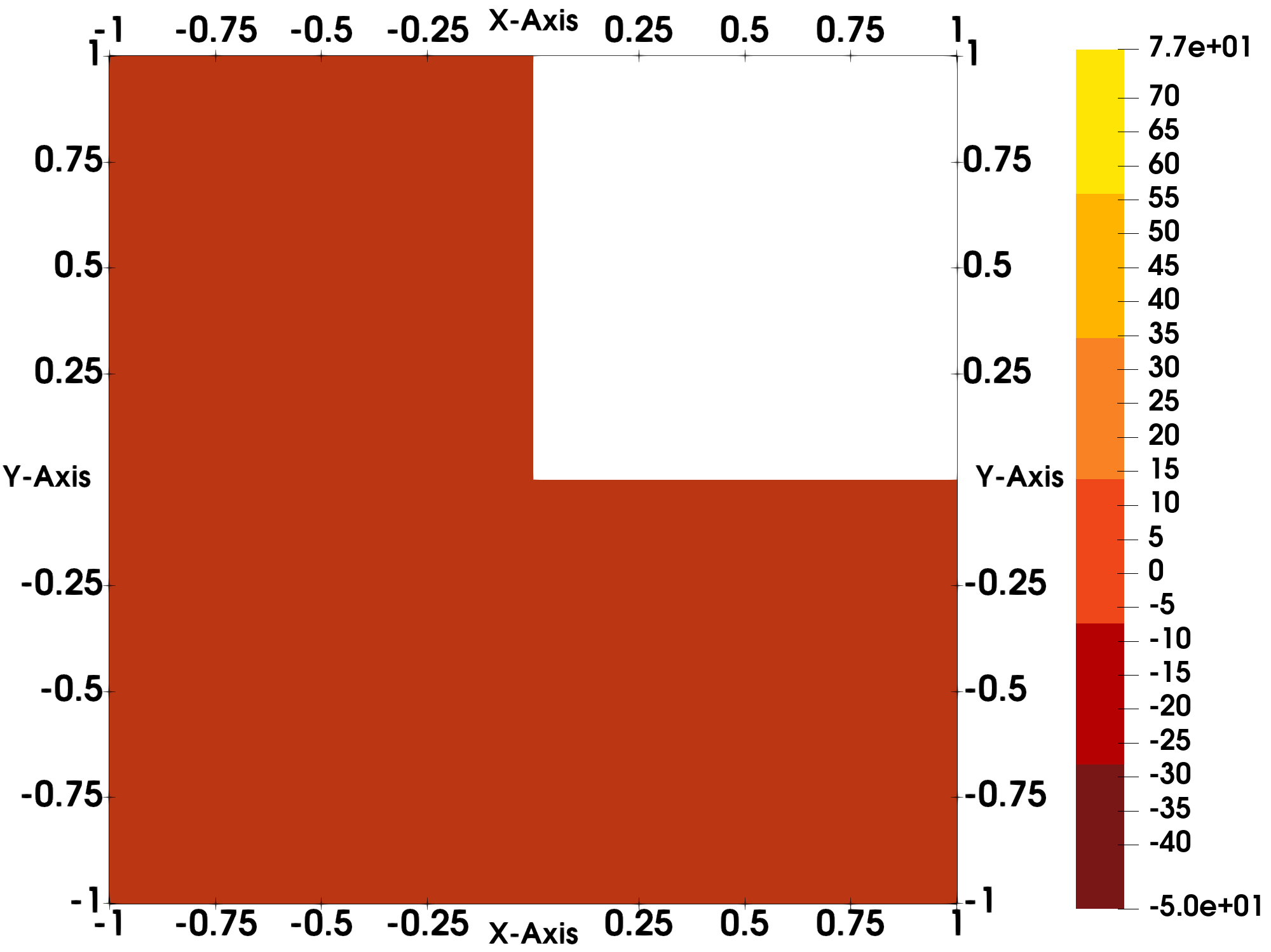}} \\
	\subfloat[$\u_{h1}$]{\includegraphics[scale=0.134]{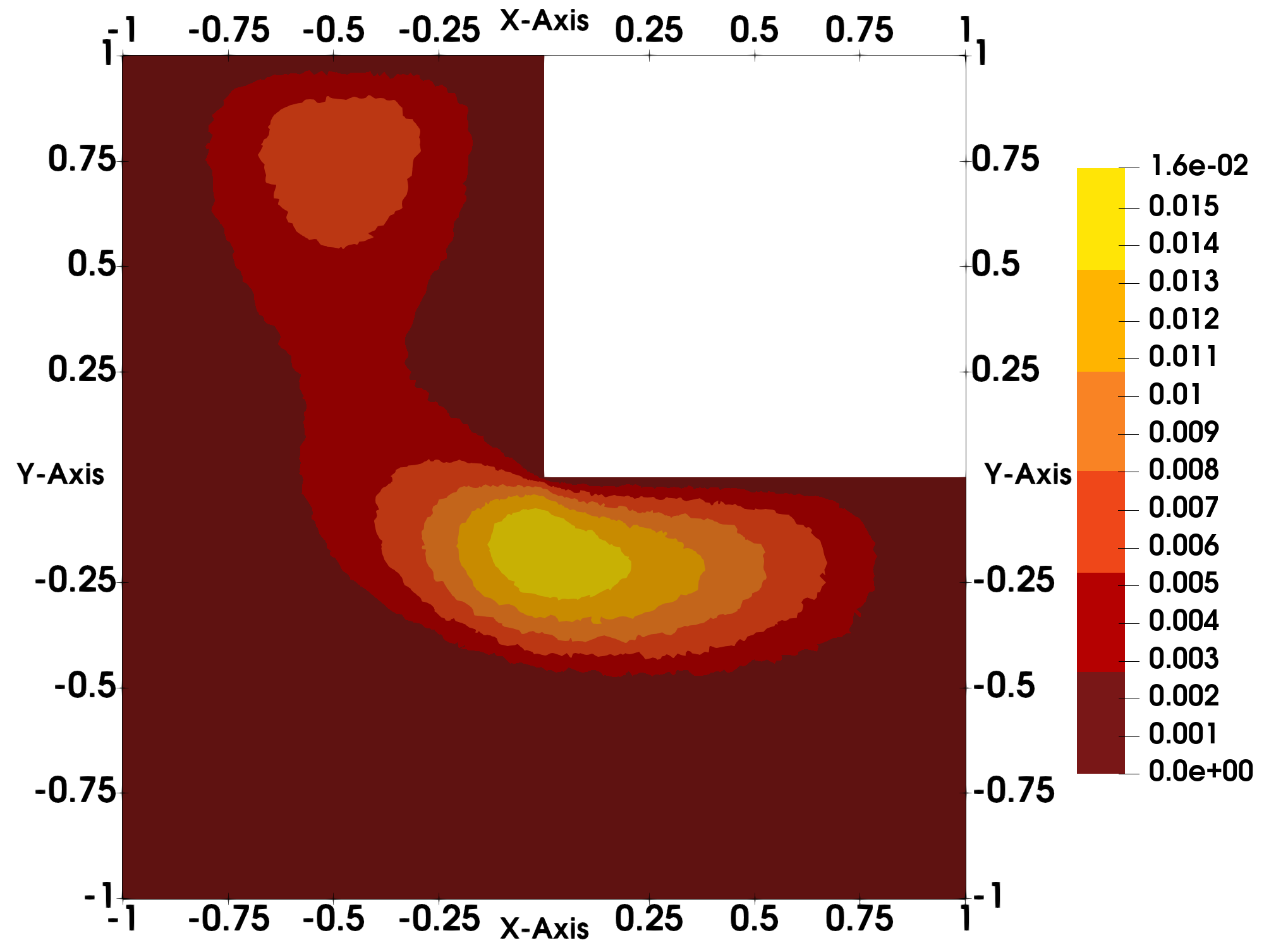}}
	\subfloat[$\u_{h2}$]{\includegraphics[scale=0.136]{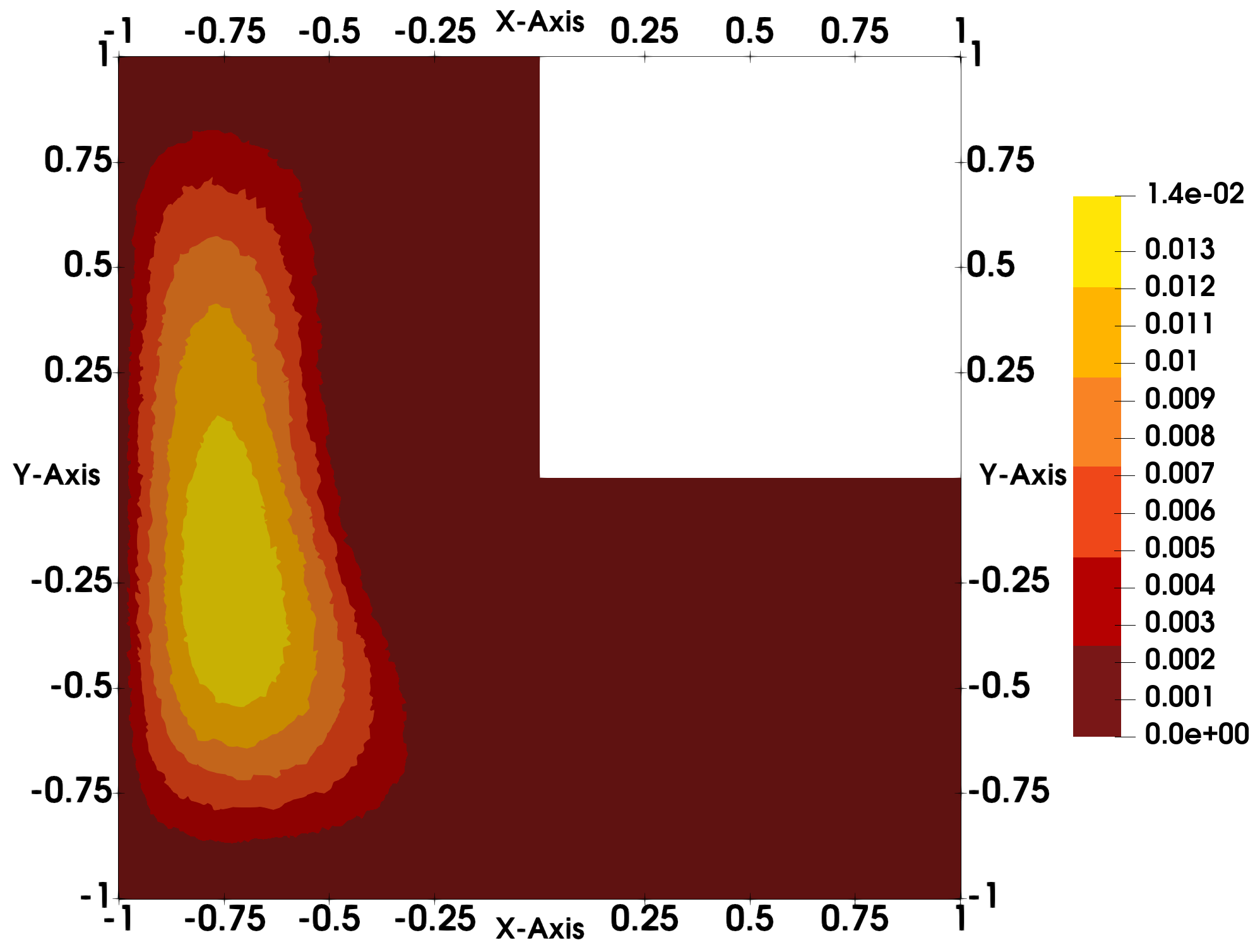}}
	\caption{Plots of numerical solutions of state velocity $(\y_{h1},\y_{h2})$, vorticity $(\kappa_h)$, pressure $(p_h)$, co-state velocity $(\w_{h1},\w_{h2})$, vorticity $(\vartheta_h)$, pressure $(q_h)$, and control $(\u_{h1},\u_{h2})$, respectively, for L-shaped domain Example~\ref{Example 5.3.}.} \label{FIGURE 8}
\end{figure}
\begin{figure}
	\centering
	\subfloat[$\y_{h1}$]{\includegraphics[scale=0.14]{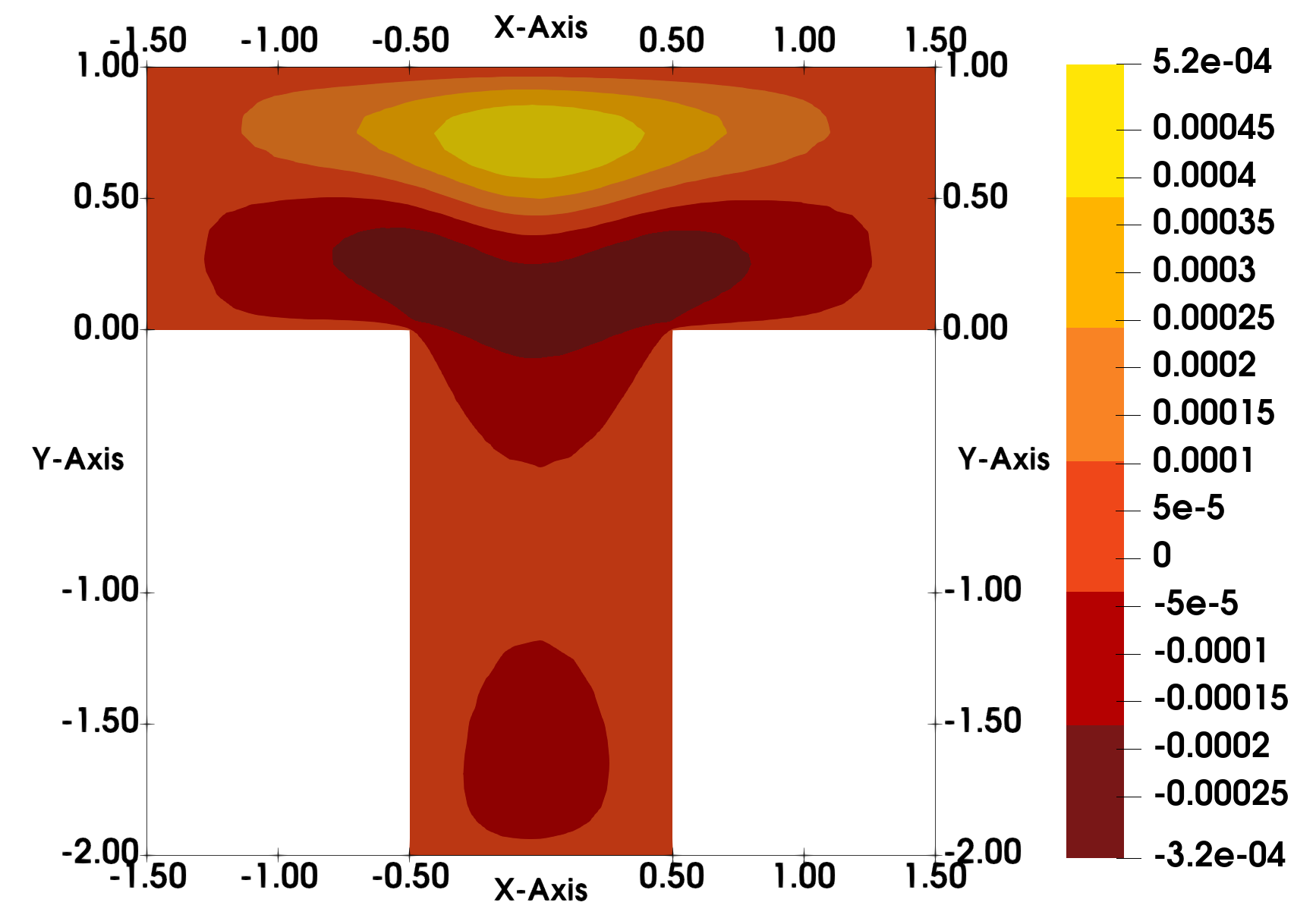}}
	\subfloat[$\y_{h2}$]{\includegraphics[scale=0.132]{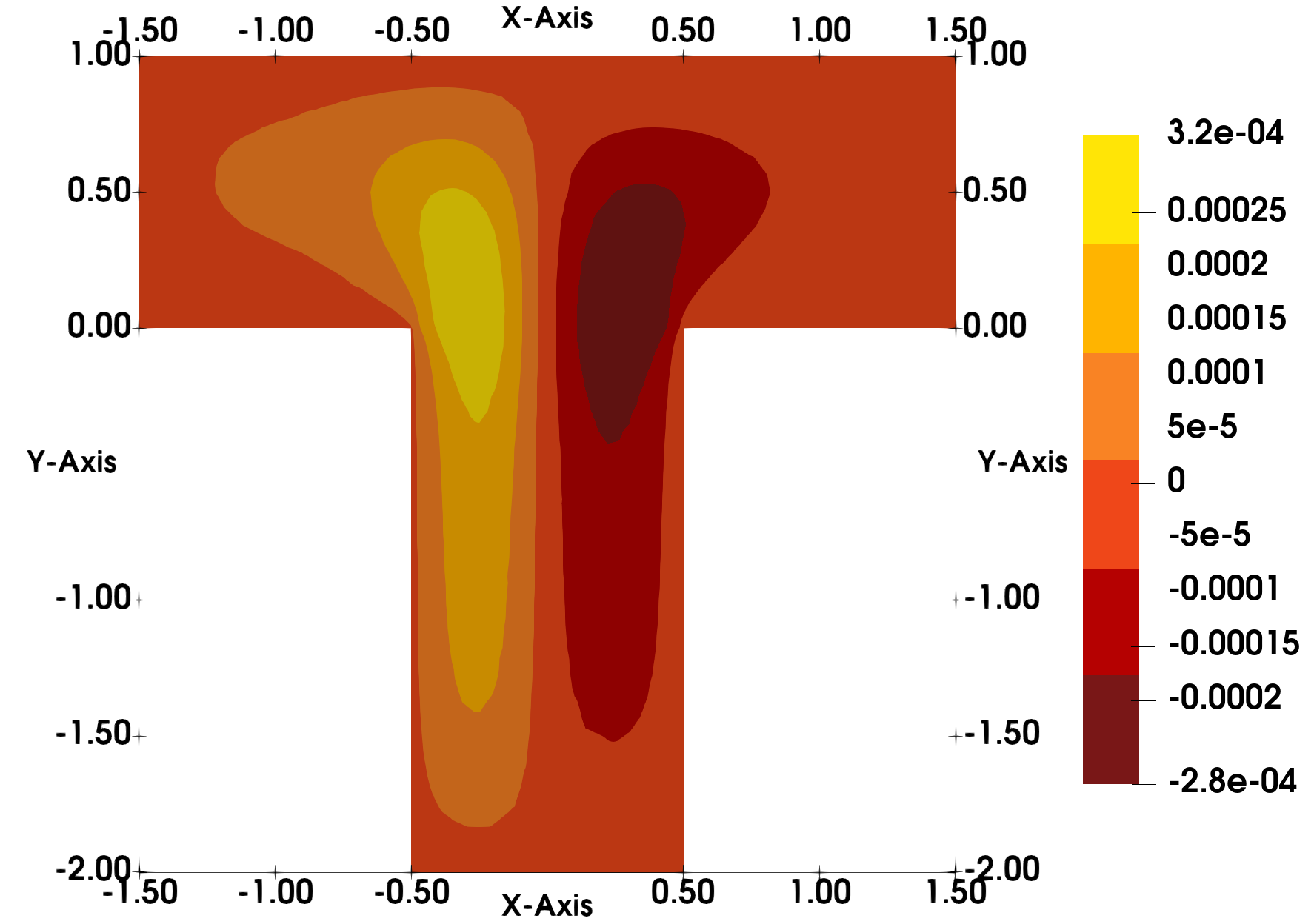}} 
	\subfloat[$\omega_h$]{\includegraphics[scale=0.132]{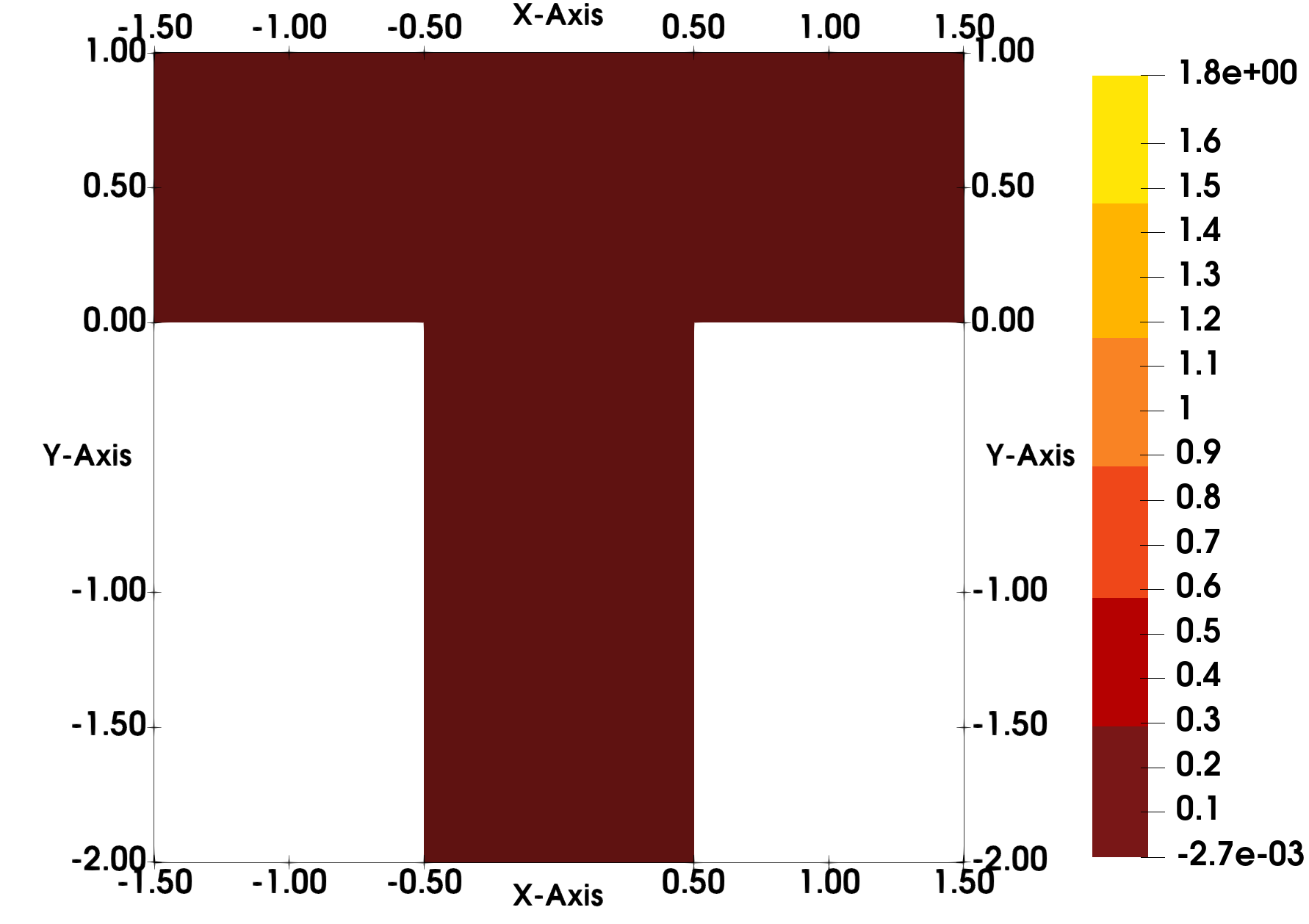}}
	\subfloat[$p_h$]{\includegraphics[scale=0.135]{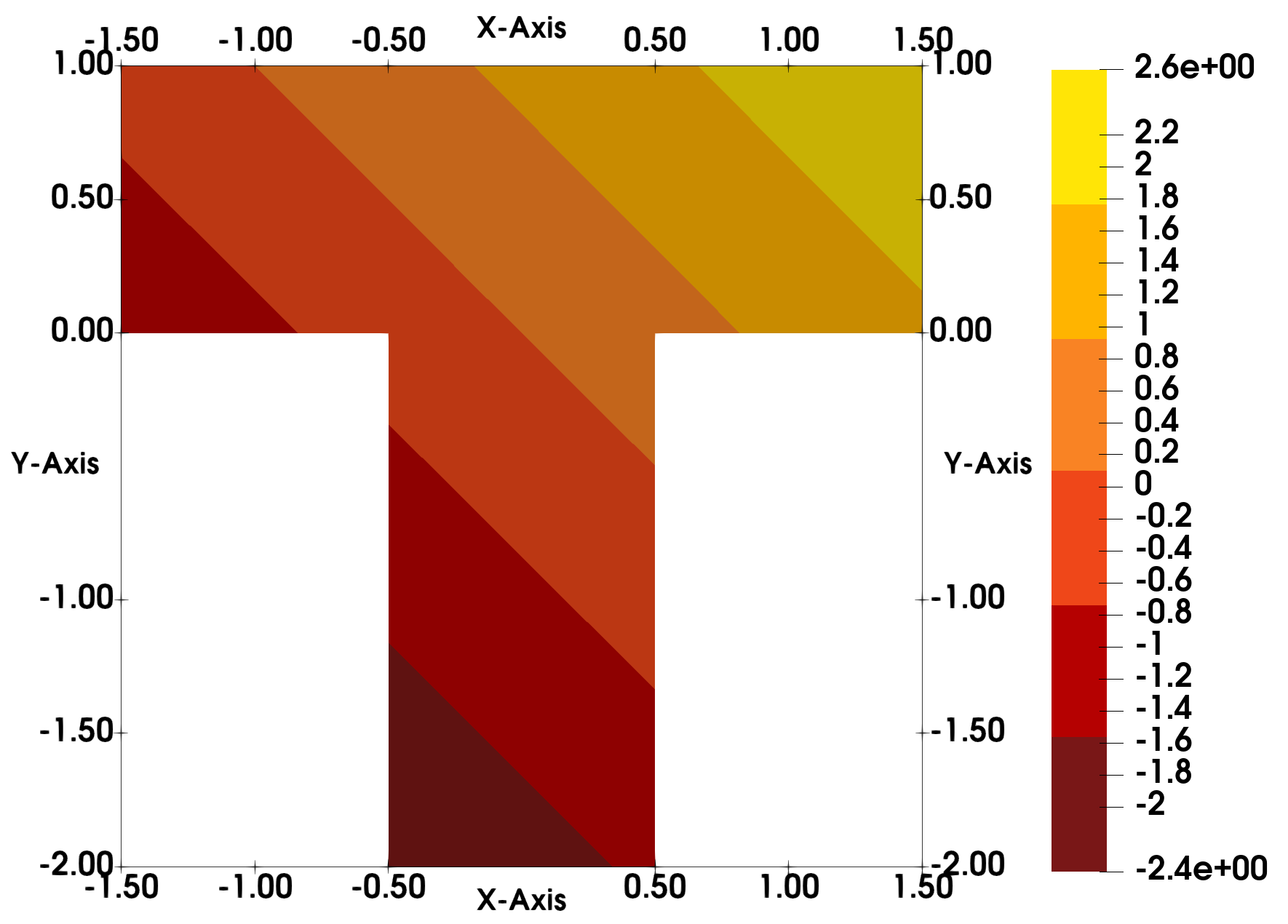}} \\
	\subfloat[$\w_{h1}$]{\includegraphics[scale=0.132]{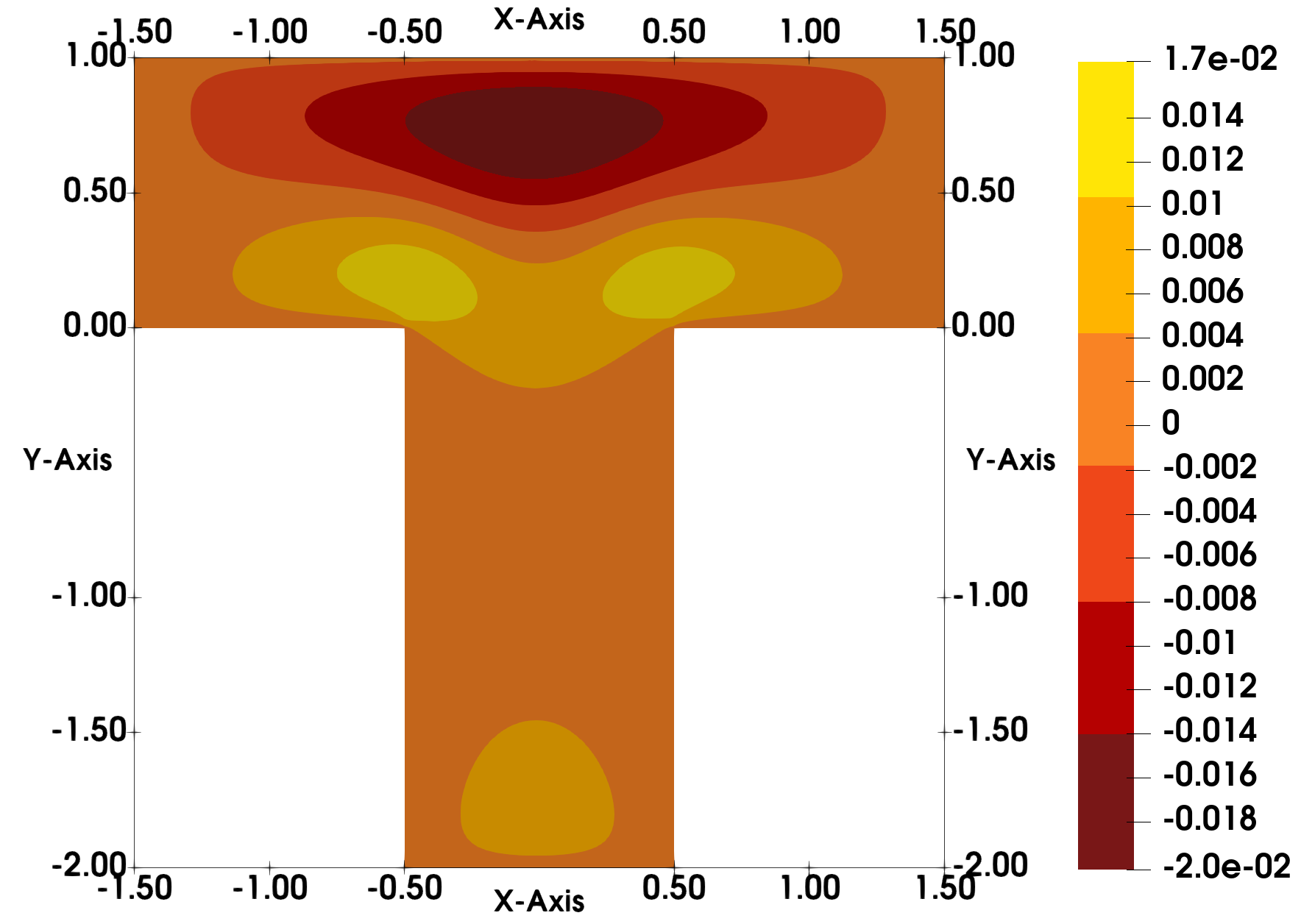}}
	\subfloat[$\w_{h2}$]{\includegraphics[scale=0.132]{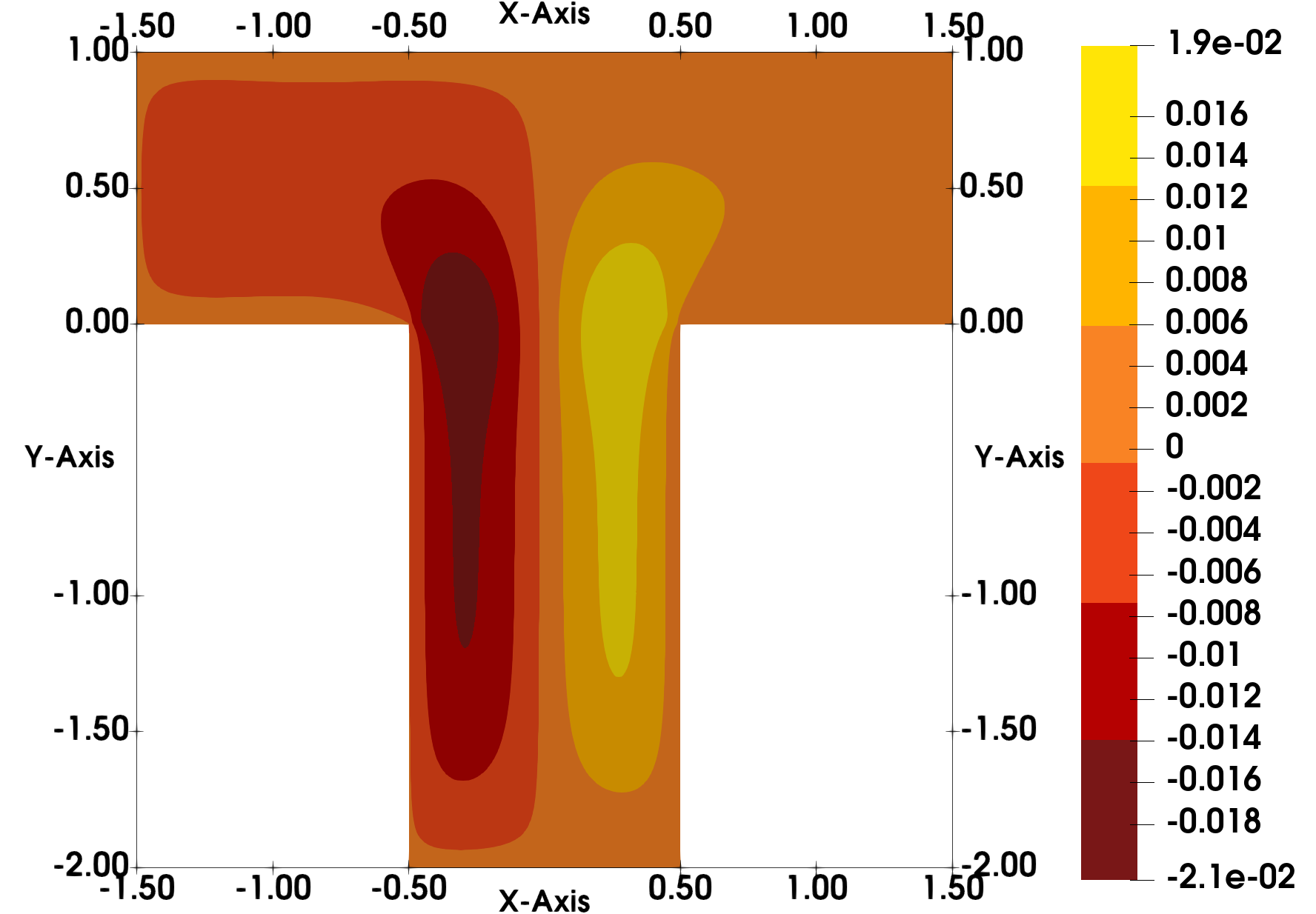}}
	\subfloat[$\vartheta_h$]{\includegraphics[scale=0.135]{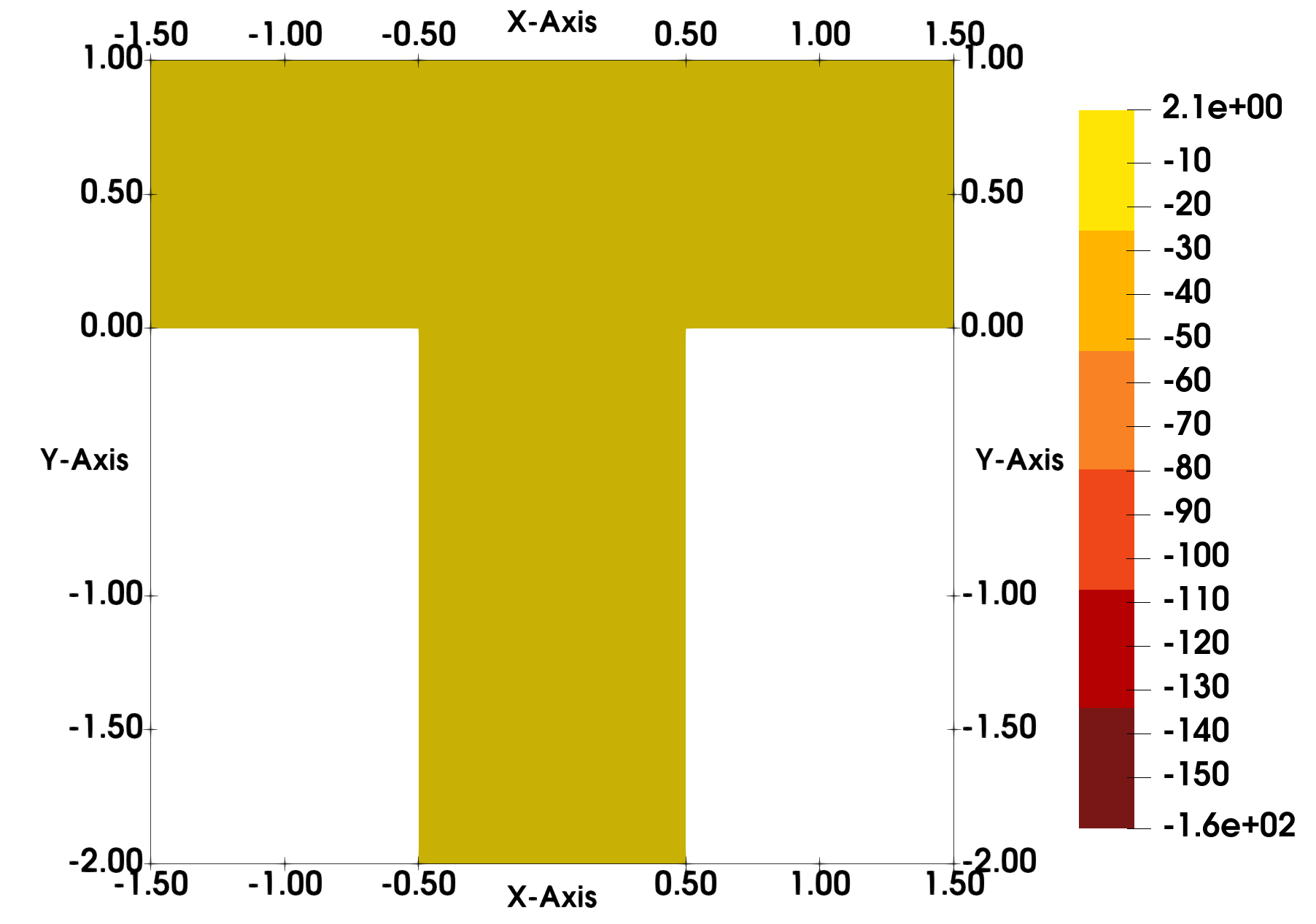}}
	\subfloat[$q_h$]{\includegraphics[scale=0.137]{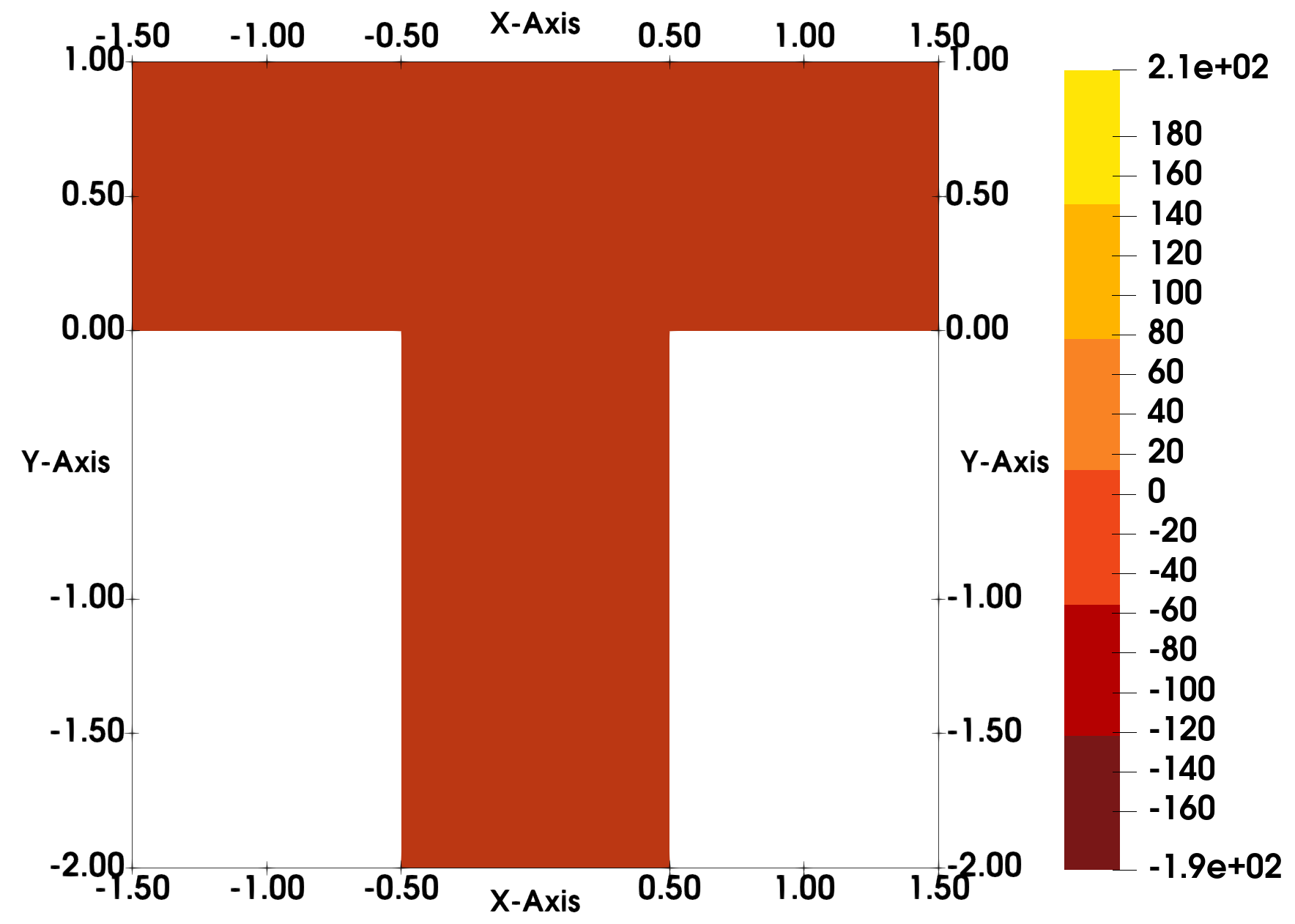}}\\
	\subfloat[$\u_{h1}$]{\includegraphics[scale=0.1420]{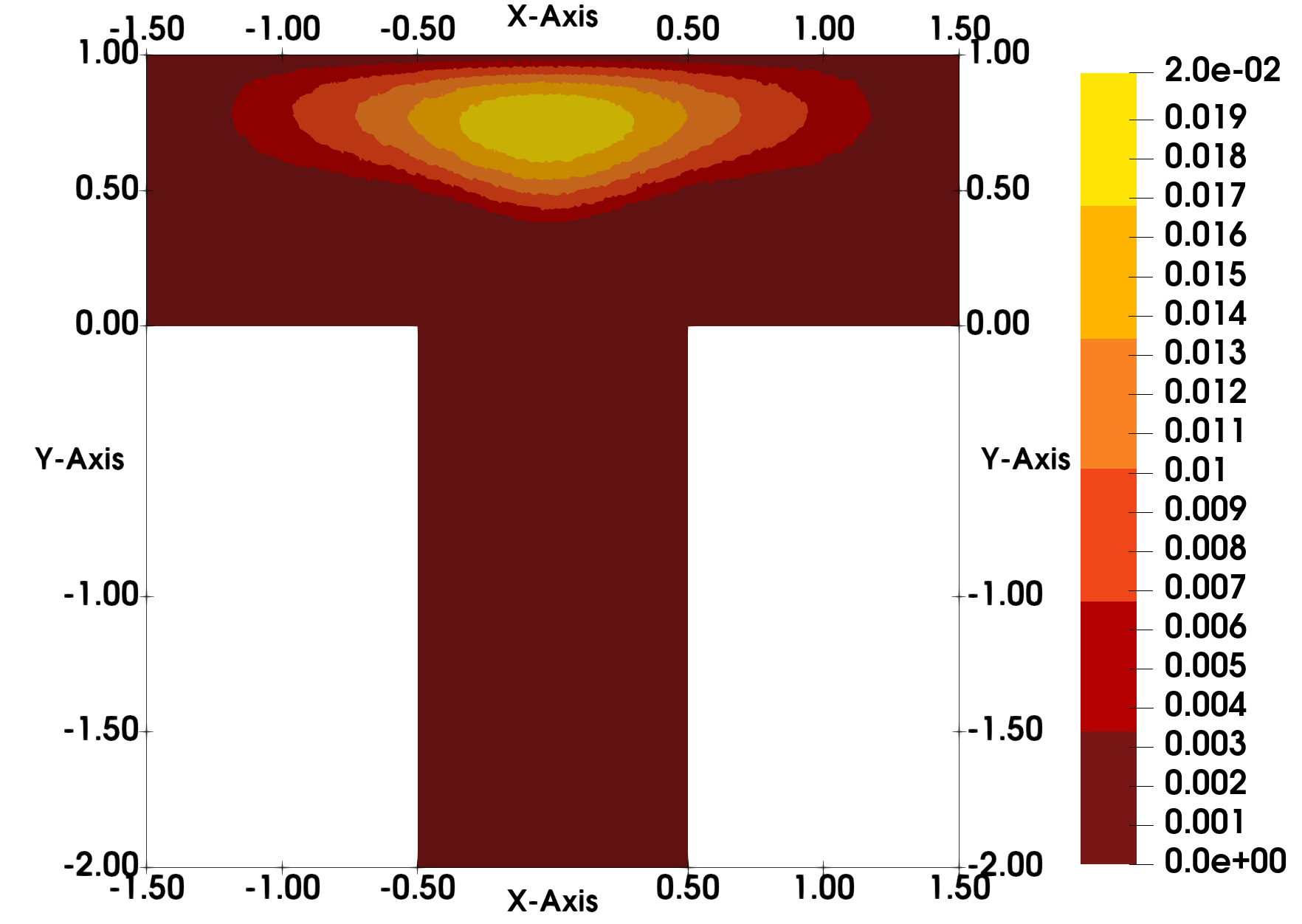}}
	\subfloat[$\u_{h2}$]{\includegraphics[scale=0.1420]{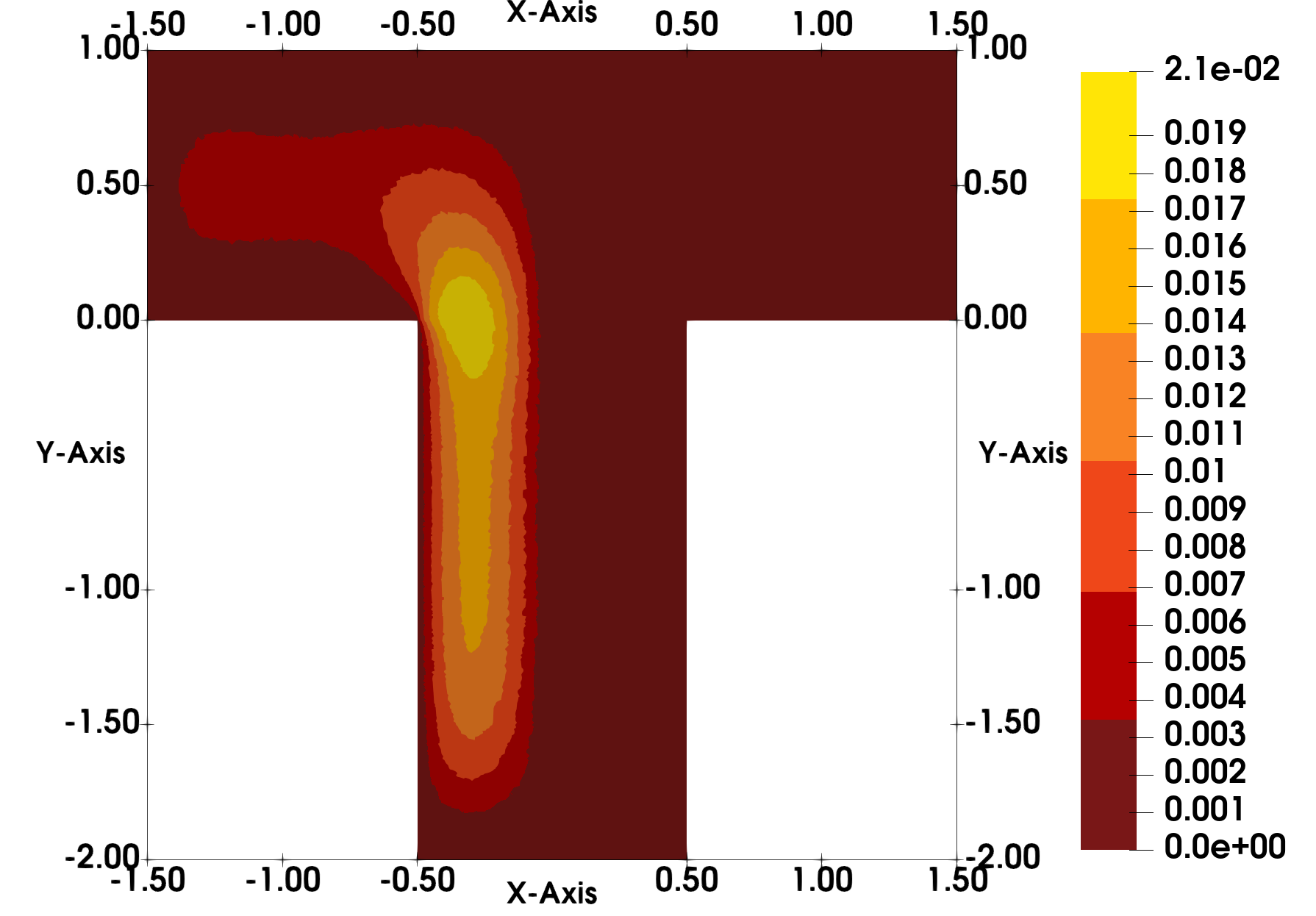}}
	\caption{Plots of numerical solutions of state velocity $(\y_{h1},\y_{h2})$, vorticity $(\kappa_h)$, pressure $(p_h)$, co-state velocity $(\w_{h1},\w_{h2})$, vorticity $(\vartheta_h)$, pressure $(q_h)$, and control $(\u_{h1},\u_{h2})$, respectively, for T-shaped domain Example~\ref{Example 5.3.}.} \label{FIGURE 9}
\end{figure}
	\subsection{Rectangular pipe flow with a circular hole \cite[Section~5]{MAMARA}}\label{Example 5.4.}
	Now, we consider a more realistic problem, where we investigate flow within a rectangular pipe $[0,2] \times [0,0.41]$ featuring an obstruction (a circular hole) centered at $(0.2,0.2)$ with a radius of $0.1$. Pressure is enforced at both the inlet and outlet boundaries, with $\y = 0$ on the boundary. The coefficients are set as $\nu = 0.1 + 0.9 x_1 x_2$, $\boldsymbol{\beta} = (1,1)^{T}$, and $\sigma = 100$. Control bounds are set as $\mathbf{a}=(0,0)^{T}$ and $\mathbf{b}=(0.1,0.1)^{T}$. The source function, desired velocity, and vorticity are the same as in the previous example.
	Although, the precise solutions to this problem are not known. We start with an initial mesh of 1094 elements and use an adaptive refinement technique. This procedure seeks to improve resolution, especially in the regions near the boundary and the circular hole, as Figure~\ref{FIGURE 10} illustrates. We find an optimal rate of decay in the indicators $\boldsymbol{\eta_{CG}}$ and $\boldsymbol{\eta_{DG}}$ as the mesh is sufficiently refined. The improved signs indicate that the adaptive refinement efficiently caught the solution's characteristics. Figure~\ref{FIGURE 12} depicts plots of the numerical solution, highlighting the accuracy and resolution achieved by the adaptive refinement approach.
		 \begin{figure}
			\centering
			\subfloat[Initial mesh]{\includegraphics[scale=0.572]{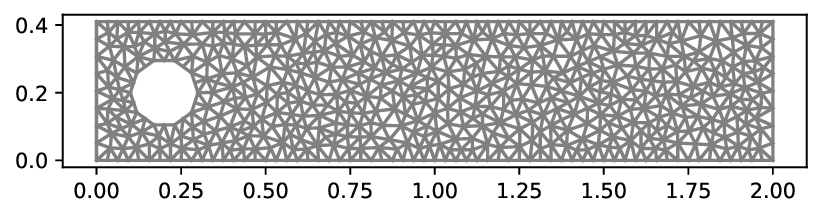}} 
			\subfloat[28668 DOF]{\includegraphics[scale=0.572]{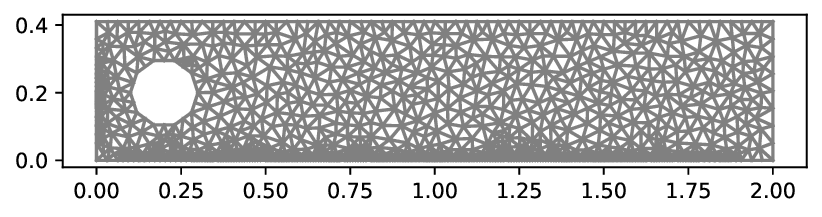}}\\
			\subfloat[44788 DOF]{\includegraphics[scale=0.572]{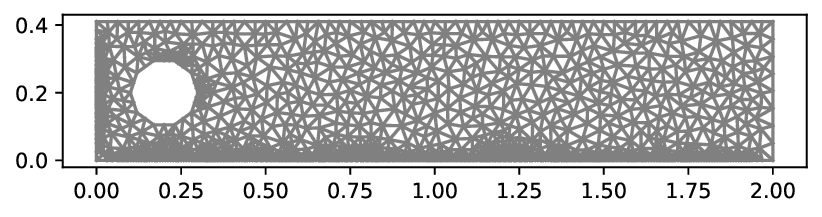}} 
			\subfloat[79884 DOF]{\includegraphics[scale=0.572]{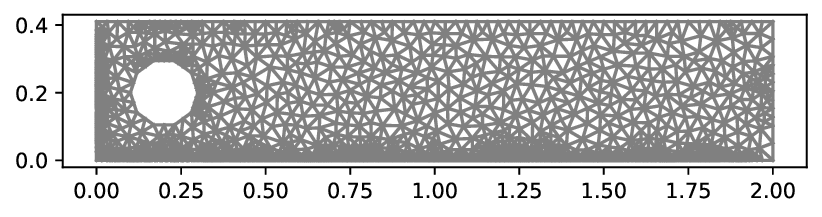}}
			\caption{Adaptively refined meshes showing refinement near boundary and obstacle for Example~\ref{Example 5.4.}.} \label{FIGURE 10}
		\end{figure}
		\begin{figure}
			\centering
			\subfloat[$\y_{h1}$]{\includegraphics[scale=0.192]{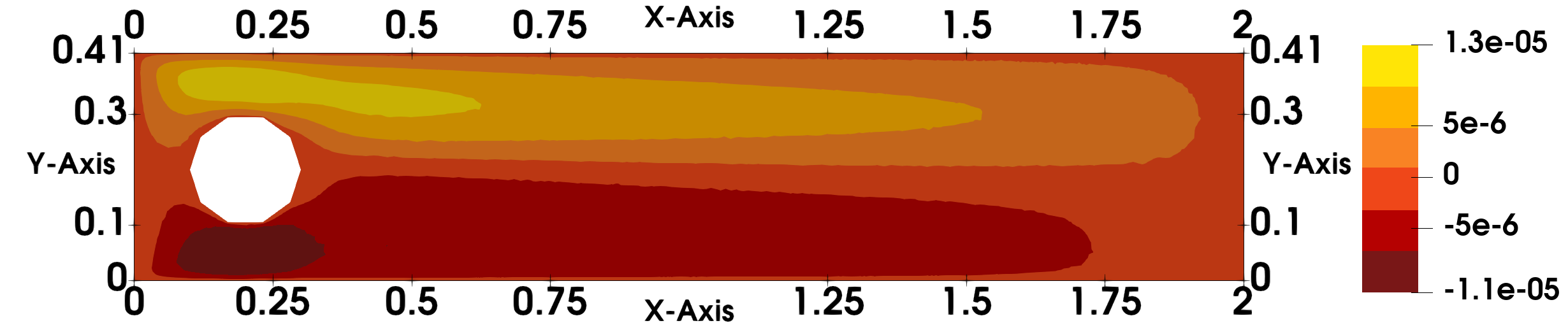}}
			\subfloat[$\y_{h2}$]{\includegraphics[scale=0.172]{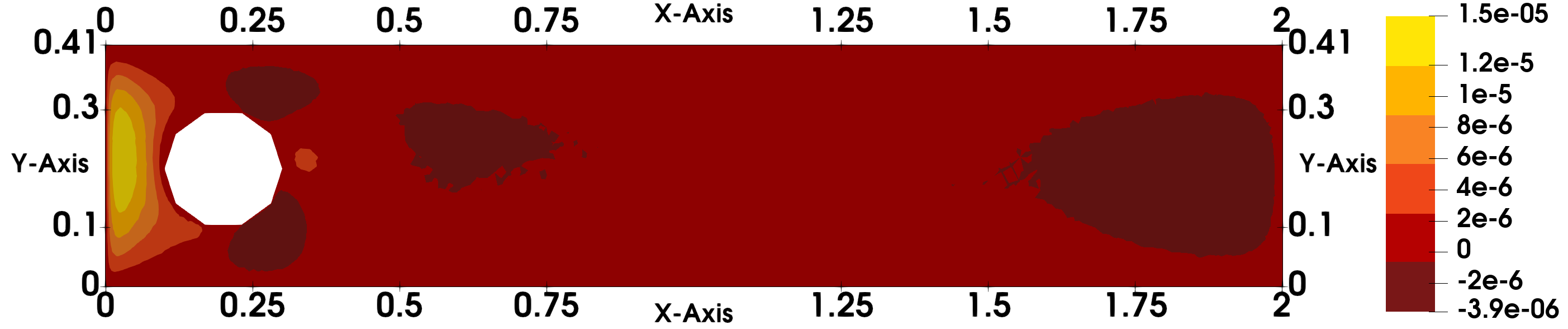}} \
			\subfloat[$\omega_h$]{\includegraphics[scale=0.172]{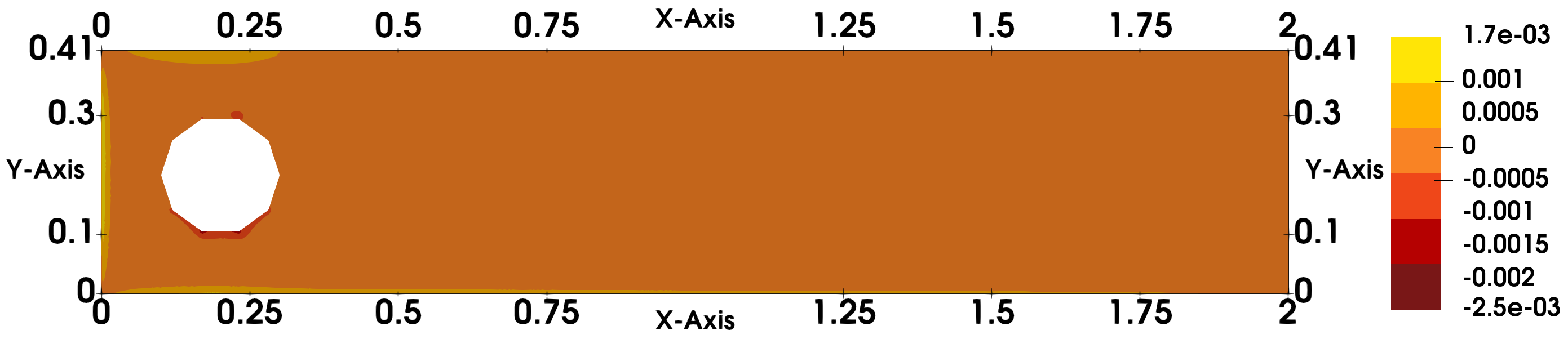}} 
			\subfloat[$p_h$]{\includegraphics[scale=0.172]{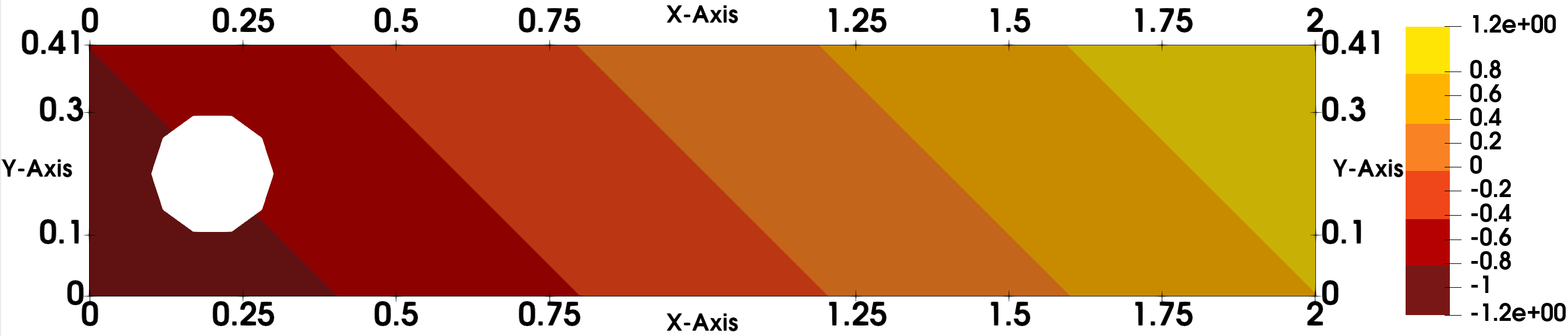}}\\
			\subfloat[$\w_{h1}$]{\includegraphics[scale=0.192]{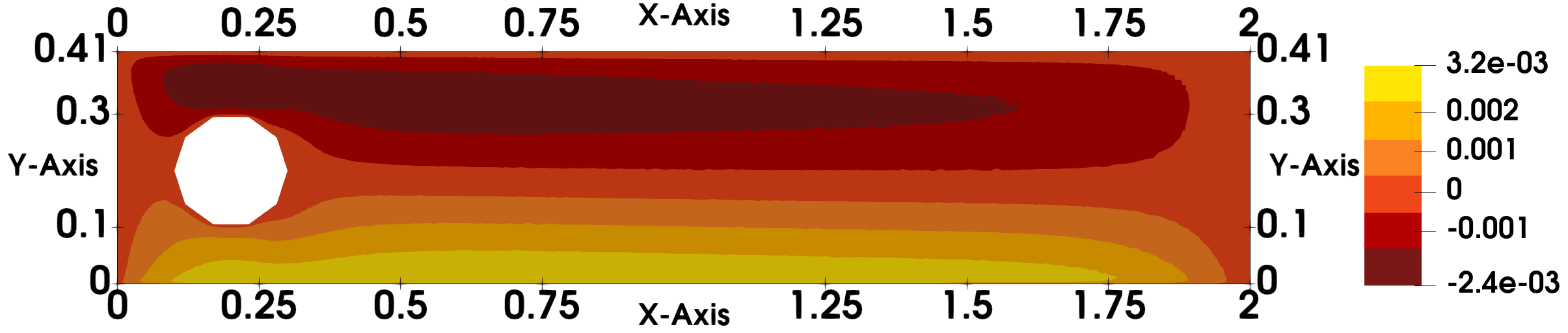}}
			\subfloat[$\w_{h2}$]{\includegraphics[scale=0.172]{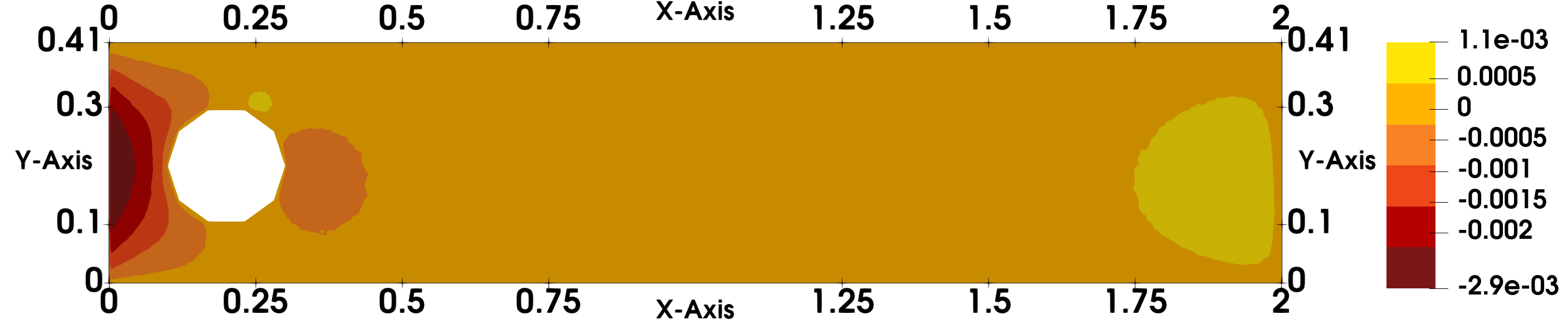}} \\
			\subfloat[$\vartheta_h$]{\includegraphics[scale=0.147]{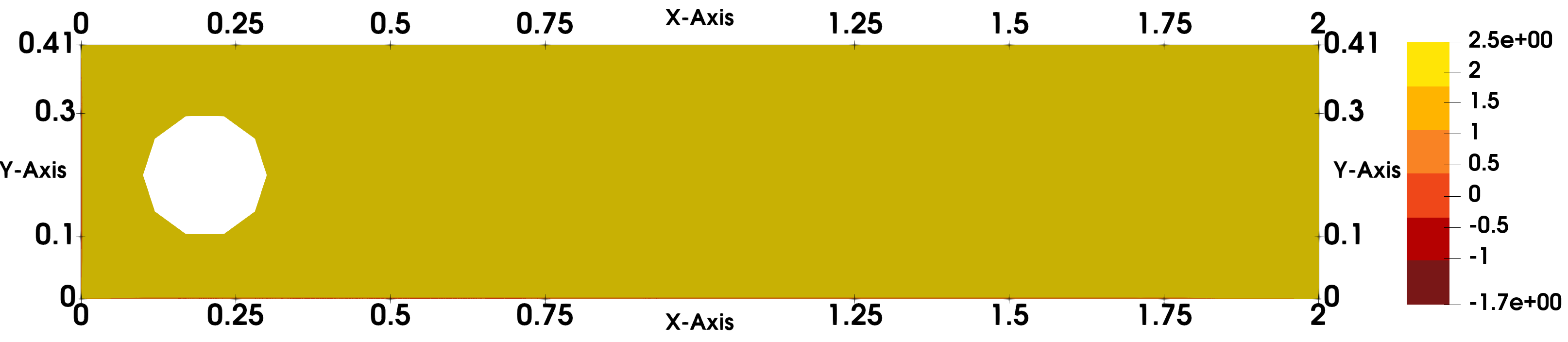}} 
			\subfloat[$q_h$]{\includegraphics[scale=0.172]{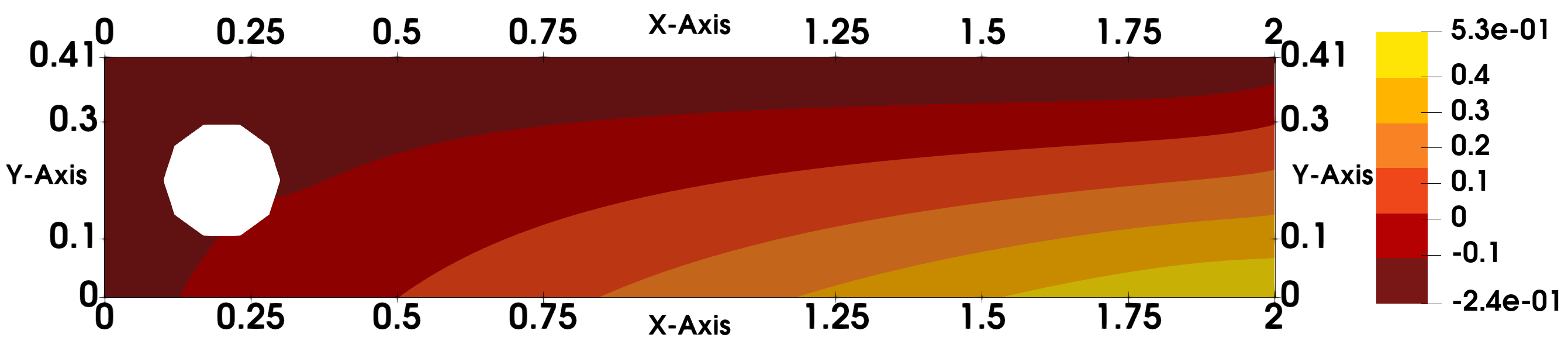}}\\
			\subfloat[$\u_{h1}$]{\includegraphics[scale=0.172]{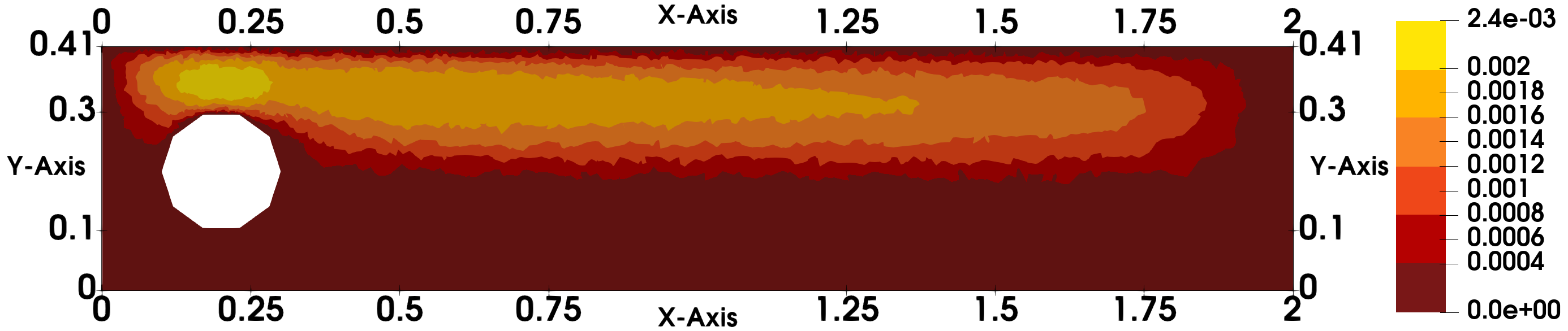}}
			\subfloat[$\u_{h2}$]{\includegraphics[scale=0.172]{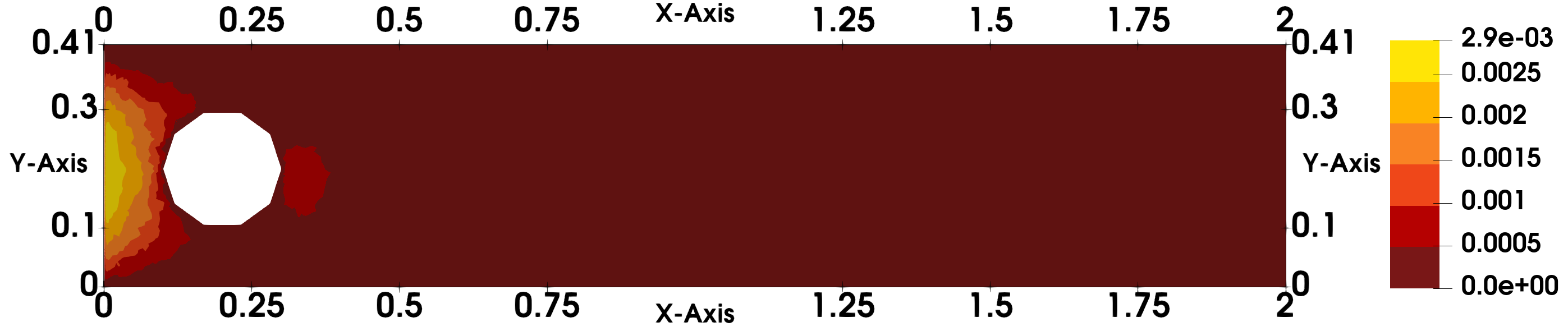}}
			\caption{Plots of numerical solutions of state velocity $(\y_{h1},\y_{h2})$, vorticity $(\kappa_h)$, pressure $(p_h)$, co-state velocity $(\w_{h1},\w_{h2})$, vorticity $(\vartheta_h)$, pressure $(q_h)$, and control $(\u_{h1},\u_{h2})$, respectively, for Example~\ref{Example 5.4.}.} \label{FIGURE 12}
		\end{figure}
	\subsection{3-D convergence test}\label{Example 5.5.}
	This numerical test aims to evaluate the method's accuracy in the 3D scenario. Let $\Omega = (0, 1)^3$, and $\f$, $\y_d$, and $\boldsymbol{\omega}_d$ are chosen such that the exact solutions are:
	\begin{align*}
		\y(x_1,x_2,x_3) &= \w(x_1,x_2,x_3) = \textbf{curl} \big((x_1 x_2 x_3(1-x_1) (1-x_2) (1-x_3))^{2}\big), \\
		\boldsymbol{\omega}(x_1,x_2,x_3) &= \boldsymbol{\vartheta}(x_1,x_2,x_3) = \textbf{curl} \big(\y\big), \quad \quad p(x_1,x_2,x_3)  = q(x_1,x_2,x_3) = 1 - x_1^{2} - x_2^{2} - x_3^{2}.
	\end{align*}
	The remaining coefficients and control bounds are choosen as: $\nu = \nu_0 + (\nu_1-\nu_0) x_1^{2} x_2^{2} x_3^{2},\ \boldsymbol{\beta}= \y,\ \sigma = 1000,\ \mathbf{a}=(-0.1,-0.1,-0.1)^{T}$ and $\mathbf{b}=(0.25,0.25,0.25)^{T}$, where $\nu_0 = 0.1$ and $\nu_1= 1$. After successive uniform refinements, both the conforming ($k=1$) and non-conforming ($k=0$) schemes demonstrate optimal convergence rates across state, co-state, and control variables. These convergence behaviors are vividly illustrated in Figures~\ref{FIGURE 13}. Additionally, Figure~\ref{FIGURE 15} showcases the numerical solutions (streamline plots) for the state and co-state variables, offering a comprehensive depiction of the model's behavior under both schemes.
	\begin{figure}
		\centering
		\subfloat[CG scheme $(k=1)$]{\includegraphics[scale=0.212]{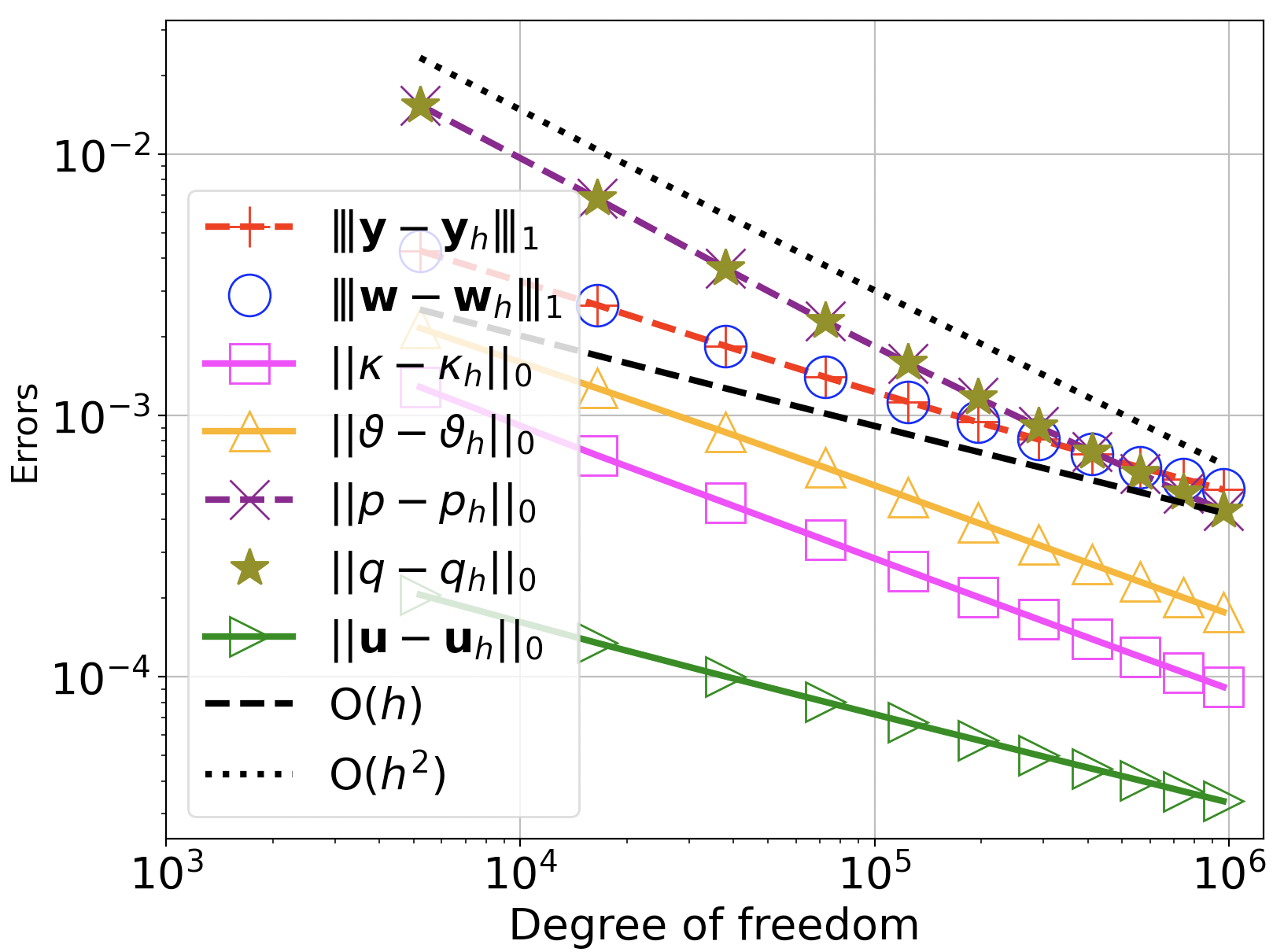}}
		\subfloat[DG scheme $(k=0)$]{\includegraphics[scale=0.212]{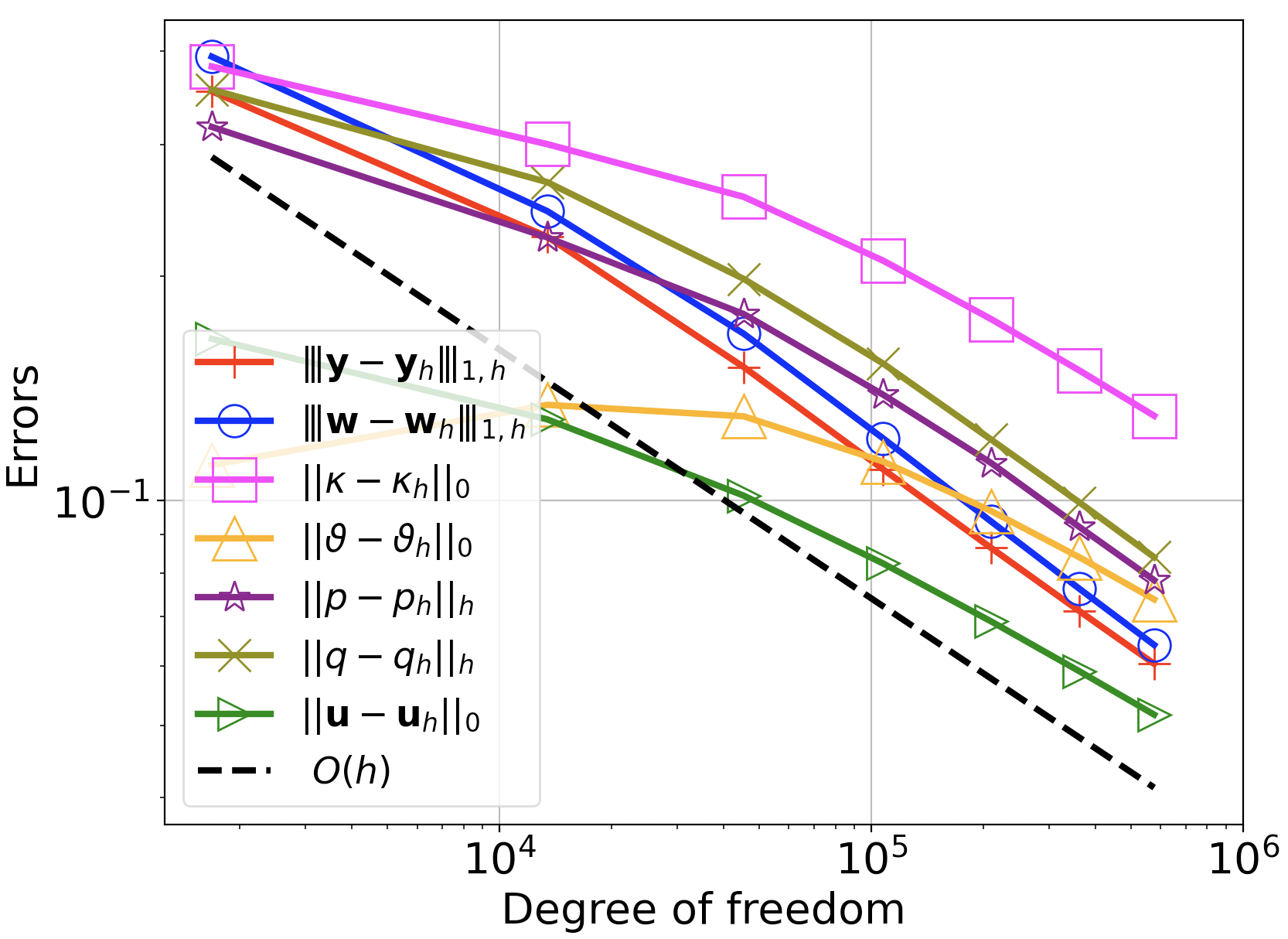}} \\
		\subfloat[Indicator-Total error (CG)]{\includegraphics[scale=0.158]{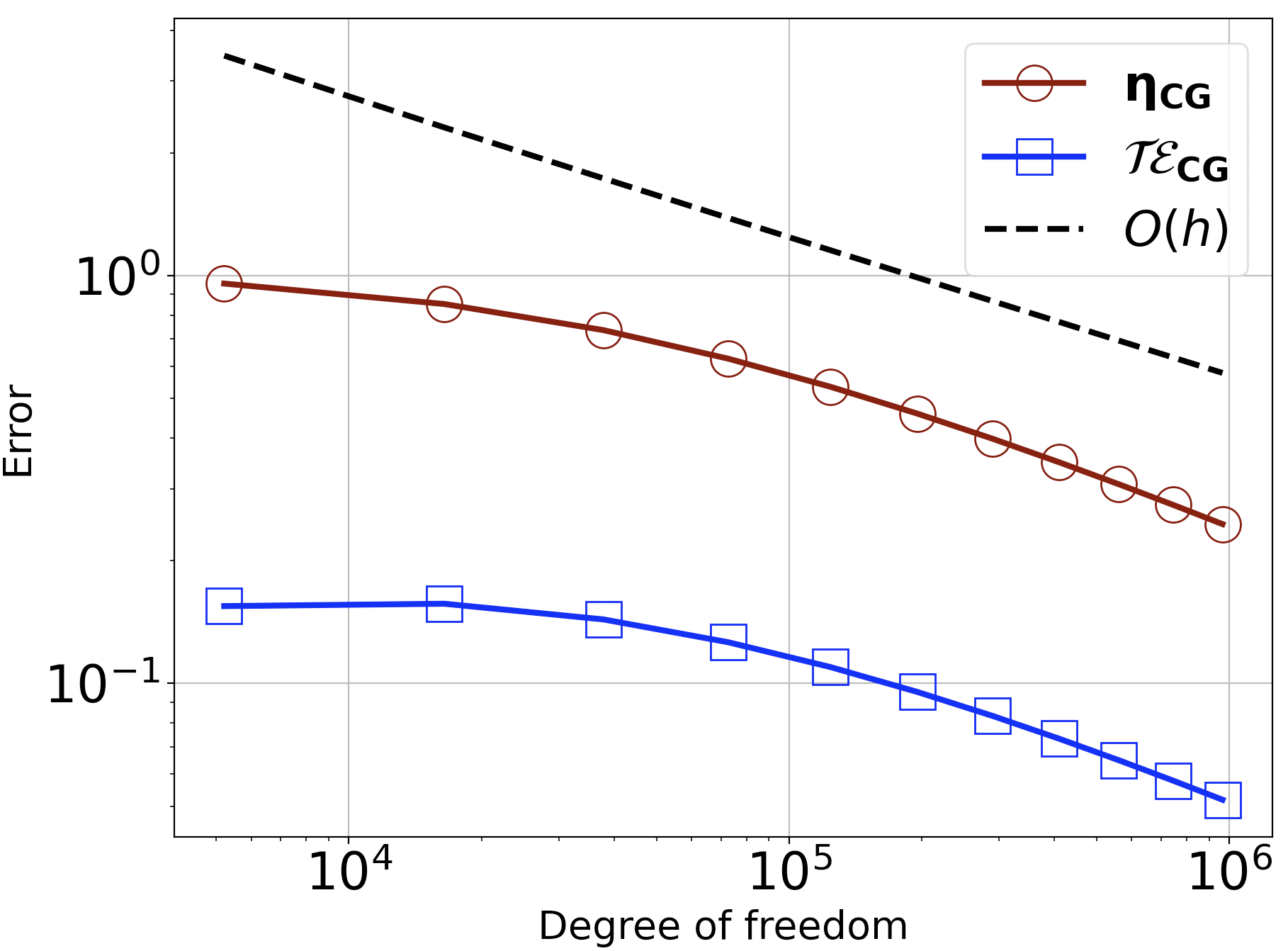}}
		\subfloat[Indicator-Total error (DG)]{\includegraphics[scale=0.158]{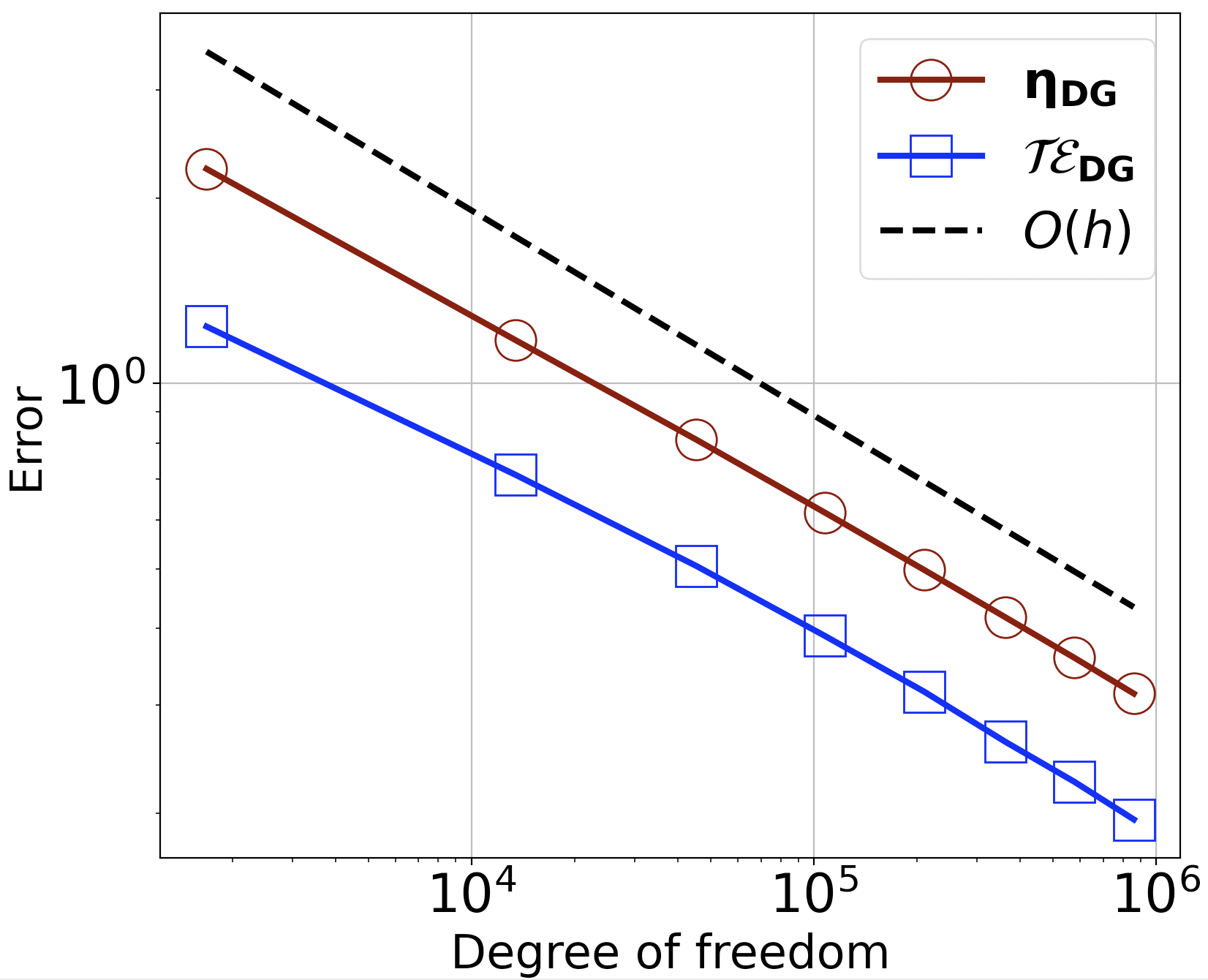}}
		\subfloat[Efficiency (CG-DG)]{\includegraphics[scale=0.162]{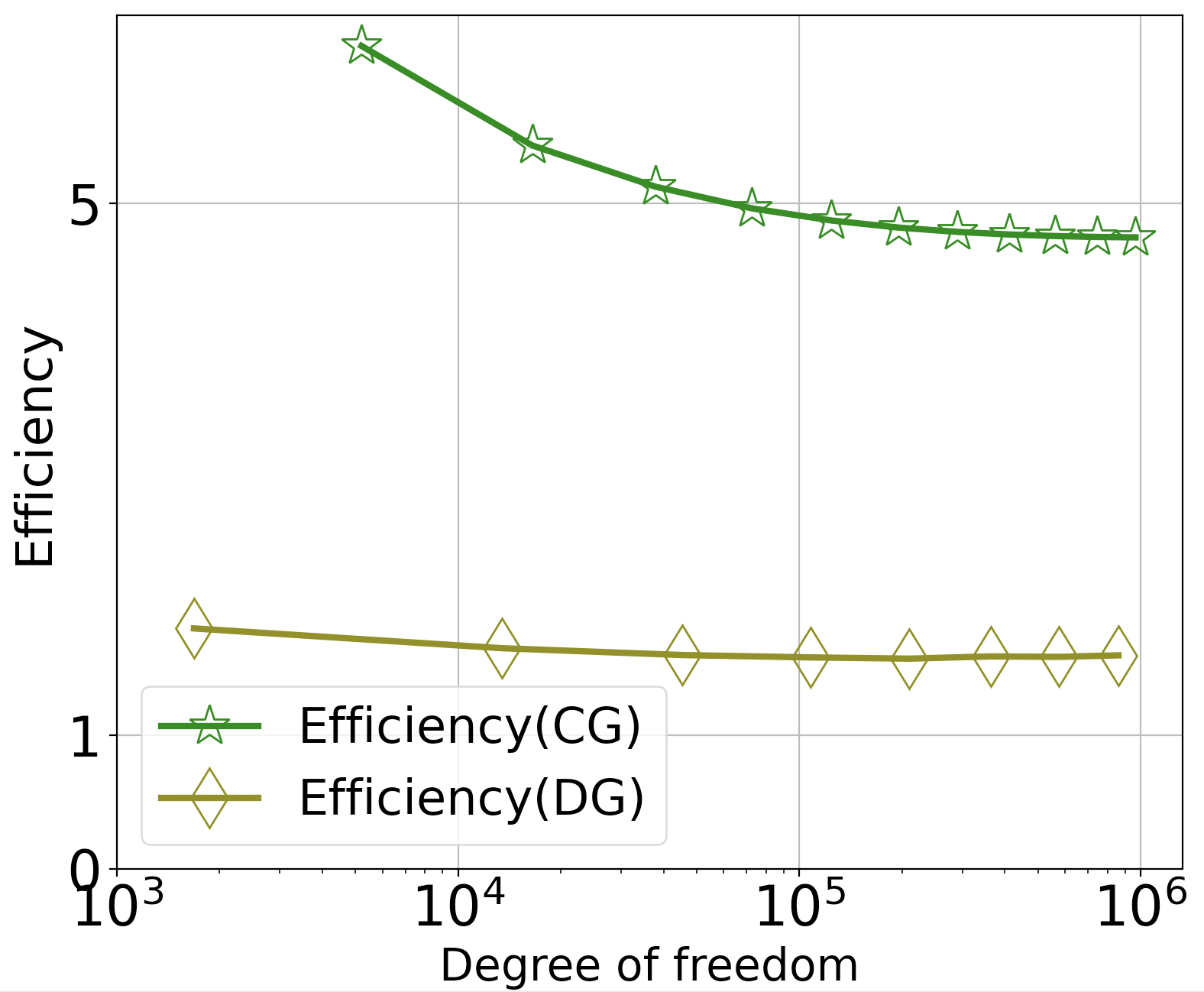}} 
		\caption{Convergence plots for the (a) CG scheme (b) DG scheme (c) Indicator- Total error (CG) (d) Indicator- Total error (DG) and (e) Efficiency for Example~\ref{Example 5.5.}.} \label{FIGURE 13}
	\end{figure}
\begin{figure}
	\centering
	\subfloat[$\y_h$]{\includegraphics[scale=0.155]{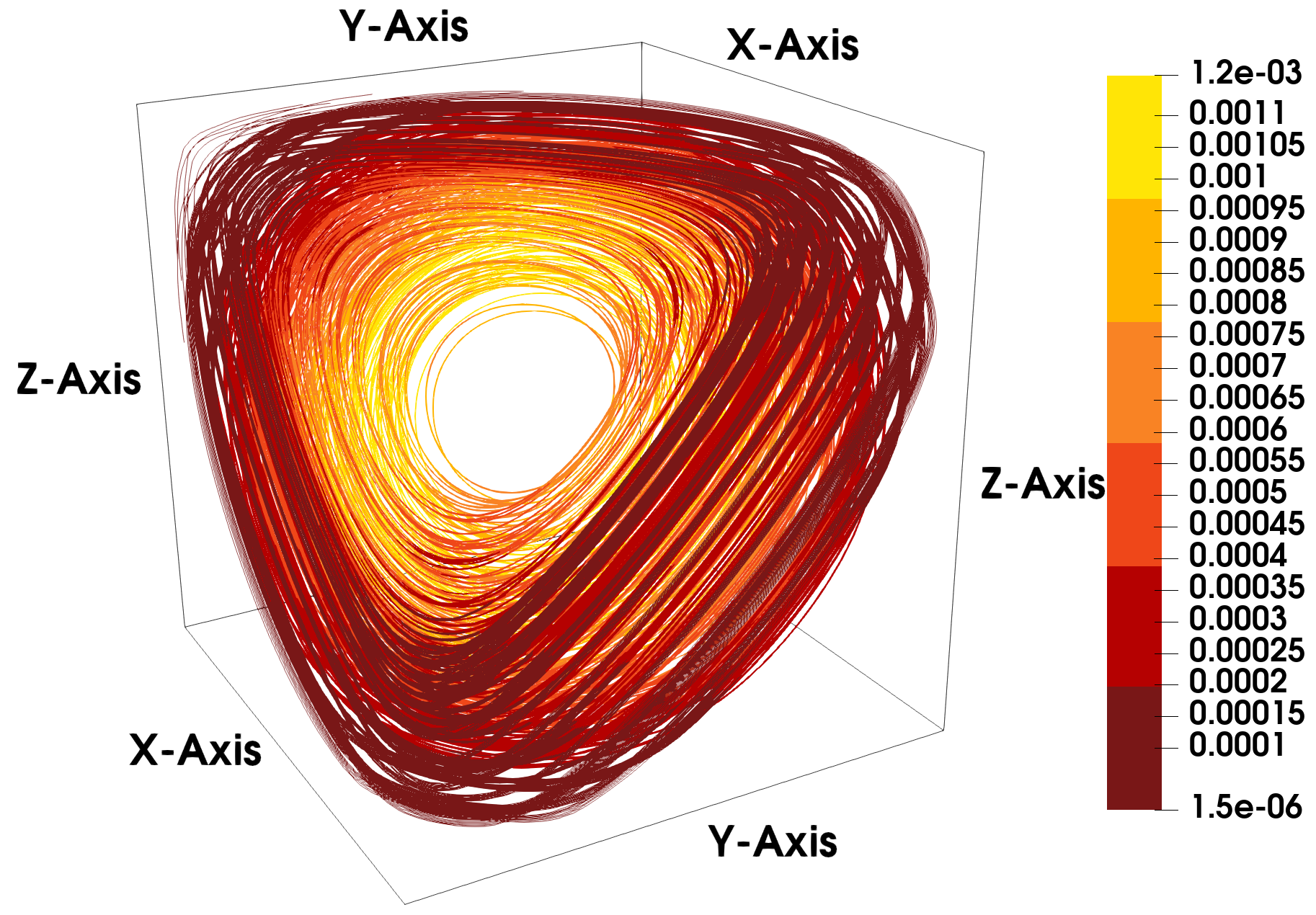}}
	\subfloat[$\boldsymbol{\omega}_h$]{\includegraphics[scale=0.155]{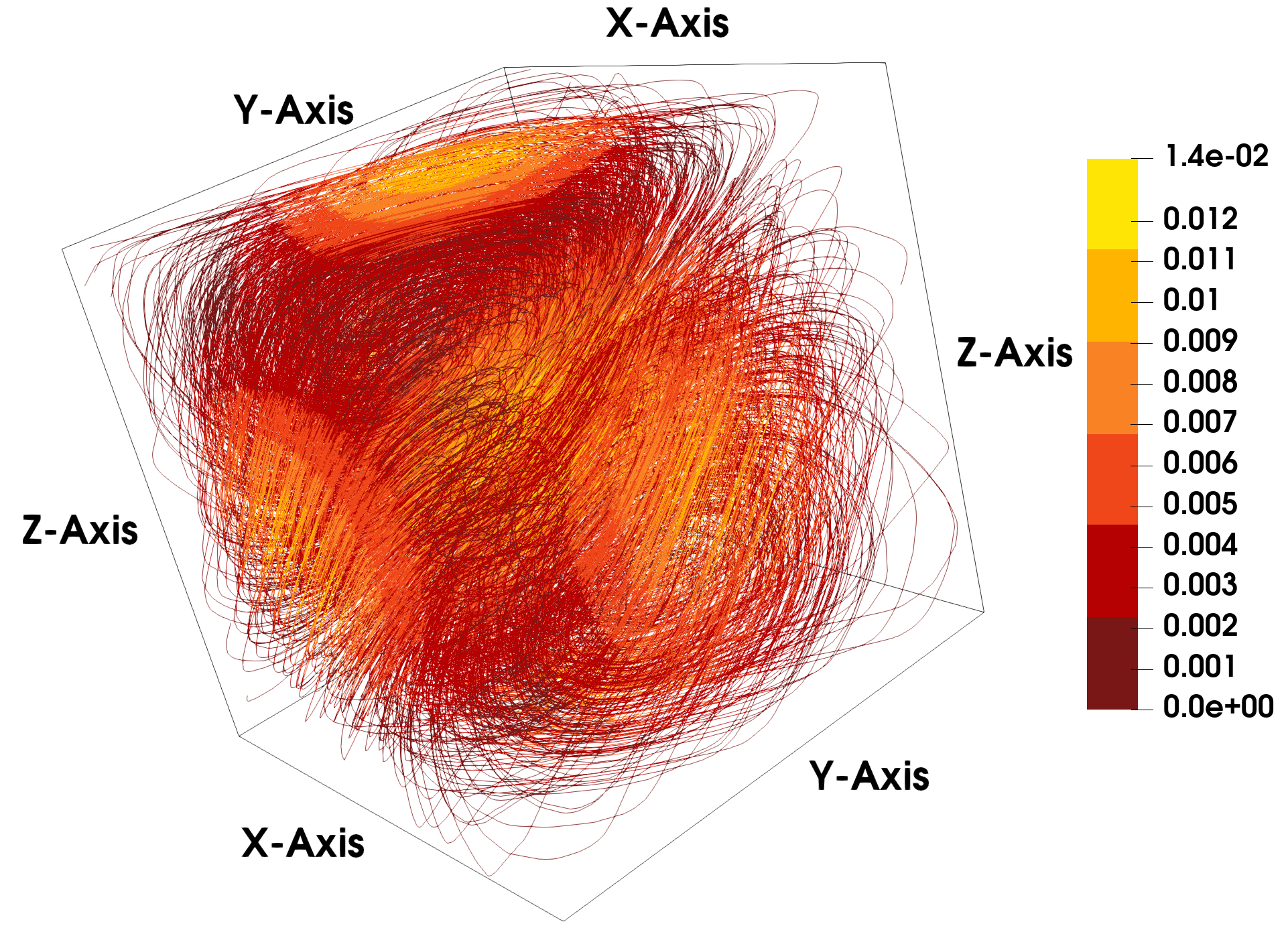}} 
	\subfloat[$p_h$]{\includegraphics[scale=0.123]{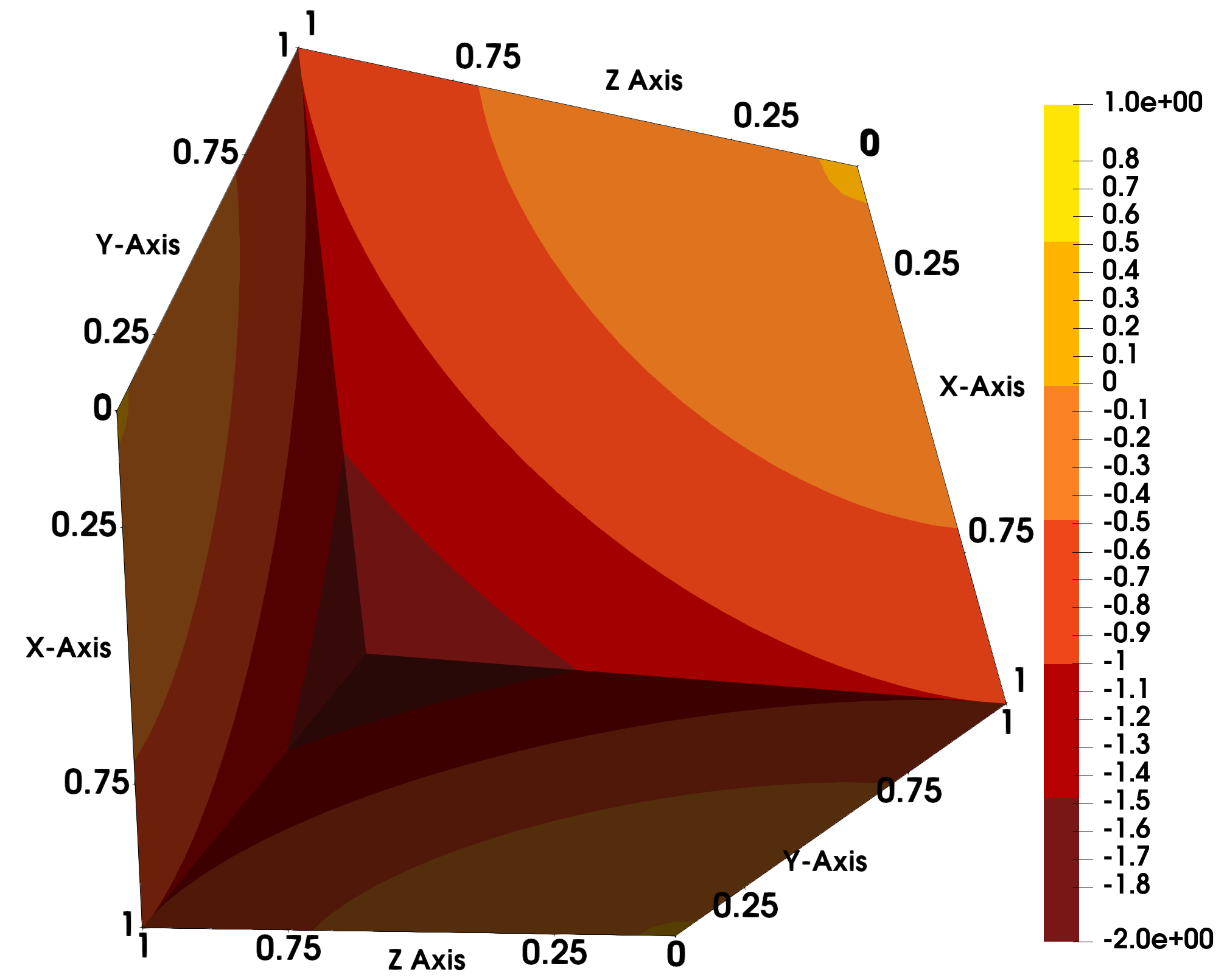}} \\
	\subfloat[$\w_h$]{\includegraphics[scale=0.155]{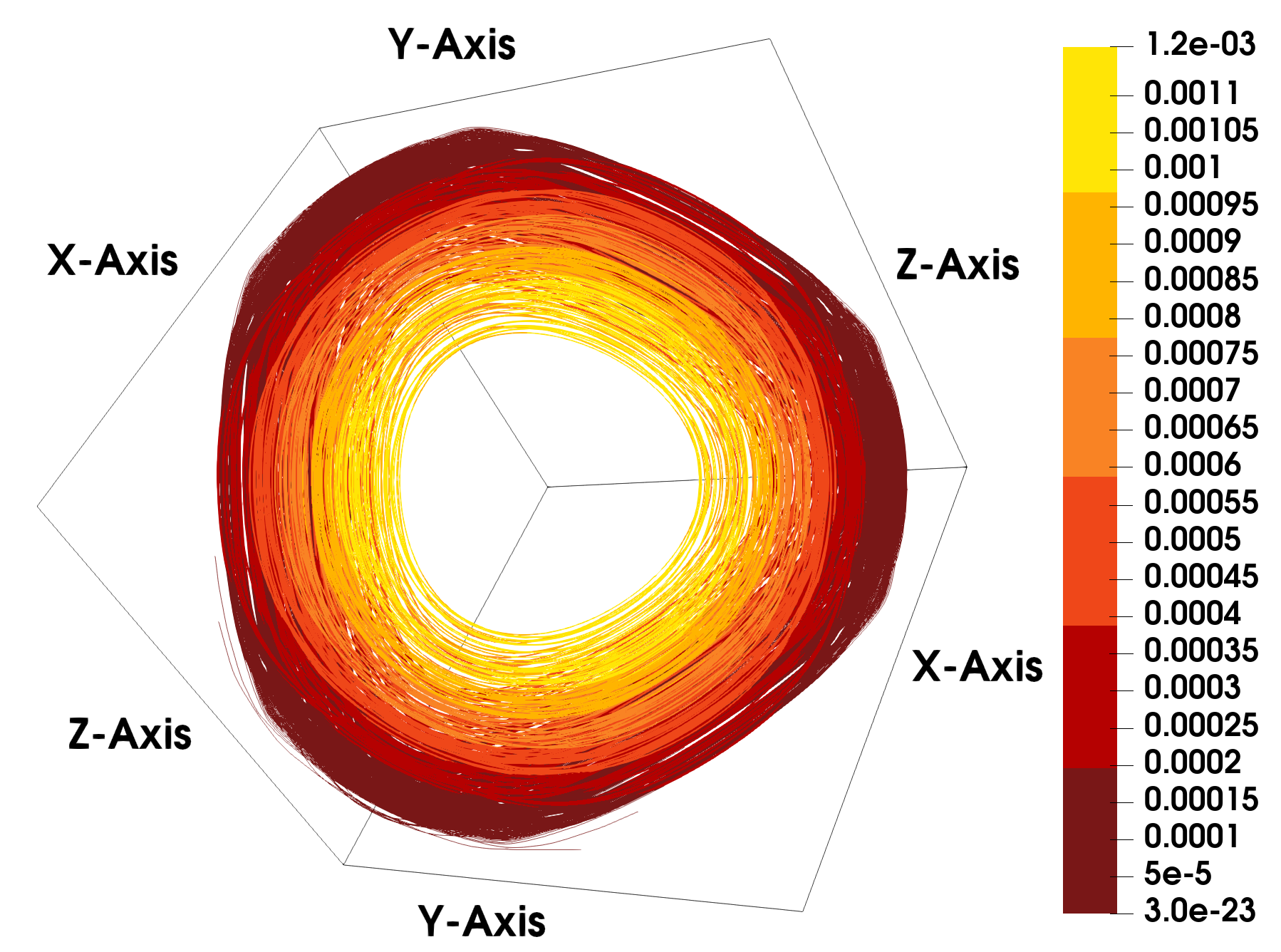}}
	\subfloat[$\boldsymbol{\vartheta}_h$]{\includegraphics[scale=0.155]{ws3dv1.png}} 
	\subfloat[$\u_{h}$]{\includegraphics[scale=0.125]{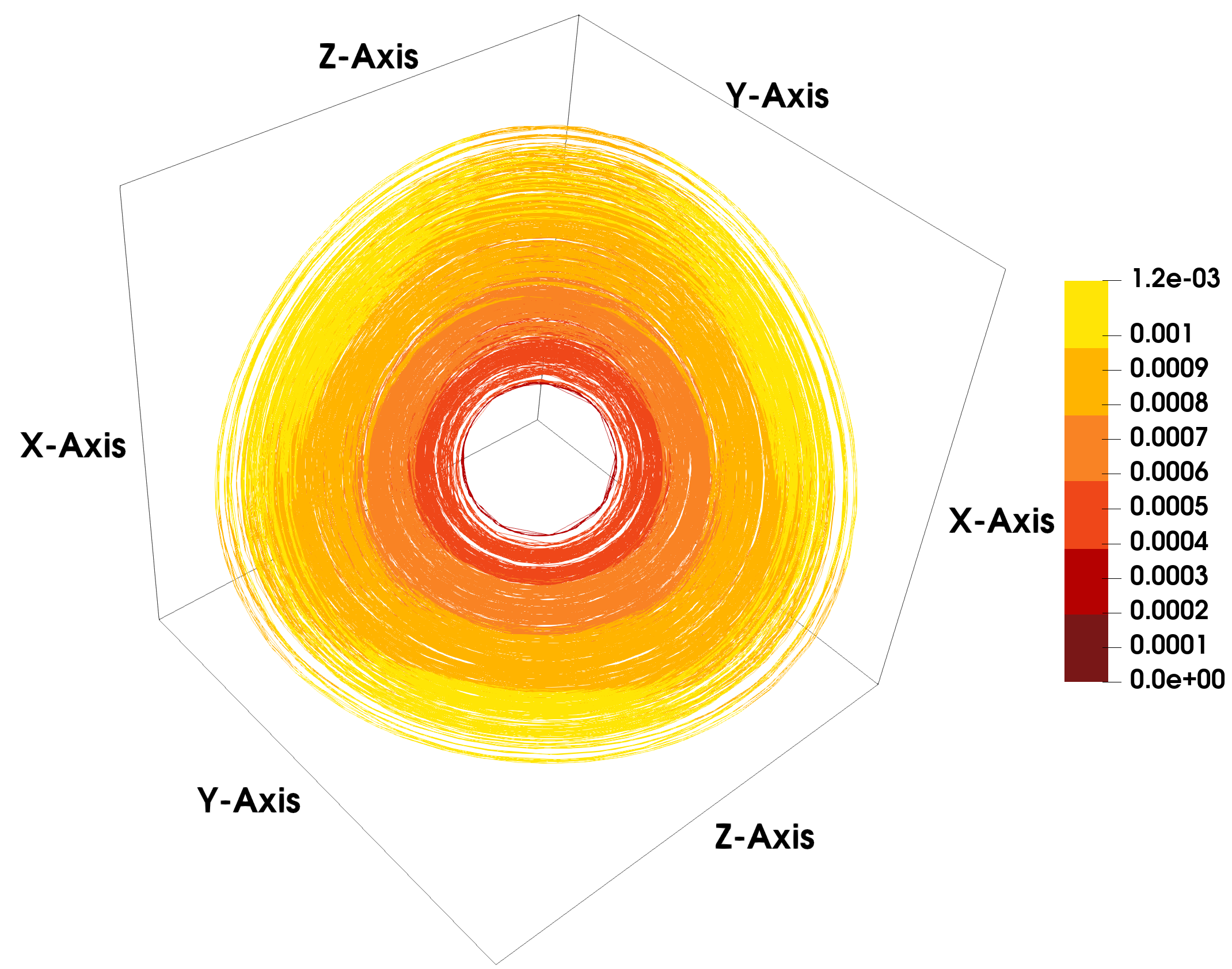}}
	\caption{Approximate solutions showing streamlines of state velocity $(\y_{h})$, vorticity $(\boldsymbol{\kappa}_h)$, pressure $(p_h)$, co-state velocity $(\w_{h})$, vorticity $(\boldsymbol{\vartheta}_h)$,
	and control $(\u_{h})$, respectively, for Example~\ref{Example 5.5.}.} \label{FIGURE 15}
\end{figure}
\begin{remark}
	In the conforming scheme, the velocity converges at the optimal rate, verifying the predictions of Theorem \ref{2.4.1.4}. The approximations for vorticity and pressure demonstrates superconvergence.
\end{remark}
\section{Conclusion}
In the present work, we propose an optimally convergent conforming augmented mixed finite element method and a discontinuous Galerkin (DG) method for the discretization of the velocity-vorticity-pressure formulation of distributed optimal control problems governed by generalized Oseen equations with non-constant viscosity. Some key features of the proposed schemes include the liberty to choose different Stokes inf-sup stable finite element families, direct and accurate access to vorticity (without applying postprocessing), and flexibility in handling Dirichlet boundary conditions for velocity. We establish optimal a priori error estimates and reliable and efficient a posteriori error estimators. Numerical experiments showcase the efficacy of the a posteriori error estimator, validating its performance for both convex and non-convex domains. This study lays the groundwork for exploring similar formulations for tackling challenges in the context of optimal control problems governed by Navier-Stokes equations. The methodologies and insights gained here can be instrumental in advancing the understanding and control of fluid dynamics in more intricate scenarios.\\
\noindent 
\textbf{Declarations:}
The authors declare that they have no conflicts of interest.
\bibliographystyle{siam}
\bibliography{reference}
\end{document}